\documentclass[12pt]{article}
\usepackage[top=1in,right=1in,left=1in,bottom=1in]{geometry}
\newtheorem{theorem}{Theorem}[section]

\newtheorem{lemma}[theorem]{Lemma}
\newtheorem{proposition}[theorem]{Proposition}
\newtheorem{corollary}[theorem]{Corollary}
\newenvironment{proof}[1][Proof]{\begin{trivlist}
\item[\hskip \labelsep {\bfseries #1}]}{\end{trivlist}}
\newenvironment{definition}[1][Definition]{\begin{trivlist}
\item[\hskip \labelsep {\bfseries #1}]}{\end{trivlist}}
\newenvironment{example}[1][Example]{\begin{trivlist}
\item[\hskip \labelsep {\bfseries #1}]}{\end{trivlist}}
\newenvironment{remark}[1][Remark]{\begin{trivlist}
\item[\hskip \labelsep {\bfseries #1}]}{\end{trivlist}}
\newcommand{\qed}{\nobreak \ifvmode \relax \else
      \ifdim\lastskip<1.5em \hskip-\lastskip
      \hskip1.5em plus0em minus0.5em \fi \nobreak
      \vrule height0.75em width0.5em depth0.25em\fi}

\newcommand{\R}[1]{\mathbb{R}^{#1}}

\newcommand{\lbar}{\underline{L}}

\newcommand{\pd}{\partial}

\newcommand{\bp}{\indent$\bullet$\hspace{.1in}}
\newcommand{\sla}[1]{\displaystyle{\not}{#1}}
\newcommand{\ul}[1]{\underline{#1}}

\newcommand{\delm}{\delta_-}
\newcommand{\delp}{\delta_+}
\newcommand{\delh}{\delta_H}
\newcommand{\tg}{\hat{\Gamma}}
\newcommand{\tsla}[1]{\tilde{\sla{#1}}}

\newcommand{\fa}{{\mathfrak{a}}}
\newcommand{\fb}{{\mathfrak{b}}}
\newcommand{\fc}{{\mathfrak{c}}}
\newcommand{\fd}{{\mathfrak{d}}}
\usepackage{amsmath,amssymb,graphicx,wasysym,mathtools}

\begin{document}

\title{Global Stability of the Nontrivial Solutions to the Wave Map Problem from Kerr $|a|\ll M$ to the Hyperbolic Plane under Axisymmetric Perturbations Preserving Angular Momentum}
\author{John Stogin}
\maketitle

\begin{abstract}
This document proves global boundedness and decay for axisymmetric perturbations of a known solution to the wave map problem from a slowly rotating $|a|\ll M$ Kerr spacetime to the hyperbolic plane. This problem is motivated by the general axisymmetric stability of Kerr conjecture and was first posed by Ionescu and Klainerman in \cite{IoKl}. Two particular developments in this paper, the treatment of terms near the axis of symmetry and the use of a decay hierarchy for energy estimates on uniformly spacelike hypersurfaces, can be used for a variety of similar problems.
\end{abstract}

\tableofcontents

\section{Introduction}

A major open problem in mathematical general relativity is to prove that the Kerr black hole solutions are stable as solutions to the Einstein Vacuum Equations. Due to the complicated nature of these equations, this problem is likely to remain open for a while. However, steps are being taken by focusing on simplified model problems. One such model problem, which was originally posed by Ionescu and Klainerman in \cite{IoKl} is the subject of this paper.

\subsection{The model problem of \cite{IoKl}}

To derive the model problem, one starts by restricting the space of solutions to those which have axisymmetry, since the Kerr solutions themselves have axisymmetry. This simplification already has a number of consequences. First, it excludes the challenging effects of the ergoregion, a neighborhood of the black hole where the effects of frame dragging are so strong that null trajectories are forced to spin in the same direction as the black hole. Second, it simplifies the set of trapped null geodesics by confining them to a single radius, $r_{trap}$. And third, it significantly simplifies the equations as we shall now discuss.

An axisymmetric spacetime $(\mathcal{M},\mathbf{g})$ can be described by the restriction $g$ of the metric $\mathbf{g}$ to the quotient spacetime under the axisymmetry, together with a new complex scalar quantity $\sigma$. The Einstein Vacuum Equations take the form
\begin{align}
\Box_g\sigma &= N[g,\sigma]\label{eve_pi_eqn} \\
Ric(g)_{ij} &= N[g,\sigma]_{ij}.\label{eve_ric_eqn}
\end{align}
For a derivation of this reduction, see \cite{weinstein}.

The model problem that is the subject of this paper comes from yet another simplification: The equation for the reduced metric (\ref{eve_ric_eqn}) is ignored, and $g$ is replaced with the reduced Kerr metric $g_{Kerr}$. So only the evolution of the scalar $\sigma$ according to equation (\ref{eve_pi_eqn}) with $g=g_{Kerr}$ is studied. (From now on, $g$ will indeed refer to the reduced metric $g_{Kerr}$.) This has the simplifying advantage that the wave dynamics are fixed.

If
$$\sigma=X+iY,$$
then equation (\ref{eve_pi_eqn}) is given by the following system of equations for $X$ and $Y$.
\begin{align}
\Box_gX &= \frac{\pd^\alpha X\pd_\alpha X}{X}-\frac{\pd^\alpha Y\pd_\alpha Y}{X}\label{wm_X_eqn} \\
\Box_gY &= 2\frac{\pd^\alpha X\pd_\alpha Y}{X}.\label{wm_Y_eqn}
\end{align}
Coincidentally, if $X$ and $Y$ are taken to be the standard coordinates for the hyperbolic plane with ranges $X\in (0,\infty)$ and $Y\in (-\infty,\infty)$, then these equations are precisely the equations that govern wave maps from the Kerr spacetime to the hyperbolic plane. For this reason, the system (\ref{wm_X_eqn}-\ref{wm_Y_eqn}) is often referred to as the \textit{wave map system}.

A particular (nontrivial) solution to this system is given by the scalar $\sigma_0$ corresponding to the Kerr metric itself.
$$\sigma_0=A+iB,$$
\begin{align*}
A &= \frac{(r^2+a^2)^2-a^2\sin^2\theta(r^2-2Mr+a^2)}{r^2+a^2\cos^2\theta}\sin^2\theta \\
B &= -2aM(3\cos\theta-(\cos\theta)^3)-\frac{2a^3M(\sin\theta)^4\cos\theta}{r^2+a^2\cos^2\theta}.
\end{align*}
The purpose of this paper is to investigate the stability of the scalar $\sigma_0=A+iB$ as an axisymmetric solution to the wave map system.

A general fact for any truly vacuum axisymmetric spacetime (ie. a spacetime that solves both equations (\ref{eve_pi_eqn}-\ref{eve_ric_eqn})) is that the imaginary part $\Im(\sigma)=Y$, called the \textit{Ernst potential}, must be constant on each connected segment of the axis. (Segments are disconnected if they lie on opposite ends of a black hole.) Furthermore, the difference between the constant values of $Y$ on opposite sides of a black hole is directly related to the angular momentum of the black hole. (One can easily see this is the case for the Kerr spacetime by evaluating $B$ at $\theta=0$ and $\theta=\pi$.) The space of perturbations of $\sigma_0$ that will be studied in this paper are such that $Y-B$ vanishes on the entire axis. We therefore interpret these perturbations as preserving angular momentum.

At this point, we can meaningfully state the informal version of the main theorem of this paper.
\begin{theorem}(Main Theorem, informal version) For a slowly rotating ($|a|\ll M$) Kerr black hole, there is a general function space (preserving angular momentum) in which the complex scalar $\sigma_0=A+iB$ is indeed stable as a solution to the wave map equations (\ref{wm_X_eqn}-\ref{wm_Y_eqn}).
\end{theorem}
The precise statement of this theorem is given by Theorem \ref{main_thm}.

When introducing this model problem, Ionescu and Klainerman outlined three main difficulties that must be overcome to find a solution.
\begin{enumerate}
\item \textit{Strong linear stability} The main result of \cite{IoKl} was a set of decay estimates for the linearized wave map system. These estimates are reproven in this paper in \S\ref{xi_a_sec}-\S\ref{ee_sec} and are summarized most concisely in Theorem \ref{p_L_thm}. It was perhaps not known at the time  \cite{IoKl} was published that these decay estimates are sufficient to handle the nonlinear problem. See \S\ref{intro_null_condition_sec}.
\item \textit{Nonlinear stability} The nonlinear terms must have some kind of special structure, including the well-known null condition, which is compatible with the linear decay estimates. In addition to this condition, there is a new structural condition on the axis. See \S\ref{intro_regularity_sec} for a brief description and \S\ref{nonlinear_sec} for a detailed examination.
\item \textit{Degeneracy on the axis} There are difficulties associated with the axisymmetric reduction of the equations, which are manifest in terms that appear to be singular on the axis. To overcome these difficulties, a new formalism is presented in \S\ref{regularity_sec}, which is essential in this paper and likely will be useful in future works on problems with axisymmetry. Again, see \S\ref{intro_regularity_sec} for a brief description.
\end{enumerate}

\subsection{The $\xi_a$ system}

There are two useful ways to linearize the wave map system about the nontrivial solution $(A,B)$. One way is to introduce a vector bundle formalism (motivated by the geometric nature of the wave map system) and derive an equation for a section $\xi_a$ of this bundle, $\mathcal{B}$. This approach is useful because it suggests an appropriate grouping of terms when deriving a Morawetz estimate. We now derive this system.

\subsubsection{General theory of wave maps}
Let $\Phi$ be a map.
$$\Phi:(M,g)\rightarrow (N,h).$$
Let $\Phi_*=d\Phi$ be its pushforward.
$$\Phi_*:T_pM\rightarrow T_{\Phi(p)}N$$
The wave map equation says that
\begin{equation}\label{wave_map_equation}
div\Phi_*=0.
\end{equation}

Before proceeding, let us take a moment to outline a few conventions for index notation. We will use greek indices to represent tensor quantities on $M$ and lower-case latin indices to represent tensor quantities on $N$. We emphasize that these shall be used for \textit{coordinate invariant} quantities only. When using a particular set of coordinates, we will use primed indices instead. For example, we can represent the pushforward $d\Phi$ by
$$d\Phi=d\Phi^a_\mu,$$
but if we specify the map $\Phi$ in coordinates $\Phi^{a'}(x^{\mu'})$, then we have the coordinate dependent equation
$$d\Phi^{a'}_{\mu'}=\pd_{\mu'}\Phi^{a'}.$$
As an exception to the primed index rule, indices $i,j,k,l$ will be used to correspond to a particular orthonormal frame for the bundle $\mathcal{B}$, which will be introduced later.

We use $\nabla$ to denote the Levi-Civita connections on both $M$ and $N$. The particular connection being used will be clear from context or use of indices. The pushforward $\Phi_*$ and the Levi-Civita connection on $N$ induce a connection $D$ taking a vector $\vec{V}\in T_pM$ to a differential operator $D_{\vec{V}}$ acting on tensors on $N$ by
$$D_{\vec{V}}:=\nabla_{\Phi_* \vec{V}}.$$

When the context is clear, we will use indices to implicitly push forward contravariant tensors on $M$ or pull back covariant tensors on $N$. Thus, for example, 
$$g^{ab}=d\Phi_\alpha^a d\Phi_{\beta}^bg^{\alpha\beta}$$
and
$$h_{\alpha\beta}=d\Phi_\alpha^a d\Phi_\beta^b h_{ab}.$$
We will never use the inverse of $\Phi$ (which may not exist) to push forward covariant tensors or pull back contravariant tensors. This allows for raising and lowering indices without ambiguity. For example, 
$$h^{\alpha\beta}=g^{\alpha\gamma}g^{\beta\delta}h_{\gamma\delta}=g^{\alpha\gamma}g^{\beta\delta}d\Phi_\gamma^c d\Phi_\delta^d h_{cd}.$$
The one drawback of this convention is that indices no longer clarify whether tensors naturally belong to $M$ or $N$. (For example, is $R_{\alpha\beta\gamma\delta}$ the Riemann curvature tensor for $M$ or the pullback of the Riemann curvature tensor for $N$?) This ambiguity must be resolved explicitly when the tensor is introduced, but is not a serious issue in practice.

Using the above index notation, one can directly calculate the following coordinate-dependent equation.
$$D_{\mu'} d\Phi^{a'}_{\nu'}=\pd_{\mu'}\pd_{\nu'}\Phi^{a'}-\Gamma_{\mu'\nu'}^{\lambda'}\pd_{\lambda'}\Phi^{a'}+\pd_{\mu'}\Phi^{b'}\pd_{\nu'}\Phi^{c'}\Gamma_{b'c'}^{a'}.$$
In particluar, the wave map equation (\ref{wave_map_equation}) in coordinates takes the form
\begin{equation}\label{general_wave_eqn_in_coordinates_eqn}
\Box_g\Phi^{a'}+g^{\mu'\nu'}\pd_{\mu'}\Phi^{b'}\pd_{\nu'}\Phi^{c'}\Gamma_{b'c'}^{a'}=0.
\end{equation}

\subsubsection{Linearized wave maps and the section $\xi_a$}

To examine the linear stability of a solution $\Phi$ to a wave map, one obtains an equation for a vectorfield 
$$\vec{\psi}:M\rightarrow TN,\hspace{.5in}\vec{\psi}(p)\in T_{\Phi(p)}N.$$
The significance of this vectorfield is that if $\Phi(s)$ is a parametrized family of solutions to the wave map equation with $\Phi(0)=\Phi$, then $\vec{\psi}=\left.\frac{d}{ds}\right|_{s=0}\Phi(s)$.

The equation for $\vec{\psi}$ is
\begin{equation}\label{vec_psi_linear_stability_eqn}
\Box_g\psi^a+R^{\lambda a}{}_{\lambda b}\psi^b=0,
\end{equation}
where $R$ is the curvature tensor of the target manifold
$$R^{\lambda a}{}_{\lambda b}=g^{\gamma\delta}d\Phi_\gamma^c d\Phi_\delta^d R_c{}^a{}_{db}.$$

By comparing the model system (\ref{wm_X_eqn}-\ref{wm_Y_eqn}) to the general wave map equation (\ref{general_wave_eqn_in_coordinates_eqn}) one can read off the Christoffel symbols $\Gamma^{a'}_{b'c'}$ and determine that the target manifold $(N,h)$ for the model problem is the hyperbolic plane--with negative constant curvature. Letting $\epsilon_{ab}$ be the volume form for the hyperbolic plane, we have
$$R_{abcd}=-\epsilon_{ab}\epsilon_{cd}.$$
Contracting equation (\ref{vec_psi_linear_stability_eqn}) with $\epsilon_{ab}$, and using the facts that $D_\mu \epsilon_{ab}=0$ and $\epsilon_{ac}\epsilon^c{}_b=h_{ab}$, we obtain a new equation.
$$\Box_g(\epsilon_{ab}\psi^b)-g_a{}^b\epsilon_{bc}\psi^c=0.$$
Therefore, we introduce the new dynamic quantity
$$\xi_a:=\epsilon_{ab}\psi^b$$
and the potential
$$V^{ab}=g^{ab}=g^{\alpha\beta}d\Phi_\alpha^a d\Phi_\beta^b,$$
and the equation (\ref{vec_psi_linear_stability_eqn}) becomes
\begin{equation}\label{xi_a_linearized_eqn}
\Box_g\xi_a-V_a{}^b\xi_b=0.
\end{equation}
We refer to the vector bundle for this equation as $\mathcal{B}$.

\subsubsection{The equations for $\xi_a$ in component form}

An orthonormal frame for $(N,h)$ is given by
$$e_1=X\pd_X,\hspace{.5in} e_2=X\pd_Y.$$
We can represent the section $\xi_a$ in terms of the dual frame $\{e^1,e^2\}$.
$$\xi_a=\xi_i(e^i)_a=\xi_1(e^1)_a+\xi_2(e^2)_a.$$
The functions $\xi_1$ and $\xi_2$ are scalar quantities, which will be the object of study in \S\ref{xi_a_sec}. In particular, these quantities linearize the wave map system (\ref{wm_X_eqn}-\ref{wm_Y_eqn}) in the following way.
\begin{align*}
X &= A-A\xi_2 \\
Y &= B+A\xi_1.
\end{align*}

For future reference, the pushforward $d\Phi$ is given by
\begin{align*}
d\Phi_\alpha^a = \frac{\pd_\alpha A}{A} (e_1)^a+\frac{\pd_\alpha B}{A} (e_2)^a,
\end{align*}
and the Christoffel symbols for the dual frame can be read from the following relations
\begin{align*}
D_\alpha e^1 &= -\frac{\pd_\alpha B}{A} e^2 \\
D_\alpha e^2 &= \frac{\pd_\alpha B}{A} e^1.
\end{align*}
Equation (\ref{xi_a_linearized_eqn}) in component form is given by the following system of equations for the scalar components $\xi_1$ and $\xi_2$.
\begin{align*}
\Box_g\xi_1 &= - 2\frac{\pd^\alpha B}{A}\pd_\alpha\xi_2+\frac{\pd^\alpha A\pd_\alpha A+\pd^\alpha B\pd_\alpha B}{A^2}\xi_1 ,\\
\Box_g\xi_2 &= 2\frac{\pd^\alpha B}{A}\pd_\alpha\xi_1 +2\frac{\pd^\alpha B\pd_\alpha B}{A^2}\xi_2 + 2\frac{\pd^\alpha A\pd_\alpha B}{A^2}\xi_1.
\end{align*}
In the Schwarzschild case, where $A=r^2\sin^2\theta$ and $B=0$, these equations reduce to
\begin{align*}
\Box_g\xi_1 &= \frac{4}{r^2}\left(1-\frac{2M}r\right)\xi_1+\frac{4\cot^2\theta}{r^2}\xi_1  \\
\Box_g\xi_2 &= 0. 
\end{align*}

\subsection{The $(\phi,\psi)$ system}\label{intro_phi_psi_system_sec}

Although the $\xi_a$ linearization naturally arises from geometric principles, it is somewhat unusual in that it has a nontrivial potential that is singular on the axis. It turns out that by replacing $\xi_1=A\psi$, the equation
$$\Box_g\xi_1 = \frac{4}{r^2}\left(1-\frac{2M}r\right)\xi_1+\frac{4\cot^2\theta}{r^2}\xi_1$$
 for $\xi_1$ in the Schwarzschild case can be transformed into a simple wave equation for $\psi$ on a modified Schwarzschild spacetime.
$$\Box_{\tilde{g}}\psi = 0.$$
Letting $\phi=-\xi_2$ so that this new linearization reads
\begin{align*}
X &= A+A\phi \\
Y &= B+A^2\psi,
\end{align*}
then in the Schwarzschild case, the linear system reduces to the following simple system.
\begin{align*}
\Box_g\phi &= 0 \\
\Box_{\tilde{g}}\psi &= 0.
\end{align*}

Perhaps more important than simplifying the equations, this new linearization captures the essential behavior of the linear system on the axis. As we shall see in the main theorem, the quantity $\psi$ will be regular on the axis, which implies that the quantity $\xi_1=A\psi$ will vanish to second order on the axis.

\subsubsection{The modified Schwarzschild spacetime $(\tilde{\mathcal{M}},\tilde{g})_{a=0}$}

It was just stated that the equation for $\psi$ naturally belongs to a modified Schwarzschild spacetime. We take a moment to explan this fact further.

Let $(\mathcal{M},g)$ denote the Schwarzschild spacetime. In the usual Boyer-Lindquist coordinate system, the metric is
$$g_{\alpha\beta}dx^\alpha dx^\beta=-\left(1-\frac{2M}r\right)dt^2+\left(1-\frac{2M}r\right)^{-1}dr^2+r^2d\omega_{S^2}^2,$$
and its volume form is
$$\sqrt{-\det g}=r^2\sin\theta.$$

By making the substitution $\xi_1=A\psi$, the equation for $\psi$ becomes
$$\Box_g\psi +2\frac{\pd^\alpha A}{A}\pd_\alpha \psi = 0.$$
For the class of axisymmetric functions considered in this paper, the linear operator in this equation is a wave operator for a different ($7+1$ dimensional) spacetime $(\tilde{\mathcal{M}},\tilde{g})$, whose metric is given by\footnote{The difference between the two metrics $g$ and $\tilde{g}$ is very subtle. For $g$, the last term is $r^2d\omega^2_{S^2}$ and for $\tilde{g}$ the last term is $r^2d\omega^2_{S^6}$.}
$$\tilde{g}_{\alpha\beta}dx^\alpha dx^\beta=-\left(1-\frac{2M}r\right)dt^2+\left(1-\frac{2M}r\right)^{-1}dr^2+r^2d\omega_{S^6}^2,$$
and whose volume form is effectively
$$\sqrt{-\det\tilde{g}}=r^6\sin^5\theta.$$

Let us take a moment to derive this fact. The sphere $S^6$ can be given coordinates $\theta_1,...,\theta_5,\phi$, where $\theta_i\in [0,\pi]$ and $\phi\in [0,2\pi]$. An axisymmetric function $f:\mathcal{M}\rightarrow\R{}$ can be written as $f(t,r,\theta)$. We consider only functions $f:\tilde{\mathcal{M}}\rightarrow\R{}$ of the form $f(t,r,\theta_1)$. We therefore identify $\theta_1$ on $\tilde{\mathcal{M}}$ with $\theta$ on $\mathcal{M}$. In the calculations to follow, $\theta_1$ and $\theta$ will be used interchangeably. The metric for $S^6$ in these coordinates is
$$d\omega^2_{S^6}=d\theta_1^2+\sin^2\theta_1(d\theta_2^2+\sin^2\theta_2(...(d\theta_5^2+\sin^2\theta_5d\phi^2)...)).$$
It follows that
$$\sqrt{-\det\tilde{g}}=r^6\sin^5\theta_1\sin^4\theta_2\sin^3\theta_3\sin^2\theta_4\sin\theta_5.$$

Assume that $\psi:\tilde{\mathcal{M}}\rightarrow\R{}$ satisfies
$$\pd_{\theta_2}\psi=...=\pd_{\theta_5}\psi=\pd_\phi\psi=0.$$
Then since
$$g^{tt}=\tilde{g}^{tt}\text{, }g^{rr}=\tilde{g}^{rr}\text{, and }g^{\theta\theta}=\tilde{g}^{\theta_1\theta_1},$$
it follows that
$$g^{\alpha\beta}\pd_\beta\psi=\tilde{g}^{\alpha\beta}\pd_\beta\psi.$$
Therefore,
\begin{align*}
\Box_{\tilde{g}}\psi &= \frac{1}{\sqrt{-\det\tilde{g}}}\pd_\alpha\left(\sqrt{-\det\tilde{g}}\tilde{g}^{\alpha\beta}\pd_\beta\psi\right) \\
&= \frac{1}{\sqrt{-\det\tilde{g}}}\pd_\alpha\left(\sqrt{-\det\tilde{g}}g^{\alpha\beta}\pd_\beta\psi\right) \\
&= \frac{1}{r^6\sin^5\theta_1}\pd_\beta\left(r^6\sin^5\theta_1g^{\alpha\beta}\pd_\beta\psi\right) \\
&= \frac{1}{r^2\sin\theta}\pd_\alpha\left(r^2\sin\theta g^{\alpha\beta}\pd_\beta\psi\right)+\frac{1}{r^4\sin^4\theta}\pd^\alpha(r^4\sin^4\theta)\pd_\alpha\psi \\
&= \frac{1}{\sqrt{-det g}}\pd_\alpha\left(\sqrt{-\det g}g^{\alpha\beta}\pd_\beta\psi\right)+\frac{2}{r^2\sin^2\theta}\pd^\alpha(r^2\sin^2\theta)\pd_\alpha\psi \\
&= \Box_g\psi +2\frac{\pd^\alpha A}{A}\pd_\alpha\psi.
\end{align*}

\subsubsection{The modified Kerr spacetime  $(\tilde{\mathcal{M}},\tilde{g})$}

The same reasoning also applies to the Kerr case, but the operator
$$\Box_g+\frac{\pd^\alpha A}{A}\pd_\alpha$$
has fewer commutators, since the function
$$A=\frac{(r^2+a^2)^2-a^2\sin^2\theta(r^2-2Mr+a^2)}{r^2+a^2\cos^2\theta}\sin^2\theta$$
depends nontrivially on $r$ and $\theta$. For this reason, we write
\begin{align*}
A &= A_1A_2 \\
A_1 &= (r^2+a^2)\sin^2\theta \\
A_2 &= \left(1+\frac{a^2\sin^2\theta}{q^2}\right)\left(1-a^2\sin^2\theta v\right) \\
v &= \frac{r^2-2Mr+a^2}{(r^2+a^2)^2},
\end{align*}
and we generalize the spacetime $\tilde{g}$ by replacing $A$ with $A_1$. (Note that $A_1$ reduces to $A$ in the Schwarzschild case--generally speaking, $A_1$ behaves very similarly to $A$ in the Schwarzschild case, while terms depending on $A_2$ are treated as error terms in very much the same way that terms depending on $B$ are treated.)

That is, we extend the calculation for Schwarzschild by defining
$$\Box_{\tilde{g}} := \Box_g\psi +\frac{\pd^\alpha A_1}{A_1}\pd_\alpha$$
and
$$\int_{\tilde{\Sigma}_t}f := \int_{\Sigma_t}f A_1^2 =\int_{r_H}^\infty\int_0^\pi f A_1^2 q^2\sin\theta d\theta dr.$$

\subsubsection{The equations for $(\phi,\psi)$}

The linearized wave map system is given by the equations
\begin{align*}
\Box_g\phi &= \mathcal{L}_\phi \\
\Box_{\tilde{g}}\psi &= \mathcal{L}_\psi,
\end{align*}
where
$$\mathcal{L}_\phi=-2\frac{\pd^\alpha B}{A}A\pd_\alpha \psi + 2\frac{\pd^\alpha B\pd_\alpha B}{A^2}\phi-4\frac{\pd^\alpha A\pd_\alpha B}{A^2} A\psi,$$
$$\mathcal{L}_\psi=-2\frac{\pd^\alpha A_2}{A_2}\pd_\alpha\psi+2\frac{\pd^\alpha B\pd_\alpha B}{A^2}\psi + 2A^{-1}\frac{\pd^\alpha B}{A}\pd_\alpha\phi,$$
and again, the modified wave operator $\Box_{\tilde{g}}$ appearing in the equation for $\psi$ is defined by
$$\Box_{\tilde{g}}:=\Box_g+2\frac{\pd^\alpha A_1}{A_1}\pd_\alpha.$$

The fully nonlinear system is given by
\begin{align}
\Box_g\phi &= \mathcal{L}_\phi+\mathcal{N}_\phi, \label{phi_nonlinear_system_eqn}\\
\Box_{\tilde{g}}\psi &= \mathcal{L}_\psi+\mathcal{N}_\psi, \label{psi_nonlinear_system_eqn}
\end{align}
where the nonlinear terms are
\begin{align*}
(1+\phi)\mathcal{N}_\phi &= \pd^\alpha \phi \pd_\alpha \phi - A\pd^\alpha \psi A\pd_\alpha \psi +2\frac{\pd^\alpha B}{A} \phi A\pd_\alpha \psi -4\frac{\pd^\alpha A}{A}A\psi A\pd_\alpha \psi \\
&\hspace{.5in}-\frac{\pd^\alpha B\pd_\alpha B}{A^2}\phi^2+4\frac{\pd^\alpha A\pd_\alpha B}{A^2}\phi A\psi -4\frac{\pd^\alpha A\pd_\alpha A}{A^2}(A\psi)^2,
\end{align*}
$$(1+\phi)\mathcal{N}_\psi = 2\pd^\alpha \phi \pd_\alpha\psi  +4\frac{\pd^\alpha A}{A}\psi\pd_\alpha \phi -2\frac{\pd^\alpha B}{A}A^{-1}\phi\pd_\alpha\phi.$$
This system will be further studied starting in \S\ref{phi_psi_sec} and throughout the remainder of the paper.

\subsubsection{Conventions for spacetime norms}

Since there are effectively two spacetimes, we briefly lay out a few conventions for the remainder of the paper. \\
\bp All functions will be assumed to depend only on $t$, $r$, and $\theta$. Therefore, any quantity can be treated as a function defined on either spacetime. \\
\bp When necessary, the tilde mark ( $\tilde{ }$ ) will be used to denote quantities corresponding to $(\tilde{\mathcal{M}},\tilde{g})$. This includes the effective volume form 
$$\tilde{\mu}=A_1^2q^2\sin\theta=q^2(r^2+a^2)^2\sin^5\theta,$$
and the constant-time hypersurface 
$$\tilde{\Sigma}_t=\{t\}\times [r_H,\infty)\times S^6.$$
\bp Integrated expressions will depend on a volume form that is implicitly defined by the manifold of integration. That is,
$$\int_{\Sigma_t}f:=\int_{r_H}^\infty\int_0^\pi\int_0^{2\pi}f(t,r,\theta)q^2\sin\theta d\phi d\theta dr.$$
and
\begin{multline*}
\int_{\tilde{\Sigma}_t}f:=\int_{r_H}^\infty\int_0^\pi\int_0^\pi\int_0^\pi\int_0^\pi\int_0^\pi\int_0^{2\pi}f(t,r,\theta)q^2(r^2+a^2)^2\sin^5\theta\sin^4\theta_2\sin^3\theta_3\sin^2\theta_4\sin\theta_5 \\
d\phi d\theta_5 d\theta_4 d\theta_3 d\theta_2 d\theta dr.
\end{multline*}
Since all relevant integral estimates in this paper are valid up to a constant, one may equivalently define
$$\int_{\Sigma_t}f:=\int_{r_H}^\infty\int_0^\pi f(t,r,\theta)q^2\sin\theta d\theta dr$$
and
$$\int_{\tilde{\Sigma}_t}f:=\int_{r_H}^\infty\int_0^\pi f(t,r,\theta)q^2(r^2+a^2)^2\sin^5\theta d\theta dr.$$
The point is that the only difference occurs in the factors that show up in the volume form. \\
\bp We also observe that $L^\infty$ estimates are weaker in the higher-dimensional spacetime. That is,
$$||f||_{L^\infty(S^2)}\lesssim \sum_{i\le 2}||\sla\nabla^if||_{L^2(S^2)},$$
while
$$||f||_{L^\infty(S^6)}\lesssim \sum_{i\le 4}||\tilde{\sla\nabla}^if||_{L^2(S^6)}.$$

\subsection{Regularity on the axis and a new structural condition}\label{intro_regularity_sec}

As previously discussed, the axis of symmetry presents additional challenges, which are manifest in the presence of apparently singular terms, ie. terms which have factors of $\sec\theta$ or $\csc\theta$. It is not a priori clear whether these challenges are related to the true dynamics of the problem or whether they are due to the fact that the coordinate system is degenerate on the axis. As we shall see, in the $(\phi,\psi)$ picture these terms are merely due to the degeneracy of the coordinate system. To handle these terms, we use a new formalism which is developed in detail in \S\ref{regularity_sec}, but a brief summary is given here.

Consider the two functions $\cos\theta$ and $\sin\theta$. While these are both smooth functions of $\theta$, if they are treated as functions on the spacetime, one of them is actually much less regular than the other. The problem is that the function $\sin\theta$ behaves like $|\theta|$ in a neighborhood of the half axis $\theta=0$, which means that $\sin\theta$ is not twice differentiable. In contrast, the function $\cos\theta$ is everywhere smooth. Informally, we will say that $\cos\theta$ belongs to a space of functions that are regular on the axis, but $\sin\theta$ does not.

The fact that $\sin\theta$ is not regular on the axis is not clear when applying two coordinate derivatives, because $\pd_\theta^2\sin\theta=-\sin\theta$ appears to be bounded. In order to measure the singular nature of $\sin\theta$ on the axis, it is necessary to use a second operator $\cot\theta\pd_\theta$, which by no coincidence appears in the spherical laplacian. Since $\cot\theta\pd_\theta\sin\theta=\cos^2\theta\csc\theta$, it is now clear that something goes wrong on the axis. Since the two operators $\pd_\theta^2$ and $\cot\theta\pd_\theta$ will often be used, they are given names.
\begin{align*}
\fa &:= \pd_\theta^2 \\
\fb &:= \cot\theta\pd_\theta.
\end{align*}

The operator $\fb$ itself seems to be in some sense singular on the axis, because it has a factor of $\csc\theta$. However, for any twice-differentiable axisymmetric function $f$, the first derivative $\pd_\theta f$ should vanish at least to first order on the axis. So for twice-differentiable axisymmetric functions, there is a cancellation effect with the factor $\csc\theta$. Essentially, \textit{the operators $\fa$ and $\fb$ preserve the space of regular functions on the axis.}

In the nonlinear problem, we will commute the fully nonlinear equations (\ref{phi_nonlinear_system_eqn}-\ref{psi_nonlinear_system_eqn}) with the following angular operators.
\begin{align*}
Q &= \fa+\fb+a^2\sin^2\theta\pd_t^2 \\
\tilde{Q} &= \fa+5\fb+a^2\sin^2\theta\pd_t^2.
\end{align*}
Let us ignore for a moment the $a^2\sin^2\theta\pd_t^2$ part of these operators. We will need to estimate the terms belonging to
$$(\fa+\fb)^l\mathcal{N}_\phi\hspace{.5in}\text{and}\hspace{.5in}(\fa+5\fb)^l\mathcal{N}_\psi.$$
We introduce the operator family $\fc^l$ to represent any term in the expansion of $(\fa+\fb)^l$. So for example,
$$\fc^2=\{\fa^2,\fa\fb,\fb\fa,\fb^2\}.$$
We will estimate $\fc^l(\mathcal{N}_\phi)$ and $\fc^l(\mathcal{N}_\psi)$, but must do so carefully to ensure that we stay in the space of regular functions on the axis.

There are a few ways things could go wrong if we are not careful. First, if at any point we expand an operator such as $\fa\fb$ or $\fb^2$, then we get terms that are truly singular on the axis. As an example, consider the following calculation.
$$\fb^2f =\cot\theta\pd_\theta(\cot\theta\pd_\theta f) =\cot^2\theta\pd_\theta^2f-\cot\theta\csc^2\theta\pd_\theta f.$$
If $f$ is regular on the axis, then $\pd_\theta f$ will vanish at least to first order on the axis, but that is not enough to ensure that each of the two terms on the right side remain regular--in general they do not. For this reason, it is necessary to treat the operator $\fb$ as an atomic operator.

The second thing that can go wrong happens when applying $\fa$ or $\fb$ to products of functions. As an example, consider the following calculation.
$$\fa(fg)=\pd_\theta^2(fg)=\fa f g+2\pd_\theta f \pd_\theta g+f\fa g.$$
If $f$ and $g$ are regular on the axis, then so are $\fa f$ and $\fa g$, but the factors $\pd_\theta f$ and $\pd_\theta g$ are not--for the same reason that $\cos\theta$ is regular on the axis, but $-\sin\theta=\pd_\theta\cos\theta$ is not. Here, it is important to observe that \textit{the product $\pd_\theta f\pd_\theta g$ is indeed regular on the axis}, because it is a product of two functions behaving like $|\theta|$, which will behave like $\theta^2$. Without being careful, it is possible to expand $\fb^2\fa(fg)$ to get a term of the form $\fb^2(\pd_\theta f)\pd_\theta g$. This term would not be regular on the axis. The correct way to handle such an expansion is illustrated by the following intermediate calculation.
$$\fb(\pd_\theta f\pd_\theta g)=\cot\theta \pd_\theta^2f\pd_\theta g+\cot\theta\pd_\theta f\pd_\theta^2 g = \fa f \fb g+\fb f\fa g.$$
Then one can apply any of the additional $\fc$ operators.

To ensure that products are regular on the axis (especially after applying $\fa$ or $\fb$), the most general form of products that can be permitted in the nonlinear term must be something like 
$$\pd_\theta \fc^{l_1}f_1\pd_\theta \fc^{l_2}f_2...\pd_\theta \fc^{l_{2k}}f_{2k}\fc^{l_{2k+1}}f_{2k+1}...\fc^{l_{2k+k'}}f_{2k+k'}.$$
That is, \textit{it is essential that an even number of single $\pd_\theta$ derivatives show up and that each $\pd_\theta$ be applied only after the $\fc$ operators are applied.} To write this more compactly, we define yet another family of operators.
$$
\fd^l := \left\{\begin{array}{cc}
\fc^{l/2} & l\in 2\mathbb{Z} \\
\pd_\theta \fc^{(l-1)/2} & l\not\in 2\mathbb{Z}.
\end{array}\right.
$$
If $f_1,...,f_k$ are each regular functions, then regular product terms will be of the form
$$\fd^{i_1}f_1...\fd^{i_k}f_k$$
where $i_1+...+i_k=2n$. We call these products \textit{terms of degree $n$}. An important fact is that if $\fa$ or $\fb$ is applied to a term of degree $n$, the result can be expressed as a sum of terms of degree $n+1$. \textbf{A new important structural condition for the nonlinear terms arising in the wave map problem is that they be of this form.}

\subsection{The null condition and decay of $p$-weighted energy norms}\label{intro_null_condition_sec}

As previously mentioned, the challenge of establishing sufficient decay for the nonlinear problem is not overcome by proving more decay than was proved in \cite{IoKl}, but instead to observe that the decay estimates in \cite{IoKl} are actually sufficient. What was perhaps overlooked is a weaker form of decay implied by these estimates that is still strong enough to estimate the nonlinear error terms. We outline the argument here, noting that \textbf{it is rather general and can be applied to a wide range of semilinear wave problems.}

\subsubsection{The $p$-weighted energy estimates}

The starting point for the decay argument is a family of spacetime estimates roughly of the form
$$E_p(t_2)+\int_{t_1}^{t_2}B_p(t)dt \lesssim E_p(t_1)+\int_{t_1}^{t_2}N_p(t)dt,$$
for values of $p$ ranging almost from $0$ to $2$. (See Theorem \ref{p_thm}.) These estimates and their higher order analogues are developed in \S\ref{ee_sec}.

For a typical wave problem, the weighted energy $E_p(t)$ is given by
$$E_p(t)=\int_{\Sigma_t}r^p\left[(L\phi)^2+|\sla\nabla\phi|^2+r^{-2}\phi^2+r^{-2}(\lbar\phi)^2\right].$$
The $r^{-2}$ weight for the $(\lbar\phi)^2$ term is necessary to ensure decay, because the term $(\lbar\phi)^2$ otherwise tends to dominate in late times and for large values of $r$. 

For the particular problem studied in this paper, the weighted energy is actually given by
\begin{multline*}
E_p(t)= \int_{\Sigma_t}r^p\left[(L\phi)^2+|\sla\nabla\phi|^2+r^{-2}\phi^2+r^{-2}(\lbar\phi)^2\right] \\
+\int_{\tilde{\Sigma}_t}r^p\left[(L\psi)^2+|\tsla\nabla\psi|^2+r^{-2}\psi^2+r^{-2}(\lbar\psi)^2\right].
\end{multline*}
Note that the second integral is over $\tilde{\Sigma}_t$, implying that there is an additional weight of $r^4\sin^4\theta$. This will generally be the case for integrals of quantities depending only on $\psi$, and has to do with the fact that the equation for $\psi$ is a wave equation naturally belonging to the modified spacetime $(\tilde{\mathcal{M}},\tilde{g})$.

The bulk quantity $B_p(t)$ is similar in weight to the weighted energy $E_{p-1}(t)$, except that it also has a degeneracy on the photon sphere.
\begin{multline*}
B_p(t)= \int_{\Sigma_t}r^{p-1}\left[\chi_{trap}(L\phi)^2+\chi_{trap}|\sla\nabla\phi|^2+r^{-2}\phi^2+r^{-2}(\pd_r\phi)^2\right] \\
+\int_{\tilde{\Sigma}_t}r^{p-1}\left[\chi_{trap}(L\psi)^2+\chi_{trap}|\tsla\nabla\psi|^2+r^{-2}\psi^2+r^{-2}(\pd_r\psi)^2\right].
\end{multline*}
The function $\chi_{trap}$ vanishes to second order at the trapping radius $r_{trap}$, which defines the photon sphere. $r_{trap}$ is the radius that maximizes the geodesic potential 
$$v=\frac{r^2+2Mr+a^2}{(r^2+a^2)^2}$$
and corresponds to the radius at which null geodesics with zero angular momentum orbit the black hole at a constant radius. It coincides with $3M$ in the Schwarzschild case.

The nonlinear quantity $N_p(t)$ is an error term that we will generally ignore in the summary that follows.

\subsubsection{Time decay}

Time decay is derived from the fact that the $B_p(t)$ norm has the same weight as the $E_{p-1}(t)$ norm. So, \textbf{if we ignore the issue of trapping}, and if we also ignore the nonlinear part, then
$$\int_{t_1}^{t_2}E_{p-1}(t)dt\lesssim E_p(t_1).$$
This roughly suggests that
$$E_{p-1}(t)\sim T^{-1}E_p(t),$$
where $T=1+t$ so that $T^{-1}(t=0)$ remains bounded. Since $p$ almost ranges from $0$ to $2$, we roughly expect that $E_2(t)$ is bounded, $E_1(t)$ behaves like $T^{-1}$, and $E_0(t)$ behaves like $T^{-2}$. In reality, we have the estimates in the range $p\in [\delm,2-\delp]$ for arbitrarily small positive constants $\delm$ and $\delp$, so the best outright decay we can conclude is 
$$E_{\delm}(t)\lesssim T^{\delm-2+\delp}.$$

\subsubsection{Weak time decay}

Although the best energy decay we can conclude is
$$E_{\delm}(t)\lesssim T^{\delm-2+\delp},$$
it turns out that there is a weaker notion of decay that allows us to estimate $E_{\delm-1}(t)$ using the $p=\delm$ estimate. That is,
$$\int_{t_1}^{t_2} E_{\delm-1}(t)dt \lesssim E_{\delm}(t_1)\lesssim T^{\delm-2+\delp}.$$
This would suggest that $E_{\delm-1}(t)$ behaves like $T^{\delm-3+\delp}$, but unfortunately, it is only possible to prove the integrated version of this estimate. So $E_{\delm-1}(t)$ decays like $T^{\delm-3+\delm}$ in the following sense.
\begin{definition} (weak decay) A function $f$ decays like $T^{-p}$ weakly if
$$\int_t^\infty f(\tau)d\tau \lesssim T^{-p+1}.$$
\end{definition}

\subsubsection{The pointwise estimates compatible with $p$-weighted energy estimates}

There are pointwise estimates that pair with the energy estimates. They are proved in \S\ref{pointwise_sec}, but a brief summary is given here.

 Let $\phi^s$ and $\psi^s$ denote quantities obtained by applying up to $s$ commutators to either $\phi$ or $\psi$, and let $E_p^s(t)$ be the generalized $s$-order energy corresponding to $\phi^s$ and $\psi^s$. One particular estimate roughly states
$$|rL\phi^s|+|r\sla\nabla\phi^s|+|\phi^s|+|\lbar\phi^s|+|r^3L\phi^s|+|r^3\sla\nabla\phi^s|+|r^2\phi^s|+|r^2\lbar\phi^s| \lesssim (E_0^{s+5}(t))^{1/2}.$$
Note that the $L$ and $\sla\nabla$ derivatives come with an additional $r$ weight compared to the $\lbar$ derivative. This is particularly relevant for the discussion on the null condition, which will follow momentarily. Note also that the $\psi$ quantities have two additional factors of $r$ compared to the $\phi$ quantities. This is due to the additional $r$ weights in the integral over $\tilde{\Sigma}_t$ when compared to $\Sigma_t$.

A more general estimate is roughly the following.
\begin{multline*}|r^{p+1}L\phi^s|+|r^{p+1}\sla\nabla\phi^s|+|r^p\phi^s|+|r^p\lbar\phi^s|+|r^{p+3}L\phi^s|+|r^{p+3}\sla\nabla\phi^s|+|r^{p+2}\phi^s|+|r^{p+2}\lbar\phi^s| \\
\lesssim (E_{2p}(t))^{1/2}\lesssim T^{(p-2+\delp+\delm)/2}.
\end{multline*}
The fact that multiplication by $r^p$ corresponds to multiplication by $T^p$ is consistent with the principle that the estimated quantities are mostly supported in the wave zone where $r\approx t$.

\subsubsection{The null condition}

The null condition is a well-known condition for the nonlinear part of a wave equation. Roughly speaking, it requires that the nonlinear part exclude terms of the form
$$\lbar\phi\lbar \phi.$$
It is built into the structure of the nonlinear term specified in \S\ref{nl_structure_sec}.

In the proof of the nonlinear problem, one of the factors must be estimated according to the $L^\infty$ estimates and translated into decay in time. The resulting time decay must be sufficiently strong so that it is integrable in time.

The null condition guarantees that the appropriate $r^p$ weight accompanies each factor. But there is still a problem. From the prevous pointwise estimate, we see that the decay is $T^{(p-2+\delp+\delm)/2}$ and we have the requirement that $p\in [\delm,2-\delp]$. So the best decay we can obtain is $T^{(-2+\delp)/2}=T^{-1+\delp/2}$, which is not integrable in time. However, if we use the weak notion of decay, this allows us to use $p=\delm-1$ and obtain decay like $T^{(-3+\delp)/2}=T^{-3/2+\delp/2}$, which is indeed integrable in time (if $\delp$ is sufficiently small). Fortunately, this weak decay is sufficient for the proof of the main theorem.

\subsection{Outline}

A brief outline of this paper is as follows.

In \S\ref{xi_a_sec}, estimates are established for the $\xi_a$ system. In particular, an energy estimate (Proposition \ref{xi_energy_estimate_prop}) the more general $h\pd_t$ estimate (Proposition \ref{xi_h_dt_prop}) and a Morawetz estimate (Proposition \ref{xi_morawetz_estimate_prop}) are proved for use in the following section.

 In \S\ref{phi_psi_sec}, the estimates for the $\xi_a$ system are translated into estimates for the $(\phi,\psi)$ system and additional $p$-type estimates near $i^0$ are established. More specifically, Proposition \ref{xi_energy_estimate_prop} (the energy estimate) is translated into Proposition \ref{translated_energy_estimate_prop}, Proposition \ref{xi_h_dt_prop} (the $h\pd_t$ estimate) is translated into Proposition \ref{translated_h_dt_prop}, and Proposition \ref{xi_morawetz_estimate_prop} (the Morawetz estimate) is translated into Proposition \ref{translated_morawetz_prop}. Then the incomplete $p$-estimates near $i^0$ (Proposition \ref{incomplete_p_estimates_prop}) are proved. Finally, Propositions \ref{translated_h_dt_prop}, \ref{translated_morawetz_prop}, and \ref{incomplete_p_estimates_prop} are combined to prove the $p$-weighted energy estimates (Proposition \ref{p_estimates_prop}).

In the beginning of \S\ref{ee_sec}, the energy estimate (Proposition \ref{translated_energy_estimate_prop}) and the $p$-weighted energy estimates (Proposition \ref{p_estimates_prop}) are combined into one single theorem (Theorem \ref{p_L_thm}). This theorem is then generalied to higher derivatives of $\phi$ and $\psi$ after applying different types of commutators. The most general and final theorem of \S\ref{ee_sec} is Theorem \ref{p_s_k_thm}.

In \S\ref{pointwise_sec}, the pointwise estimates that are compatible with the weighted energy estimates are proved. They are summarized in Proposition \ref{infinity_prop}.

In \S\ref{nonlinear_sec}, the nonlinear terms $\mathcal{N}_\phi$ and $\mathcal{N}_\psi$ are examined in detail and their structure is defined for the eventual proof of the main theorem.

Finally, in \S\ref{main_thm_sec}, the main theorem is stated and proved.

% Currently removed to keep sections consistent
%\section{A Review of the Scalar Wave Equation}

\section{Estimates for the $\xi_a$ System}\label{xi_a_sec}

As discussed in the introduction, the geometric nature of the wave map problem suggests a vector bundle formalism with which one can describe a linearization of the wave map system (\ref{wm_X_eqn}-\ref{wm_Y_eqn}). In this section, the vector bundle formalism is used to derive the most delicate spacetime estimates.

The derivation of the Morawetz estimate (Proposition \ref{xi_morawetz_estimate_prop}) is rather involved. The first step is to prove a \textit{partial Morawetz estimate}, which is just an estimate for the divergence of a current $J$ that will eventually need to be slightly modified to prove the actual Morawetz estimate. Even the partial Morawetz estimate is quite difficult to prove. It is first proved in the Schwarzschild case (see \S\ref{partial_morawetz_szd_sec}) and then for the slowly rotating Kerr case (see \S\ref{partial_morawetz_kerr_sec}).

Once the partial Morawetz estimate is proved, focus shifts to proving spacetime estimates. First, a family of estimates, called the $h\pd_t$ estimates (so named, because they use a current depending on a vectorfield of the form $h(r)\pd_t$) are proved. A special case of these estimates (where $h=1$) is the classic energy estimate (Proposition \ref{xi_energy_estimate_prop}). The more general estimate is Proposition \ref{xi_h_dt_prop}. These are both proved in \S\ref{xi_h_dt_sec}.

Finally, in \S\ref{xi_morawetz_sec}, the Morawetz estimate is proved, drawing on the partial Morawetz estimate derived earlier.

\subsection{Results from the scalar wave equation}

We begin by reviewing a few facts that are known from analysis of the simpler scalar wave equation problem. These are mainly for reference.

\begin{lemma}\label{Kmunu_lem}
Let $X$ be a vectorfield and $w$ a scalar function. Define
$$K^{\mu\nu}=2\nabla^{(\mu}X^{\nu)}+(w-divX)g^{\mu\nu}.$$
For all $\epsilon_{temper}>0$ sufficiently small, there exists a function $u(r)$ so that the following hold true.

If $X=uv\pd_r$ and $w=v\pd_r u$, where
$$v=\frac{\Delta}{(r^2+a^2)^2}$$
is the geodesic potential, then for any axisymmetric function $\psi$,
$$K^{\mu\nu}\pd_\mu\psi\pd_\nu\psi = 2\left(u'-\frac{2ru}{r^2+a^2}\right)\frac{\Delta^2}{q^2(r^2+a^2)^2}(\pd_r\psi)^2-\frac{u\pd_r v}{q^2}Q^{\mu\nu}\pd_\mu\psi\pd_\nu\psi.$$
Furthermore, we define
\begin{align*}
K^{tt}&=0 \\
K^{rr}&=2\left(u'-\frac{2ru}{r^2+a^2}\right)\frac{\Delta^2}{q^2(r^2+a^2)^2} \\
K_{Q}^{\mu\nu} &= -\frac{u\pd_r v}{q^2}Q^{\mu\nu}.
\end{align*}
Then
\begin{align*}
K^{rr} &\sim \frac{M^2}{r^3}\left(1-\frac{r_H}r\right)^2 \\
K_{Q}^{\theta\theta} &\sim \frac1{r^3}\left(1-\frac{r_{trap}}r\right) \\
K_{Q}^{tt} &\sim \frac{a^2\sin^2\theta}{r^3}\left(1-\frac{r_{trap}}r\right).
\end{align*}
Also,
$$\frac{M}{r^4}1_{r\ge r_*}-q^{-2}V_{\epsilon_{temper}} \le -\frac12\Box_gw,$$
where $V_{\epsilon_{temper}}$ is a positive function supported near the event horizon and satisfying 
$$||V_{\epsilon_{temper}}||_{L^1(r)} < \epsilon_{temper}.$$

Let $r_{trap}$ be the radius where $\pd_rv$ changes sign from positive to negative, and let $r_*$ be the radius where $\pd_r(2rv)$ changes sign from positive to negative. (In Schwarzschild, $r_{trap}=3M$ and $r_*=4M$.) Then $u$ changes sign from negative to positive at $r_{trap}$ and $\pd_ru=2r$ for $r>r_*$.
\end{lemma}

The spacetime estimates will derive from the following simple application of the divergence theorem.
\begin{proposition}\label{general_divergence_estimate_prop}
Let $J$ be a current vanishing at an appropriate rate as $r\rightarrow\infty$. Then
$$\int_{H_{t_1}^{t_2}}J^r+\int_{\Sigma_{t_2}}-J^t+\int_{t_1}^{t_2}\int_{\Sigma_t}div J=\int_{\Sigma_{t_1}}-J^t.$$
\end{proposition}

\subsection{A current template for the bundle $\mathcal{B}$}

We generalize the well-known estimates for the scalar wave equation by using a vector-bundle formalism. To begin, we prove a few lemmas that make this possible.

The first lemma defines a Morawetz current template for the bundle $\mathcal{B}$.
\begin{lemma}\label{vectorized_current_template_lem}
For a vectorfield $X$, a function $w$, and a $\mathcal{B}'\otimes\mathcal{B}'$-valued one-form $m_\mu^{ab}$, define
$$J[X,w,m]_\mu:=T_{\mu\nu}X^\nu+w\xi\cdot D_\mu\xi-\frac12|\xi|^2\pd_\mu w+m_\mu^{ab}\xi_a\xi_b,$$
where
$$T_{\mu\nu}=2D_\mu\xi^a D_\nu\xi_a-g_{\mu\nu}D^\lambda\xi^a D_\lambda\xi_a-g_{\mu\nu}V^{ab}\xi_a\xi_b.$$
Then
\begin{multline*}
div J = K^{\mu\nu}D_\mu\xi\cdot D_\nu\xi-\frac12\Box_g w|\xi|^2+((w-divX)V^{ab}-D_XV^{ab})\xi_a\xi_b+D^\mu m_\mu^{ab}\xi_a\xi_b+2m_\mu^{ab}\xi_aD^\mu\xi_b \\
+2R_{\mu\nu a b}X^\mu \xi^a D^\nu \xi^b+(\Box_g\xi_a-V_a{}^b\xi_b)(2D_X\xi^a+w\xi^a).
\end{multline*}
where
$$K^{\mu\nu}=2\nabla^{(\mu}X^{\nu)}+(w-divX)g^{\mu\nu}$$
is the same tensor as defined in Lemma \ref{Kmunu_lem}.
\end{lemma}
\begin{remark}
The formula for $divJ$ in the lemma seems rather complicated, but it can be broken down into the following parts. (i) The part $K^{\mu\nu}D_\mu\xi\cdot D_\nu\xi-\frac12\Box_gw|\xi|^2$ is directly analogous to the scalar wave equation with no potential. (ii) The part $((w-divX)V^{ab}-D_XV^{ab})\xi_a\xi_b$ is new, because it depends on the potential $V^{ab}$ that is introduced by the equation. (If the scalar equation also had a potential, this part would be analogous to an additional part for that scalar equation.) (iii) The part $D^\mu m_\mu^{ab}\xi_a\xi_b+2m_\mu^{ab}\xi_aD^\mu \xi_b$ is also new, but only because it will be helpful, since it could be excluded by choosing $m_\mu^{ab}=0$. (iv) The part $2R_{\mu\nu a b}X^\mu \xi^aD^\nu \xi^b$ is also new, but this time it is purely due to the use of a bundle instead of a scalar. It will be considered an error term because the curvature tensor vanishes in the Schwarzschild case (see Lemma \ref{bundle_R_calculation_lem} below). (v) Finally, the part $(\Box_g\xi_a-V_a{}^b\xi_b)(2D_X\xi^a+w\xi^a)$ is a nonlinear error term, because it vanishes when the linear equation is satisfied.
\end{remark}
\begin{proof}
To start, we compute the divergence of the energy-momentum tensor.
\begin{align*}
\nabla^\nu T_{\mu\nu} &= 2D_\mu\xi^a D^\nu D_\nu\xi_a + 2D^\nu D_\mu\xi^a D_\nu\xi_a - 2D_\mu D^\lambda\xi^a D_\lambda\xi_a -D_\mu V^{ab}\xi_a\xi_b -2V^{ab}\xi_aD_\mu\xi_b \\
&= 2D_\mu\xi^a\Box_g\xi_a+2(D_\nu D_\mu\xi_a-D_\mu D_\nu\xi_a)D^\nu\xi^a - D_\mu V^{ab}\xi_a\xi_b -2V_a{}^b\xi_bD_\mu \xi^a \\
&= 2(\Box_g\xi_a-V_a{}^b\xi_b)D_\mu\xi^a+2R_{\nu\mu ab}\xi^b D^\nu\xi^a -D_\mu V^{ab}\xi_a\xi_b \\
&= 2(\Box_g\xi_a-V_a{}^b\xi_b)D_\mu\xi^a+2R_{\mu\nu ab}\xi^a D^\nu\xi^b -D_\mu V^{ab}\xi_a\xi_b.
\end{align*}
It follows that
\begin{align*}
\nabla^\nu(T_{\mu\nu}X^\mu) &= \nabla^\nu T_{\mu\nu}X^\mu + T_{\mu\nu}\nabla^{(\mu}X^{\nu)} \\
&= 2(\Box_g\xi_a-V_a{}^b\xi_b)D_X\xi^a+2R_{\mu\nu ab}X^\mu \xi^aD^\nu\xi^b-D_XV^{ab}\xi_a\xi_b \\
&\hspace{.2in}+(2\nabla^{(\mu}X^{\nu)}-divX g^{\mu\nu})D_\mu\xi^aD_\nu\xi_a-(divX) V^{ab}\xi_a\xi_b.
\end{align*}
Also,
\begin{align*}
\nabla^\mu &(w\xi^a D_\mu\xi_a-\frac12\xi^a\xi_a\pd_\mu w) \\
&= w\xi^a D^\mu D_\mu\xi_a+wD^\mu\xi^a D_\mu\xi_a +\pd^\mu w\xi^a D_\mu\xi_a -\pd_\mu w\xi^aD^\mu\xi_a-\frac12\xi^a\xi_a\nabla^\mu \pd_\mu w \\
&= w\xi^a \Box_g\xi_a+wD^\mu\xi^a D_\mu\xi_a -\frac12\Box_g w \xi^a\xi_a \\
&= (\Box_g\xi_a-V_a{}^b\xi_b)w\xi^a + wV^{ab}\xi_a\xi_b+wD^\mu\xi^a D_\mu\xi_a -\frac12\Box_g w \xi^a\xi_a.
\end{align*}
And
$$\nabla^\mu(m_\mu^{ab}\xi_a\xi_b) = D^\mu m_\mu^{ab}\xi_a\xi_b + 2m_\mu^{ab}\xi_aD^\mu\xi_b.$$
Summing all these terms proves the statement of the lemma. \qed
\end{proof}

The next lemma condenses the part of $divJ$ having to do with the potential.
\begin{lemma}\label{muvVab_lem}
If $X=uv\pd_r$ and $w=v\pd_ru$, then
$$(w-divX)V^{ab}-D_XV^{ab}=-\frac{u}{\mu}D_r(\mu v V^{ab}).$$
\end{lemma}
\begin{proof}
\begin{align*}
(w-divX)V^{ab}-X^\mu D_\mu V^{ab} &= (v\pd_ru - \frac1\mu\pd_r(\mu u v))V^{ab}-uvD_rV^{ab} \\
&= -\frac{u}{\mu}\pd_r(\mu v)V^{ab}-uv D_rV^{ab} \\
&= -\frac{u}{\mu}D_r(\mu v V^{ab}).
\end{align*}
\qed
\end{proof}

The final lemma gives a formula for the curvature tensor $R_{\mu\nu ab}$ for the bundle $\mathcal{B}$. Note in particular that in the Schwarzschild case, since $B=0$, the curvature tensor vanishes.
\begin{lemma}\label{bundle_R_calculation_lem}
The curvature tensor is given by
$$R_{\mu\nu a b} = -2\frac{\pd_{[\mu} A\pd_{\nu]} B}{A^2}\epsilon_{ab}.$$
\end{lemma}
\begin{proof}
We calculate $R_{\mu\nu a b}$ according to the formula
$$D_\mu D_\nu (e^i)_a -D_\nu D_\mu (e^i)_a = R_{\mu\nu a b}(e^i)^b.$$
We start with
\begin{align*}
D_\mu D_\nu e^1 &= D_\mu\left(-\frac{\pd_\nu B}{A} e^2\right) \\
&= -\frac{\nabla_\mu \pd_\nu B}{A}e^2+\frac{\pd_\mu A\pd_\nu B}{A^2}e^2-\frac{\pd_\mu B\pd_\nu B}{A^2}e^1.
\end{align*}
Therefore,
$$D_\mu D_\nu e^1 -D_\nu D_\mu e^1 = 2\frac{\pd_{[\mu }A\pd_{\nu]}B}{A^2}e^2.$$
Likewise, we calculate
\begin{align*}
D_\mu D_\nu e^2 &= D_\mu\left(\frac{\pd_\nu B}{A} e^1\right) \\
&= \frac{\nabla_\mu \pd_\nu B}{A}e^1-\frac{\pd_\mu A\pd_\nu B}{A^2}e^1-\frac{\pd_\mu B\pd_\nu B}{A^2}e^2.
\end{align*}
Therefore,
$$D_\mu D_\nu e^2 -D_\nu D_\mu e^2 = -2\frac{\pd_{[\mu }A\pd_{\nu]}B}{A^2}e^1.$$
\end{proof}
The formula is verified by the following two relations.
$$(D_\mu D_\nu e^1 -D_\nu D_\mu e^1)_a = 2\frac{\pd_{[\mu}A\pd_{\nu]}B}{A^2}(e^2)_a = -2\frac{\pd_{[\mu}A\pd_{\nu]}B}{A^2}\epsilon_{ab}(e^1)^b,$$
$$(D_\mu D_\nu e^2 -D_\nu D_\mu e^2)_a = -2\frac{\pd_{[\mu}A\pd_{\nu]}B}{A^2}(e^1)_a = -2\frac{\pd_{[\mu}A\pd_{\nu]}B}{A^2}\epsilon_{ab}(e^2)^b.$$
This completes the proof. \qed

\subsection{The partial Morawetz estimate for the Schwarzschild case}\label{partial_morawetz_szd_sec}

Here, we prove the partial Morawetz estimate for the Schwarzschild case in a way that it can be adapted to Kerr for $|a|\ll M$. Since the proof of the estimate in Kerr is so complicated, it can help to understand the simpler proof for the Schwarzschidl case first.

\begin{proposition}\label{partial_morawetz_szd_prop} (partial Morawetz estimate) Suppose $a=0$, and suppose $\xi_a$ satisfies the equation
$$\Box_g\xi_a - V_a{}^b\xi_b=0,$$
where $V_a{}^b$ is the potential defined previously.

Then there exists a current $J$ such that
\begin{multline*}
\frac{M^2}{r^3}\left(1-\frac{2M}r\right)^2(\pd_r\xi_1)^2+\frac{\chi_{trap}}{r}|\sla\nabla\xi_1|^2+\frac{1}{r^3}(\xi_1)^2 + \chi_{trap}\frac{\cot^2\theta}{r^3}(\xi_1)^2 \\
+ \frac{M^2}{r^3}\left(1-\frac{2M}r\right)^2(\pd_r\xi_2)^2+\frac{\chi_{trap}}{r}|\sla\nabla\xi_2|^2+\frac{M}{r^4}1_{r\ge 4M}(\xi_2)^2 \\
-r^{-2}V_{\epsilon_{temper}}((\xi_1)^2+(\xi_2)^2) \\
\lesssim div J,
\end{multline*}
where $\chi_{trap}=\left(1-\frac{3M}r\right)^2$, and $V_{\epsilon_{temper}}$ is the potential defined in Lemma \ref{Kmunu_lem}.
\end{proposition}

\begin{proof}
We choose
$$J_\mu=J[X,w,m]_\mu,$$
where $J[X,w,m]$ is the current template defined in Lemma \ref{vectorized_current_template_lem}.

We use the same vectorfield
$$X=uv\pd_r$$
and scalar function
$$w=v\pd_ru$$
that are used for the scalar wave equation (see Lemma \ref{Kmunu_lem}).

However, we now choose an additional one-form $m_\mu^{ab}$ with components
\begin{align*}
m_\mu^{ab}&=m_\mu^{ij} (e_i)^a(e_j)^b \\
m^{11}_\mu dx^\mu&=\frac{4\chi v}{r^2}dr+(2-\epsilon)u\pd_rv\cot\theta d\theta \\
m^{12}_\mu dx^\mu&=m^{21}=m^{22}=0,
\end{align*}
where the function $\chi$ will be defined in Lemma \ref{0_K_r_lem}. It is not common to include an $m$ term  with a nonzero $d\theta$ component. The reason will become clear in the proof of Lemma \ref{0_K_theta_lem}. (See the remark following the proof of Lemma \ref{0_K_theta_lem}.)

By Lemma \ref{vectorized_current_template_lem}, we have that
\begin{multline}\label{0_divJ_initial_terms_eqn}
div J = K^{\mu\nu}D_\mu\xi\cdot D_\nu\xi-\frac12\Box_g w|\xi|^2+((w-divX)V^{ab}-D_XV^{ab})\xi_a\xi_b+D^\mu m_\mu^{ab}\xi_a\xi_b+2m_\mu^{ab}\xi_aD^\mu\xi_b \\
+2R_{\mu\nu a b}X^\mu \xi^a D^\nu \xi^b+(\Box_g\xi_a-V_a{}^b\xi_b)(2D_X\xi^a+w\xi^a).
\end{multline}
We rearrange these terms according to the following lemma.
\begin{lemma}\label{0_divJ_rearranged_lem}
$$divJ =\mathcal{K}-\frac12\Box_g w |\xi|^2+(\Box_g\xi_a-V_a{}^b\xi_b)(2D_X\xi^a+w\xi^a),$$
where
$$\mathcal{K}= \mathcal{K}_{(r)}+\mathcal{K}_{(\theta)}+\mathcal{K}_{(2)}$$
$$\mathcal{K}_{(r)}=K^{rr}(\pd_r\xi_1)^2+((w-divX)V_{(r)}^{11}-X^\mu\pd_\mu V_{(r)}^{11})(\xi_1)^2+\nabla^rm_r^{11}(\xi_1)^2+2m_r^{11}\xi_1\pd^r\xi_1$$
$$\mathcal{K}_{(\theta)}=K^{\theta\theta}(\pd_\theta\xi_1)^2+((w-divX)V_{(\theta)}^{11}-X^\mu\pd_\mu V_{(\theta)}^{11})(\xi_1)^2+\nabla^\theta m_\theta^{11}(\xi_1)^2+2m_\theta^{11}\xi_1\pd^\theta\xi_1$$
$$\mathcal{K}_{(2)}=K^{rr}(\pd_r\xi_2)^2+K^{\theta\theta}(\pd_\theta\xi_2)^2,$$
and
\begin{align*}
V_{(r)}^{11} &= g^{rr}A^{-2}(\pd_rA)^2 = \frac4{r^2}\left(1-\frac{2M}r\right) \\
V_{(\theta)}^{11} &= g^{\theta\theta}A^{-2}(\pd_\theta A)^2 = \frac{4}{r^2}\cot^2\theta.
\end{align*}
\end{lemma}
\begin{proof}
Comparing with equation (\ref{0_divJ_initial_terms_eqn}), since $R_{\mu\nu ab}=0$, it suffices to show that
$$\mathcal{K}=K^{\mu\nu}D_\mu\xi\cdot D_\nu\xi +((w-divX)V^{ab}-D_XV^{ab})\xi_a\xi_b+D^\mu m_\mu^{ab}\xi_a\xi_b+2m_\mu^{ab}\xi_aD^\mu \xi_b.$$
Since $K^{tt}=0$, this can be directly verified.
\qed
\end{proof}

In what follows, we examine the coercive properties of each of the terms $\mathcal{K}_{(\theta)}$ (Lemma \ref{0_K_theta_lem}), $\mathcal{K}_{(r)}$ (Lemma \ref{0_K_r_lem}), and $\mathcal{K}_{(2)}$ (Lemma \ref{0_K_2_lem}).

First, we show that even after subtracting a good term, the quantity $\mathcal{K}_{(\theta)}$ controls two angular terms. See the remark following the proof of this lemma.
\begin{lemma}\label{0_K_theta_lem}
For all $\epsilon>0$,
$$\frac{\chi_{trap}}{r}|\sla\nabla\xi_1|^2+\chi_{trap}\frac{\cot^2\theta}{r^3}(\xi_1)^2\lesssim_\epsilon \mathcal{K}_{(\theta)}-(2-\epsilon)\left(-\frac{u\pd_r v}{r^2}\right)(\xi_1)^2.$$
\end{lemma}
\begin{proof}
Recall that
$$\mathcal{K}_{(\theta)}=K^{\theta\theta}(\pd_\theta\xi_1)^2+((w-divX)V_{(\theta)}^{11}-X^\mu\pd_\mu V_{(\theta)}^{11})(\xi_1)^2+D^\theta m_\theta^{11}(\xi_1)^2+2m_\theta^{11}\xi_1\pd^\theta\xi_1.$$
We calculate (using Lemma \ref{Kmunu_lem} in the first line and Lemma \ref{muvVab_lem} in the second line)
\begin{align*}
K^{\theta\theta}(\pd_\theta\xi_1)^2&=-\frac{u\pd_rv}{r^2}(\pd_\theta\xi_1)^2 \\
((w-divX)V_{(\theta)}^{11}-X^r\pd_rV_{(\theta)}^{11})(\xi_1)^2&=-\frac{u\pd_rv}{r^2}4\cot^2\theta(\xi_1)^2 \\
\nabla^\theta m_\theta^{11}(\xi_1)^2&=\frac1{r^2\sin\theta}\pd_\theta(\sin\theta (2-\epsilon)u\pd_rv\cot\theta)(\xi_1)^2 \\
&=-\frac{u\pd_rv}{r^2}(2-\epsilon)(\xi_1)^2 \\
2m_\theta^{11}\xi_1\pd^\theta\xi_1&=2(2-\epsilon)u\pd_rv\cot\theta \xi_1r^{-2}\pd_\theta\xi_1.
\end{align*}
Summing these, we obtain
$$\mathcal{K}_{(\theta)}=-\frac{u\pd_rv}{r^2}\left[(\pd_\theta\xi_1)^2+4\cot^2\theta(\xi_1)^2+(2-\epsilon)(\xi_1)^2-2(2-\epsilon)\cot\theta\xi_1\pd_\theta\xi_1\right].$$
Now,
\begin{align*}
-2(2-\epsilon)\cot\theta\xi_1\pd_\theta\xi_1 &=  \left(1-\frac{\epsilon}2\right)\left(-2\cot\theta\xi_1\pd_\theta\xi_1\right) \\
&=\left(1-\frac{\epsilon}2\right)\left( (\pd_\theta\xi_1-2\cot\theta\xi_1)^2-(\pd_\theta\xi_1)^2-4\cot^2\theta(\xi_1)^2\right).
\end{align*}
Therefore,
$$\mathcal{K}_{(\theta)}=-\frac{u\pd_rv}{r^2}\left[\frac{\epsilon}{2}(\pd_\theta\xi_1)^2+\frac{\epsilon}{2}4\cot^2\theta(\xi_1)^2+\left(1-\frac{\epsilon}2\right)(\pd_\theta\xi_1-2\cot\theta\xi_1)^2+(2-\epsilon)(\xi_1)^2\right].$$
Put another way,
\begin{multline*}
\mathcal{K}_{(\theta)}-(2-\epsilon)\left(-\frac{u\pd_rv}{r^2}\right)(\xi_1)^2 \\
=-\frac{u\pd_rv}{r^2}\left[\frac{\epsilon}{2}(\pd_\theta\xi_1)^2+\frac{\epsilon}{2}4\cot^2\theta(\xi_1)^2+\left(1-\frac{\epsilon}2\right)(\pd_\theta\xi_1-2\cot\theta\xi_1)^2\right].
\end{multline*}
Since $-u\pd_rv\sim \frac{\chi_{trap}}{r}$, the bound stated in Lemma \ref{0_K_theta_lem} follows.
\qed
\end{proof}
\begin{remark}
The case $\epsilon=0$ corresponds directly to a calculation for the $(\phi,\psi)$ system. Indeed,
$$\pd_\theta\xi_1-2\cot\theta\xi_1=A\pd_\theta\psi.$$
This is precisely the need for the $m_\theta^{11}$ component. It accounts for the fact that integrated quantities for $\xi_a$ will differ from their $(\phi,\psi)$ counterparts up to the divergence of a term with a $\theta$ component. For another example, see Lemma \ref{xi_le_phi_psi_2_lem}.
\end{remark}

Next, we show that if the term subtracted from $\mathcal{K}_{(\theta)}$ is added to the term $\mathcal{K}_{(r)}$, then the result is a coercive quantity.
\begin{lemma}\label{0_K_r_lem}
If $\epsilon>0$ is sufficiently small, then there exists a function $\chi$ (determining the component $m_r^{11}$) such that
$$\frac{M^2}{r^3}\left(1-\frac{2M}r\right)^2(\pd_r\xi_1)^2+\frac{1}{r^3}(\xi_1)^2 \lesssim \mathcal{K}_{(r)}+(2-\epsilon)\left(-\frac{u\pd_r v}{r^2}\right)(\xi_1)^2.$$
\end{lemma}
\begin{proof}
Recall that
$$\mathcal{K}_{(r)}=K^{rr}(\pd_r\xi_1)^2+((w-divX)V_{(r)}^{11}-X^\mu\pd_\mu V_{(r)}^{11})(\xi_1)^2+D^rm_r^{11}(\xi_1)^2+2m_r^{11}\xi_1\pd^r\xi_1.$$
We calculate  (using Lemma \ref{Kmunu_lem} in the first line and Lemma \ref{muvVab_lem} in the second line)
\begin{align*}
K^{rr}(\pd_r\xi_1)^2 &= 2\left(\frac{u'}{r^2}-\frac{2u}{r^3}\right)\left(1-\frac{2M}r\right)^2(\pd_r\xi_1)^2 \\
((w-divX)V_{(r)}^{11}-X^\mu\pd_\mu V_{(r)}^{11})(\xi_1)^2 &= -\frac{u}{r^2}\pd_r(r^2 vV_{(r)}^{11}) (\xi_1)^2\\
&= -u\pd_r\left(\frac{4}{r^2}\left(1-\frac{2M}r\right)^2\right)(r^{-1}\xi_1)^2 \\
\nabla^rm_r^{11}(\xi_1)^2 &= \frac1{r^2}\pd_r(r^2 g^{rr}m_r^{11}) (\xi_1)^2\\
&= \pd_r\left(\frac{4\chi}{r^2}\left(1-\frac{2M}r\right)^2\right)(r^{-1}\xi_1)^2 \\
2m_r^{11}\xi_1\pd^r\xi_1 &= \frac{4\chi}{r^3}\left(1-\frac{2M}r\right)^22(r^{-1}\xi_1)\pd_r\xi_1.
\end{align*}
Summing these, we obtain
\begin{multline*}
\mathcal{K}_{(r)} =
2\left(\frac{u'}{r^2}-\frac{2u}{r^3}\right)\left(1-\frac{2M}r\right)^2(\pd_r\xi_1)^2
+(\chi-u)\pd_r\left(\frac{4}{r^2}\left(1-\frac{2M}r\right)^2\right)(r^{-1}\xi_1)^2 \\
+\chi'\frac{4}{r^2}\left(1-\frac{2M}r\right)^2(r^{-1}\xi_1)^2 + \frac{4\chi}{r^3}\left(1-\frac{2M}r\right)^22(r^{-1}\xi_1)\pd_r\xi_1.
\end{multline*}
Now,
\begin{multline*}
\frac{4\chi}{r^3}\left(1-\frac{2M}r\right)^22(r^{-1}\xi_1)\pd_r\xi_1 \\
= \frac{4\chi}{r^3}\left(1-\frac{2M}r\right)^2(\pd_r\xi_1+r^{-1}\xi_1)^2 - \frac{4\chi}{r^3}\left(1-\frac{2M}r\right)^2(\pd_r\xi_1)^2 -\frac{4\chi}{r^3}\left(1-\frac{2M}r\right)^2(r^{-1}\xi_1)^2.
\end{multline*}
Therefore,
\begin{multline*}
\mathcal{K}_{(r)} =
2\left(\frac{u'}{r^2}-\frac{2u}{r^3}-\frac{2\chi}{r^3}\right)\left(1-\frac{2M}r\right)^2(\pd_r\xi_1)^2
+(\chi-u)\pd_r\left(\frac{4}{r^2}\left(1-\frac{2M}r\right)^2\right)(r^{-1}\xi_1)^2 \\
+4\left(\frac{\chi'}{r^2}-\frac{\chi}{r^3}\right)\left(1-\frac{2M}r\right)^2(r^{-1}\xi_1)^2 + \frac{4\chi}{r^3}\left(1-\frac{2M}r\right)^2(\pd_r\xi_1+r^{-1}\xi_1)^2.
\end{multline*}
We make an initial choice for $\chi$.
$$\chi=\left\{
\begin{array}{ll}
  0 & r<3M \\
  u & r\in[3M,4M]. \\
  u(4M) & 4M\le r
\end{array}
\right.$$
Note that since $u(3M)=0$, this piecewise function is continuous. With this choice for $\chi$, we derive estimates for the terms appearing in the above expression for $\mathcal{K}_{(r)}$. These estimates are given in Lemmas \ref{0_K_r_1_lem}-\ref{0_K_r_3_lem}.
\begin{lemma}\label{0_K_r_1_lem}
For all $r>r_H$,
$$\frac{M^2}{r^3}\left(1-\frac{2M}r\right)^2(\pd_r\xi_1)^2 \lesssim 2\left(\frac{u'}{r^2}-\frac{2u}{r^3}-\frac{2\chi}{r^3}\right)\left(1-\frac{2M}r\right)^2(\pd_r\xi_1)^2.$$
\end{lemma}
\begin{proof}
First, we will show that
\begin{equation}\label{up_4u_eqn}
\left(\frac{u'}{r^2}-\frac{4u}{r^3}\right)1_{r\le 4M}>0.
\end{equation}
Note that $u'\ge 0$ everywhere and that $u<0$ for $r< 3M$. The challenge is to prove the inequality for $r\in [3M,4M]$. Since $u$ vanishes at $3M$,
$$\frac{u'}{r^2}-\frac{4u}{r^3}=\frac{u'}{r^2}-\frac{4}{r^3}\int_{3M}^ru'.$$
We now divide by $w=vu'$, which is constant in the interval $[3M,4M]$.
\begin{align*}
\frac1w\left(\frac{u'}{r^2}-\frac{u}{r^3}\right) &= \frac1{r^2v}-\frac{4}{r^3}\int_{3M}^r \frac1{v}.
\end{align*}
Next, we multiply by $r^2v$ so that
\begin{align*}
\frac{r^2v}w\left(\frac{u'}{r^2}-\frac{u}{r^3}\right) &= 1-\frac{4v}{r}\int_{3M}^r \frac1{v}.
\end{align*}
Now, observe that the quantity $v^{-1}$ is increasing starting at $r=3M$, where $v$ has a minimum. Therefore,
$$1-\frac{4v}{r}\int_{3M}^r\frac1v >1-\frac{4v}{r}\left(\frac1v(r-3M)\right)=1-4\left(1-\frac{3M}r\right)=-3\left(1-\frac{4M}r\right).$$
Given the strict inequality in the first step, this establishes (\ref{up_4u_eqn}).

It remains to examine the interval $[4M,\infty)$. In this range, $u=r^2-c_1^2$ and $\chi=c_2^2$ for constants $c_1$ and $c_2$. Thus,
$$\left(\frac{u'}{r^2}-\frac{2u}{r^3}-\frac{2\chi}{r^3}\right)1_{r\ge 4M} = \left(\frac{2r}{r^2}-\frac{2(r^2-c_1^2)}{r^3}-\frac{2c_2^2}{r^3}\right)1_{r\ge 4M} = \frac{2(c_1^2-c_2^2)}{r^3}1_{r>4M}.$$
It follows that $\frac{u'}{r^2}-\frac{2u}{r^3}-\frac{2\chi}{r^3}$ does not change sign and behaves like $O(r^{-3})$ for large $r$. \qed
\end{proof}

\begin{lemma}\label{0_K_r_2_lem}
Since $\chi-u$ has the same sign as $-\left(1-\frac{4M}r\right)$ and $\chi-u=O(-r^2)$ for large $r$,
\begin{multline*}
-\frac1{M^3}\left(1-\frac{2M}r\right)\left(1-\frac{3M}r\right)1_{r\le 3M}(\xi_1)^2+\frac1{r^3} \left(1-\frac{4M}r\right)^21_{r\ge 4M}(\xi_1)^2 \\
\lesssim (\chi-u)\pd_r\left(\frac{4}{r^2}\left(1-\frac{2M}r\right)^2\right)(r^{-1}\xi_1)^2.
\end{multline*}
\end{lemma}
\begin{proof}
We compute
$$(\chi-u)\pd_r\left(\frac{4}{r^2}\left(1-\frac{2M}r\right)^2\right)=-\frac{4(\chi-u)}{r^3}\left(1-\frac{2M}r\right)\left(1-\frac{4M}r\right).$$
By the choice of $\chi$, this quantity vanishes on the interval $[3M,4M]$. Then, $\chi-u>0$ for $r< 3M$ and $\chi-u<0$ for $r>4M$. The fact that $\chi-u$ vanishes linearly as $r\nearrow 3M$ and as $r\searrow 4M$ accounts for the factor $1-\frac{3M}r$ and the additional factor $1-\frac{4M}r$ respectively. Finally, since $\chi-u=O(-r^2)$, this accounts for the appropriate $r^{-3}$ weight for large $r$. \qed
\end{proof}

\begin{lemma}\label{0_K_r_3_lem}
For all $r>r_H$,
\begin{multline*}
\frac{1}{M^3}\left(1-\frac{3M}r\right)^21_{r\le 3M}(\xi_1)^2+\frac{1}{M^3}1_{3M\le r\le 4M}(\xi_1)^2+\frac{1}{r^3}\left(1-\frac{4M}r\right)1_{r\ge 4M}(\xi_1)^2 \\
\lesssim 4\left(\frac{\chi'}{r^2}-\frac{\chi}{r^3}\right)\left(1-\frac{2M}r\right)^2(r^{-1}\xi_1)^2-2\frac{u\pd_rv}{r^2}(\xi_1)^2.
\end{multline*}
\end{lemma}
\begin{proof}
For the region $r\le 3M$, we have that $\chi=0$, so the inequality reduces to the fact that $\left(1-\frac{3M}r\right)^2\lesssim -u\pd_rv$ for $r\le 3M$. For the region $3M\le r\le 4M$, we have that $\chi=u$. We ignore the $-u\pd_rv$ term, because it has the right sign. Then the inequality reduces to the fact that 
$$\frac{\chi'}{r^2}-\frac{\chi}{r^3}=\frac{u'}{r^3}-\frac{u}{r^4}\ge \frac{u'}{r^3}-\frac{4u}{r^4}\ge 0.$$
The last inequality follows from (\ref{up_4u_eqn}).

It is in the region $r\ge 4M$ where the term $-2\frac{u\pd_rv}{r^2}(\xi_1)^2$ is essential. In this region, $\chi'=0$, so
\begin{multline*}
4\left(\frac{\chi'}{r^2}-\frac{\chi}{r^3}\right)\left(1-\frac{2M}r\right)^2(r^{-1}\xi_1)^2-2\frac{u\pd_rv}{r^2}(\xi_1)^2 \\
= -\frac{4\chi}{r^3}\left(1-\frac{2M}r\right)^2(r^{-1}\xi_1)^2 -2\chi\pd_rv(r^{-1}\xi_1)^2-2(u-\chi)\pd_rv (r^{-2}\xi_1)^2.
\end{multline*}
Now,
$$\frac1{r^3}\left(1-\frac{4M}r\right)(\xi_1)^2 \lesssim -2(u-\chi)\pd_rv(r^{-1}\xi_1)^2.$$
It remains to show that
$$0\le -\frac{4\chi}{r^3}\left(1-\frac{2M}r\right)^2(r^{-1}\xi_1)^2-2\chi\pd_rv(r^{-1}\xi_1)^2.$$
Indeed,
\begin{align*}
-\frac{4\chi}{r^3}\left(1-\frac{2M}r\right)^2-2\chi\pd_rv &= -\frac{4\chi}{r^3}\left(1-\frac{2M}r\right)^2+\frac{4\chi}{r^3}\left(1-\frac{3M}r\right) \\
&= \frac{4\chi}{r^3}\left[-1+\frac{4M}r-\frac{4M^2}{r^2} + 1 -\frac{3M}r\right] \\
&= \frac{4\chi}{r^3}\left[\frac{M}{r}-\frac{4M^2}{r^2}\right] \\
&= \frac{4\chi}{r^3}\frac{M}{r}\left(1-\frac{4M}r\right).
\end{align*}
This completes the proof of Lemma \ref{0_K_r_3_lem}. \qed
\end{proof}

Combining Lemmas \ref{0_K_r_1_lem}-\ref{0_K_r_3_lem}, we have the following estimate.
$$\frac{M^2}{r^3}\left(1-\frac{2M}r\right)^2(\pd_r\xi_1)^2+\frac{1}{r^3}\left|1-\frac{3M}r\right|\left|1-\frac{4M}r\right|(\xi_1)^2 \lesssim \mathcal{K}_{(r)}-2\frac{u\pd_r v}{r^2}(\xi_1)^2.$$
By slightly modifying $\chi$, the weak degeneracies at $r=3M$ and $r=4M$ can be removed and the following estimate can be established.
$$\frac{M^2}{r^3}\left(1-\frac{2M}r\right)^2(\pd_r\xi_1)^2+\frac{1}{r^3}(\xi_1)^2 \lesssim \mathcal{K}_{(r)}-2\frac{u\pd_r v}{r^2}(\xi_1)^2.$$
It follows that if $\epsilon>0$ is sufficiently small, then 
$$\frac{M^2}{r^3}\left(1-\frac{2M}r\right)^2(\pd_r\xi_1)^2+\frac{1}{r^3}(\xi_1)^2 \lesssim \mathcal{K}_{(r)}-(2-\epsilon)\frac{u\pd_r v}{r^2}(\xi_1)^2.$$
This completes the proof of Lemma \ref{0_K_r_lem}. \qed
\end{proof}

\begin{lemma}\label{0_K_2_lem}
$$\frac{M^2}{r^3}\left(1-\frac{2M}r\right)^2(\pd_r\xi_2)^2+\frac{\chi_{trap}}{r}|\sla\nabla\xi_2|^2 \lesssim \mathcal{K}_{(2)}$$
\end{lemma}
\begin{proof}
This was established for the scalar wave equation (see Lemma \ref{Kmunu_lem}). \qed
\end{proof}

Finally, we are prepared to complete the proof of the partial Morawetz estimate in Schwarzschild. From Lemma \ref{0_divJ_rearranged_lem}, since $\Box_g\xi_a-V_a{}^b\xi_b=0$,
\begin{align*}
divJ &= \mathcal{K}_{(\theta)}+\mathcal{K}_{(r)}+\mathcal{K}_{(2)}-\frac12\Box_g w |\xi|^2 \\
&= \left(\mathcal{K}_{(\theta)}-(2-\epsilon)\left(-\frac{u\pd_r v}{r^2}\right)(\xi_1)^2\right) 
+\left(\mathcal{K}_{(r)}+(2-\epsilon)\left(-\frac{u\pd_r v}{r^2}\right)(\xi_1)^2\right) +\mathcal{K}_{(2)}-\frac12\Box_g w|\xi|^2.
\end{align*}
From Lemma \ref{0_K_theta_lem},
$$\frac{\chi_{trap}}{r}|\sla\nabla\xi_1|^2+\chi_{trap}\frac{\cot^2\theta}{r^3}(\xi_1)^2\lesssim_\epsilon \mathcal{K}_{(\theta)}-(2-\epsilon)\left(-\frac{u\pd_r v}{r^2}\right)(\xi_1)^2.$$
From Lemma \ref{0_K_r_lem},
$$\frac{M^2}{r^3}\left(1-\frac{2M}r\right)^2(\pd_r\xi_1)^2+\frac{1}{r^3}(\xi_1)^2 \lesssim \mathcal{K}_{(r)}+(2-\epsilon)\left(-\frac{u\pd_r v}{r^2}\right)(\xi_1)^2.$$
From Lemma \ref{0_K_2_lem},
$$\frac{M^2}{r^3}\left(1-\frac{2M}r\right)^2(\pd_r\xi_2)^2+\frac{\chi_{trap}}{r}|\sla\nabla\xi_2|^2 \lesssim \mathcal{K}_{(2)}.$$
And from Lemma \ref{Kmunu_lem},
$$\frac{M}{r^4}1_{r\ge 4M}((\xi_1)^2+(\xi_2)^2)-r^{-2}V_{\epsilon_{temper}}((\xi_1)^2+(\xi_2)^2)
\lesssim -\frac12\Box_g w|\xi|^2.$$
The bound stated in the partial Morawetz estimate is obtained by summing each of these. \qed
\end{proof}

\subsection{The partial Morawetz estimate for Kerr $|a|\ll M$}\label{partial_morawetz_kerr_sec}

Here, we prove the partial Morawetz estimate for slowly rotating Kerr spacetimes by slightly generalizing the much simpler proof for the Schwarzschild case presented previously.

First, we begin with a calculation of quantites that arise in Kerr.
\begin{lemma}\label{small_a_quantities_lem}
The following identities hold.
$$\frac{\pd_\mu A}{A} = \frac{\pd_\mu A_1}{A_1}+\frac{\pd_\mu A_2}{A_2},$$
\begin{align*}
\frac{\pd_r A_1}{A_1} &= \frac{2r}{r^2+a^2} \\
\frac{\pd_\theta A_1}{A_1} &= 2\cot\theta \\
\frac{\pd_rA_2}{A_2} &= a^2\sin^2\theta\left(\frac{-2r}{q^2(r^2+a^2)}+\frac{-\pd_rv}{1-a^2\sin^2\theta v}\right) \\
\frac{\pd_\theta A_2}{A_2} &= a^2\sin\theta\left(\frac{4Mr\cos\theta}{q^2(r^2+a^2)(1-a^2\sin^2\theta v)}\right) \\
\frac{\pd_rB}{A} &= a^3\sin^2\theta\left(\frac{4Mr\cos\theta}{q^2(r^2+a^2)^2(1-a^2\sin^2\theta v)}\right) \\
\frac{\pd_\theta B}{A} &= a\sin\theta\left(\frac{2M(2r^2(r^2+a^2)+(r^2-a^2)q^2)}{q^2(r^2+a^2)^2(1-a^2\sin^2\theta v)}\right).
\end{align*}
In particular,
\begin{align*}
\left|\frac{\pd_r A_2}{A_2}\right| &\lesssim \frac{a^2}{r^3} \\
\left|\frac{\pd_\theta A_2}{A_2}\right| &\lesssim \frac{a^2M}{r^3} \\
\left|\frac{\pd_r B}{A}\right| &\lesssim \frac{|a|^3M}{r^5} \\
\left|\frac{\pd_\theta B}{A}\right| &\lesssim \frac{|a|M}{r^2}.
\end{align*}
\end{lemma}
\begin{proof}
These follow from direct calculation. \qed
\end{proof}

We also prove the following lemma, which allows us to estimate pure partial derivatives.
\begin{lemma}\label{Dxi_lem}
$$\left(1-\left|\frac{r\pd_r B}{A}\right|\right)\left[(\pd_r\xi_1)^2+(\pd_r\xi_2)^2\right] \le |D_r\xi|^2 +\left|\frac{r\pd_r B}{A}\right|\frac{|\xi|^2}{r^2}$$
$$\left(1-\left|\frac{\pd_{\theta} B}{A}\right|\right)\left[(\pd_{\theta}\xi_1)^2+(\pd_{\theta}\xi_2)^2\right] \le |D_{\theta}\xi|^2+\left|\frac{\pd_{\theta} B}{A}\right||\xi|^2.$$
\end{lemma}
\begin{proof}
For arbitrary quantities $x$ and $y$ and some positive function $f$, we have
$$(x\pm fy)^2+f (x^2+y^2) = x^2 \pm 2fxy+f^2y^2+fx^2+fy^2 = x^2+f(x \pm y)^2+f^2y^2.$$
Therefore,
$$x^2\le (x \pm fy)^2+ f(x^2+y^2),$$
whence
$$(1-f)x^2 \le (x \pm fy)^2+fy^2.$$

Now, we have that
$$|D_r\xi|^2 = \left(\pd_r\xi_1 +\frac{r\pd_rB}{A}\frac{\xi_2}{r}\right)^2 + \left(\pd_r\xi_2 -\frac{r\pd_rB}{A}\frac{\xi_1}{r}\right)^2$$
and
$$|D_{\theta}\xi|^2 = \left(\pd_{\theta}\xi_1 +\frac{\pd_{\theta}B}{A}\xi_2\right)^2 + \left(\pd_{\theta}\xi_2 -\frac{\pd_{\theta}B}{A}\xi_1\right)^2.$$
At this point, both estimates in the lemma can be easily deduced. \qed
\end{proof}

Finally, we turn to the partial Morawetz estimate for slowly rotating Kerr spacetimes. 
\begin{proposition}\label{partial_morawetz_kerr_prop} (partial Morawetz estimate) Suppose $|a|/M$ is sufficiently small, and suppose $\xi_a$ satisfies the equation
$$\Box_g\xi_a - V_a{}^b\xi_b=0,$$
where $V_a{}^b$ is the potential defined previously.

Then there exists a current $J$ such that
\begin{multline*}
\frac{M^2}{r^3}\left(1-\frac{r_H}r\right)^2(\pd_r\xi_1)^2+\frac{\chi_{trap}}{r}|\sla\nabla\xi_1|^2+\frac{1}{r^3}(\xi_1)^2 + \chi_{trap}\frac{\cot^2\theta}{r^3}(\xi_1)^2 \\
+ \frac{M^2}{r^3}\left(1-\frac{r_H}r\right)^2(\pd_r\xi_2)^2+\frac{\chi_{trap}}{r}|\sla\nabla\xi_2|^2+\frac{M}{r^4}1_{r\ge r_*}(\xi_2)^2 \\
-q^{-2}V_{\epsilon_{temper}}((\xi_1)^2+(\xi_2)^2) -q^{-2}V_{\epsilon_a}(\xi_2)^2\\
\lesssim div J,
\end{multline*}
where $\chi_{trap}=\left(1-\frac{r_{trap}}r\right)^2$, $V_{\epsilon_{temper}}$ is the potential defined in Lemma \ref{Kmunu_lem}, and $V_{\epsilon_a}$ is a positive function supported on $r\in [r_H,r_*]$ and satisfying $||V_{\epsilon_a}||_{L^1(r)}\le \epsilon_a$ when $|a|/M$ is chosen sufficiently small.
\end{proposition}
The reader is encouraged to compare the following proof to the proof of Proposition \ref{partial_morawetz_szd_prop}.
\begin{proof}
We again choose the current
$$J_\mu = J[X,w,m]_\mu,$$
where $J[X,w,m]$ is defined in Lemma \ref{vectorized_current_template_lem}.

We use the same vectorfield
$$X=uv\pd_r$$
and scalar function
$$w=v\pd_r u$$
that are used in the scalar wave equation for Kerr (see Lemma \ref{Kmunu_lem}).

As in the simpler proof for Schwarzschild, we also use a one-form $m_\mu^{ab}$ with components
\begin{align*}
m_\mu^{ab}&=m_\mu^{ij} (e_i)^a(e_j)^b \\
m^{11}_\mu dx^\mu&=\frac{4r^2\chi v}{(r^2+a^2)^2}dr+(2-\epsilon)u\pd_rv\cot\theta d\theta \\
m^{12}_\mu dx^\mu&=m^{21}=m^{22}=0,
\end{align*}
where the function $\chi$ will be defined in Lemma \ref{a_K_r_lem}.

By Lemma \ref{vectorized_current_template_lem}, we have that
\begin{multline*}
div J = K^{\mu\nu}D_\mu\xi\cdot D_\nu\xi-\frac12\Box_g w|\xi|^2+((w-divX)V^{ab}-D_XV^{ab})\xi_a\xi_b+D^\mu m_\mu^{ab}\xi_a\xi_b+2m_\mu^{ab}\xi_aD^\mu\xi_b \\
+2R_{\mu\nu a b}X^\mu \xi^a D^\nu \xi^b+(\Box_g\xi_a-V_a{}^b\xi_b)(2D_X\xi^a+w\xi^a).
\end{multline*}
We rearrange these terms according to the following lemma.
\begin{lemma}\label{a_divJ_rearranged_lem}
$$divJ = \mathcal{K}-\frac12\Box_g w|\xi|^2+(\Box_g\xi_a-V_a{}^b\xi_b)(2D_X\xi^a+w\xi^a),$$
where
$$\mathcal{K} = \mathcal{K}_{(r)}+\mathcal{K}_{(\theta)}+\mathcal{K}_{(2)}+\mathcal{K}_{(t)}+\mathcal{K}_{(a)},$$
$$\mathcal{K}_{(r)} = K^{rr}\left(\pd_r\xi_1+\frac{\pd_r B}{A}\xi_2\right)^2+((w-divX)V_{(r)}^{11}+X^\mu\pd_\mu V_{(r)}^{11})(\xi_1)^2+\nabla^r (m_r^{11})(\xi_1)^2+2m_r^{11}\xi_1\pd^r\xi_1$$
$$\mathcal{K}_{(\theta)} = K_{Q}^{\theta\theta}\left(\pd_\theta\xi_1+\frac{\pd_\theta B}{A}\xi_2\right)^2+((w-divX)V_{(\theta)}^{11}+X^\mu\pd_\mu V_{(\theta)}^{11})(\xi_1)^2+\nabla^\theta (m_\theta^{11})(\xi_1)^2+2m_\theta^{11}\xi_1\pd^\theta\xi_1$$
$$\mathcal{K}_{(2)}=K^{rr}\left(\pd_r\xi_2-\frac{\pd_r B}{A}\xi_1\right)^2+K_{Q}^{\theta\theta}\left(\pd_\theta\xi_2-\frac{\pd_\theta B}{A}\xi_1\right)^2$$
$$\mathcal{K}_{(t)}=K_Q^{tt}\left((\pd_t\xi_1)^2+(\pd_t\xi_2)^2\right)$$
\begin{multline*}
\mathcal{K}_{(a)} = ((w-divX)V_{(a)}^{ij}+X^\mu \pd_\mu V_{(a)}^{ij})\xi_i\xi_j +X^\mu V^{ij}D_\mu(e_i^a e_j^b)\xi_a\xi_b +m_\mu^{ij}D^\mu(e_i^a e_j^b)\xi_i\xi_j \\
+2m_\mu^{ij} e_i^aD^\mu (e_j^b)\xi_a\xi_b +2R_{\mu\nu ab}X^\mu\xi^aD^\nu\xi^b
\end{multline*}
and
$$V_{(r)}^{11}=\frac{\Delta}{q^2}\left(\frac{2r}{r^2+a^2}\right)^2$$
$$V_{(\theta)}^{11}=\frac1{q^2}\left(2\cot\theta\right)^2$$
$$V_{(a)}^{ij}=V^{ij}-(V_{(r)}^{11}+V_{(\theta)}^{11})\delta_1^i\delta_1^j.$$
\end{lemma}
\begin{proof}
The proof is the same as the proof of Lemma \ref{0_divJ_rearranged_lem}, except that there are new terms which are grouped into the quantities $\mathcal{K}_{(t)}$ and $\mathcal{K}_{(a)}$. \qed
\end{proof}
In what follows, we examine the coercive properties of each of the terms $\mathcal{K}_{(\theta)}$ (Lemma \ref{a_K_theta_lem}), $\mathcal{K}_{(r)}$ (Lemma \ref{a_K_r_lem}), and $\mathcal{K}_{(2)}$ (Lemma \ref{a_K_2_lem}). We also estimate the error term $\mathcal{K}_{(a)}$ (Lemma \ref{a_K_a_lem}). The term $\mathcal{K}_{(t)}$ has a good sign, but also a factor of $a^2/M^2$, so it is generally ignored.

First, we show that even after subtracting a good term, the quantity $\mathcal{K}_{(\theta)}$ almost controls two angular terms.
\begin{lemma}\label{a_K_theta_lem}
For all $\epsilon>0$, if $|a|/M$ is sufficiently small compared to $\epsilon$, then
$$\frac{\chi_{trap}}{r}|\sla\nabla\xi_1|^2+\chi_{trap}\frac{\cot^2\theta}{r^3}(\xi_1)^2\lesssim_\epsilon \mathcal{K}_{(\theta)}-(2-\epsilon)\left(-\frac{u\pd_r v}{q^2}\right)(\xi_1)^2+\chi_{trap}\frac{aM}{r^5}(\xi_2)^2.$$
\end{lemma}
\begin{proof}
Recall that
$$\mathcal{K}_{(\theta)} = K^{\theta\theta}\left(\pd_\theta\xi_1+\frac{\pd_\theta B}{A}\xi_2\right)^2+((w-divX)V_{(\theta)}^{11}+X^\mu\pd_\mu V_{(\theta)}^{11})(\xi_1)^2+\nabla^\theta (m_\theta^{11})(\xi_1)^2+2m_\theta^{11}\xi_1\pd^\theta\xi_1.$$
By Lemma \ref{Dxi_lem}, since $K^{\theta\theta}\ge 0$,
$$\left(1-\left|\frac{\pd_\theta B}{A}\right|\right)K^{\theta\theta}(\pd_\theta\xi_1)^2 \le K^{\theta\theta}\left(\pd_\theta\xi_1+\frac{\pd_\theta B}{A}\xi_2\right)^2 +K^{\theta\theta}\left|\frac{\pd_\theta B}{A}\right|(\xi_2)^2.$$
We calculate  (using Lemma \ref{Kmunu_lem} in the first line and Lemma \ref{muvVab_lem} in the second line)
\begin{align*}
\left(1-\left|\frac{\pd_\theta B}{A}\right|\right)K^{\theta\theta}(\pd_\theta\xi_1)^2 &= \left(1-\left|\frac{\pd_\theta B}{A}\right|\right)\left(-\frac{u\pd_r v}{q^2}\right)(\pd_\theta\xi_1)^2 \\
((w-div X)V^{11}_{(\theta)}-X^\mu \pd_\mu V_{(\theta)}^{11})(\xi_1)^2 &= -\frac{u\pd_r v}{q^2}4\cot^2\theta (\xi_1)^2 \\
\nabla^\theta m_\theta^{11}(\xi_1)^2 &= \frac{1}{q^2\sin\theta}\pd_\theta\left(q^2\sin\theta \frac{u\pd_rv}{q^2}\cot\theta\right)(\xi_1)^2 \\
&= -\frac{u\pd_r v}{q^2}(2-\epsilon)(\xi_1)^2 \\
2m_\theta^{11}\xi_1\pd^\theta \xi_1 &= 2(2-\epsilon)u\pd_r v\cot\theta \xi_1 q^{-2}\pd_\theta\xi_1
\end{align*}
Proceeding exactly as in the Schwarzschild case, we conclude that 
\begin{multline*}
-\frac{u\pd_r v}{q^2}\left[\left(\frac{\epsilon}{2}-\left|\frac{\pd_\theta B}{A}\right|\right)(\pd_\theta\xi_1)^2+\frac{\epsilon}{2}4\cot^2\theta(\xi_1)^2+\left(1-\frac{\epsilon}{2}\right)(\pd_\theta\xi_1-2\cot\theta\xi_1)^2\right] \\
\le \mathcal{K}_{(\theta)}-(2-\epsilon)\left(-\frac{u\pd_r v}{q^2}\right)(\xi_1)^2+\left|\frac{u\pd_r v}{q^2}\frac{\pd_\theta B}{A}\right|(\xi_2)^2.
\end{multline*}
Lastly, since $-\frac{u\pd_r v}{q^2}\sim \frac{\chi_{trap}}{r^3}$ and $\left|\frac{\pd_\theta B}{A}\right|\lesssim \frac{aM}{r^2}$,
$$\left|\frac{u\pd_r v}{q^2}\frac{\pd_\theta B}{A}\right|(\xi_2)^2 \lesssim \chi_{trap}\frac{aM}{r^5}(\xi_2)^2.$$
Thus, the bound stated in Lemma \ref{a_K_theta_lem} follows. \qed
\end{proof}

Next, we show that if the term subtracted from $\mathcal{K}_{(\theta)}$ is added to the term $\mathcal{K}_{(r)}$, then the result is almost a coercive quantity.
\begin{lemma}\label{a_K_r_lem}
If $\epsilon>0$ and $|a|/M$ are both sufficiently small, then there exists a function $\chi$ (determining the component $m_r^{11}$) such that 
$$\frac{M^2}{r^3}\left(1-\frac{r_H}r\right)^2(\pd_r\xi_1)^2+\frac1{r^3}(\xi_1)^2\lesssim \mathcal{K}_{(r)}+(2-\epsilon)\left(-\frac{u\pd_rv}{q^2}\right)(\xi_1)^2+\frac{a^3M^3}{r^9}\left(1-\frac{r_H}r\right)^2(\xi_2)^2.$$
\end{lemma}
\begin{proof}
Recall that
$$\mathcal{K}_{(r)} = K^{rr}\left(\pd_r\xi_1+\frac{\pd_r B}{A}\xi_2\right)^2+((w-divX)V_{(r)}^{11}+X^\mu\pd_\mu V_{(r)}^{11})(\xi_1)^2+\nabla^r (m_r^{11})(\xi_1)^2+2m_r^{11}\xi_1\pd^r\xi_1.$$
By Lemma \ref{Dxi_lem}, since $K^{rr}\ge 0$,
$$\left(1-\left|\frac{r\pd_r B}{A}\right|\right)K^{rr}(\pd_r\xi_1)^2 \le K^{rr}\left(\pd_\theta\xi_1+\frac{\pd_rB}{A}\xi_2\right)^2+K^{rr}\left|\frac{r\pd_rB}{A}\right|\frac{(\xi_2)^2}{r^2}.$$
We calculate (using Lemma \ref{Kmunu_lem} in the first line and Lemma \ref{muvVab_lem} in the second line)
\begin{align*}
\left(1-\left|\frac{r\pd_r B}{A}\right|\right)K^{rr}(\pd_r\xi_1)^2 &= 2\left(1-\left|\frac{r\pd_r B}{A}\right|\right)\left(\frac{u'}{r^2+a^2}-\frac{2ru}{(r^2+a^2)^2}\right)\frac{\Delta^2}{q^2(r^2+a^2)}(\pd_r\xi_1)^2 \\
((w-divX)V_{(r)}^{11}+X^\mu\pd_\mu V_{(r)}^{11})(\xi_1)^2 &= -\frac{u}{q^2}\left(q^2v V_{(r)}^{11}\right) (\xi_1)^2 \\
&=-\frac{u}{q^2}\left(q^2\frac{\Delta}{(r^2+a^2)^2}\frac{\Delta}{q^2}\frac{4r^2}{(r^2+a^2)^2}\right)(\xi_1)^2 \\
&=-\frac{u}{q^2}\left(\frac{4r^2\Delta^2}{(r^2+a^2)^4}\right)(\xi_1)^2 \\
\nabla^r m_r^{11} (\xi_1)^2 &= \frac1{q^2}\pd_r\left(q^2 g^{rr}m_r^{11}\right) (\xi_1)^2 \\
&= \frac1{q^2}\pd_r\left(q^2\frac{\Delta}{q^2}\frac{4r^2\chi v}{(r^2+a^2)^2}\right)(\xi_1)^2 \\
&= \frac1{q^2}\pd_r\left(\frac{4r^2\chi\Delta^2}{(r^2+a^2)^4}\right) (\xi_1)^2 \\
2m_r^{11}\xi_1\pd^r\xi_1 &= \frac{4r^2\chi v}{(r^2+a^2)^2}\frac{\Delta}{q^2}  2\xi_1\pd_r\xi_1.
\end{align*}
Summing these, we obtain
\begin{multline*}
2\left(1-\left|\frac{r\pd_r B}{A}\right|\right)\left(\frac{u'}{r^2+a^2}-\frac{2ru}{(r^2+a^2)^2}\right) \frac{\Delta^2}{q^2(r^2+a^2)}(\pd_r\xi_1)^2  \\
+\frac{(\chi-u)}{q^2}\pd_r\left(\frac{4r^2\Delta^2}{(r^2+a^2)^4}\right)(\xi_1)^2 
+\frac{\chi'}{q^2}\frac{4r^2\Delta^2}{(r^2+a^2)^4}(\xi_1)^2 + \frac{4r^2\chi\Delta^2}{q^2(r^2+a^2)^4}(2\xi_1\pd_r\xi_1)
\\ \le \mathcal{K}_{(r)} + 2\frac{\Delta^2}{q^2(r^2+a^2)}\pd_r\left(\frac{u}{r^2+a^2}\right)\frac{r\pd_r B}{A} r^{-2}(\xi_2)^2.
\end{multline*}
Now,
\begin{multline*}
\frac{4r^2\chi\Delta^2}{q^2(r^2+a^2)^4}2\xi_1\pd_r\xi_1  \\
= \frac{4r\chi\Delta^2}{q^2(r^2+a^2)^3}\left(\pd_r\xi_1+\frac{r}{r^2+a^2}\xi_1\right)^2 -\frac{4r\chi\Delta^2}{q^2(r^2+a^2)^3}(\pd_r\xi_1)^2-\frac{4r^3\chi \Delta^2}{q^2(r^2+a^2)^5}(\xi_1)^2.
\end{multline*}
Therefore,
\begin{multline*}
2\left(\left(1-\left|\frac{r\pd_r B}{A}\right|\right)\left(\frac{u'}{r^2+a^2}-\frac{2ru}{(r^2+a^2)^2}\right)-\frac{2r\chi}{(r^2+a^2)^2} \right) \frac{\Delta^2}{q^2(r^2+a^2)}(\pd_r\xi_1)^2  \\
+\frac{(\chi-u)}{q^2}\pd_r\left(\frac{4r^2\Delta^2}{(r^2+a^2)^4}\right)(\xi_1)^2 
+\left(\frac{\chi'}{r^2+a^2}-\frac{r\chi}{(r^2+a^2)^2}\right)\frac{4r^2\Delta^2}{q^2(r^2+a^2)^3}(\xi_1)^2 \\
+ \frac{4r\chi\Delta^2}{q^2(r^2+a^2)^3}\left(\pd_r\xi_1+\frac{r}{r^2+a^2}\xi_1\right)^2
\\ \le \mathcal{K}_{(r)} + 2\frac{\Delta^2}{q^2(r^2+a^2)}\pd_r\left(\frac{u}{r^2+a^2}\right)\left|\frac{r\pd_r B}{A}\right| r^{-2}(\xi_2)^2.
\end{multline*}
The error term can be estimated as
\begin{align*}
2\frac{\Delta^2}{q^2(r^2+a^2)}\pd_r\left(\frac{u}{r^2+a^2}\right)\left|\frac{r\pd_r B}{A}\right| r^{-2}(\xi_2)^2 &\lesssim \left(1-\frac{r_H}r\right)^2\frac{M^2}{r^3}\frac{a^3M}{r^4}\frac{(\xi_2)^2}{r^2}(\xi_2)^2 \\
&\lesssim \frac{a^3M^3}{r^9}\left(1-\frac{r_H}r\right)^2(\xi_2)^2.
\end{align*}
We make an initial choice for $\chi$.
$$\chi=\left\{
\begin{array}{ll}
  0 & r<r_{trap} \\
  u & r\in[r_{trap},r_*]. \\
  u(r_*) & r_*\le r
\end{array}
\right.$$
Note that since $u(r_{trap})=0$, this piecewise function is continuous. With this choice for $\chi$, we derive estimates for the terms appearing in the expression for $\mathcal{K}_{(r)}$ above. These estimates are given in Lemmas \ref{a_K_r_1_lem}-\ref{a_K_r_3_lem}.
\begin{lemma}\label{a_K_r_1_lem}
If $|a|/M$ is sufficiently small, then for all $r>r_H$,
\begin{multline*}
\frac{M^2}{r^3}\left(1-\frac{r_H}r\right)^2(\pd_r\xi_1) \\
\lesssim 2\left(\left(1-\left|\frac{r\pd_r B}{A}\right|\right)\left(\frac{u'}{r^2+a^2}-\frac{2ru}{(r^2+a^2)^2}\right)-\frac{2r\chi}{(r^2+a^2)^2} \right) \frac{\Delta^2}{q^2(r^2+a^2)}(\pd_r\xi_1)^2
\end{multline*}
\end{lemma}
\begin{proof}
The case $a=0$ reduces to Lemma \ref{0_K_r_1_lem}. One can check that the inequality is not affected by taking $|a|/M$ to be small.\qed
\iffalse
This is a consequence of Lemma \ref{0_K_r_1_lem} and the fact that the expression
$$\left(\left(1-\frac{r\pd_r B}{A}\right)\left(\frac{u'}{r^2+a^2}-\frac{2ru}{(r^2+a^2)^2}\right)-\frac{2r\chi}{(r^2+a^2)^2} \right)$$
is continuous in $a$. \qed
\fi
\end{proof}

\begin{lemma}\label{a_K_r_2_lem}
For all  $|a|<M$,
\begin{multline*}
-\frac1{M^3}\left(1-\frac{r_H}r\right)\left(1-\frac{r_{trap}}r\right)1_{r\le r_{trap}}(\xi_1)^2 +\frac1{r^3}\left(1-\frac{r_*}r\right)^21_{r\ge r_*}(\xi_1)^2 \\
\lesssim \frac{(\chi-u)}{q^2}\pd_r\left(\frac{4r^2\Delta^2}{(r^2+a^2)^4}\right)(\xi_1)^2.
\end{multline*}
\end{lemma}
\begin{proof}
The case $a=0$ reduces to Lemma \ref{0_K_r_2_lem}. Recall that $r_*$ is by definition the radius at which the function $\frac{2r\Delta}{(r^2+a^2)^2}$ has a maximum value. Therefore the function $\frac{4r^2\Delta^2}{(r^2+a^2)^4}$ also has a maximum value at $r_*$. With this fact in mind, one can check that the argument in the proof of Lemma \ref{0_K_r_2_lem} applies for the general case $|a|<M$ after replacing $3M$ and $4M$ with $r_{trap}$ and $r_*$ respectively. \qed
\end{proof}

\begin{lemma}\label{a_K_r_3_lem}
If $|a|/M$ is sufficiently small, then for all $r>r_H$, there exists a constant $c$ such that
\begin{multline*}
\frac1{M^3}\left(1-\frac{r_{trap}}r\right)^21_{r\le r_{trap}}(\xi_1)^2+\frac1{M^3}1_{r_{trap}\le r\le r_*}(\xi_1)^2+\frac1{r^3}\left(1-\frac{r_*}r-c\frac{|a|}{M}\right)1_{r\ge r_*}(\xi_1)^2 \\
\lesssim \left(\frac{\chi'}{r^2+a^2}-\frac{r\chi}{(r^2+a^2)^2}\right)\frac{4r^2\Delta^2}{q^2(r^2+a^2)^3}(\xi_1)^2 +2\left(-\frac{u\pd_r v}{q^2} \right)(\xi_1)^2.
\end{multline*}
\end{lemma}
\begin{proof}
The case $a=0$ reduces to Lemma \ref{0_K_r_3_lem}. One can check that the inequality is not affected by taking $|a|/M$ to be small except for a small neighborhood of $r_*$, which will have an error term for $r>r_*$. This term is accounted for by the introduction of a constant $c$ so that the term
$$\frac{1}{r^3}\left(1-\frac{r_*}{r}-c\frac{|a|}{M}\right)1_{r\ge r_*}(\xi_1)^2,$$
which is slightly negative near $r_*$, is still bounded by the right hand side. \qed
\end{proof}
Combining Lemmas \ref{a_K_r_1_lem}-\ref{a_K_r_3_lem}, we have the following estimate.
\begin{multline*}
\frac{M^2}{r^3}\left(1-\frac{r_H}r\right)^2(\pd_r\xi_1)^2+\frac{1}{r^3}\left|1-\frac{r_{trap}}r\right|\left|1-\frac{r_*}r\right|(\xi_1)^2 -c\frac{|a|}{M}r^{-3}1_{r\approx r_*}(\xi_1)^2 \\
\lesssim \mathcal{K}_{(r)}-2\frac{u\pd_r v}{q^2}(\xi_1)^2+\frac{a^3M^3}{r^9}\left(1-\frac{r_H}r\right)^2(\xi_2)^2.
\end{multline*}
If $|a|/M$ is sufficiently small, then by slightly modifying $\chi$, the weak degeneracies at $r_{trap}$ and $r_*$ can be removed as well as the small error term $-c\frac{|a|}{M}r^{-3}1_{r\approx r_*}(\xi_1)^2$. Then the following estimate can be established.
$$\frac{M^2}{r^3}\left(1-\frac{r_H}r\right)^2(\pd_r\xi_1)^2+\frac{1}{r^3}(\xi_1)^2 
\lesssim \mathcal{K}_{(r)}-2\frac{u\pd_r v}{q^2}(\xi_1)^2+\frac{a^3M^3}{r^9}\left(1-\frac{r_H}r\right)^2(\xi_2)^2.$$
It follows that if $\epsilon>0$ is sufficiently small, 
$$\frac{M^2}{r^3}\left(1-\frac{r_H}r\right)^2(\pd_r\xi_1)^2+\frac{1}{r^3}(\xi_1)^2 
\lesssim \mathcal{K}_{(r)}-(2-\epsilon)\frac{u\pd_r v}{q^2}(\xi_1)^2+\frac{a^3M^3}{r^9}\left(1-\frac{r_H}r\right)^2(\xi_2)^2.$$
This completes the proof of Lemma \ref{a_K_r_lem}. \qed
\end{proof}

\begin{lemma}\label{a_K_2_lem}
If $|a|/M$ is sufficiently small, then
$$\frac{M^2}{r^3}\left(1-\frac{r_H}r\right)^2(\pd_r\xi_2)^2+\frac{\chi_{trap}}{r}|\sla\nabla\xi_2|^2 \lesssim \mathcal{K}_{(2)} +\frac{|a|M}{r^5}(\xi_1)^2.$$
\end{lemma}
\begin{proof}
Recall that
\begin{align*}
\mathcal{K}_{(2)}&=K^{rr}\left(\pd_r\xi_2-\frac{\pd_r B}{A}\xi_1\right)^2+K_Q^{\theta\theta}\left(\pd_\theta\xi_2-\frac{\pd_\theta B}{A}\xi_1\right)^2 \\
&=K^{rr}\left(\pd_r\xi_2-\frac{r\pd_r B}{A}r^{-1}\xi_1\right)^2+K_Q^{\theta\theta}\left(\pd_\theta\xi_2-\frac{\pd_\theta B}{A}\xi_1\right)^2.
\end{align*}
By applying the procedure in the proof of Lemma \ref{Dxi_lem}, we conclude that
$$\left(1-\left|\frac{r\pd_r B}{A}\right|\right)K^{rr}(\pd_r\xi_2)^2 \le K^{rr}\left(\pd_r\xi_2-\frac{r\pd_r B}{A}r^{-1}\xi_1\right)^2 +\left|\frac{r\pd_r B}{A}\right|K^{rr}r^{-2}(\xi_1)^2.$$
Therefore, if $|a|/M$ is sufficiently small,
$$K^{rr}(\pd_r\xi_2)^2 \lesssim K^{rr}\left(\pd_r\xi_2-\frac{\pd_r B}{A}\xi_1\right)^2 +\frac{|a|^3M}{r^4}K^{rr}r^{-2}(\xi_1)^2.$$
Again by applying the procedure in the proof of Lemma \ref{Dxi_lem}, we also conclude that
$$\left(1-\left|\frac{\pd_\theta B}{A}\right|\right)K_Q^{\theta\theta}(\pd_\theta\xi_2)^2 \le K_Q^{\theta\theta}\left(\pd_\theta\xi_2-\frac{\pd_\theta B}{A}\xi_1\right)^2+\left|\frac{\pd_\theta B}{A}\right|K_Q^{\theta\theta}(\xi_1)^2.$$
Therefore, if $|a|/M$ is sufficiently small,
$$K_Q^{\theta\theta}(\pd_\theta\xi_2)^2 \lesssim K_Q^{\theta\theta}\left(\pd_\theta\xi_2-\frac{\pd_\theta B}{A}\xi_1\right)^2+\frac{|a|M}{r^2}K_Q^{\theta\theta}(\xi_1)^2.$$
From both estimates, since $K^{rr}=O(M^2/r^3)$ and $K_Q^{\theta\theta}=O(r^{-3})$, we have
\begin{multline*}
K^{rr}(\pd_r\xi_2)^2+K_Q^{\theta\theta}(\pd_\theta\xi_2)^2 \\
\lesssim K^{rr}\left(\pd_r\xi_2-\frac{r\pd_r B}{A}\xi_1\right)^2+ K_Q^{\theta\theta}\left(\pd_\theta\xi_2-\frac{\pd_\theta B}{A}\xi_1\right)^2+\frac{|a|M}{r^5}(\xi_1)^2 \\
\lesssim \mathcal{K}_{(2)}+\frac{|a|M}{r^5}(\xi_1)^2.
\end{multline*}
The lemma follows from the lower bound estimates for $K^{rr}$ and $K_Q^{\theta\theta}$. \qed
\end{proof}

\begin{lemma}\label{a_K_a_lem}
The following bound holds.
$$|\mathcal{K}_{(a)}| \lesssim \chi_{trap}\frac{|a|M}{r^3}\left(|\sla\nabla\xi_1|^2+|\sla\nabla\xi_2|^2\right)+ \frac{|a|M}{r^5}|\xi|^2.$$
\end{lemma}
\begin{proof}
Recall that
\begin{multline*}
\mathcal{K}_{(a)} = ((w-divX)V_{(a)}^{ij}+X^\mu \pd_\mu V_{(a)}^{ij})\xi_i\xi_j +X^\mu V^{ij}D_\mu(e_i^a e_j^b)\xi_a\xi_b +m_\mu^{ij}D^\mu(e_i^a e_j^b)\xi_i\xi_j \\
+2m_\mu^{ij} e_i^aD^\mu (e_j^b)\xi_a\xi_b +2R_{\mu\nu ab}X^\mu\xi^aD^\nu\xi^b,
\end{multline*}
where
$$V_{(a)}^{ij}=V^{ij}-(V_{(r)}^{11}+V_{(\theta)}^{11})\delta_1^i\delta_1^j.$$

There is one term (the last term) in the expression for $\mathcal{K}_{(a)}$ that contains a derivative of $\xi$. We calculate it explicitly using Lemma \ref{bundle_R_calculation_lem} and the fact that $X$ has only $t$ and $r$ components.
\begin{align*}
2R_{\mu\nu a b}X^\mu\xi^aD^\nu\xi^b &= -4\frac{\pd_{[\mu}A\pd_{\nu]}B}{A^2}\epsilon_{ab}X^\mu \xi^aD^\nu\xi^b \\
&= -4\frac{\pd_{[r}A\pd_{\theta]}B}{A^2}\epsilon_{ab}X^r\xi^aD^\theta\xi^b \\
&= -4\frac{\pd_{[r}A\pd_{\theta]}B}{A^2}X^rg^{\theta\theta}\epsilon^{ab}\xi_aD_\theta\xi_b.
\end{align*}
Now,
\begin{align*}
\epsilon^{ab}\xi_aD_\theta\xi_b &= \epsilon^{ab}\xi_aD_\theta(\xi_j(e^j)_b) \\
&= \epsilon^{ij}\xi_i\pd_\theta\xi_j +\epsilon^{ab}(e^i)_a(D_\theta e^j)_b\xi_i\xi_j
\end{align*}
Thus, the one term that contains a derivative of $\xi_i$ is
\begin{multline*}
-4\frac{\pd_{[r}A\pd_{\theta]}B}{A^2}X^rg^{\theta\theta}\epsilon^{ij}\xi_i\pd_\theta\xi_j \\
=2\left(\frac{\pd_\theta A\pd_r B}{A^2}-\frac{\pd_r A\pd_\theta B}{A^2}\right)\frac{uv}{q^2}(\xi_1\pd_\theta\xi_2-\xi_2\pd_\theta\xi_1) \\
\lesssim \frac{uv}{q^2}\left(\frac{|a|^3M}{r^5}+\frac1r\frac{|a|M}{r^2}\right)(|\xi_1\pd_\theta\xi_2|+|\xi_2\pd_\theta\xi_1|) \\
\lesssim \frac{u^2v^2}{q^2}\frac{|a|M}{r^3}((\pd_\theta\xi_1)^2+(\pd_\theta\xi_2)^2)+\frac1{q^2}\frac{|a|M}{r^3}((\xi_1)^2+(\xi_2)^2) \\
\lesssim \chi_{trap}\frac{|a|M}{r^3}\left(|\sla\nabla\xi_1|^2+|\sla\nabla\xi_2|^2\right) + \frac{|a|M}{r^5}|\xi|^2.
\end{multline*}
The remaining terms to estimate are
\begin{multline*}
((w-divX)V_{(a)}^{ij}+X^\mu \pd_\mu V_{(a)}^{ij})\xi_i\xi_j +X^\mu V^{ij}D_\mu(e_i^a e_j^b)\xi_a\xi_b +m_\mu^{ij}D^\mu(e_i^a e_j^b)\xi_i\xi_j \\
+2m_\mu^{ij} e_i^aD^\mu (e_j^b)\xi_a\xi_b -4\frac{\pd_{[r}A\pd_{\theta]}B}{A^2}X^rg^{\theta\theta}\epsilon^{ab}(e^i)_a(D_\theta e^j)_b\xi_i\xi_j.
\end{multline*}
Each of these terms can be estimated by $\frac{|a|M}{r^5}|\xi|^2$. \qed
\end{proof}

Finally, we are prepared to complete the proof of the partial Morawetz estimate in slowly rotating Kerr spacetimes. From Lemma \ref{a_divJ_rearranged_lem}, since $\Box_g\xi_a-V_a{}^b\xi_b=0$,
\begin{align*}
divJ &= \mathcal{K}_{(\theta)}+\mathcal{K}_{(r)}+\mathcal{K}_{(2)}+\mathcal{K}_{(t)}+\mathcal{K}_{(a)}-\frac12\Box_g w |\xi|^2 \\
&= \left(\mathcal{K}_{(\theta)}-(2-\epsilon)\left(-\frac{u\pd_r v}{r^2}\right)(\xi_1)^2\right) 
+\left(\mathcal{K}_{(r)}+(2-\epsilon)\left(-\frac{u\pd_r v}{r^2}\right)(\xi_1)^2\right) \\
&\hspace{3in}+\mathcal{K}_{(2)}+\mathcal{K}_{(t)}+\mathcal{K}_{(a)}-\frac12\Box_g w|\xi|^2.
\end{align*}
From Lemma \ref{a_K_theta_lem},
$$\frac{\chi_{trap}}{r}|\sla\nabla\xi_1|^2+\chi_{trap}\frac{\cot^2\theta}{r^3}(\xi_1)^2\lesssim_\epsilon \mathcal{K}_{(\theta)}-(2-\epsilon)\left(-\frac{u\pd_r v}{q^2}\right)(\xi_1)^2+\chi_{trap}\frac{|a|M}{r^5}(\xi_2)^2.$$
From Lemma \ref{a_K_r_lem},
$$\frac{M^2}{r^3}\left(1-\frac{r_H}r\right)^2(\pd_r\xi_1)^2+\frac1{r^3}(\xi_1)^2\lesssim \mathcal{K}_{(r)}+(2-\epsilon)\left(-\frac{u\pd_rv}{q^2}\right)(\xi_1)^2+\frac{|a|^3M^3}{r^9}\left(1-\frac{r_H}r\right)^2(\xi_2)^2.$$
From Lemma \ref{a_K_2_lem},
$$\frac{M^2}{r^3}\left(1-\frac{r_H}r\right)^2(\pd_r\xi_2)^2+\frac{\chi_{trap}}{r}|\sla\nabla\xi_2|^2 \lesssim \mathcal{K}_{(2)} +\frac{|a|M}{r^5}(\xi_1)^2.$$
From Lemma \ref{a_K_a_lem},
$$|\mathcal{K}_{(a)}| \lesssim \chi_{trap}\frac{|a|M}{r^3}\left(|\sla\nabla\xi_1|^2+|\sla\nabla\xi_2|^2\right) + \frac{|a|M}{r^5}|\xi|^2.$$
And from Lemma \ref{Kmunu_lem},
$$\frac{M}{r^4}1_{r\ge r_*}((\xi_1)^2+(\xi_2)^2)-q^{-2}V_{\epsilon_{temper}}((\xi_1)^2+(\xi_2)^2)
\lesssim -\frac12\Box_g w|\xi|^2.$$
We ignore the term
$$\mathcal{K}_{(t)}=K_Q^{tt}((\pd_t\xi_1)^2+(\pd_t\xi_2)^2)\approx \chi_{trap}\frac{a^2\sin^2\theta}{r^3}((\pd_t\xi_1)^2+(\pd_t\xi_2)^2),$$
because it has a good sign, but vanishes on the axis and is of order $a^2/M^2$.

Combining each of these estimates, we obtain the following.
\begin{multline*}
\frac{M^2}{r^3}\left(1-\frac{r_H}r\right)^2(\pd_r\xi_1)^2+\frac{\chi_{trap}}{r}|\sla\nabla\xi_1|^2+\frac{1}{r^3}(\xi_1)^2 + \chi_{trap}\frac{\cot^2\theta}{r^3}(\xi_1)^2 \\
+ \frac{M^2}{r^3}\left(1-\frac{r_H}r\right)^2(\pd_r\xi_2)^2+\frac{\chi_{trap}}{r}|\sla\nabla\xi_2|^2+\frac{M}{r^4}1_{r\ge r_*}(\xi_2)^2 \\
-q^{-2}V_{\epsilon_{temper}}((\xi_1)^2+(\xi_2)^2) \\
\lesssim div J +\chi_{trap}\frac{|a|M}{r^3}\left(|\sla\nabla\xi_1|^2+|\sla\nabla\xi_2|^2\right) + \frac{|a|M}{r^5}|\xi|^2.
\end{multline*}
By taking $|a|/M$ sufficiently small, the new terms on the right hand side can be absorbed into terms on the left, except for the term $\frac{|a|M}{r^5}(\xi_2)^2$, which only can be absorbed for $r\ge r_*$. This is the reason for the new error term $-q^{-2}V_{\epsilon_a}(\xi_2)^2$ on the left hand side of the bound stated in the partial Morawetz estimate.

This concludes the proof of the partial Morawetz estimate in slowly rotating Kerr spacetimes. \qed
\end{proof}

\subsection{Estimates using the vectorfields $h\pd_t$}\label{xi_h_dt_sec}

We now derive two spacetime estimates using the vectorfield $h\pd_t$ in the current template $J$. The first estimate will turn out to be the wave map perturbation analogue of the classic energy estimate, and it is obtained by taking $h=1$.

We start by proving an identity for components of $J$ that will appear in the spacetime estimate when applying Proposition \ref{general_divergence_estimate_prop}.
\begin{lemma}\label{h_dt_J_components_lem}
On the event horizon $H_{t_1}^{t_2}$,
$$J^r[h\pd_t] \sim h|D_t\xi|^2$$
and on a timelike hypersurface $\Sigma_t$,
$$-J^t[h\pd_t] \sim h\left[\chi_H|D_r\xi|^2+|D_t\xi|^2+q^{-2}|D_\theta\xi|^2+V^{ab}\xi_a\xi_b\right],$$
where $\chi_H=1-\frac{r_H}r$.
\end{lemma}
\begin{proof}
For simplicity, we take $h=1$. From Lemma \ref{vectorized_current_template_lem},
$$J^\mu[\pd_t] = 2g^{\mu\lambda}D_\lambda\xi\cdot D_t\xi -\delta^\mu{}_t D^\lambda\xi\cdot D_\lambda\xi-\delta^\mu{}_tV^{ab}\xi_a\xi_b.$$
On the event horizon $H_{t_1}^{t_2}$, since $g^{rr}$ vanishes and $g^{rt}>0$, it follows that
$$J^r[\pd_t] = 2g^{rt}|D_t\xi|^2 \sim |D_t\xi|^2.$$
On a timelike hypersurface $\Sigma_t$,
\begin{multline*}
-J^t[\pd_t] = -2g^{tt}|D_t\xi|^2-2g^{tr}D_r\xi\cdot D_t\xi \\
+\left(g^{tt}|D_t\xi|^2+2g^{tr}D_r\xi\cdot D_t\xi+g^{rr}|D_r\xi|^2+q^{-2}|D_\theta\xi|^2+\frac{a^2\sin^2\theta}{q^2}|D_t\xi|^2\right) \\
+V^{ab}\xi_a\xi_b.
\end{multline*}
Thus,
$$-J^t[\pd_t] = \left(-g^{tt}+\frac{a^2\sin^2\theta}{q^2}\right)|D_t\xi|^2+g^{rr}|D_r\xi|^2+q^{-2}|D_\theta\xi|^2+V^{ab}\xi_a\xi_b.$$
\qed
\end{proof}

Next, we define a null pair.
\begin{definition} (null pair)
\begin{align*}
L &= \pd_t+\alpha \pd_r \\
\lbar &= \pd_t-\alpha\pd_r \\
\alpha &= \frac{\Delta}{r^2+a^2}.
\end{align*}
The following identities will often be used.
$$\frac{q^2}{r^2+a^2}g^{rr}=\alpha\text{ and }-\frac{q^2}{r^2+a^2}g^{tt}=\alpha^{-1}.$$
\end{definition}

The null pair shows up in the calculation of the divergence of $J[h\pd_t]$. See the following lemma.
\begin{lemma}\label{h_dt_divJ_lem}
If $h=h(r)$ is constant in the interval $r\in[r_H,r_H+\delh]$, and
$$\Box_g\xi_a-V_a{}^b\xi_b=0,$$
then
$$\frac{q^2}{r^2+a^2}divJ[h\pd_t]= \frac{h'}{2}\left[|D_L\xi|^2-|D_{\lbar}\xi|^2\right].$$
\end{lemma}
\begin{proof}
Recall from Lemma \ref{vectorized_current_template_lem} that
$$divJ[X] = K^{\mu\nu}D_\mu\xi\cdot D_\nu\xi - divXV^{ab}-D_XV^{ab}+2R_{\mu\nu a b}X^\mu\xi^aD^\nu\xi^b.$$
If $X=h\pd_t$, then since
$$div(h\pd_t)=0$$
and
$$D_{\pd_t}V^{ab}=0$$
and
$$2R_{\mu\nu a b}(\pd_t)^\mu=0,$$
the only nonzero term is $K^{\mu\nu}D_\mu\xi\cdot D_\nu\xi$.

Recall also from Lemma \ref{vectorized_current_template_lem} that
\begin{align*}
K^{\mu\nu}&=2\nabla^{(\mu}X^{\nu)}+(w-divX)g^{\mu\nu}
\end{align*}
Since $divX=0$ and $w=0$, we have
$$K^{\mu\nu}=2\nabla^{(\mu}X^{\nu)}= g^{\mu\lambda}\pd_\lambda X^\nu+g^{\nu\lambda}\pd_\lambda X^\mu - X^\lambda\pd_\lambda(g^{\mu\nu}).$$
Since $\pd_t$ is killing, $\pd_t(g^{\mu\nu})=0$. Also, the only component of $X$ is the $t$ component and the only nonzero derivative of that component is the $\pd_r$ derivative, so
\begin{align*}
g^{\mu\lambda}\pd_\lambda X^\nu +g^{\nu\lambda}\pd_\lambda X^\mu &=
g^{\mu r}\pd_r X^\nu +g^{\nu r}\pd_r X^\mu.
\end{align*}
It follows that the only possible nonzero $K^{\mu\nu}$ components are
\begin{align*}
K^{tt} &= 2g^{tr}h' \\
K^{tr}+K^{rt} &= 2g^{rr}h'.
\end{align*}
Since $h'=0$ in the region $r\in[r_H,r_H+\delh]$, the first of these actually vanishes. We conclude that
\begin{align*}
\frac{q^2}{r^2+a^2}divJ[h\pd_t]&=2\frac{q^2}{r^2+a^2}g^{rr}h'D_r\xi\cdot D_t\xi \\
&= 2\alpha h'D_r\xi\cdot D_t\xi \\
&= \frac{h'}2 \left[|\alpha D_r\xi+D_t\xi|^2-|\alpha D_r\xi-D_t\xi|^2\right] \\
&= \frac{h'}2 \left[|D_L\xi|^2-|D_{\lbar}\xi|^2\right]
\end{align*}
This completes the proof. \qed
\end{proof}

Taking $h=1$, we immediately obtain the classic energy estimate.
\begin{proposition}\label{xi_energy_estimate_prop}
\begin{multline*}
\int_{H_{t_1}^{t_2}}|D_t\xi|^2+\int_{\Sigma_{t_2}}\left[\chi_H|D_r\xi|^2+|D_t\xi|^2+q^{-2}|D_\theta\xi|^2+V^{ab}\xi_a\xi_b\right] \\
\lesssim \int_{\Sigma_{t_1}}\left[\chi_H|D_r\xi|^2+|D_t\xi|^2+q^{-2}|D_\theta\xi|^2+V^{ab}\xi_a\xi_b\right] + Err_{nl},
\end{multline*}
where $\chi_H=1-\frac{r_H}{r}$ and
$$Err_{nl} = \int_{t_1}^{t_2}\int_{\Sigma_t} |D_t\xi^a(\Box_g\xi_a-V_a{}^b\xi_b)|.$$
\end{proposition}
\begin{proof}
The estimate follows from Proposition \ref{general_divergence_estimate_prop}, the fact that $divJ=0$ in the linear case (since $h'=0$, see Lemma \ref{h_dt_divJ_lem}) and the estimates for the components of $J[\pd_t]$ in Lemma \ref{h_dt_J_components_lem}. \qed
\end{proof}

Taking $h$ to be a positive function decreasing to zero as $r\rightarrow\infty$ at a particular rate, we obtain a similar but slightly more complicated estimate.
\begin{proposition}\label{xi_h_dt_prop}
Fix $\delp>0$ and let $p\le 2-\delp$. Let $R>r_H+\delh$ be any given radius. Then for all $\epsilon >0$, there is a small constant $c_\epsilon$ and a large constant $C_\epsilon$, such that
\begin{multline*}
\int_{H_{t_1}^{t_2}}|D_t\xi|^2+\int_{\Sigma_{t_2}}r^{p-2}\left[\chi_H|D_r\xi|^2+|D_t\xi|^2+q^{-2}|D_\theta\xi|^2 +V^{ab}\xi_a\xi_b\right] \\
+\int_{t_1}^{t_2}\int_{\Sigma_t\cap\{R+M<r\}}c_\epsilon r^{p-3}|D_{\lbar}\xi|^2 \\
\lesssim \int_{\Sigma_{t_1}}C_\epsilon r^{p-2}\left[\chi_H|D_r\xi|^2+|D_t\xi|^2+q^{-2}|D_\theta\xi|^2+V^{ab}\xi_a\xi_b\right] + Err,
\end{multline*}
where $\chi_H=1-\frac{r_H}{r}$ and
\begin{align*}
Err&=Err_1+Err_{nl} \\
Err_1 &= \int_{t_1}^{t_2}\int_{\Sigma_t\cap\{R<r\}}\epsilon r^{-1}|D_L\xi|^2 \\
Err_{nl} &= \int_{t_1}^{t_2}\int_{\Sigma_t}C_\epsilon r^{p-2}|D_t\xi^a(\Box_g\xi_a-V_a{}^b\xi_b)|.
\end{align*}
\end{proposition}
\begin{proof}
We observe that for all $p<2$ and $\epsilon>0$, there is a function $h$ satisfying the conditions
$$h=1\text{ for }r\le R,$$
$$-h'\ge 0,$$
$$-h'\le 2\epsilon r^{-1},$$
$$\exists c_\epsilon >0 \text{ such that } 2c_\epsilon r^{p-3}\le -h'\text{ for }r\ge R+M,$$
and
$$h=O(r^{p-2})\text{ for large }r.$$
(To construct the function $h$, it is perhaps easier to construct a positive function $-h'$ supported on the interval $r\in[R,\infty)$ and satisfying $\int_R^\infty -h'(r)dr=1$.)

The estimate follows from Proposition \ref{general_divergence_estimate_prop}, the fact that
\begin{align*}
\frac{q^2}{r^2+a^2}divJ &= -\frac{h'}{2}\left(|D_{\lbar}\xi|^2-|D_L\xi|^2\right)+2hD_t\xi^a(\Box_g\xi_a-V_a{}^b\xi_b) \\
&\ge c_\epsilon r^{p-3}|D_{\lbar}\xi|^2-\epsilon r^{-1}|D_L\xi|^2-C_\epsilon r^{p-2}|D_t\xi^a(\Box_g\xi_a-V_a{}^b\xi_b)|
\end{align*}
(see Lemma \ref{h_dt_divJ_lem}) and the estimates for the components of $J[\pd_t]$ in Lemma \ref{h_dt_J_components_lem}. \qed
\end{proof}

\subsection{The Morawetz estimate}\label{xi_morawetz_sec}

Having proved the simpler spacetime estimates (Propositions \ref{xi_energy_estimate_prop} and \ref{xi_h_dt_prop}) we now prove the Morawetz estimate, drawing on the partial Morawetz estimate established earlier.
\begin{proposition}\label{xi_morawetz_estimate_prop} (Morawetz estimate)
Suppose $|a|/M$ is sufficiently small. Then
\begin{align*}
&\int_{H_{t_1}^{t_2}}q^{-2}|D_\theta \xi|^2+V^{ab}\xi_a\xi_b
+\int_{\Sigma_{t_2}} |D_L\xi|^2+q^{-2}|D_\theta\xi|^2+V^{ab}\xi_a\xi_b+r^{-2}|\xi|^2 +\frac{M^2}{r^2}|D_r\xi|^2  \\
&+\int_{t_1}^{t_2}\int_{\Sigma_t} \left[\frac{M^2}{r^3}(\pd_r\xi_1)^2+\chi_{trap}\left(\frac{M^2}{r^3}(\pd_t\xi_1)^2+\frac{1}{r}|\sla\nabla\xi_1|^2+\frac{\cot^2\theta}{r^3}(\xi_1)^2\right)+\frac{1}{r^3}(\xi_1)^2\right. \\
&\hspace{1.25in}\left.+ \frac{M^2}{r^3}(\pd_r\xi_2)^2+\chi_{trap}\left(\frac{M^2}{r^3}(\pd_t\xi_2)^2+\frac{1}{r}|\sla\nabla\xi_2|^2\right)+\frac{M}{r^4}(\xi_2)^2\right] \\
&\lesssim \int_{\Sigma_{t_1}} |D_L\xi|^2+q^{-2}|D_\theta\xi|^2+V^{ab}\xi_a\xi_b+r^{-2}|\xi|^2 +\frac{M^2}{r^2}|D_r\xi|^2  + Err,
\end{align*}
where $\chi_{trap}=\left(1-\frac{r_{trap}}{r}\right)^2$ and
\begin{align*}
Err &= Err_1+Err_{nl} \\
Err_1 &= \int_{H_{t_1}^{t_2}}|D_t\xi|^2 + \int_{\Sigma_{t_2}}r^{-1}|\xi\cdot D_L\xi|+\frac{M^2}{r^2}\left[\chi_H|D_r\xi|^2+|D_t\xi|^2+r^{-2}|\xi|^2\right] \\
Err_{nl} &= \int_{t_1}^{t_2}\int_{\Sigma_t}|(2D_X\xi^a+w\xi^a)(\Box_g\xi_a-V_a{}^b\xi_b)|,
\end{align*}
where $\chi_H=1-\frac{r_H}{r}$.
\end{proposition}
\begin{proof}
Let
\begin{align*}
X &= X_{\epsilon_{temper}}+\epsilon_{redshift}Y+\pd_t \\
w &= w_{\epsilon_{temper}}+\epsilon_{\pd_t}w_{\pd_t},
\end{align*}
where $X_{\epsilon_{temper}}$ and $w_{\epsilon_{temper}}$ are the vectorfield and function used in the proof of the partial Morawetz estimate (the dependence on the parameter $\epsilon_{temper}$ is made more explicit for the argument that will follow), $Y$ is a redshift vectorfield, and $w_{\pd_t}$ is a new function to be defined in the following lemma.

First, we establish an estimate for the bulk term.
\begin{lemma}
If $X$ and $w$ are as defined above and $m$ is the one-form used in the proof of the partial Morawetz estimate, then
\begin{multline*}
\int_{t_1}^{t_2}\int_{\Sigma_t} \left[\frac{M^2}{r^3}(\pd_r\xi_1)^2+\chi_{trap}\left(\frac{M^2}{r^3}(\pd_t\xi_1)^2+\frac{1}{r}|\sla\nabla\xi_1|^2+\frac{\cot^2\theta}{r^3}(\xi_1)^2\right)+\frac{1}{r^3}(\xi_1)^2\right. \\
\hspace{1.25in}\left.+ \frac{M^2}{r^3}(\pd_r\xi_2)^2+\chi_{trap}\left(\frac{M^2}{r^3}(\pd_t\xi_2)^2+\frac{1}{r}|\sla\nabla\xi_2|^2\right)+\frac{M}{r^4}(\xi_2)^2\right] \\
\lesssim \int_{t_1}^{t_2}\int_{\Sigma_t}div J[X,w,m]
\end{multline*}
\end{lemma}
\begin{proof}
To prove this estimate, we start with the partial Morawetz estimate (Proposition \ref{partial_morawetz_kerr_prop}) and make a few slight modifications.  We have
\begin{multline*}
\frac{M^2}{r^3}\left(1-\frac{r_H}r\right)^2(\pd_r\xi_1)^2+\frac{\chi_{trap}}{r}|\sla\nabla\xi_1|^2+\frac{1}{r^3}(\xi_1)^2 + \chi_{trap}\frac{\cot^2\theta}{r^3}(\xi_1)^2 \\
+ \frac{M^2}{r^3}\left(1-\frac{r_H}r\right)^2(\pd_r\xi_2)^2+\frac{\chi_{trap}}{r}|\sla\nabla\xi_2|^2+\frac{M}{r^4}1_{r\ge r_*}(\xi_2)^2 \\
-q^{-2}V_{\epsilon_{temper}}((\xi_1)^2+(\xi_2)^2) -q^{-2}V_{\epsilon_a}(\xi_2)^2\\
\lesssim div J[X_{\epsilon_{temper}},w_{\epsilon_{temper}},m].
\end{multline*}
By applying a small constant times the redshift vectorfield $Y$, the degeneracy of the $(\pd_r\xi_i)^2$ terms near the horizon can be removed without significant consequence. That is
\begin{multline*}
\frac{M^2}{r^3}(\pd_r\xi_1)^2+\frac{\chi_{trap}}{r}|\sla\nabla\xi_1|^2+\frac{1}{r^3}(\xi_1)^2 + \chi_{trap}\frac{\cot^2\theta}{r^3}(\xi_1)^2 \\
+ \frac{M^2}{r^3}(\pd_r\xi_2)^2+\frac{\chi_{trap}}{r}|\sla\nabla\xi_2|^2+\frac{M}{r^4}1_{r\ge r_*}(\xi_2)^2 \\
-q^{-2}V_{\epsilon_{temper}}((\xi_1)^2+(\xi_2)^2) -q^{-2}V_{\epsilon_a}(\xi_2)^2\\
\lesssim div J[X_{\epsilon_{temper}}+\epsilon_{redshift}Y,w_{\epsilon_{temper}},m].
\end{multline*}
It is worth mention that $Y$ is supported near the event horizon and $Y^r<0$.

Next, by choosing $\epsilon_{temper}$ and $\epsilon_a$ sufficiently small and applying a local Hardy estimate, we obtain in an integrated sense
\begin{multline*}
\int_{t_1}^{t_2}\int_{\Sigma_t}\left[\frac{M^2}{r^3}(\pd_r\xi_1)^2+\frac{\chi_{trap}}{r}|\sla\nabla\xi_1|^2+\frac{1}{r^3}(\xi_1)^2 + \chi_{trap}\frac{\cot^2\theta}{r^3}(\xi_1)^2\right. \\
\left.+\frac{M^2}{r^3}(\pd_r\xi_2)^2+\frac{\chi_{trap}}{r}|\sla\nabla\xi_2|^2+\frac{M}{r^4}(\xi_2)^2\right] \\
\lesssim \int_{t_1}^{t_2}\int_{\Sigma_t}div J[X_{\epsilon_{temper}}+\epsilon_{redshift}Y,w_{\epsilon_{temper}},m].
\end{multline*}

Next, we add to $w$ a small constant $\epsilon_{\pd_t}$ times a new function $w_{\pd_t}$. Note that according to Lemma \ref{vectorized_current_template_lem}, in the linear case,
$$divJ[0,w_{\pd_t}] = w_{\pd_t}D^\lambda\xi\cdot D_\lambda\xi-\frac12\Box_g w_{\pd_t}+w_{\pd_t}V^{ab}\xi_a\xi_b.$$
One particular term in the above formula is $w_{\pd_t}g^{tt}|D_t\xi|^2=w_{\pd_t}g^{tt}[(\pd_t\xi_1)^2+(\pd_t\xi_2)^2]$. Thus, if $w_{\pd_t}\le 0$ so that $w_{\pd_t}g^{tt}\ge 0$, then the result is added control of the time derivatives. As long as $w_{\pd_t}$ vanishes to second order at the trapping radius and decays like $O(M^2/r^2)$, by taking $\epsilon_{\pd_t}$ is sufficiently small, the remaining terms have no significant effect.
\begin{multline*}
\int_{t_1}^{t_2}\int_{\Sigma_t} \left[\frac{M^2}{r^3}(\pd_r\xi_1)^2+\chi_{trap}\left(\frac{M^2}{r^3}(\pd_t\xi_1)^2+\frac{1}{r}|\sla\nabla\xi_1|^2+\frac{\cot^2\theta}{r^3}(\xi_1)^2\right)+\frac{1}{r^3}(\xi_1)^2\right. \\
\hspace{1.25in}\left.+ \frac{M^2}{r^3}(\pd_r\xi_2)^2+\chi_{trap}\left(\frac{M^2}{r^3}(\pd_t\xi_2)^2+\frac{1}{r}|\sla\nabla\xi_2|^2\right)+\frac{M}{r^4}(\xi_2)^2\right] \\
\lesssim \int_{t_1}^{t_2}\int_{\Sigma_t}div J[X_{\epsilon_{temper}}+\epsilon_{redshift}Y,w_{\epsilon_{temper}}+\epsilon_{\pd_t}w_{\pd_t},m].
\end{multline*}
Finally, since
$$divJ[\pd_t]=0,$$
the estimate is not affected by adding $\pd_t$ to the vectorfield $X$. (The purpose of $\pd_t$ is for the boundary terms.)

This concludes the proof of the lemma. \qed
\end{proof}

Now, all that remains to be done is to investigate the boundary terms. This is the purpose of the remaining few calculations.

We first approximate the vectorfield $X$ and function $w$ in order to compute the boundary terms. Since $X_{\epsilon_{temper}}=uv\pd_r = \frac{u\Delta}{(r^2+a^2)^2}\pd_r$, and since $\pd_ru=2r$ and $w_{\epsilon_{temper}}=\frac{2r\Delta}{(r^2+a^2)^2}$ for $r>r_*$, it follows that for $r>r_*$,
\begin{align*}
X&=\frac{(r^2+a^2-c^2)\Delta}{(r^2+a^2)^2}\pd_r+\pd_t=L+O(M^2/r^2)\pd_r, \\
w&=\frac{2r\Delta}{(r^2+a^2)^2} +\epsilon_{\pd_t}w_{\pd_t}=\frac{2r\Delta}{(r^2+a^2)^2}+O(M^2/r^3).
\end{align*}
We have the following lemma.
\begin{lemma}
\begin{multline*}
-J^t\left[L,\frac{2r\alpha}{r^2+a^2}\right] 
= \frac{r^2+a^2}{q^2}\left(1-\alpha\frac{a^2\sin^2\theta}{r^2+a^2}\right)\alpha^{-1}|D_L\xi|^2+\frac1{q^2}|D_\theta\xi|^2+V^{ab}\xi_a\xi_b \\
+\frac{\alpha+r\alpha'}{q^2}|\xi|^2-\frac1{q^2}\pd_r(r\alpha|\xi|^2)+Err',
\end{multline*}
where
$$|Err'|\lesssim r^{-1}|\xi\cdot D_L\xi|+\frac{a^2}{r^2}\left[\left(1-\frac{r_H}r\right)^2|D_r\xi|^2+r^{-2}|\xi|^2\right].$$
\end{lemma}
\begin{proof}
By a direct calculation,
\begin{align*}
-J^t[L] &= -2D^t\xi\cdot D_L\xi+L^tD^\lambda\xi\cdot D_\lambda\xi+L^tV^{ab}\xi_a\xi_b \\
&= -2(g^{tt}+{}^{(Q)}g^{tt})D_t\xi\cdot D_L\xi+\left[(g^{tt}+{}^{(Q)}g^{tt})|D_t\xi|^2+g^{rr}|D_r\xi|^2+{}^{(Q)}g^{\theta\theta}|D_\theta\xi|^2\right]+V^{ab}\xi_a\xi_b \\
&= -(g^{tt}+{}^{(Q)}g^{tt})|D_t\xi|^2-2\alpha (g^{tt}-{}^{(Q)}g^{tt})D_t\xi\cdot D_r\xi+g^{rr}|D_r\xi|^2+{}^{(Q)}g^{\theta\theta}|D_\theta\xi|^2+V^{ab}\xi_a\xi_b \\
&= \frac{r^2+a^2}{q^2}\left(1-\alpha\frac{a^2\sin^2\theta}{r^2+a^2}\right)\alpha^{-1}|D_t\xi|^2+\frac{r^2+a^2}{q^2}\left(1-\alpha\frac{a^2\sin^2\theta}{r^2+a^2}\right)D_t\xi\cdot D_r\xi  \\
&\hspace{3in}+\frac{r^2+a^2}{q^2}\alpha|D_r\xi|^2+\frac1{q^2}|D_\theta\xi|^2+V^{ab}\xi_a\xi_b \\
&= \frac{r^2+a^2}{q^2}\left(1-\alpha\frac{a^2\sin^2\theta}{r^2+a^2}\right)\alpha^{-1}|D_L\xi|^2+\frac{r^2+a^2}{q^2}\alpha^2\frac{a^2\sin^2\theta}{r^2+a^2}|D_r\xi|^2+\frac1{q^2}|D_\theta\xi|^2+V^{ab}\xi_a\xi_b.
\end{align*}
Also,
\begin{align*}
-J^t\left[0,\frac{2r\alpha}{r^2+a^2}\right] &= -\frac{2r\alpha}{r^2+a^2}\xi\cdot D^t\xi \\
&= -\frac{2r\alpha}{r^2+a^2}(g^{tt}+{}^{(Q)}g^{tt})\xi\cdot D_t\xi \\
&= \frac{2r}{q^2}\left(1-\alpha\frac{a^2\sin^2\theta}{r^2+a^2}\right)\xi\cdot D_t\xi \\
&= \frac{2r}{q^2}\left(1-\alpha\frac{a^2\sin^2\theta}{r^2+a^2}\right)\xi\cdot D_L\xi-\frac{2r\alpha}{q^2}\left(1-\alpha\frac{a^2\sin^2\theta}{r^2+a^2}\right)\xi\cdot D_r\xi \\
&= \frac{2r}{q^2}\left(1-\alpha\frac{a^2\sin^2\theta}{r^2+a^2}\right)\xi\cdot D_L\xi-\frac{2r\alpha}{q^2}\xi\cdot D_r\xi +\frac{2r\alpha^2 a^2\sin^2\theta}{q^2(r^2+a^2)}\xi\cdot D_r\xi  \\
&= \frac{2r}{q^2}\left(1-\alpha\frac{a^2\sin^2\theta}{r^2+a^2}\right)\xi\cdot D_L\xi +\left(-\frac{1}{q^2}\pd_r(r\alpha |\xi|^2)+\frac{\alpha+r\alpha'}{q^2}|\xi|^2\right) \\
&\hspace{3.5in}+\frac{2r\alpha^2 a^2\sin^2\theta}{q^2(r^2+a^2)}\xi\cdot D_r\xi.
\end{align*}
Comparing to the identity given for $-J^t\left[L,\frac{2r\alpha}{r^2+a^2}\right]$, we see that
$$Err' = \frac{2r}{q^2}\left(1-\alpha\frac{a^2\sin^2\theta}{r^2+a^2}\right)\xi\cdot D_L\xi+ \alpha^2\frac{a^2\sin^2\theta}{q^2}|D_r\xi|^2 +\frac{2r\alpha^2 a^2\sin^2\theta}{q^2(r^2+a^2)}\xi\cdot D_r\xi.$$
Thus,
$$Err'\lesssim r^{-1}|\xi\cdot D_L\xi| +\left(1-\frac{r_H}r\right)\frac{a^2}{r^2}|D_r\xi|^2 +\frac{a^2}{r^2}|\xi|^2.$$
This concludes the proof of the lemma. \qed
\end{proof}

Given that
\begin{align*}
X-L&=O(M^2/r^2)\pd_r, \\
w-\frac{2r\alpha}{r^2+a^2}&=O(M^2/r^2)r^{-1}.
\end{align*}
We can estimate the remainder
$$\left|J^t[X,w]-J^t\left[L,\frac{2r\alpha}{r^2+a^2}\right]\right| \lesssim \frac{M^2}{r^2}\left[\left(1-\frac{r_H}r\right)|D_r\xi|^2+|D_t\xi|^2+r^{-2}|\xi|^2\right].$$

But it is also important to specifically take into account the effect of the redshift vectorfield $Y$ near the horizon, because it will remove the degeneracy of the term $\chi_H|D_r\xi|^2$ on the spacelike hypersurface $\Sigma_{t_2}$.
\begin{lemma}
On the event horizon $H_{t_1}^{t_2}$,
$$q^{-2}|D_\theta \xi|^2+V^{ab}\xi_a\xi_b \lesssim J^r[Y] + |D_t\xi|^2$$
and on the timelike hypersurface $\Sigma_{t_2}$,
$$\frac{M^2}{r^2}|D_r\xi|^2\lesssim -J^t[Y] +\frac{M^2}{r^2}\left[\chi_H|D_r\xi|^2+|D_t\xi|^2\right],$$
where $\chi_H=1-\frac{r_H}r$.
\end{lemma}
\begin{proof}
Since $Y\approx -\pd_r+c\pd_t$ near the event horizon, given the estimates for $J^\mu[\pd_t]$ in Lemma \ref{h_dt_J_components_lem}, it suffices to compute the components $J^r[-\pd_r]$ and $-J^t[-\pd_r]$.

From Lemma \ref{vectorized_current_template_lem},
$$J^\mu[-\pd_r] = -2g^{\mu\lambda}D_\lambda\xi\cdot D_r\xi +\delta^\mu{}_r D^\lambda\xi\cdot D_\lambda\xi+\delta^\mu{}_rV^{ab}\xi_a\xi_b.$$
Therefore,
\begin{multline*}
J^r[-\pd_r] = -2g^{rr}|D_r\xi|^2-2g^{rt}D_t\xi\cdot D_r\xi \\
+\left(g^{tt}|D_t\xi|^2+2g^{tr}D_r\xi\cdot D_t\xi+g^{rr}|D_r\xi|^2+q^{-2}|D_\theta\xi|^2+\frac{a^2\sin^2\theta}{q^2}|D_t\xi|^2\right) \\
+V^{ab}\xi_a\xi_b.
\end{multline*}
On the event horizon $H_{t_1}^{t_2}$, since $g^{rr}$ vanishes,
$$J^r[-\pd_r]=q^{-2}|D_\theta\xi|^2+V^{ab}\xi_a\xi_b +\left(g^{tt}+\frac{a^2\sin^2\theta}{q^2}\right)|D_t\xi|^2.$$
This implies the first estimate.

On the timelike hypersurface $\Sigma_{t_2}$,
$$-J^t[-\pd_r] = 2g^{tr}|D_r\xi|^2+2g^{tt}D_t\xi\cdot D_r\xi.$$
Since $g^{tr}>0$ near the horizon, this implies the second estimate. \qed
\end{proof}
This accounts for all of the boundary terms in the Morawetz estimate. \qed
\end{proof}

\section{Estimates for the $(\phi,\psi)$ System}\label{phi_psi_sec}

The purpose of this section is to prove the energy estimate (Proposition \ref{translated_energy_estimate_prop}) and $p$-weighted estimates (Proposition \ref{p_estimates_prop}) for the $(\phi,\psi)$ system. The energy estimate is simply a translation of Proposition \ref{xi_energy_estimate_prop}, which is in terms of $\xi_a$. The $p$-weighted estimates are a combination of three other estimates, the $h\pd_t$ estimate (Proposition \ref{translated_h_dt_prop}, translated from \ref{xi_h_dt_prop}), the Morawetz estimate (Proposition \ref{translated_morawetz_prop}, translated from \ref{xi_morawetz_estimate_prop}), and the incomplete $p$-weighted estimate (Proposition \ref{incomplete_p_estimates_prop}).

This section is split into three parts. In \S\ref{phi_psi_translating_sec}, the estimates from the previous section, which are written in terms of $\xi_a$ are translated to be in terms of $(\phi,\psi)$. In \S\ref{phi_psi_incomplete_sec}, the incomplete $p$-estimates are proved. Finally, in \S\ref{phi_psi_p_ee_sec} the $p$-estimates are proved.

Recall that the $\xi_a$ system corresponds to the linearization
\begin{align*}
X &= A-A\xi_2 \\
Y &= B+A\xi_1
\end{align*}
of the full nonlinear wave map system (\ref{wm_X_eqn}-\ref{wm_Y_eqn}). To solve the full nonlinear wave map system, we will instead use the linearization
\begin{align*}
X &= A+A\phi \\
Y &= B+A^2\psi.
\end{align*}
The motivation for this second linearization is given in \S\ref{intro_phi_psi_system_sec}. For this reason, from now on, we assume that
$$(\xi_1,\xi_2)=(A\psi,-\phi).$$

\subsection{Translating estimates from the $\xi_a$ system to the $(\phi,\psi)$ system}\label{phi_psi_translating_sec}

To begin, we prove Lemmas \ref{phi_psi_le_xi_lem}-\ref{translate_nl_lem}, which allow us to translate estimates for the $\xi_a$ system into estimates for the $(\phi,\psi)$ system. To prove these lemmas, we will make repeated use of the following calculations.
\begin{align}
|D_\mu\xi|^2 &= \left(\pd_\mu\xi_1+\frac{\pd_\mu B}{A}\xi_2\right)^2+\left(\pd_\mu\xi_2-\frac{\pd_\mu B}{A}\xi_1\right)^2 \nonumber \\
&= \left(A\pd_\mu \psi +\frac{\pd_\mu A}{A}A\psi-\frac{\pd_\mu B}{A}\phi\right)^2+\left(\pd_\mu \phi+\frac{\pd_\mu B}{A}A\psi\right)^2 \nonumber \\
&= \left(A\pd_\mu \psi +\frac{\pd_\mu A_1}{A_1}A\psi+\frac{\pd_\mu A_2}{A_2}A\psi-\frac{\pd_\mu B}{A}\phi\right)^2+\left(\pd_\mu \phi+\frac{\pd_\mu B}{A}A\psi\right)^2.\label{D_mu_xi_eqn}
\end{align}
\begin{align}
V^{ab}\xi_a\xi_b &= g^{\mu\nu}\left(\frac{\pd_\mu A}{A}\xi_1+\frac{\pd_\mu B}{A}\xi_2\right)\left(\frac{\pd_\nu A}{A}\xi_1+\frac{\pd_\nu B}{A}\xi_2\right) \nonumber \\
&= \frac{\Delta}{q^2}\left(\frac{\pd_r A}{A}\xi_1+\frac{\pd_r B}{A}\xi_2\right)^2+\frac{1}{q^2}\left(\frac{\pd_\theta A}{A}\xi_1+\frac{\pd_\theta B}{A}\xi_2\right)^2 \nonumber \\
&= \frac{\Delta}{q^2}\left(\frac{\pd_r A}{A}A\psi-\frac{\pd_r B}{A}\phi\right)^2+\frac{1}{q^2}\left(\frac{\pd_\theta A}{A}A\psi-\frac{\pd_\theta B}{A}\phi\right)^2 \nonumber \\
&= \frac{\Delta}{q^2}\left(\frac{\pd_r A_1}{A_1}A\psi+\frac{\pd_rA_2}{A_2}A\psi-\frac{\pd_r B}{A}\phi\right)^2+\frac{1}{q^2}\left(\frac{\pd_\theta A_1}{A_1}A\psi+\frac{\pd_\theta A_2}{A_2}A\psi-\frac{\pd_\theta B}{A}\phi\right)^2. \label{V_xi_xi_eqn}
\end{align}
\begin{equation}\label{D_r_A1_eqn}
\frac{\pd_r A_1}{A_1}=\frac{2r}{r^2+a^2}
\end{equation}
\begin{equation}\label{D_theta_A1_eqn}
\frac{\pd_\theta A_1}{A_1}=2\cot\theta
\end{equation}
\begin{equation}\label{D_r_A2_B_eqn}
\left(\frac{\pd_r A_2}{A}\right)^2+\left(\frac{\pd_r B}{A}\right)^2\lesssim \frac{a^2}{M^2}r^{-2}
\end{equation}
\begin{equation}\label{D_theta_A2_B_eqn}
\left(\frac{\pd_\theta A_2}{A}\right)^2+\left(\frac{\pd_\theta B}{A}\right)^2\lesssim \frac{a^2}{M^2}
\end{equation}

The first lemma allows us to estimate $(\phi,\psi)$ terms by $\xi_a$ terms.
\begin{lemma}\label{phi_psi_le_xi_lem}
$$(\pd_t\phi)^2+A^2(\pd_t\psi)^2=|D_t\xi|^2,$$
$$(\pd_r\phi)^2+A^2\left(\pd_r\psi+\frac{2r}{r^2+a^2}\psi\right)^2\lesssim |D_r\xi|^2+\frac{a^2}{M^2}r^{-2}(\phi^2+A^2\psi^2),$$
$$(\pd_\theta\phi)^2+A^2\left(\pd_\theta\psi+2\cot\theta\psi\right)^2 \lesssim |D_\theta\xi|^2+\frac{a^2}{M^2}(\phi^2+A^2\psi^2)$$
$$A^2\chi_H\left(\frac{2r}{r^2+a^2}\psi\right)^2+A^2q^{-2}\left(2\cot\theta\psi\right)^2\lesssim V^{ab}\xi_a\xi_b+\frac{a^2}{M^2}r^{-2}(\phi^2+A^2\psi^2),$$
where $\chi_H=1-\frac{r_H}r$.
\end{lemma}
\begin{proof}
From equation (\ref{D_mu_xi_eqn}),
\begin{align*}
|D_\mu\xi|^2 = \left(A\pd_\mu \psi +\frac{\pd_\mu A_1}{A_1}A\psi+\frac{\pd_\mu A_2}{A_2}A\psi-\frac{\pd_\mu B}{A}\phi\right)^2+\left(\pd_\mu \phi+\frac{\pd_\mu B}{A}A\psi\right)^2.
\end{align*}
In the case $\mu=t$, since $\pd_tA_1=\pd_tA_2=\pd_t B=0$, we get the identity
$$(\pd_t\phi)^2+A^2(\pd_t\psi)^2=|D_t\xi|^2.$$
Given equation (\ref{D_r_A1_eqn}) and estimate (\ref{D_r_A2_B_eqn}),
\begin{align*}
(\pd_r\phi)^2+A^2\left(\pd_r\psi+\frac{2r}{r^2+a^2}\psi\right)^2 &= (\pd_r\phi)^2+\left(A\pd_r\psi +\frac{\pd_r A_1}{A_1}A\psi\right)^2 \\
&\lesssim |D_r\xi|^2+\frac{a^2}{M^2}r^{-2}(\phi^2+A^2\psi^2).
\end{align*}
Similarly, given equation (\ref{D_theta_A1_eqn}) and estimate (\ref{D_theta_A2_B_eqn}),
\begin{align*}
(\pd_\theta\phi)^2+A^2\left(\pd_\theta\psi+2\cot\theta\psi\right)^2 &= (\pd_\theta\phi)^2+\left(A\pd_\theta\psi +\frac{\pd_\theta A_1}{A_1}A\psi\right)^2 \\
&\lesssim |D_\theta\xi|^2+\frac{a^2}{M^2}(\phi^2+A^2\psi^2).
\end{align*}
Finally, from equation (\ref{V_xi_xi_eqn}),
\begin{align*}
V^{ab}\xi_a\xi_b &= \frac{\Delta}{q^2}\left(\frac{\pd_r A_1}{A_1}A\psi+\frac{\pd_rA_2}{A_2}A\psi-\frac{\pd_r B}{A}\phi\right)^2+\frac{1}{q^2}\left(\frac{\pd_\theta A_1}{A_1}A\psi+\frac{\pd_\theta A_2}{A_2}A\psi-\frac{\pd_\theta B}{A}\phi\right)^2.
\end{align*}
Using equations (\ref{D_r_A1_eqn}) and (\ref{D_theta_A1_eqn}) together with estimates (\ref{D_r_A2_B_eqn}) and (\ref{D_theta_A2_B_eqn}) again, we conclude
\begin{align*}
A^2\chi_H\left(\frac{2r}{r^2+a^2}\psi\right)^2+A^2q^{-2}\left(2\cot\theta\psi\right)^2 &\lesssim \frac{\Delta}{q^2}\left(\frac{\pd_r A_1}{A_1}A\psi\right)^2+\frac1{q^2}\left(\frac{\pd_\theta A_1}{A_1}A\psi\right)^2\\
&\lesssim V^{ab}\xi_a\xi_b+\frac{a^2}{M^2}r^{-2}(\phi^2+A^2\psi^2).
\end{align*}
 \qed
\end{proof}

The next lemma allows us to estimate certain $\xi_a$ terms by $(\phi,\psi)$ terms.
\begin{lemma}\label{xi_le_phi_psi_1_lem}
$$|D_t\xi|^2 = (\pd_t\phi)^2+A^2(\pd_t\psi)^2$$
$$|D_r\xi|^2 \lesssim (\pd_r\phi)^2+A^2\left(\pd_r\psi+\frac{2r}{r^2+a^2}\psi\right)^2+\frac{a^2}{M^2}r^{-2}(\phi^2+A^2\psi^2).$$
\end{lemma}
\begin{proof}
The identity for $|D_t\xi|^2$ was already proved in the previous lemma, but is simply restated in this lemma for the sake of completeness. 

To prove the estimate for $|D_r\xi|^2$, we use equation (\ref{D_mu_xi_eqn}), then equation (\ref{D_r_A1_eqn}) and finally estimate (\ref{D_r_A2_B_eqn}).
\begin{align*}
|D_r\xi|^2 &= \left(A\pd_r \psi +\frac{\pd_r A_1}{A_1}A\psi+\frac{\pd_r A_2}{A_2}A\psi-\frac{\pd_r B}{A}\phi\right)^2+\left(\pd_r \phi+\frac{\pd_r B}{A}A\psi\right)^2 \\
&= \left(A\pd_r \psi +\frac{2r}{r^2+a^2}A\psi+\frac{\pd_r A_2}{A_2}A\psi-\frac{\pd_r B}{A}\phi\right)^2+\left(\pd_r \phi+\frac{\pd_r B}{A}A\psi\right)^2 \\
&\lesssim (\pd_r\phi)^2+A^2\left(\pd_r\psi+\frac{2r}{r^2+a^2}\psi\right)^2+\frac{a^2}{M^2}r^{-2}(\phi^2+A^2\psi^2).
\end{align*}
 \qed
\end{proof}

Note that the previous lemma did not estimate $|D_\theta\xi|^2$ or $V^{ab}\xi_a\xi_b$. The reason is that both of these terms are singluar on the axis and would require a term like $A^2\cot^2\theta \psi^2$ on the right hand side of an estimate. It turns out that if these terms are combined in just the right way, then up to a divergence term (with only a $\theta$ component), the singularities cancel. This fact should be compared to the remark following the proof of Lemma \ref{0_K_theta_lem}.
\begin{lemma}\label{xi_le_phi_psi_2_lem}
For an arbitrary function $f(r)$,
$$\int_{\Sigma_t}f(r)\left[q^{-2}|D_\theta\xi|^2+V^{ab}\xi_a\xi_b\right] \lesssim \int_{\Sigma_t}f(r)\left[q^{-2}(\pd_\theta\phi)^2+r^{-2}\phi^2\right] +\int_{\tilde{\Sigma}_t}f(r)\left[q^{-2}(\pd_\theta\psi)^2+r^{-2}\psi^2\right].$$
\end{lemma}
\begin{proof}
From equation (\ref{D_mu_xi_eqn}),
\begin{align*}
|D_\theta\xi|^2 &= \left(A\pd_\theta \psi +\frac{\pd_\theta A}{A}A\psi-\frac{\pd_\theta B}{A}\phi\right)^2+\left(\pd_\theta \phi+\frac{\pd_\theta B}{A}A\psi\right)^2,
\end{align*}
and from equation (\ref{V_xi_xi_eqn}),
\begin{align*}
V^{ab}\xi_a\xi_b &= \frac{\Delta}{q^2}\left(\frac{\pd_r A}{A}A\psi-\frac{\pd_r B}{A}\phi\right)^2+\frac{1}{q^2}\left(\frac{\pd_\theta A}{A}A\psi-\frac{\pd_\theta B}{A}\phi\right)^2.
\end{align*}
The only challenge arises when dealing with the terms that contain the factor $\frac{\pd_\theta A}{A}$, because this quantity diverges on the axis. We will now show that after an integration by parts on the sphere, the divergent parts cancel. For the sake of simplicity, we take $f(r)=1$.
\begin{multline*}
\int_{\Sigma_t}q^{-2} \left(A\pd_\theta \psi +\frac{\pd_\theta A}{A}A\psi-\frac{\pd_\theta B}{A}\phi\right)^2+q^{-2}\left(\frac{\pd_\theta A}{A}A\psi-\frac{\pd_\theta B}{A}\phi\right)^2 \\
=2\pi\int_{r_H}^{\infty}\int_0^\pi \left(A\pd_\theta \psi +\frac{\pd_\theta A}{A}A\psi-\frac{\pd_\theta B}{A}\phi\right)^2+\left(\frac{\pd_\theta A}{A}A\psi-\frac{\pd_\theta B}{A}\phi\right)^2 \sin\theta d\theta dr.
\end{multline*}
The singular part is contained in the following expression, which we evaluate by expanding the squares and integrating by parts.
\begin{multline*}
\int_0^\pi \left(A\pd_\theta \psi+\frac{\pd_\theta A}{A}A\psi\right)^2+\left(\frac{\pd_\theta A}{A}A\psi\right)^2 \sin\theta d\theta \\
= \int_0^\pi \left[(\pd_\theta\psi)^2+2\frac{\pd_\theta A}{A}\psi\pd_\theta\psi +2\left(\frac{\pd_\theta A}{A}\right)^2\psi^2\right]A^2\sin\theta d\theta \\
= \int_0^\pi \left[(\pd_\theta\psi)^2-\frac{1}{A^2\sin\theta}\pd_\theta\left(A^2\sin\theta \frac{\pd_\theta A}{A}\right)\psi^2 +2\left(\frac{\pd_\theta A}{A}\right)^2\psi^2\right]A^2\sin\theta d\theta.
\end{multline*}
It now suffices to show that the quantity
$$-\frac{1}{A^2\sin\theta}\pd_\theta\left(A^2\sin\theta \frac{\pd_\theta A}{A}\right)+2\left(\frac{\pd_\theta A}{A}\right)^2$$
is regular on the axis. For the Schwarzschild case, where $A=r^2\sin^2\theta$, this expression is rather simple.
\begin{align*}
-\frac{1}{A^2\sin\theta}&\pd_\theta\left(A^2\sin\theta \frac{\pd_\theta A}{A}\right)+2\left(\frac{\pd_\theta A}{A}\right)^2 \\
&=-\frac{1}{r^4\sin^5\theta}\pd_\theta\left(r^4\sin^5\theta (2\cot\theta)\right)+2(2\cot\theta)^2 \\
&= -\frac{1}{\sin^5\theta}\pd_\theta(2\sin^4\theta\cos\theta)+8\cot^2\theta \\
&= -8\cot^2\theta +2 + 8\cot^2\theta \\
&= 2.
\end{align*}
\begin{remark}
Recall that a term was exchanged between Lemmas \ref{0_K_theta_lem} and \ref{0_K_r_lem} with a factor of $2-\epsilon$. Furthermore, there was a remark following the proof of Lemma \ref{0_K_theta_lem} claiming that $\epsilon=0$ corresponds to the $(\phi,\psi)$ system. The above calculation that yields the number $2$ is directly related to the $2-\epsilon$ factor in the exchanged term.
\end{remark}

In the Kerr case, since $A_1=(r^2+a^2)\sin^2\theta$, we also have
$$-\frac{1}{A_1^2\sin\theta}\pd_\theta\left(A_1^2\sin\theta \frac{\pd_\theta A_1}{A_1}\right)+2\left(\frac{\pd_\theta A_1}{A_1}\right)^2=2.$$
This fact, together with the fact that $\frac{\pd_\theta A_2}{A_2}=O(a^2/r^2)\sin\theta$, implies
\begin{align*}
-\frac{1}{A^2\sin\theta}&\pd_\theta\left(A^2\sin\theta \frac{\pd_\theta A}{A}\right)+2\left(\frac{\pd_\theta A}{A}\right)^2 \\
&= -\frac{1}{A_1^2A_2^2\sin\theta}\pd_\theta\left(A_1^2A_2^2\sin\theta \left(\frac{\pd_\theta A_1}{A_1}+\frac{\pd_\theta A_2}{A_2}\right)\right)+2\left(\frac{\pd_\theta A_1}{A_1}+\frac{\pd_\theta A_2}{A_2}\right)^2 \\
&= 2 + O(a^2/r^2).
\end{align*}
This concludes the proof. \qed
\end{proof}

Finally, we prove an identity that relates the linear equation for $\xi_a$ with the linear equation for $(\phi,\psi)$. In particular, the proof will demonstrate the cancellation of two singular terms in the expression involving $\Box_g\psi$, which is directly related to the choice $\xi_1=A\psi$.
\begin{lemma}\label{translate_nl_lem}
If $(\xi_1,\xi_2)=(A\psi,-\phi)$, then
\begin{align*}
(e_1)^a(\Box_g\xi_a-V_a{}^b\xi_b) &= A\left(\Box_g\psi+2\frac{\pd^\alpha A}{A}\pd_\alpha \psi -2\frac{\pd^\alpha B}{A^2}\pd_\alpha \phi -2\frac{\pd^\alpha B\pd_\alpha B}{A^2}\psi \right) \\
&= A\left(\Box_{\tilde{g}}\psi+2\frac{\pd^\alpha A_2}{A_2}\pd_\alpha \psi -2\frac{\pd^\alpha B}{A^2}\pd_\alpha \phi -2\frac{\pd^\alpha B\pd_\alpha B}{A^2}\psi \right) \\
&= A(\Box_{\tilde{g}}\psi-\mathcal{L}_\psi) \\
-(e_2)^a(\Box_g\xi_a-V_a{}^b\xi_b) &= \Box_g\phi +2\pd^\alpha B\pd_\alpha \psi -2\frac{\pd^\alpha B\pd_\alpha B}{A^2}\phi +4\frac{\pd^\alpha A\pd_\alpha B}{A}\psi \\
&= \Box_g\phi-\mathcal{L}_\phi.
\end{align*}
\end{lemma}
\begin{proof}
Recall the following identities.
$$\xi=\xi_1e^1+\xi_2e^2$$
$$D_\alpha e^1 = -\frac{\pd_\alpha B}{A}e^2$$
$$D_\alpha e^2 = \frac{\pd_\alpha B}{A}e^1$$
Since $(A,B)$ solves the nonlinear wave map system (\ref{wm_X_eqn}-\ref{wm_Y_eqn}), we also have the following two equations.
\begin{align*}
\Box_g A &= \frac{\pd^\alpha A\pd_\alpha A}{A}-\frac{\pd^\alpha B\pd_\alpha B}{A} \\
\Box_g B &= 2\frac{\pd^\alpha A\pd_\alpha B}{A}.
\end{align*}
Using the identity for $\Box_g B$, we compute
\begin{align*}
\Box_g e^1 &= D^\alpha\left(-\frac{\pd_\alpha B}{A}e^2\right) \\
&= -\frac{\Box_g B}{A}e^2+\frac{\pd^\alpha A\pd_\alpha B}{A^2}e^2-\frac{\pd^\alpha B\pd_\alpha B}{A^2}e^1 \\
&= -\frac{\pd^\alpha A\pd_\alpha B}{A^2}e^2-\frac{\pd^\alpha B\pd_\alpha B}{A^2}e^1.
\end{align*}
\begin{align*}
\Box_g e^2 &= D^\alpha\left(\frac{\pd_\alpha B}{A}e^1\right) \\
&= \frac{\Box_g B}{A}e^1-\frac{\pd^\alpha A\pd_\alpha B}{A^2}e^1-\frac{\pd^\alpha B\pd_\alpha B}{A^2}e^2 \\
&=\frac{\pd^\alpha A\pd_\alpha B}{A^2}e^1-\frac{\pd^\alpha B\pd_\alpha B}{A^2}e^2.
\end{align*}
Therefore,
\begin{align*}
\Box_g\xi &= \Box_g(\xi_1e^2)+\Box_g(\xi_2e^2) \\
&= (\Box_g\xi_1 e^1+2\pd^\alpha\xi_1 D_\alpha e^1+\xi_1\Box_g e^1)+(\Box_g\xi_2 e^2+2\pd^\alpha \xi_2 D_\alpha e^2+\xi_2\Box_g e^2) \\
&= \left(\Box_g \xi_1 +2\frac{\pd^\alpha B}{A}\pd_\alpha \xi_2 -\frac{\pd^\alpha B\pd_\alpha B}{A^2}\xi_1+\frac{\pd^\alpha A\pd_\alpha B}{A^2}\xi_2\right)e^1 \\
&\hspace{.5in} +\left(\Box_g\xi_2-2\frac{\pd^\alpha B}{A}\pd_\alpha \xi_1 -\frac{\pd^\alpha A\pd_\alpha B}{A^2}\xi_1-\frac{\pd^\alpha B\pd_\alpha B}{A^2}\xi_2\right)e^2.
\end{align*}

Recall also that
\begin{align*}
V&=g^{\alpha\beta}d\Phi_\alpha d\Phi_\beta \\
&=g^{\alpha\beta}\left(\frac{\pd_\alpha A}{A}e_1+\frac{\pd_\alpha B}{A}e_2\right)\left(\frac{\pd_\beta A}{A}e_1+\frac{\pd_\beta B}{A}e_2\right).
\end{align*}
It follows that
$$V\cdot\xi =\left(\frac{\pd^\alpha A}{A}e^1+\frac{\pd^\alpha B}{A}e^2\right)\left(\frac{\pd_\alpha A}{A}\xi_1+\frac{\pd_\alpha B}{A}\xi_2\right).$$

Combining both identities, we conclude
\begin{align*}
\Box_g\xi-V\cdot\xi &= 
\left(\Box_g \xi_1 +2\frac{\pd^\alpha B}{A}\pd_\alpha \xi_2 -\frac{\pd^\alpha A\pd_\alpha A}{A^2}\xi_1-\frac{\pd^\alpha B\pd_\alpha B}{A^2}\xi_1\right)e^1 \\
&\hspace{.5in} +\left(\Box_g\xi_2-2\frac{\pd^\alpha B}{A}\pd_\alpha \xi_1 -2\frac{\pd^\alpha A\pd_\alpha B}{A^2}\xi_1-2\frac{\pd^\alpha B\pd_\alpha B}{A^2}\xi_2\right)e^2.
\end{align*}
Now, we replace $(\xi_1,\xi_2)=(A\psi,-\phi)$. In the first calculation, we also use the identity for $\Box_gA$, and see that the term having the factor $\Box_gA$ cancels with the term having the factor $\frac{\pd^\alpha A\pd_\alpha A}{A^2}$. Since these terms are singular on the axis, this cancellation is in some sense the purpose of the choice $\xi_1=A\psi$.
\begin{align*}
(e_1)^a(\Box_g\xi_a-V_a{}^b\xi_b) &= \Box_g \xi_1 +2\frac{\pd^\alpha B}{A}\pd_\alpha \xi_2 -\frac{\pd^\alpha A\pd_\alpha A}{A^2}\xi_1-\frac{\pd^\alpha B\pd_\alpha B}{A^2}\xi_1 \\
&= \Box_g(A\psi) +2\frac{\pd^\alpha B}{A}\pd_\alpha(-\phi) -\frac{\pd^\alpha A\pd_\alpha A}{A^2}(A\psi)-\frac{\pd^\alpha B\pd_\alpha B}{A^2}(A\psi) \\
&= (A\Box_g \psi +2\pd^\alpha A\pd_\alpha\psi +\psi \Box_gA) -2\frac{\pd^\alpha B}{A}\pd_\alpha\phi -\frac{\pd^\alpha A\pd_\alpha A}{A}\psi-\frac{\pd^\alpha B\pd_\alpha B}{A}\psi \\
&= A\Box_g\psi +2\pd^\alpha A\pd_\alpha\psi-2\frac{\pd^\alpha B}{A}\pd_\alpha\phi +\left(\Box_gA-\frac{\pd^\alpha A\pd_\alpha A}{A}-\frac{\pd^\alpha B\pd_\alpha B}{A}\right)\psi \\
&= A\Box_g\psi +2\pd^\alpha A\pd_\alpha\psi-2\frac{\pd^\alpha B}{A}\pd_\alpha\phi -2\frac{\pd^\alpha B\pd_\alpha B}{A}\psi \\
&= A\left(\Box_g\psi +2\frac{\pd^\alpha A}{A}\pd_\alpha \psi -2\frac{\pd^\alpha B}{A^2}\pd_\alpha\phi -2\frac{\pd^\alpha B\pd_\alpha B}{A^2}\psi\right).
\end{align*}
This verifies the first identity of the lemma.
\begin{align*}
-(e_2)^a(\Box_g\xi_a-V_a{}^b\xi_b) &= -\left(\Box_g\xi_2-2\frac{\pd^\alpha B}{A}\pd_\alpha \xi_1 -2\frac{\pd^\alpha A\pd_\alpha B}{A^2}\xi_1-2\frac{\pd^\alpha B\pd_\alpha B}{A^2}\xi_2\right) \\
&=-\left(\Box_g(-\phi)-2\frac{\pd^\alpha B}{A}\pd_\alpha (A\psi) -2\frac{\pd^\alpha A\pd_\alpha B}{A^2}(A\psi)-2\frac{\pd^\alpha B\pd_\alpha B}{A^2}(-\phi)\right) \\
&= \Box_g\phi+2\pd^\alpha B\pd_\alpha \psi +4\frac{\pd^\alpha A\pd_\alpha B}{A}\psi-2\frac{\pd^\alpha B\pd_\alpha B}{A^2}\phi.
\end{align*}
This verifies the second identity of the lemma. \qed
\end{proof}

\subsubsection{Translating the energy estimate from the $\xi_a$ system}

We rewrite the energy estimate (Proposition \ref{xi_energy_estimate_prop}) in terms of $(\phi,\psi)$.
\begin{proposition}\label{translated_energy_estimate_prop}(Energy Estimate)
\begin{multline*}
\int_{\Sigma_{t_2}}\chi_H(\pd_r\phi)^2+(\pd_t\phi)^2+|\sla\nabla\phi|^2+r^{-2}\phi^2 +\int_{\tilde\Sigma_{t_2}}\chi_H(\pd_r\psi)^2+(\pd_t\psi)^2+|\tsla\nabla\psi|^2+r^{-2}\psi^2 \\
\lesssim \int_{\Sigma_{t_1}}\chi_H(\pd_r\phi)^2+(\pd_t\phi)^2+|\sla\nabla\phi|^2+r^{-2}\phi^2 +\int_{\tilde\Sigma_{t_1}}\chi_H(\pd_r\psi)^2+(\pd_t\psi)^2+|\tsla\nabla\psi|^2+r^{-2}\psi^2 \\
+Err_{nl},
\end{multline*}
where $\chi_H=1-\frac{r_H}r$ and
$$Err_{nl} = \int_{t_1}^{t_2}\int_{\Sigma_t}|\pd_t\phi(\Box_g\phi-\mathcal{L}_\phi)|
+\int_{t_1}^{t_2}\int_{\tilde{\Sigma}_t}|\pd_t\psi(\Box_{\tilde{g}}\psi-\mathcal{L}_\psi)|.$$
\end{proposition}
\begin{proof}
According to Proposition \ref{xi_energy_estimate_prop},
\begin{multline*}
\int_{H_{t_1}^{t_2}}|D_t\xi|^2+\int_{\Sigma_{t_2}}\chi_H|D_r\xi|^2+|D_t\xi|^2+q^{-2}|D_\theta\xi|^2+V^{ab}\xi_a\xi_b \\
\lesssim \int_{\Sigma_{t_1}}\chi_H|D_r\xi|^2+|D_t\xi|^2+q^{-2}|D_\theta\xi|^2+V^{ab}\xi_a\xi_b + \int_{t_1}^{t_2}\int_{\Sigma_t} |D_t\xi^a(\Box_g\xi_a-V_a{}^b\xi_b)|.
\end{multline*}
We ignore the term on the horizon, which has an appropriate sign.

By a standard Hardy estimate and then Lemma \ref{phi_psi_le_xi_lem},
\begin{multline*}
\int_{\Sigma_{t_2}}\chi_H(\pd_r\phi)^2+(\pd_t\phi)^2+|\sla\nabla\phi|^2+r^{-2}\phi^2 +\int_{\tilde\Sigma_{t_2}}\chi_H(\pd_r\psi)^2+(\pd_t\psi)^2+|\tsla\nabla\psi|^2+r^{-2}\psi^2 \\
\lesssim \int_{\Sigma_{t_2}}\chi_H(\pd_r\phi)^2+(\pd_t\phi)^2+|\sla\nabla\phi|^2 +\int_{\tilde\Sigma_{t_2}}\chi_H(\pd_r\psi)^2+(\pd_t\psi)^2+|\tsla\nabla\psi|^2 \\
\lesssim \int_{\Sigma_{t_2}}\chi_H|D_r\xi|^2+|D_t\xi|^2+q^{-2}|D_\theta\xi|^2+V^{ab}\xi_a\xi_b +\int_{\Sigma_{t_2}}\frac{|a|}{M}r^{-2}(\phi^2+A^2\psi^2).
\end{multline*}
By taking $|a|/M$ sufficiently small, the final term on the right hand side can be absorbed into the left hand side.

By Lemmas \ref{xi_le_phi_psi_1_lem} and \ref{xi_le_phi_psi_2_lem},
\begin{multline*}
\int_{\Sigma_{t_1}}\chi_H|D_r\xi|^2+|D_t\xi|^2+q^{-2}|D_\theta\xi|^2+V^{ab}\xi_a\xi_b \\
\lesssim \int_{\Sigma_{t_1}}\chi_H(\pd_r\phi)^2+(\pd_t\phi)^2+|\sla\nabla\phi|^2+r^{-2}\phi^2 +\int_{\tilde\Sigma_{t_1}}\chi_H(\pd_r\psi)^2+(\pd_t\psi)^2+|\tsla\nabla\psi|^2+r^{-2}\psi^2.
\end{multline*}

Finally, by Lemma \ref{translate_nl_lem},
$$\int_{t_1}^{t_2}\int_{\Sigma_t} |D_t\xi^a(\Box_g\xi_a-V_a{}^b\xi_b)|
\le \int_{t_1}^{t_2}\int_{\Sigma_t}|\pd_t\phi(\Box_g\phi-\mathcal{L}_\phi)|
+\int_{t_1}^{t_2}\int_{\tilde{\Sigma}_t}|\pd_t\psi(\Box_{\tilde{g}}\psi-\mathcal{L}_\psi)|.$$

These estimates together prove the proposition. \qed
\end{proof}

\subsubsection{Translating the $h\pd_t$ estimate from the $\xi_a$ system}

We rewrite the $h\pd_t$ estimate (Proposition \ref{xi_h_dt_prop}) in terms of $(\phi,\psi)$.
\begin{proposition}\label{translated_h_dt_prop}
Fix $\delp>0$ and let $p\le 2-\delp$. Then for all $\epsilon>0$, there is a small constant $c_\epsilon$ and a large constant $C_\epsilon$, such that
\begin{align*}
&\hspace{.5in}\int_{\Sigma_{t_2}}r^{p-2}\left[\chi_H(\pd_r\phi)^2+(\pd_t\phi)^2+|\sla\nabla\phi|^2\right] +\int_{\tilde\Sigma_{t_2}}r^{p-2}\left[\chi_H(\pd_r\psi)^2+(\pd_t\psi)^2+|\tsla\nabla\psi|^2\right] \\
&+\int_{H_{t_1}^{t_2}}(\pd_t\phi)^2 +\int_{\tilde{H}_{t_1}^{t_2}}(\pd_t\psi)^2 \\
&\hspace{1in}+\int_{t_1}^{t_2}\int_{\Sigma_t\cap \{6M<r\}} c_\epsilon r^{p-3}(\lbar\phi)^2 + \int_{t_1}^{t_2}\int_{\tilde\Sigma_t\cap \{6M<r\}} c_\epsilon r^{p-3}\left(\lbar\psi+\frac{\lbar A}{A}\psi\right)^2 \\
&\lesssim \int_{\Sigma_{t_1}}C_\epsilon r^{p-2}\left[\chi_H(\pd_r\phi)^2+(\pd_t\phi)^2+|\sla\nabla\phi|^2+r^{-2}\phi^2\right] \\
&\hspace{2in}+\int_{\tilde\Sigma_{t_1}}C_\epsilon r^{p-2}\left[\chi_H(\pd_r\psi)^2+(\pd_t\psi)^2+|\tsla\nabla\psi|^2+r^{-2}\psi^2\right]  +Err,
\end{align*}
where $\chi_H=1-\frac{r_H}r$ and
\begin{align*}
Err &= Err_1 + Err_2 + Err_3 + Err_{nl} \\
Err_1 &= \int_{t_1}^{t_2}\int_{\Sigma_t\cap\{5M<r\}} \epsilon r^{-1}((L\phi)^2+r^{-2}\phi^2) + \int_{t_1}^{t_2}\int_{\tilde\Sigma_t\cap\{5M<r\}} \epsilon r^{-1}((L\psi)^2+r^{-2}\psi^2) \\
Err_2 &= \int_{\Sigma_{t_2}}\frac{|a|}{M}r^{p-4}(\phi^2+A^2\psi^2) \\
Err_3 &= \int_{t_1}^{t_2}\int_{\Sigma_t\cap\{6M<r\}} c_\epsilon\frac{|a|}{M}r^{p-5}(\phi^2+A^2\psi^2) \\
Err_{nl} &= \int_{t_1}^{t_2}\int_{\Sigma_t}C_\epsilon r^{p-2}|\pd_t\phi(\Box_g\phi-\mathcal{L}_\phi)| +\int_{t_1}^{t_2}\int_{\tilde{\Sigma}_t}C_\epsilon r^{p-2}|\pd_t\psi(\Box_{\tilde{g}}\psi-\mathcal{L}_\psi)|.
\end{align*}
\end{proposition}

\begin{proof}
By Proposition \ref{xi_h_dt_prop} with $R=5M$,
\begin{multline*}
\int_{H_{t_1}^{t_2}}|D_t\xi|^2+\int_{\Sigma_{t_2}}r^{p-2}\left[\chi_H|D_r\xi|^2+|D_t\xi|^2+q^{-2}|D_\theta\xi|^2 +V^{ab}\xi_a\xi_b\right] \\
+\int_{t_1}^{t_2}\int_{\Sigma_t\cap\{6M<r\}}c_\epsilon r^{p-3}|D_{\lbar}\xi|^2 \\
\lesssim \int_{\Sigma_{t_1}}C_\epsilon r^{p-2}\left[\chi_H|D_r\xi|^2+|D_t\xi|^2+q^{-2}|D_\theta\xi|^2+V^{ab}\xi_a\xi_b\right] + \int_{t_1}^{t_2}\int_{\Sigma_t\cap\{5M<r\}}\epsilon r^{-1}|D_L\xi|^2 \\
+  \int_{t_1}^{t_2}\int_{\Sigma_t}C_\epsilon r^{p-2}|D_t\xi^a(\Box_g\xi_a-V_a{}^b\xi_b)|.
\end{multline*}

By Lemma \ref{phi_psi_le_xi_lem},
\begin{multline*}\int_{\Sigma_{t_2}}r^{p-2}\left[\chi_H(\pd_r\phi)^2+(\pd_t\phi)^2+|\sla\nabla\phi|^2\right] +\int_{\tilde\Sigma_{t_2}}r^{p-2}\left[\chi_H(\pd_r\psi)^2+(\pd_t\psi)^2+|\tsla\nabla\psi|^2\right] \\
\lesssim \int_{\Sigma_{t_2}}r^{p-2}\left[\chi_H|D_r\xi|^2+|D_t\xi|^2+q^{-2}|D_\theta\xi|^2+V^{ab}\xi_a\xi_b\right] +\int_{\Sigma_{t_2}}\frac{|a|}{M}r^{p-4}(\phi^2+A^2\psi^2).
\end{multline*}
and
$$\int_{H_{t_1}^{t_2}}(\pd_t\phi)^2+\int_{\tilde{H}_{t_1}^{t_2}}(\pd_t\psi)^2 = \int_{H_{t_1}^{t_2}}|D_t\xi|^2,$$
and
\begin{multline*}
\int_{t_1}^{t_2}\int_{\Sigma_t\cap \{6M<r\}} c_\epsilon r^{p-3}(\lbar\phi)^2 + \int_{t_1}^{t_2}\int_{\tilde\Sigma_t\cap \{6M<r\}} c_\epsilon r^{p-3}\left(\lbar\psi+\frac{\lbar A}{A}\psi\right)^2 \\
\lesssim \int_{t_1}^{t_2}\int_{\Sigma_t\cap\{6M<r\}}c_\epsilon r^{p-3}|D_{\lbar}\xi|^2 + \int_{t_1}^{t_2}\int_{\Sigma_t\cap\{6M<r\}}c_\epsilon \frac{|a|}{M}r^{p-5}(\phi^2+A^2\psi^2).
\end{multline*}
At this point, we have estimated all the terms on the left hand side of the main estimate. Now, we must estimate the additional terms:
\begin{multline*}
\int_{\Sigma_{t_2}}\frac{|a|}{M}r^{p-4}(\phi^2+A^2\psi^2) + \int_{t_1}^{t_2}\int_{\Sigma_t\cap\{6M<r\}}c_\epsilon \frac{|a|}{M}r^{p-5}(\phi^2+A^2\psi^2) \\
+  \int_{\Sigma_{t_1}}C_\epsilon r^{p-2}\left[\chi_H|D_r\xi|^2+|D_t\xi|^2+q^{-2}|D_\theta\xi|^2+V^{ab}\xi_a\xi_b\right] \\
+ \int_{t_1}^{t_2}\int_{\Sigma_t\cap\{5M<r\}}\epsilon r^{-1}|D_L\xi|^2 +\int_{t_1}^{t_2}\int_{\Sigma_t}C_\epsilon r^{p-2}|D_t\xi^a(\Box_g\xi_a-V_a{}^b\xi_b)|.
\end{multline*}
By definition,
$$\int_{\Sigma_{t_2}}\frac{|a|}{M}r^{p-4}(\phi^2+A^2\psi^2) + \int_{t_1}^{t_2}\int_{\Sigma_t\cap\{6M<r\}}c_\epsilon\frac{|a|}{M}r^{p-5}(\phi^2+A^2\psi^2) = Err_2+Err_3.$$
Also, by Lemmas \ref{xi_le_phi_psi_1_lem} and \ref{xi_le_phi_psi_2_lem},
\begin{multline*}
\int_{\Sigma_{t_1}}C_\epsilon r^{p-2}\left[\chi_H|D_r\xi|^2+|D_t\xi|^2+q^{-2}|D_\theta\xi|^2+V^{ab}\xi_a\xi_b\right] \\
\lesssim  \int_{\Sigma_{t_1}}C_\epsilon r^{p-2}\left[\chi_H(\pd_r\phi)^2+(\pd_t\phi)^2+|\sla\nabla\phi|^2+r^{-2}\phi^2\right] \\
+\int_{\tilde\Sigma_{t_1}}C_\epsilon r^{p-2}\left[\chi_H(\pd_r\psi)^2+(\pd_t\psi)^2+|\tsla\nabla\psi|^2+r^{-2}\psi^2\right]
\end{multline*}
Also,
\begin{multline*}
\int_{t_1}^{t_2}\int_{\Sigma_t\cap\{5M<r\}}\epsilon r^{-1}|D_L\xi|^2 \\
\lesssim \int_{t_1}^{t_2}\int_{\Sigma_t\cap\{5M<r\}} \epsilon r^{-1}((L\phi)^2+r^{-2}\phi^2) + \int_{t_1}^{t_2}\int_{\tilde\Sigma_t\cap\{5M<r\}} \epsilon r^{-1}((L\psi)^2+r^{-2}\psi^2) \\
= Err_1.
\end{multline*}
Finally, by Lemma \ref{translate_nl_lem},
\begin{multline*}
\int_{t_1}^{t_2}\int_{\Sigma_t}C_\epsilon r^{p-2}|D_t\xi^a(\Box_g\xi_a-V_a{}^b\xi_b)| \\
\le \int_{t_1}^{t_2}\int_{\Sigma_t}C_\epsilon r^{p-2}|\pd_t\phi(\Box_g\phi-\mathcal{L}_\phi)| +\int_{t_1}^{t_2}\int_{\tilde{\Sigma}_t}C_\epsilon r^{p-2}|\pd_t\psi(\Box_{\tilde{g}}\psi-\mathcal{L}_\psi)| \\
\le Err_{nl}.
\end{multline*}
These estimates complete the proof. \qed
\end{proof}

\subsubsection{Translating the Morawetz estimate from the $\xi_a$ system}

We rewrite the Morawetz estimate (Proposition \ref{xi_morawetz_estimate_prop}) in terms of $(\phi,\psi)$.
\begin{proposition}\label{translated_morawetz_prop} Suppose $|a|/M$ is sufficiently small. Then
\begin{multline*}
\int_{\Sigma_{t_2}}(L\phi)^2+|\sla\nabla\phi|^2+r^{-2}\phi^2+\frac{M^2}{r^2}(\pd_r\phi)^2 +\int_{\tilde{\Sigma}_{t_2}}(L\psi)^2+|\tsla\nabla\psi|^2+r^{-2}\psi^2+\frac{M^2}{r^2}(\pd_r\psi)^2 \\
+\int_{t_1}^{t_2}\int_{\Sigma_t}\frac{M^2}{r^3}(\pd_r\phi)^2+\chi_{trap}\left(\frac{M^2}{r^3}(\pd_t\phi)^2+\frac1r|\sla\nabla\phi|^2\right)+\frac{M}{r^4}\phi^2 \\
+\int_{t_1}^{t_2}\int_{\tilde{\Sigma}_t}\frac{M^2}{r^3}(\pd_r\psi)^2+\chi_{trap}\left(\frac{M^2}{r^3}(\pd_t\psi)^2+\frac1r|\tsla\nabla\psi|^2\right)+\frac{1}{r^3}\psi^2 \\
\lesssim \int_{\Sigma_{t_1}}(L\phi)^2+|\sla\nabla\phi|^2+r^{-2}\phi^2+\frac{M^2}{r^2}(\pd_r\phi)^2 +\int_{\tilde{\Sigma}_{t_1}}(L\psi)^2+|\tsla\nabla\psi|^2+r^{-2}\psi^2+\frac{M^2}{r^2}(\pd_r\psi)^2 \\
+Err,
\end{multline*}
where $\chi_{trap}=\left(1-\frac{r_{trap}}{r}\right)^2$ and
\begin{align*}
Err &= Err_1+Err_2+Err_{nl} \\
Err_1 &= \int_{\Sigma_{t_2}}r^{-1}|\phi L\phi|+\int_{\tilde{\Sigma}_{t_2}}r^{-1}|\psi L\psi|+r^{-2}\psi^2 \\
Err_2 &= \int_{H_{t_1}^{t_2}}(\pd_t\phi)^2+\int_{H_{t_1}^{t_2}}(\pd_t\psi)^2 +\int_{\Sigma_{t_2}}\frac{M^2}{r^2}\left[\chi_H(\pd_r\phi)^2+(\pd_t\phi)^2\right] \\
&\hspace{1in} +\int_{\tilde{\Sigma}_{t_2}}\frac{M^2}{r^2}\left[\chi_H(\pd_r\psi)^2+(\pd_t\psi)^2\right] \\
Err_{nl} &= \int_{t_1}^{t_2}\int_{\Sigma_t}|(2X(\phi)+w\phi+w_{(a)}\psi)(\Box_g\phi-\mathcal{L}_\phi)| \\
&\hspace{1in} +\int_{t_1}^{t_2}\int_{\tilde{\Sigma}_t}|(2X(\psi)+\tilde{w}\psi+\tilde{w}_{(a)}\phi)(\Box_{\tilde{g}}\psi-\mathcal{L}_\psi)|,
\end{align*}
where $\chi_H=1-\frac{r_H}r$ and the new functions $\tilde{w}$, $w_{(a)}$, and $\tilde{w}_{(a)}$, which are defined in terms of the original vectorfield $X$ and function $w$ used in the proof of Proposition \ref{xi_morawetz_estimate_prop}, are given by the following relations.
\begin{align*}
2X(\phi)+w\phi+w_{(a)}\psi &= X(\phi)+w\phi+\frac{X(B)}{A}A\psi \\
2X(\psi)+\tilde{w}\psi+\tilde{w}_{(a)}\phi &= 2X(\psi)+\left(\frac{X(A)}{A}+w\right)\psi-\frac{X(B)}{A}A^{-1}\phi.
\end{align*}
\end{proposition}
\begin{proof}
From Proposition \ref{xi_morawetz_estimate_prop}, we have
\begin{align*}
&\int_{H_{t_1}^{t_2}}q^{-2}|D_\theta \xi|^2+V^{ab}\xi_a\xi_b
+\int_{\Sigma_{t_2}} |D_L\xi|^2+q^{-2}|D_\theta\xi|^2+V^{ab}\xi_a\xi_b+r^{-2}|\xi|^2 +\frac{M^2}{r^2}|D_r\xi|^2  \\
&+\int_{t_1}^{t_2}\int_{\Sigma_t} \left[\frac{M^2}{r^3}(\pd_r\xi_1)^2+\chi_{trap}\left(\frac{M^2}{r^3}(\pd_t\xi_1)^2+\frac{1}{r}|\sla\nabla\xi_1|^2+\frac{\cot^2\theta}{r^3}(\xi_1)^2\right)+\frac{1}{r^3}(\xi_1)^2\right. \\
&\hspace{1.25in}\left.+ \frac{M^2}{r^3}(\pd_r\xi_2)^2+\chi_{trap}\left(\frac{M^2}{r^3}(\pd_t\xi_2)^2+\frac{1}{r}|\sla\nabla\xi_2|^2\right)+\frac{M}{r^4}(\xi_2)^2\right] \\
&\lesssim \int_{\Sigma_{t_1}} |D_L\xi|^2+q^{-2}|D_\theta\xi|^2+V^{ab}\xi_a\xi_b+r^{-2}|\xi|^2 +\frac{M^2}{r^2}|D_r\xi|^2  + Err',
\end{align*}
where
\begin{align*}
Err' &= Err'_1+Err'_{nl} \\
Err'_1 &= \int_{H_{t_1}^{t_2}}|D_t\xi|^2 + \int_{\Sigma_{t_2}}r^{-1}|\xi\cdot D_L\xi|+\frac{M^2}{r^2}\left[\chi_H|D_r\xi|^2+|D_t\xi|^2+r^{-2}|\xi|^2\right] \\
Err'_{nl} &= \int_{t_1}^{t_2}\int_{\Sigma_t}|(2D_X\xi^a+w\xi^a)(\Box_g\xi_a-V_a{}^b\xi_b)|.
\end{align*}
We ignore the term on the horizon, which has an appropriate sign.

By Lemma \ref{phi_psi_le_xi_lem},
\begin{multline*}
\int_{\Sigma_{t_2}}(L\phi)^2+|\sla\nabla\phi|^2+r^{-2}\phi^2+\frac{M^2}{r^2}(\pd_r\phi)^2 +\int_{\tilde{\Sigma}_{t_2}}(L\psi)^2+|\tsla\nabla\psi|^2+r^{-2}\psi^2+\frac{M^2}{r^2}(\pd_r\psi)^2 \\
\lesssim \int_{\Sigma_{t_2}} |D_L\xi|^2+q^{-2}|D_\theta\xi|^2+V^{ab}\xi_a\xi_b+r^{-2}|\xi|^2 +\frac{M^2}{r^2}|D_r\xi|^2.
\end{multline*}
and
\begin{multline*}
\int_{t_1}^{t_2}\int_{\Sigma_t}\frac{M^2}{r^3}(\pd_r\phi)^2+\chi_{trap}\left(\frac{M^2}{r^3}(\pd_t\phi)^2+\frac1r|\sla\nabla\phi|^2\right)+\frac{M}{r^4}\phi^2 \\
+\int_{t_1}^{t_2}\int_{\tilde{\Sigma}_t}\frac{M^2}{r^3}(\pd_r\psi)^2+\chi_{trap}\left(\frac{M^2}{r^3}(\pd_t\psi)^2+\frac1r|\tsla\nabla\psi|^2\right)+\frac{1}{r^3}\psi^2 \\
\lesssim \int_{t_1}^{t_2}\int_{\Sigma_t} \left[\frac{M^2}{r^3}(\pd_r\xi_1)^2+\chi_{trap}\left(\frac{M^2}{r^3}(\pd_t\xi_1)^2+\frac{1}{r}|\sla\nabla\xi_1|^2+\frac{\cot^2\theta}{r^3}(\xi_1)^2\right)+\frac{1}{r^3}(\xi_1)^2\right. \\
\hspace{1.25in}\left.+ \frac{M^2}{r^3}(\pd_r\xi_2)^2+\chi_{trap}\left(\frac{M^2}{r^3}(\pd_t\xi_2)^2+\frac{1}{r}|\sla\nabla\xi_2|^2\right)+\frac{M}{r^4}(\xi_2)^2\right].
\end{multline*}
By Lemmas \ref{xi_le_phi_psi_1_lem} and \ref{xi_le_phi_psi_2_lem},
\begin{multline*}
\int_{\Sigma_{t_1}} |D_L\xi|^2+q^{-2}|D_\theta\xi|^2+V^{ab}\xi_a\xi_b+r^{-2}|\xi|^2 +\frac{M^2}{r^2}|D_r\xi|^2  \\
\lesssim \int_{\Sigma_{t_1}}(L\phi)^2+|\sla\nabla\phi|^2+r^{-2}\phi^2+\frac{M^2}{r^2}(\pd_r\phi)^2 +\int_{\tilde{\Sigma}_{t_1}}(L\psi)^2+|\tsla\nabla\psi|^2+r^{-2}\psi^2+\frac{M^2}{r^2}(\pd_r\psi)^2.
\end{multline*}
Since 
$$Err'_1\lesssim Err_1+Err_2,$$
it remains to check the nonlinear terms.

We calculate
\begin{align*}
2D_X\xi+w\xi &= 2D_X(\xi_1e^1+\xi_2e^2)+w(\xi_1e^1+\xi_2e^2) \\
&= \left(2X(\xi_1)+w\xi_1+\frac{X(B)}{A}\xi_2\right)e^1+\left(2X(\xi_2)+w\xi_2-\frac{X(B)}{A}\xi_1\right)e^2 \\
&= A\left(2X(\psi)+\left(\frac{X(A)}{A}+w\right)\psi-\frac{X(B)}{A}A^{-1}\phi\right)e^1
-\left(2X(\phi)+w\phi+\frac{X(B)}{A}A\psi\right)e^2.
\end{align*}
This motivates the definition
\begin{align*}
2X(\phi)+w\phi+w_{(a)}\psi &= X(\phi)+w\phi+\frac{X(B)}{A}A\psi \\
2X(\psi)+\tilde{w}\psi+\tilde{w}_{(a)}\phi &= 2X(\psi)+\left(\frac{X(A)}{A}+w\right)\psi-\frac{X(B)}{A}A^{-1}\phi.
\end{align*}
Since according to Lemma \ref{translate_nl_lem}
\begin{align*}
(e_1)^a(\Box_g\xi_a-V_a{}^b\xi_b) &= A(\Box_{\tilde{g}}\psi-\mathcal{L}_\psi) \\
-(e_2)^a(\Box_g\xi_a-V_a{}^b\xi_b) &= \Box_g\phi -\mathcal{L}_\phi,
\end{align*}
we conclude that
\begin{multline*}
(2D_X\xi^a+w\xi^a)(\Box_g\xi_a-V_a{}^b\xi_b) \\
=A^2(2X(\psi)+\tilde{w}\psi+\tilde{w}_{(a)}\phi)(\Box_{\tilde{g}}-\mathcal{L}_\psi)+(2X(\phi)+w\phi+w_{(a)}\psi)(\Box_g\phi-\mathcal{L}_\phi).
\end{multline*}
From this identity, it is clear that $Err'_{nl}\le Err_{nl}$. This completes the proof. \qed
\end{proof}

\subsection{The incomplete $p$-weighted estimate near $i^0$}\label{phi_psi_incomplete_sec}

At this point, all of the relevant estimates from \S\ref{xi_a_sec} have been rewritten in terms of $(\phi,\psi)$. But there is one more type of estimate that is needed near $i^0$. In fact, it is a family of estimates depending on a parameter $p\in (0,2)$. One can think of this type of estimate as exploting the asymptotic flatness of the spacetime. Near $i^0$, the equations for $\phi$ and $\psi$ are similar to the homogeneous wave equations $\Box_g\phi=0$ and $\Box_{\tilde{g}}\psi=0$, so these proofs do not take into account the linear terms $\mathcal{L}_\phi$ and $\mathcal{L}_\psi$.

The $p$-weighted estimate relies on two spacetime identities (one for the spacetime $\mathcal{M}$ and one for the spacetime $\tilde{\mathcal{M}}$), which are sufficiently complicated that they are each stated and proved seperately in \S\ref{p_identity_M_sec} and \S\ref{p_identity_M_tilde_sec} respectively. Finally, they are combined in \S\ref{p_identity_combined_sec}.

\subsubsection{A $p$-identity for $\mathcal{M}$}\label{p_identity_M_sec}

We prove the following lemma, which has a much simpler, well-known analogue in Minkowski spacetime.
\begin{lemma}\label{p_identity_phi_lem} ($p$ identity for $\mathcal{M}$) Let $\alpha=\frac{\Delta}{r^2+a^2}$ and $L=\alpha\pd_r+\pd_t$. For any function $f=f(r)$ supported where $r>r_H+\delh$, the following identity holds.
\begin{align*}
&\int_{\Sigma_{t_2}}\left[\left(1-\alpha\frac{a^2\sin^2\theta}{r^2+a^2}\right)\frac{r^2+a^2}{q^2}f\left(\alpha^{-1}L\phi+\frac{r}{r^2+a^2}\phi\right)^2 +\frac{\alpha^{-1}f}{q^2}(\pd_\theta\phi)^2+\epsilon\frac{rf'}{q^2}\phi^2\right. \\
&\hspace{1.5in}\left.+\alpha\frac{a^2\sin^2\theta}{q^2}f\left(\pd_r\phi+\frac{r}{r^2+a^2}\phi\right)^2 +\frac{a^2f}{q^2(r^2+a^2)}\phi^2-\frac1{q^2}\pd_r(rf\phi^2)\right] \\
&+\int_{t_1}^{t_2}\int_{\Sigma_t}
\left[
  \left(\frac{2rf}{r^2+a^2}-f'\right)\frac{Q^{\alpha\beta}}{q^2}\pd_\alpha\phi\pd_\beta\phi
+ \alpha f'\frac{r^2+a^2}{q^2}\left(\alpha^{-1}L\phi+\frac{(1-\epsilon)r}{r^2+a^2}\phi\right)^2 
\vphantom{
+ \epsilon\alpha\left((1-\epsilon)f'-rf''\right)\frac{\phi^2}{q^2} 
+ \alpha'\left(\frac{r^2-a^2}{r^2+a^2}f-\epsilon r f'\right)\frac{\phi^2}{q^2}
+ \frac{a^2}{r^2+a^2}\left(\alpha((1-\epsilon)^2-2)f'-\frac{4\alpha rf}{r^2+a^2}\right)\frac{\phi^2}{q^2}
}\right. \\
&\hspace{2.5in} + \epsilon\alpha\left((1-\epsilon)f'-rf''\right)\frac{\phi^2}{q^2} -\alpha'\alpha^{-2}f\frac{r^2+a^2}{q^2}(L\phi)^2 \\
&\hspace{1.1in}
\left.\vphantom{
  \left(\frac{2rf}{r^2+a^2}-f'\right)\frac{Q^{\alpha\beta}}{q^2}\pd_\alpha\phi\pd_\beta\phi
+ \alpha f'\frac{r^2+a^2}{q^2}\left(\alpha^{-1}L\phi+\frac{(1-\epsilon)r}{r^2+a^2}\phi\right)^2 
+ \epsilon\alpha\left((1-\epsilon)f'-rf''\right)\frac{\phi^2}{q^2} 
}
+ \alpha'\left(-\epsilon r f' +\frac{r^2-a^2}{r^2+a^2}f\right)\frac{\phi^2}{q^2}
+ \frac{a^2}{r^2+a^2}\left(-(1+\epsilon)\alpha f'-\frac{4\alpha rf}{r^2+a^2}\right)\frac{\phi^2}{q^2}
\right] \\
=&\int_{\Sigma_{t_1}}\left[\left(1-\alpha\frac{a^2\sin^2\theta}{r^2+a^2}\right)\frac{r^2+a^2}{q^2}f\left(\alpha^{-1}L\phi+\frac{r}{r^2+a^2}\phi\right)^2 +\frac{\alpha^{-1}f}{q^2}(\pd_\theta\phi)^2+\epsilon\frac{rf'}{q^2}\phi^2\right. \\
&\hspace{1.5in}\left.+\alpha\frac{a^2\sin^2\theta}{q^2}f\left(\pd_r\phi+\frac{r}{r^2+a^2}\phi\right)^2 +\frac{a^2f}{q^2(r^2+a^2)}\phi^2-\frac1{q^2}\pd_r(rf\phi^2)\right] \\
&+\int_{t_1}^{t_2}\int_{\Sigma_t}-\left(2\alpha^{-1}fL\phi+\frac{2rf}{r^2+a^2}\phi\right)\Box_g\phi.
\end{align*}
\end{lemma}

\begin{proof}
We will use Proposition \ref{general_divergence_estimate_prop} together with the following current template.
$$J_{(\phi)}[X,w,m]_\mu = T_{\mu\nu} X^\nu +w\phi\pd_\mu\phi-\frac12\phi^2\pd_\mu w+m_\mu \phi^2,$$
$$T_{\mu\nu}=2\pd_\mu\phi\pd_\nu\phi-g_{\mu\nu}\pd^\lambda\phi\pd_\lambda\phi.$$

Assume for now that $\Box_g\phi=0$. Let $\alpha = \frac{\Delta}{r^2+a^2}$, and observe that
$$L=\alpha\pd_r+\pd_t,$$
$$q^2g^{rr}=(r^2+a^2)\alpha,$$
$$q^2g^{tt}=-(r^2+a^2)\alpha^{-1}.$$

\begin{lemma}\label{divJphiX_lem}
Without appealing directly to the particular expression for $\alpha$, one can deduce the following.
$$\frac{q^2}{r^2+a^2}divJ_{(\phi)}[\alpha^{-1}fL]=(\alpha^{-1}f)'(L\phi)^2-\frac{2rf}{r^2+a^2}\left(\alpha(\pd_r\phi)^2-\alpha^{-1}(\pd_t\phi)^2\right)-f'\frac{Q^{\alpha\beta}}{r^2+a^2}\pd_\alpha\phi\pd_\beta\phi.$$
\end{lemma}
\begin{proof}
Note that
$$div J_{(\phi)}[X] = K^{\mu\nu}\pd_\mu\phi\pd_\nu\phi,$$
where
$$K^{\mu\nu}=2g^{\mu\lambda}\pd_\lambda X^\nu-X^\lambda \pd_\lambda(g^{\mu\nu})-div X g^{\mu\nu}.$$

Set $X=\alpha^{-1}f(\alpha\pd_r+\pd_t)=f\pd_r+\alpha^{-1}f\pd_t$. From the above formula, since $g^{rt}=0$,
$$\frac{q^2}{r^2+a^2}(K^{tr}+K^{rt})=2\frac{q^2}{r^2+a^2}g^{rr}\pd_r X^t=2\alpha\pd_r(\alpha^{-1} f).$$
Thus, the expression for $\frac{q^2}{r^2+a^2}divJ_{(\phi)}[\alpha^{-1}fL]$ will have a mixed term of the form 
$$2\alpha\pd_r(\alpha^{-1}f)\pd_r\phi\pd_t\phi.$$
Note that
\begin{align*}
(\alpha^{-1} f)'(L\phi)^2 &= (\alpha^{-1} f)'(\alpha\pd_r\phi+\pd_t\phi)^2 \\
&= \alpha^2(\alpha^{-1}f)'(\pd_r\phi)^2+2\alpha(\alpha^{-1}f)'\pd_r\phi\pd_t\phi +(\alpha^{-1} f)'(\pd_t\phi)^2.
\end{align*}
We now compute the $(\pd_r\phi)^2$ and $(\pd_t\phi)^2$ components, subtracting the part that will be grouped with the $(L\phi)^2$ term.
\begin{align*}
\frac{q^2}{r^2+a^2}K^{rr}-\alpha^2(\alpha^{-1}f)' &= \frac{q^2}{r^2+a^2}\left[2g^{rr}\pd_rX^r-X^r\pd_r g^{rr}-\frac1{q^2}\pd_r(q^2X^r)g^{rr}\right]-\alpha^2(\alpha^{-1}f)' \\
&= \frac{q^2}{r^2+a^2}\left[2g^{rr}\pd_rX^r-\frac1{q^2}\pd_r(q^2g^{rr}X^r)\right]-\alpha^2(\alpha^{-1}f)' \\
&= 2\alpha\pd_r f -\frac{1}{r^2+a^2}\pd_r\left((r^2+a^2)\alpha f\right)-\alpha^2(\alpha^{-1}f)' \\
&= -\frac{2r \alpha f}{r^2+a^2}
\end{align*}
and
\begin{align*}
\frac{q^2}{r^2+a^2}K^{tt}-(\alpha^{-1}f)' &= \frac{q^2}{r^2+a^2}\left[-X^r\pd_r g^{tt}-\frac{1}{q^2}\pd_r(q^2X^r)g^{tt}\right]-(\alpha^{-1}f)' \\
&= -\frac{q^2}{r^2+a^2}\frac{1}{q^2}\pd_r(q^2 g^{tt} X^r) -(\alpha^{-1}f)' \\
&= -\frac{1}{r^2+a^2}\pd_r\left((r^2+a^2)(-\alpha^{-1}) f\right)-(\alpha^{-1}f)' \\
&= \frac{2r\alpha^{-1}f}{r^2+a^2}.
\end{align*}
Finally,
\begin{align*}
\frac{q^2}{r^2+a^2}{}^{(Q)}K^{\alpha\beta} &= \frac{q^2}{r^2+a^2}\left[-X^r\pd_r {}^{(Q)}g^{\alpha\beta}-\frac1{q^2}\pd_r\left(q^2X^r\right){}^{(Q)}g^{\alpha\beta}\right] \\
&= \frac{q^2}{r^2+a^2}\left[-\frac{1}{q^2}\pd_r\left(q^2{}^{(Q)}g^{\alpha\beta}X^r\right)\right] \\
&= -\frac{1}{r^2+a^2}\pd_r(Q^{\alpha\beta}f) \\
&= -f' \frac{Q^{\alpha\beta}}{r^2+a^2}.
\end{align*}
Combining all these terms gives the identity stated in the lemma. \qed
\end{proof}

Next, we choose $w=\frac{2rf}{r^2+a^2}$ to directly cancel with the middle term in the above lemma.
\begin{lemma}\label{divJphiXw_lem}
\begin{multline*}
\frac{q^2}{r^2+a^2}divJ_{(\phi)}\left[\alpha^{-1}fL,\frac{2rf}{r^2+a^2}\right] \\
= (\alpha^{-1}f)'(L\phi)^2+\left(\frac{2rf}{r^2+a^2}-f'\right)\frac{Q^{\alpha\beta}}{r^2+a^2}\pd_\alpha\phi\pd_\beta\phi-\frac12\frac{q^2}{r^2+a^2}\Box_g\left(\frac{2rf}{r^2+a^2}\right)\phi^2.
\end{multline*}
\end{lemma}
\begin{proof}
Note that
$$divJ_{(\phi)}[0,w]=wg^{\mu\nu}\pd_\mu\phi\pd_\nu\phi-\frac12\Box_gw \phi^2.$$
We compute the new terms only.
\begin{multline*}
\frac{q^2}{r^2+a^2}divJ_{(\phi)}\left[0,\frac{2rf}{r^2+a^2}\right] = \frac{2rf}{r^2+a^2}\frac{q^2 g^{\alpha\beta}}{r^2+a^2}\pd_\alpha\phi\pd_\beta\phi -\frac12\frac{q^2}{r^2+a^2}\Box_g\left(\frac{2rf}{r^2+a^2}\right)\phi^2 \\
= \frac{2rf}{r^2+a^2}\left(\alpha (\pd_r\phi)^2-\alpha^{-1}(\pd_t\phi)^2\right) +\frac{2rf}{r^2+a^2}\frac{Q^{\alpha\beta}}{r^2+a^2}\pd_\alpha\phi\pd_\beta\phi -\frac12\frac{q^2}{r^2+a^2}\Box_g\left(\frac{2rf}{r^2+a^2}\right)\phi^2.
\end{multline*}
When adding these terms to the expression in Lemma \ref{divJphiX_lem}, the $\alpha(\pd_r\phi)^2-\alpha^{-1}(\pd_t\phi)^2$ terms cancel (this was the reason for the choice of $w=\frac{2rf}{r^2+a^2}$) and the result is as desired. \qed
\end{proof}

The term $-\frac12\frac{q^2}{r^2+a^2}\Box_g\left(\frac{2rf}{r^2+a^2}\right)\phi^2$ is like $-r^{-1}f''\phi^2$. In the future, when $f\sim r^p$, this will have a sign $-p(p-1)$. The sign will be negative if $p>1$, which is bad. So we include a divergence term to fix it. (But in doing so, we almost lose some other good terms--this is why we need a small parameter $\epsilon$.) This is the point of the following Lemma.
\begin{lemma}
\begin{multline*}
\alpha^{-1}f' (L\phi)^2+\frac{q^2}{r^2+a^2}\left[-\frac12 \Box_g\left(\frac{2rf}{r^2+a^2}\right)\phi^2+(1-\epsilon)div\left(\phi^2\frac{r}{q^2}f'L\right)\right] \\
=\alpha f' \left(\alpha^{-1}L\phi+\frac{(1-\epsilon)r}{r^2+a^2}\phi\right)^2 + \epsilon\alpha \frac{\left((1-\epsilon)f'-rf''\right)}{r^2+a^2}\phi^2 
+ \alpha'\left(-\frac{\epsilon r f'}{r^2+a^2}+\frac{(r^2-a^2)f}{(r^2+a^2)^2}\right)\phi^2 \\
+ \frac{a^2}{r^2+a^2}\left(\frac{-(1+\epsilon)\alpha f'}{r^2+a^2}-\frac{4\alpha rf}{(r^2+a^2)^2}\right)\phi^2.
\end{multline*}
\end{lemma}
\begin{proof}
First, we calculate
\begin{align*}
-\frac{q^2}{r^2+a^2}\frac12\Box_g\left(\frac{2rf}{r^2+a^2}\right)\phi^2 &=-\frac{1}{r^2+a^2}\pd_r\left((r^2+a^2)\alpha\pd_r\left(\frac{rf}{r^2+a^2}\right)\right)\phi^2 \\
&= -\frac{\alpha}{r^2+a^2}\pd_r\left((r^2+a^2)\pd_r\left(\frac{rf}{r^2+a^2}\right)\right) \phi^2 -\alpha'\pd_r\left(\frac{rf}{r^2+a^2}\right)\phi^2 \\
&= -\frac{\alpha r f''}{r^2+a^2}\phi^2 \\
&\hspace{.75in}-\alpha'\pd_r\left(\frac{rf}{r^2+a^2}\right)\phi^2 -\frac{2\alpha a^2}{(r^2+a^2)^2}\left(f'+\frac{2r}{r^2+a^2}f\right)\phi^2.
\end{align*}
We also calculate
\begin{align*}
\frac{q^2}{r^2+a^2}div\left(\phi^2\frac{r}{q^2}f'L\right) &= \frac{1}{r^2+a^2}\pd_\alpha\left(\phi^2rf' L^\alpha\right) \\
&=\frac{rf'}{r^2+a^2}2\phi L\phi +\frac{\pd_r(rf'\alpha)}{r^2+a^2}\phi^2 \\
&=\frac{r f'}{r^2+a^2}2\phi L\phi +\frac{\alpha f'}{r^2+a^2}\phi^2 +\frac{\alpha r f''}{r^2+a^2}\phi^2+\frac{\alpha' r f'}{r^2+a^2}\phi^2.
\end{align*}
The first two terms in the last line almost complete a square (up to a term on the order of $\frac{a^2}{r^2+a^2}$) with the term $\alpha^{-1} f' (L\phi)^2$. The third term cancels with the first term from the previous calculation. However, it will be beneficial to introduce the factor $1-\epsilon$ that appears in the lemma, so that a good term appears with an $\epsilon$ factor. This is summarized by the following two calculations.
\begin{multline*}
\alpha^{-1}f'(L\phi)^2+(1-\epsilon)\left(\frac{rf'}{r^2+a^2}2\phi L\phi+\frac{\alpha f'}{r^2+a^2}\phi^2\right) \\
= \alpha f' \left(\alpha^{-1}L\phi+\frac{(1-\epsilon)r}{r^2+a^2}\phi\right)^2 -\frac{(1-\epsilon)^2r^2\alpha f'}{(r^2+a^2)^2}\phi^2 +\frac{(1-\epsilon)\alpha f'}{r^2+a^2}\phi^2 \\
= \alpha f' \left(\alpha^{-1}L\phi+\frac{(1-\epsilon)r}{r^2+a^2}\phi\right)^2 +\frac{\epsilon (1-\epsilon)\alpha f'}{r^2+a^2}\phi^2 +\frac{a^2(1-\epsilon)\alpha f'}{(r^2+a^2)^2}\phi^2
\end{multline*}
and
$$-\frac{\alpha r f''}{r^2+a^2}\phi^2 +(1-\epsilon)\frac{\alpha r f''}{r^2+a^2}\phi^2=-\epsilon \frac{\alpha r f''}{r^2+a^2}\phi^2.$$
Adding these terms together and ignoring the term with the $a^2$ factor yields
$$\alpha f' \left(\alpha^{-1}L\phi+\frac{(1-\epsilon)r}{r^2+a^2}\phi\right)^2 +\epsilon \alpha \frac{\left((1-\epsilon)f'-rf''\right)}{r^2+a^2}\phi^2.$$
All the remaining terms (which either contain a factor of $\alpha'\sim \frac{M}{r^2}$ or $\frac{a^2}{r^2+a^2}$) are
$$
\alpha'\left[-\pd_r\left(\frac{rf}{r^2+a^2}\right)+(1-\epsilon)\frac{rf'}{r^2+a^2}\right]\phi^2 
+\frac{a^2}{r^2+a^2}\left[-\frac{2\alpha}{r^2+a^2}\left(f'+\frac{2r}{r^2+a^2}f\right)+\frac{(1-\epsilon)\alpha f'}{r^2+a^2}\right]\phi^2
$$
Adding both of these yields the result. \qed
\end{proof}

Thus, we have shown that if $\Box_g\phi=0$, then
\begin{multline*}
\frac{q^2}{r^2+a^2}divJ_{(\phi)}\left[\alpha^{-1} f L,\frac{2rf}{r^2+a^2},(1-\epsilon)\frac{rf'}{q^2}L\right] \\
= \alpha f' \left(\alpha^{-1}L\phi+\frac{(1-\epsilon)r}{r^2+a^2}\phi\right)^2 + \epsilon\alpha \frac{\left((1-\epsilon)f'-rf''\right)}{r^2+a^2}\phi^2 +\left(\frac{2rf}{r^2+a^2}-f'\right)\frac{Q^{\alpha\beta}}{r^2+a^2}\pd_\alpha\phi\pd_\beta\phi \\
- \alpha' \alpha^{-2}f(L\phi)^2
+ \alpha'\left(-\frac{\epsilon r f'}{r^2+a^2}+\frac{(r^2-a^2)f}{(r^2+a^2)^2}\right)\phi^2  \\
+ \frac{a^2}{r^2+a^2}\left(\frac{-(1+\epsilon)\alpha f'}{r^2+a^2}-\frac{4\alpha rf}{(r^2+a^2)^2}\right)\phi^2.
\end{multline*}
If we remove the assumption that $\Box_g\phi=0$, there is an additional term
$$\left(2X(\phi)+w\phi\right)\Box_g\phi = \left(2\alpha^{-1}fL\phi+\frac{2rf}{r^2+a^2}\phi\right)\Box_g\phi$$
appearing in the expression for $divJ_{(\phi)}$.

Finally, we turn to the boundary terms. Since we have assumed that $f$ is supported away from the event horizon, it suffices to compute $-J_{(\phi)}^t$.
\begin{lemma}
\begin{multline*}
-J_{(\phi)}^t\left[\alpha^{-1}fL,\frac{2rf}{r^2+a^2},(1-\epsilon)\frac{rf'}{q^2}L\right] \\
=\left(1-\alpha\frac{a^2\sin^2\theta}{r^2+a^2}\right)\frac{r^2+a^2}{q^2}f\left(\alpha^{-1}L\phi+\frac{r}{r^2+a^2}\phi\right)^2 +\frac{\alpha^{-1}f}{q^2}(\pd_\theta\phi)^2+\epsilon\frac{rf'}{q^2}\phi^2 \\
+\alpha\frac{a^2\sin^2\theta}{q^2}f\left(\pd_r\phi+\frac{r}{r^2+a^2}\phi\right)^2 +\frac{a^2f}{q^2(r^2+a^2)}\phi^2-\frac1{q^2}\pd_r(rf\phi^2).
\end{multline*}
\end{lemma}
\begin{proof}
We have
\begin{align*}
-J_{(\phi)}^t[\alpha^{-1}fL] &= -2\pd^t\phi\alpha^{-1}fL\phi+\alpha^{-1}fL^t\pd^\lambda\phi\pd_\lambda\phi \\
&=-2\alpha^{-1}f(g^{tt}+{}^{(Q)}g^{tt})\pd_t\phi L\phi \\
&\hspace{1.5in}+\alpha^{-1}f\left((g^{tt}+{}^{(Q)}g^{tt})(\pd_t\phi)^2+g^{rr}(\pd_r\phi)^2+{}^{(Q)}g^{\theta\theta}(\pd_\theta\phi)^2\right) \\
&=-\alpha^{-1}f(g^{tt}+{}^{(Q)}g^{tt})(\pd_t\phi)^2-2\alpha^{-1}f(g^{tt}+{}^{(Q)}g^{tt})\pd_t\phi\alpha\pd_r\phi+\alpha^{-1}fg^{rr}(\pd_r\phi)^2 \\
&\hspace{4in}+\alpha^{-1}f{}^{(Q)}g^{\theta\theta}(\pd_\theta\phi)^2 \\
&=\frac{r^2+a^2}{q^2}\left(1-\alpha\frac{a^2\sin^2\theta}{r^2+a^2}\right)\alpha^{-2}f(\pd_t\phi)^2+2\frac{r^2+a^2}{q^2}\left(1-\alpha\frac{a^2\sin^2\theta}{r^2+a^2}\right)\alpha^{-1}f\pd_t\phi\pd_r\phi \\
&\hspace{3.5in}+\frac{r^2+a^2}{q^2}f(\pd_r\phi)^2+\frac{\alpha^{-1}f}{q^2}(\pd_\theta\phi)^2 \\
&=\frac{r^2+a^2}{q^2}\left(1-\alpha\frac{a^2\sin^2\theta}{r^2+a^2}\right)\alpha^{-2}f(L\phi)^2+\alpha\frac{a^2\sin^2\theta}{q^2}f(\pd_r\phi)^2+\frac{\alpha^{-1}f}{q^2}(\pd_\theta\phi)^2.
\end{align*}
Also,
\begin{align*}
-J^t_{(\phi)}\left[0,\frac{2rf}{r^2+a^2}\right] &= -\frac{2rf}{r^2+a^2}\phi\pd^t\phi \\
&= -\frac{2rf}{r^2+a^2}(g^{tt}+{}^{(Q)}g^{tt})\phi\pd_t\phi \\
&= \frac{2r\alpha^{-1}f}{q^2}\left(1-\alpha\frac{a^2\sin^2\theta}{r^2+a^2}\right)\phi\pd_t\phi \\
&= \frac{2r\alpha^{-1}f}{q^2}\left(1-\alpha\frac{a^2\sin^2\theta}{r^2+a^2}\right)\phi L\phi - \frac{2rf}{q^2}\left(1-\alpha\frac{a^2\sin^2\theta}{r^2+a^2}\right)\phi\pd_r\phi \\
&= \frac{2r\alpha^{-1}f}{q^2}\left(1-\alpha\frac{a^2\sin^2\theta}{r^2+a^2}\right)\phi L\phi -\frac{2rf}{q^2}\phi\pd_r\phi +\frac{2rf}{q^2}\alpha\frac{a^2\sin^2\theta}{r^2+a^2}\phi\pd_r\phi \\
&= \frac{2r\alpha^{-1}f}{q^2}\left(1-\alpha\frac{a^2\sin^2\theta}{r^2+a^2}\right)\phi L\phi +\left(-\frac{1}{q^2}\pd_r(rf\phi^2)+\frac{f+rf'}{q^2}\phi^2\right) \\
&\hspace{3.75in}+\frac{2rf}{q^2}\alpha\frac{a^2\sin^2\theta}{r^2+a^2}\phi\pd_r\phi \\
&= \left(1-\alpha\frac{a^2\sin^2\theta}{r^2+a^2}\right)\left(\frac{2r\alpha^{-1}f}{q^2}\phi L\phi+\frac{f}{q^2}\phi^2\right)+\alpha\frac{a^2\sin^2\theta}{r^2+a^2}\left(\frac{2rf}{q^2}\phi\pd_r\phi+\frac{f}{q^2}\phi^2\right) \\
&\hspace{3.75in}+\frac{rf'}{q^2}\phi^2-\frac{1}{q^2}\pd_r(rf\phi^2).
\end{align*}
Now, observe that
$$\frac{r^2+a^2}{q^2}\alpha^{-2}f(L\phi)^2+\frac{2r\alpha^{-1}f}{q^2}\phi L\phi+\frac{f}{q^2}\phi^2 = \frac{r^2+a^2}{q^2}f\left(\alpha^{-1}L\phi+\frac{r}{r^2+a^2}\phi\right)^2+\frac{a^2f}{q^2(r^2+a^2)}\phi^2$$
and
$$\frac{r^2+a^2}{q^2}f(\pd_r\phi)^2 +\frac{2rf}{q^2}\phi\pd_r\phi+\frac{f}{q^2}\phi^2 = \frac{r^2+a^2}{q^2}f\left(\pd_r\phi+\frac{r}{r^2+a^2}\phi\right)^2+\frac{a^2f}{q^2(r^2+a^2)}\phi^2.$$
Thus,
\begin{multline*}
-J_{(\phi)}^t\left[\alpha^{-1}fL,\frac{2rf}{r^2+a^2}\right] \\
=\left(1-\alpha\frac{a^2\sin^2\theta}{r^2+a^2}\right)\frac{r^2+a^2}{q^2}f\left(\alpha^{-1}L\phi+\frac{r}{r^2+a^2}\phi\right)^2+\alpha\frac{a^2\sin^2\theta}{r^2+a^2}\frac{r^2+a^2}{q^2}f\left(\pd_r\phi+\frac{r}{r^2+a^2}\phi\right)^2 \\
+\frac{a^2f}{q^2(r^2+a^2)}\phi^2+\frac{\alpha^{-1}f}{q^2}(\pd_\theta\phi)^2+\frac{rf'}{q^2}\phi^2-\frac1{q^2}\pd_r(rf\phi^2).
\end{multline*}
Also,
$$-J_{(\phi)}^t\left[0,0,(1-\epsilon)\frac{rf'}{q^2}L\right] = -(1-\epsilon)\frac{rf'}{q^2}\phi^2 L^t =-(1-\epsilon)\frac{rf'}{q^2}\phi^2.$$
Adding these two expressions together yields the result. \qed
\end{proof}

This concludes the proof. \qed
\end{proof}

\subsubsection{A $p$-identity for $\tilde{\mathcal{M}}$}\label{p_identity_M_tilde_sec}
We prove the following lemma for the spacetime $\tilde{\mathcal{M}}$ by following a similar procedure to that in the proof of Lemma \ref{p_identity_phi_lem}.
\begin{lemma}\label{p_identity_psi_lem} ($p$ identity for $\tilde{\mathcal{M}}$) Let $\alpha=\frac{\Delta}{r^2+a^2}$ and $L=\alpha\pd_r+\pd_t$. For any function $f=f(r)$ supported where $r>r_H+\delh$, the following identity holds.
\begin{align*}
&\int_{\tilde{\Sigma}_{t_2}}\left[\left(1-\alpha\frac{a^2\sin^2\theta}{r^2+a^2}\right)\frac{r^2+a^2}{q^2}f\left(\alpha^{-1}L\psi+\frac{3r}{r^2+a^2}\psi\right)^2 +\frac{\alpha^{-1}f}{q^2}(\pd_\theta\psi)^2 +6f\frac{\psi^2}{q^2} \right.\\
&\hspace{1.25in}+\left.\frac{a^2\sin^2\theta \alpha f}{q^2}\left(\pd_r\psi+\frac{3r}{r^2+a^2}\psi\right)^2 -\frac{a^23f}{r^2+a^2}\frac{\psi^2}{q^2}-\frac1{A_1^2q^2}\pd_r(A_1^23rf\psi^2)\right] \\
&+\int_{t_1}^{t_2}\int_{\tilde{\Sigma}_t}
\left[
  \left(\frac{2rf}{r^2+a^2}-f'\right)\frac{Q^{\alpha\beta}}{q^2}\pd_\alpha\psi\pd_\beta\psi
+ \alpha f'\frac{r^2+a^2}{q^2}\left(\alpha^{-1}L\psi+\frac{3r}{r^2+a^2}\psi\right)^2 
\vphantom{
+ \epsilon\alpha\left((1-\epsilon)f'-rf''\right)\frac{\phi^2}{q^2} 
+ \alpha'\left(\frac{r^2-a^2}{r^2+a^2}f-\epsilon r f'\right)\frac{\phi^2}{q^2}
+ \frac{a^2}{r^2+a^2}\left(\alpha((1-\epsilon)^2-2)f'-\frac{4\alpha rf}{r^2+a^2}\right)\frac{\phi^2}{q^2}
}\right. \\
&\hspace{3in} + \frac{6r\alpha(2f-rf')}{r^2+a^2}\frac{\psi^2}{q^2} -\alpha'\alpha^{-2}f\frac{r^2+a^2}{q^2}(L\psi)^2 \\
&\hspace{1.5in}
\left.\vphantom{
  \left(\frac{2rf}{r^2+a^2}-f'\right)\frac{Q^{\alpha\beta}}{q^2}\pd_\alpha\phi\pd_\beta\phi
+ \alpha f'\frac{r^2+a^2}{q^2}\left(\alpha^{-1}L\phi+\frac{(1-\epsilon)r}{r^2+a^2}\phi\right)^2 
+ \epsilon\alpha\left((1-\epsilon)f'-rf''\right)\frac{\phi^2}{q^2} 
}
- \alpha'(r^2+a^2)\pd_r\left(\frac{3r}{r^2+a^2}\right)f\frac{\psi^2}{q^2}
- \frac{a^2}{r^2+a^2}\left(3\alpha f'+\frac{36\alpha rf}{r^2+a^2}\right)\frac{\psi^2}{q^2}
\right] \\
=&\int_{\tilde{\Sigma}_{t_1}}\left[\left(1-\alpha\frac{a^2\sin^2\theta}{r^2+a^2}\right)\frac{r^2+a^2}{q^2}f\left(\alpha^{-1}L\psi+\frac{3r}{r^2+a^2}\psi\right)^2 +\frac{\alpha^{-1}f}{q^2}(\pd_\theta\psi)^2 +6f\frac{\psi^2}{q^2} \right.\\
&\hspace{1.25in}+\left.\frac{a^2\sin^2\theta \alpha f}{q^2}\left(\pd_r\psi+\frac{3r}{r^2+a^2}\psi\right)^2 -\frac{a^23f}{r^2+a^2}\frac{\psi^2}{q^2}-\frac1{A_1^2q^2}\pd_r(A_1^23rf\psi^2)\right] \\
&+\int_{t_1}^{t_2}\int_{\tilde{\Sigma}_t}-\left(2\alpha^{-1}fL\psi+\frac{6rf}{r^2+a^2}\psi\right)\Box_{\tilde{g}}\psi.
\end{align*}
\end{lemma}

\begin{proof}
The proof is similar to the proof of Lemma \ref{p_identity_phi_lem}, however there are a few subtle differences, some of which actually simplify the proof. Again, we will use Proposition \ref{general_divergence_estimate_prop} together with the following current template.
$$J_{(\psi)}[X,w,m]_\mu = T_{\mu\nu} X^\nu +w\psi\pd_\mu\psi-\frac12\psi^2\pd_\mu w+m_\mu \psi^2,$$
$$T_{\mu\nu}=2\pd_\mu\psi\pd_\nu\psi-g_{\mu\nu}\pd^\lambda\psi\pd_\lambda\psi.$$

Assume for now that $\Box_{\tilde{g}}\psi=0$. Let $\alpha = \frac{\Delta}{r^2+a^2}$, and observe that
$$L=\alpha\pd_r+\pd_t,$$
$$q^2g^{rr}=(r^2+a^2)\alpha,$$
$$q^2g^{tt}=-(r^2+a^2)\alpha^{-1}.$$

\begin{lemma}\label{divJpsiX_lem}
Without appealing directly to the particular expression for $\alpha$, one can deduce the following.
\begin{multline*}
\frac{q^2}{r^2+a^2}divJ_{(\psi)}[\alpha^{-1}fL]=(\alpha^{-1}f)'(L\psi)^2-\frac{6rf}{r^2+a^2}\left(\alpha(\pd_r\psi)^2-\alpha^{-1}(\pd_t\psi)^2\right) \\
-\left(\frac{4rf}{r^2+a^2}+f'\right)\frac{Q^{\alpha\beta}}{r^2+a^2}\pd_\alpha\psi\pd_\beta\psi.
\end{multline*}
\end{lemma}
\begin{proof}
Note that
$$div J_{(\psi)}[X] = K^{\mu\nu}\pd_\mu\psi\pd_\nu\psi,$$
where
$$K^{\mu\nu}=2g^{\mu\lambda}\pd_\lambda X^\nu-X^\lambda \pd_\lambda(g^{\mu\nu})-div X g^{\mu\nu}.$$

Set $X=\alpha^{-1}f(\alpha\pd_r+\pd_t)=f\pd_r+\alpha^{-1}f\pd_t$. From the above formula, since $g^{rt}=0$,
$$\frac{q^2}{r^2+a^2}(K^{tr}+K^{rt})=2\frac{q^2}{r^2+a^2}g^{rr}\pd_r X^t=2\alpha\pd_r(\alpha^{-1} f).$$
Thus, the expression for $\frac{q^2}{r^2+a^2}divJ_{(\psi)}[\alpha^{-1}fL]$ will have a mixed term of the form 
$$2\alpha\pd_r(\alpha^{-1}f)\pd_r\psi\pd_t\psi.$$
Note that
\begin{align*}
(\alpha^{-1} f)'(L\psi)^2 &= (\alpha^{-1} f)'(\alpha\pd_r\psi+\pd_t\psi)^2 \\
&= \alpha^2(\alpha^{-1}f)'(\pd_r\psi)^2+2\alpha(\alpha^{-1}f)'\pd_r\psi\pd_t\psi +(\alpha^{-1} f)'(\pd_t\psi)^2.
\end{align*}
We now compute the $(\pd_r\psi)^2$ and $(\pd_t\psi)^2$ components, subtracting the part that will be grouped with the $(L\psi)^2$ term.
\begin{align*}
\frac{q^2}{r^2+a^2}K^{rr}-\alpha^2(\alpha^{-1}f)' &= \frac{q^2}{r^2+a^2}\left[2g^{rr}\pd_rX^r-X^r\pd_r g^{rr}-\frac1{A_1^2q^2}\pd_r(A_1^2q^2X^r)g^{rr}\right]-\alpha^2(\alpha^{-1}f)' \\
&= \frac{q^2}{r^2+a^2}\left[2g^{rr}\pd_rX^r-\frac1{A_1^2q^2}\pd_r(A_1^2q^2g^{rr}X^r)\right]-\alpha^2(\alpha^{-1}f)' \\
&= 2\alpha\pd_r f -\frac{1}{(r^2+a^2)^3}\pd_r\left((r^2+a^2)^3\alpha f\right)-\alpha^2(\alpha^{-1}f)' \\
&= -\frac{6r \alpha f}{r^2+a^2}
\end{align*}
and
\begin{align*}
\frac{q^2}{r^2+a^2}K^{tt}-(\alpha^{-1}f)' &= \frac{q^2}{r^2+a^2}\left[-X^r\pd_r g^{tt}-\frac{1}{A_1^2q^2}\pd_r(A_1^2q^2X^r)g^{tt}\right]-(\alpha^{-1}f)' \\
&= -\frac{q^2}{r^2+a^2}\frac{1}{A_1^2q^2}\pd_r(A_1^2q^2 g^{tt} X^r) -(\alpha^{-1}f)' \\
&= -\frac{1}{(r^2+a^2)^3}\pd_r\left((r^2+a^2)^3(-\alpha^{-1}) f\right)-(\alpha^{-1}f)' \\
&= \frac{6r\alpha^{-1}f}{r^2+a^2}.
\end{align*}
Finally,
\begin{align*}
\frac{q^2}{r^2+a^2}{}^{(Q)}K^{\alpha\beta} &= \frac{q^2}{r^2+a^2}\left[-X^r\pd_r {}^{(Q)}g^{\alpha\beta}-\frac1{A_1^2q^2}\pd_r\left(A_1^2q^2X^r\right){}^{(Q)}g^{\alpha\beta}\right] \\
&= \frac{q^2}{r^2+a^2}\left[-\frac{1}{A_1^2q^2}\pd_r\left(A_1^2q^2{}^{(Q)}g^{\alpha\beta}X^r\right)\right] \\
&= -\frac{1}{(r^2+a^2)^3}\pd_r((r^2+a^2)^2Q^{\alpha\beta}f) \\
&= -\left(\frac{4rf}{r^2+a^2}+f'\right) \frac{Q^{\alpha\beta}}{r^2+a^2}.
\end{align*}
Combining all these terms gives the identity stated in the lemma. \qed
\end{proof}

Next, we choose $w=\frac{6rf}{r^2+a^2}$ to directly cancel with the middle term in the above lemma.
\begin{lemma}\label{divJpsiXw_lem}
\begin{multline*}
\frac{q^2}{r^2+a^2}divJ_{(\psi)}\left[\alpha^{-1}fL,\frac{6rf}{r^2+a^2}\right] \\
= (\alpha^{-1}f)'(L\psi)^2+\left(\frac{2rf}{r^2+a^2}-f'\right)\frac{Q^{\alpha\beta}}{r^2+a^2}\pd_\alpha\psi\pd_\beta\psi-\frac12\frac{q^2}{r^2+a^2}\Box_{\tilde{g}}\left(\frac{6rf}{r^2+a^2}\right)\psi^2.
\end{multline*}
\end{lemma}
\begin{proof}
Note that
$$divJ_{(\psi)}[0,w]=wg^{\mu\nu}\pd_\mu\psi\pd_\nu\psi-\frac12\Box_{\tilde{g}}w \psi^2.$$
We compute the new terms only.
\begin{multline*}
\frac{q^2}{r^2+a^2}divJ_{(\psi)}\left[0,\frac{6rf}{r^2+a^2}\right] = \frac{6rf}{r^2+a^2}\frac{q^2 g^{\alpha\beta}}{r^2+a^2}\pd_\alpha\psi\pd_\beta\psi -\frac12\frac{q^2}{r^2+a^2}\Box_{\tilde{g}}\left(\frac{6rf}{r^2+a^2}\right)\psi^2 \\
= \frac{6rf}{r^2+a^2}\left(\alpha (\pd_r\psi)^2-\alpha^{-1}(\pd_t\psi)^2\right) +\frac{6rf}{r^2+a^2}\frac{Q^{\alpha\beta}}{r^2+a^2}\pd_\alpha\psi\pd_\beta\psi -\frac12\frac{q^2}{r^2+a^2}\Box_{\tilde{g}}\left(\frac{6rf}{r^2+a^2}\right)\psi^2.
\end{multline*}
When adding these terms to the expression in Lemma \ref{divJpsiX_lem}, the $\alpha(\pd_r\psi)^2-\alpha^{-1}(\pd_t\psi)^2$ terms cancel (this was the reason for the choice of $w=\frac{2rf}{r^2+a^2}$) and the result is as desired. \qed
\end{proof}

The term $-\frac12\frac{q^2}{r^2+a^2}\Box_{\tilde{g}}\left(\frac{6rf}{r^2+a^2}\right)\psi^2$ is like $-6r^{-1}f''-24r^{-1}\pd_r(r^{-1}f)\phi^2$. In the future, when $f\sim r^p$, this will have a sign $-p^2-3p+4=-(p+4)(p-1)$. The sign will be negative if $p>1$, which is bad. So we include a divergence term to fix it. (Unlike in the analogous step for the $\phi$ version, there will not be a need for a smallness parameter $\epsilon$.) This is the point of the following Lemma.
\begin{lemma}
\begin{multline*}
\alpha^{-1}f'(L\psi)^2+\frac{q^2}{r^2+a^2}\left[-\frac12\Box_{\tilde{g}}\left(\frac{6rf}{r^2+a^2}\right)\psi^2+div\left(\psi^2\frac{3rf'}{q^2}L\right)\right] \\
= \alpha f'\left(\alpha^{-1}L\psi+\frac{3r}{r^2+a^2}\psi\right)^2+\frac{6r\alpha(2f-rf')}{(r^2+a^2)^2}\psi^2 \\
-\alpha'\pd_r\left(\frac{3r}{r^2+a^2}\right)f\psi^2 +\frac{a^2}{r^2+a^2}\left(\frac{-3\alpha f'}{r^2+a^2}+\frac{-36\alpha r f}{(r^2+a^2)^2}\right)\psi^2.
\end{multline*}
\end{lemma}
\begin{proof}
First, borrowing from a calculation for the $\phi$ version, we obtain
\begin{align*}
-\frac{q^2}{r^2+a^2}\frac12\Box_{\tilde{g}}\left(\frac{6rf}{r^2+a^2}\right)\psi^2 
&= -\frac{1}{A_1^2(r^2+a^2)}\pd_r\left(A_1^2(r^2+a^2)\alpha\pd_r\left(\frac{3rf}{r^2+a^2}\right)\right)\psi^2 \\
&= -\frac{q^2}{r^2+a^2}\Box_g\left(\frac{3rf}{r^2+a^2}\right)\psi^2-\frac{\pd_rA_1^2}{A_1^2}\alpha\pd_r\left(\frac{3rf}{r^2+a^2}\right)\psi^2 \\
&= -\frac{3\alpha rf''}{r^2+a^2}\psi^2-\frac{\pd_rA_1^2}{A_1^2}\alpha\pd_r\left(\frac{3rf}{r^2+a^2}\right)\psi^2 \\
&\hspace{.5in}-\alpha'\pd_r\left(\frac{3rf}{r^2+a^2}\right)\psi^2-\frac{6\alpha a^2}{(r^2+a^2)^2}\left(f'+\frac{2r}{r^2+a^2}f\right)\psi^2
\end{align*}
We also calculate
\begin{multline*}
\frac{q^2}{r^2+a^2}div\left(\psi^2\frac{3r}{q^2}f'L\right) = \frac{1}{A_1^2(r^2+a^2)}\pd_\alpha(A_1^2\psi^2 3rf'L^\alpha) \\
=\frac{3rf'}{r^2+a^2}2\psi L\psi + \frac{\pd_r(3r f'\alpha)}{r^2+a^2}\psi^2+\frac{\pd_rA_1^2}{A_1^2}\frac{3r\alpha f'}{r^2+a^2}\psi^2 \\
= \frac{3rf'}{r^2+a^2}2\psi L\psi +\frac{3\alpha f'}{r^2+a^2}\psi^2 + \frac{3\alpha rf''}{r^2+a^2}\psi^2 +\frac{\pd_rA_1^2}{A_1^2}\frac{3r\alpha f'}{r^2+a^2}\psi^2 +\alpha'\frac{3 r f'}{r^2+a^2}\psi^2  \\
= \frac{3rf'}{r^2+a^2}2\psi L\psi +\frac{9\alpha f'}{r^2+a^2}\psi^2 + \frac{3\alpha rf''}{r^2+a^2}\psi^2 -\frac{6\alpha f'}{r^2+a^2}\psi^2+\frac{\pd_rA_1^2}{A_1^2}\frac{3r\alpha f'}{r^2+a^2}\psi^2 +\alpha'\frac{3 r f'}{r^2+a^2}\psi^2 
\end{multline*}
The first two terms in the last line almost complete a square (up to a term on the order of $\frac{a^2}{r^2+a^2}$) with the term $\alpha^{-1}f'(L\psi)^2$. The third term cancels with the first term from the previous calculation. Due to the fourth and fifth terms, which did not show in the calculation for $\phi$, the $\epsilon$ parameter is not needed here, allowing for a slightly simpler calculation.
$$\alpha^{-1}f'(L\psi)^2+\frac{3rf'}{r^2+a^2}2\psi L\psi +\frac{9\alpha f'}{r^2+a^2}\psi^2=\alpha f'\left(\alpha^{-1}L\psi+\frac{3r}{r^2+a^2}\psi\right)^2+\frac{a^2 9\alpha f'}{(r^2+a^2)^2}\psi^2$$
and
$$-\frac{3\alpha rf''}{r^2+a^2}\psi^2+\frac{3\alpha rf''}{r^2+a^2}\psi^2=0.$$
The new terms are
\begin{align*}
-\frac{\pd_rA_1^2}{A_1^2}&\alpha\pd_r\left(\frac{3rf}{r^2+a^2}\right)\psi^2 -\frac{6\alpha f'}{r^2+a^2}\psi^2+\frac{\pd_rA_1^2}{A_1^2}\frac{3r\alpha f'}{r^2+a^2}\psi^2 \\
&= -\frac{6\alpha f'}{r^2+a^2}\psi^2 -\frac{\pd_r A_1^2}{A_1^2} \pd_r\left(\frac{3r}{r^2+a^2}\right)\alpha f \psi^2 \\
&= -\frac{6\alpha f'}{r^2+a^2}\psi^2 -\left(\frac{4r}{r^2+a^2}\right)\left(\frac{3}{r^2+a^2}-\frac{6r^2}{(r^2+a^2)^2}\right)\alpha f \psi^2 \\
&= -\frac{6r^2\alpha f'}{(r^2+a^2)^2}\psi^2 -\frac{6a^2\alpha f'}{(r^2+a^2)^2}-\left(\frac{4r}{r^2+a^2}\right)\left(-\frac{3}{r^2+a^2}+\frac{6a^2}{(r^2+a^2)^2}\right)\alpha f\psi^2 \\
&= \frac{-6r^2\alpha f'+12r\alpha f}{(r^2+a^2)^2}\psi^2 +\frac{a^2}{r^2+a^2}\left(\frac{-6(r^2+a^2)\alpha f'-24r\alpha f}{(r^2+a^2)^2}\right)\psi^2
\end{align*}
Adding these terms together and ignoring terms with an $a^2$ factor yields
$$\alpha f'\left(\alpha^{-1}L\psi+\frac{3r}{r^2+a^2}\psi\right)^2+\frac{6r\alpha(2f-rf')}{(r^2+a^2)^2}\psi^2.$$
All the remaining terms (which either contain a factor of $\alpha'\sim \frac{M}{r^2}$ or $\frac{a^2}{r^2+a^2}$) are
\begin{align*}
-\alpha'&\pd_r\left(\frac{3rf}{r^2+a^2}\right)\psi^2-\frac{6\alpha a^2}{(r^2+a^2)^2}\left(f'+\frac{2r}{r^2+a^2}f\right)\psi^2 +\alpha'\frac{3 r f'}{r^2+a^2}\psi^2 +\frac{a^2 9\alpha f'}{(r^2+a^2)^2}\psi^2 \\
&\hspace{3in}+\frac{a^2}{r^2+a^2}\left(\frac{-6(r^2+a^2)\alpha f'-24r\alpha f}{(r^2+a^2)^2}\right)\psi^2 \\
&= -\alpha'\pd_r\left(\frac{3r}{r^2+a^2}\right)f\psi^2 +\frac{a^2}{r^2+a^2}\left(\frac{(-6+9-6)\alpha f'}{r^2+a^2}+\frac{(-12-24)\alpha r f}{(r^2+a^2)^2}\right)\psi^2 \\
&= -\alpha'\pd_r\left(\frac{3r}{r^2+a^2}\right)f\psi^2 +\frac{a^2}{r^2+a^2}\left(\frac{-3\alpha f'}{r^2+a^2}+\frac{-36\alpha r f}{(r^2+a^2)^2}\right)\psi^2.
\end{align*}
Adding both of these yields the result. \qed
\end{proof}

Thus, we have shown that if $\Box_{\tilde{g}}\psi = 0$, then
\begin{multline*}
\frac{q^2}{r^2+a^2}div J_{(\psi)}\left[\alpha^{-1}fL,\frac{6rf}{r^2+a^2},\frac{3rf'}{q^2}L\right] \\
= \alpha f'\left(\alpha^{-1}L\psi+\frac{3r}{r^2+a^2}\psi\right)^2+\frac{6r\alpha(2f-rf')}{(r^2+a^2)^2}\psi^2 +\left(\frac{2rf}{r^2+a^2}-f'\right)\frac{Q^{\alpha\beta}}{r^2+a^2}\pd_\alpha\psi\pd_\beta\psi \\
-\alpha'\alpha^{-2}f(L\psi)^2-\alpha'\pd_r\left(\frac{3r}{r^2+a^2}\right)f\psi^2 +\frac{a^2}{r^2+a^2}\left(\frac{-3\alpha f'}{r^2+a^2}+\frac{-36\alpha r f}{(r^2+a^2)^2}\right)\psi^2.
\end{multline*}
If we remove the assumption that $\Box_{\tilde{g}}\psi=0$, there is an additional term
$$(2X(\psi)+w\psi)\Box_{\tilde{g}}\psi =\left(2\alpha^{-1}fL\psi+\frac{6rf}{r^2+a^2}\psi\right)\Box_{\tilde{g}}\psi$$
appearing in the expression for $div J_{(\psi)}$.

Finally, we turn to the boundary terms. Since we have assumed that $f$ is supported away from the event horizon, it suffices to compute $-J^t_{(\psi)}$.
\begin{lemma}
\begin{multline*}
-J_{(\psi)}^t\left[\alpha^{-1}fL,\frac{6rf}{r^2+a^2},\frac{3rf'}{q^2}L\right] \\
=\left(1-\alpha\frac{a^2\sin^2\theta}{r^2+a^2}\right)\frac{r^2+a^2}{q^2}f\left(\alpha^{-1}L\psi+\frac{3r}{r^2+a^2}\psi\right)^2 +\frac{\alpha^{-1}f}{q^2}(\pd_\theta\psi)^2 +\frac{6f}{q^2}\psi^2\\
+\alpha\frac{a^2\sin^2\theta}{q^2}f\left(\pd_r\psi+\frac{3r}{r^2+a^2}\psi\right)^2 
-\frac{3a^2f}{q^2(r^2+a^2)}\psi^2-\frac1{A_1^2q^2}\pd_r(A_1^23rf\psi^2).
\end{multline*}
\end{lemma}
\begin{proof}
Borrowing a calculation for the $\phi$ version, we have
$$-J_{(\psi)}^t[\alpha^{-1}fL] = \frac{r^2+a^2}{q^2}\left(1-\alpha\frac{a^2\sin^2\theta}{r^2+a^2}\right)\alpha^{-2}f(L\psi)^2+\alpha\frac{a^2\sin^2\theta}{q^2}f(\pd_r\psi)^2+\frac{\alpha^{-1}f}{q^2}(\pd_\theta\psi)^2.$$
Again, borrowing a calculation for the $\phi$ version, we have
\begin{align*}
-J_{(\psi)}^t &\left[0,\frac{6rf}{r^2+a^2}\right] \\
&= \left(1-\alpha\frac{a^2\sin^2\theta}{r^2+a^2}\right)\left(\frac{6r\alpha^{-1}f}{q^2}\psi L\psi+\frac{3f}{q^2}\psi^2\right)+\alpha\frac{a^2\sin^2\theta}{r^2+a^2}\left(\frac{6rf}{q^2}\psi\pd_r\psi+\frac{3f}{q^2}\psi^2\right) \\
&\hspace{4in}+\frac{3rf'}{q^2}\psi^2-\frac{3}{q^2}\pd_r(rf\psi^2) \\
&= \left(1-\alpha\frac{a^2\sin^2\theta}{r^2+a^2}\right)\left(\frac{6r\alpha^{-1}f}{q^2}\psi L\psi+\frac{9f}{q^2}\psi^2\right)+\alpha\frac{a^2\sin^2\theta}{r^2+a^2}\left(\frac{6rf}{q^2}\psi\pd_r\psi+\frac{9f}{q^2}\psi^2\right) \\
&\hspace{2in}-\frac{6f}{q^2}\psi^2+\frac{3rf'}{q^2}\psi^2+\frac{\pd_rA_1^2}{A_1^2}\frac{3rf}{q^2}\psi^2-\frac{1}{A_1^2q^2}\pd_r(A_1^23rf\psi^2).
\end{align*}
Following a similar procedure as for the $\phi$ case, we notice that
$$\frac{r^2+a^2}{q^2}\alpha^{-2}f(L\psi)^2+\frac{6r\alpha^{-1}f}{q^2}\psi L\psi+\frac{9f}{q^2}\psi^2=\frac{r^2+a^2}{q^2}f\left(\alpha^{-1}L\psi+\frac{3r}{r^2+a^2}\psi\right)^2+\frac{9a^2f}{q^2(r^2+a^2)}\psi^2$$
and
$$\frac{r^2+a^2}{q^2}f(\pd_r\psi)^2+\frac{6rf}{q^2}\psi\pd_r\psi+\frac{9f}{q^2}\psi^2 = \frac{r^2+a^2}{q^2}f\left(\pd_r\psi+\frac{3r}{r^2+a^2}\psi\right)^2+\frac{9a^2f}{q^2(r^2+a^2)}\psi^2.$$
Also, there are two new terms, which we now combine.
$$-\frac{6f}{q^2}\psi^2+\frac{\pd_rA_1^2}{A_1^2}\frac{3rf}{q^2}\psi^2 = -\frac{6f}{q^2}\psi^2 +\frac{12r^2f}{q^2(r^2+a^2)}\psi^2 
= \frac{6f}{q^2}\psi^2-\frac{12a^2f}{q^2(r^2+a^2)}\psi^2.$$
Thus,
\begin{multline*}
-J_{(\psi)}^t\left[\alpha^{-1}fL,\frac{6rf}{r^2+a^2}\right] \\
=\left(1-\alpha\frac{a^2\sin^2\theta}{r^2+a^2}\right)\frac{r^2+a^2}{q^2}f\left(\alpha^{-1}L\psi+\frac{3r}{r^2+a^2}\psi\right)^2+\alpha\frac{a^2\sin^2\theta}{r^2+a^2}\frac{r^2+a^2}{q^2}f\left(\pd_r\psi+\frac{3r}{r^2+a^2}\psi\right)^2 \\
-\frac{3a^2f}{q^2(r^2+a^2)}\psi^2+\frac{\alpha^{-1}f}{q^2}(\pd_\theta\psi)^2+\frac{6f}{q^2}\psi^2+\frac{3rf'}{q^2}\psi^2-\frac1{A_1^2q^2}\pd_r(A_1^23rf\psi^2).
\end{multline*}
Finally,
$$-J_{(\psi)}^t\left[0,0,\frac{3rf'}{q^2}L\right]=-\frac{3rf'}{q^2}\psi^2L^t=-\frac{3rf'}{q^2}\psi^2.$$
Adding these two expressions together yields the result. \qed
\end{proof}

This concludes the proof. \qed
\end{proof}

\subsubsection{The incomplete $p$-weighted estimate near $i^0$}\label{p_identity_combined_sec}
Now we combine the identies from Lemmas \ref{p_identity_phi_lem} and \ref{p_identity_psi_lem} and make a choice for the function $f$ (so that $f=r^p$ for large $r$) to prove the following.
\begin{proposition}\label{incomplete_p_estimates_prop}
Fix $\delm,\delp>0$. Let $R$ be a sufficiently large radius. Then for all $p\in[\delm,2-\delp]$, the following estimate holds if $\phi$ and $\psi$ decay sufficiently fast as $r\rightarrow\infty$.
\begin{align*}
&\int_{\Sigma_{t_2}\cap\{r>2R\}}r^p\left[(L\phi)^2+|\sla\nabla\phi|^2+r^{-2}\phi^2\right]
+\int_{\tilde{\Sigma}_{t_2}\cap\{r>2R\}}r^p\left[(L\psi)^2+|\sla\nabla\psi|^2+r^{-2}\psi^2\right] \\
&+ \int_{t_1}^{t_2}\int_{\Sigma_t\cap\{r>2R\}}r^{p-1}\left[(L\phi)^2+|\sla\nabla\phi|^2+r^{-2}\phi^2\right] + \int_{t_1}^{t_2}\int_{\tilde{\Sigma}_t\cap\{r>2R\}}r^{p-1}\left[(L\psi)^2+|\sla\nabla\psi|^2+r^{-2}\psi^2\right] \\
&\lesssim \int_{\Sigma_{t_2}\cap\{r>2R\}}r^p\left[(L\phi)^2+|\sla\nabla\phi|^2+r^{-2}\phi^2\right]
+\int_{\tilde{\Sigma}_{t_2}\cap\{r>2R\}}r^p\left[(L\psi)^2+|\sla\nabla\psi|^2+r^{-2}\psi^2\right] +Err,
\end{align*}
where
\begin{align*}
Err &= Err_1 + Err_2 + Err_\Box \\
Err_1 &= \int_{t_1}^{t_2}\int_{\Sigma_t\cap\{R<r<2R\}}(L\phi)^2+|\sla\nabla\phi|^2+\frac{a^2}{M^2}(\pd_t\phi)^2+\phi^2  \\
&\hspace{1in}+ \int_{t_1}^{t_2}\int_{\tilde{\Sigma}_t\cap\{R<r<2R\}}(L\psi)^2+|\sla\nabla\psi|^2+\frac{a^2}{M^2}(\pd_t\phi)^2+\psi^2 \\
Err_2 &= \int_{\Sigma_{t_1}\cap\{r>R\}}a^2r^{p-2}(\pd_r\phi)^2 + \int_{\tilde{\Sigma}_{t_1}\cap\{r>R\}}a^2r^{p-2}(\pd_r\psi)^2 \\
Err_\Box &= \int_{t_1}^{t_2}\int_{\Sigma_t\cap\{R<r\}}r^p(|L\phi|+r^{-1}|\phi|)|\Box_g\phi| + \int_{t_1}^{t_2}\int_{\tilde{\Sigma}_t\cap\{R<r\}}r^p(|L\psi|+r^{-1}|\psi|)|\Box_{\tilde{g}}\psi|.
\end{align*}
\end{proposition}
\begin{remark}
Unlike most estimates, this estimate could have been separated into two valid estimates--one estimate depending only on $\phi$ and the other depending only on $\psi$. The reason is that the linear terms $\mathcal{L}_\phi$ and $\mathcal{L}_\psi$, which cause the mixing, were ignored.
\end{remark}
\begin{proof}
The estimate follows from the identities given in Lemmas \ref{p_identity_phi_lem} and \ref{p_identity_psi_lem}, and a particular choice for the function $f$.
$$f(r)=\rho^p,$$
where
$$
\rho = \left\{
\begin{array}{ll}
0 & r\le R \\
smooth & r\in[R,2R] \\
r & 2R< r.
\end{array}
\right.
$$
With this choice, we have 
$$f\ge 0$$
$$f'\ge 0$$
and for $r>2R$,
$$f = r^p$$
$$f'=pr^{p-1}.$$
Furthermore, for $r>2R$,
$$\frac{2rf}{r^2+a^2}-f' = \frac{2r^{p+1}}{r^2+a^2}-pr^{p-1} = \frac{(2-p)r^{p+1}}{r^2+a^2}-\frac{a^2pr^{p-1}}{r^2+a^2} \ge \frac{1}{1+a^2/(4R^2)}\left(2-p-\frac{a^2p}{4R^2}\right)r^{p-1}.$$
It follows that if $p\le 2-\delp$ and $R$ is sufficiently large so that $\frac{a^2p}{4R^2}\le \delp/2$, then for $r>2R$,
$$r^{p-1}\lesssim \frac{2rf}{r^2+a^2}-f'.$$
Also, for $r>2R$,
$$\epsilon\alpha ((1-\epsilon)f'-rf'') = \epsilon\alpha ((1-\epsilon)pr^{p-1}-p(p-1)r^{p-1}) = \epsilon\alpha p (2-\epsilon -p)r^{p-1}.$$
If $R$ is sufficiently large so that $\alpha>3/4$ and $p\le 2-\delp$ and $\epsilon\le \delp/2$, then
$$\epsilon r^{p-1} \lesssim \epsilon\alpha ((1-\epsilon)f'-rf'').$$
Also, for $r>2R$, if $p\le 2-\delp$, then
$$\frac{6r\alpha(2f-rf')}{r^2+a^2} = \frac{6r\alpha (2-p)r^p}{r^2+a^2} \sim (2-p)r^{p-1}.$$

We also note that there are some error terms that either have a factor of $\alpha'$ or $\frac{a^2}{r^2+a^2}$. Each of these terms has a smallness parameter available, since $R$ can be taken to be very large and
$$\alpha'\lesssim \frac{M}{R}r^{-1}$$
and
$$\frac{a^2}{r^2+a^2}\lesssim \frac{M^2}{R^2}.$$

Finally, we observe that if $\phi$ and $\psi$ vanish sufficiently fast as $r\rightarrow\infty$, then since $f$ is supported for $r>R$, we have
$$\int_{\Sigma_t}-\frac1{q^2}\pd_r(rf\phi^2) =0$$
and
$$\int_{\tilde{\Sigma}_t}-\frac{1}{A_1^2q^2}\pd_r(A_1^23rf\psi^2) = 0.$$

With these facts having been established, it is straightforward to check that the estimate follows from Lemmas  \ref{p_identity_phi_lem} and \ref{p_identity_psi_lem}. \qed
\end{proof}

\subsection{The $p$-weighted energy estimate}\label{phi_psi_p_ee_sec}

We conclude this section by proving the $p$-weighted energy estimate. This is a combination of the $h\pd_t$ estimate (Proposition \ref{translated_h_dt_prop}), the Morawetz estimate (Proposition \ref{translated_morawetz_prop}), and the incomplete $p$-weighted estimate (Proposition \ref{incomplete_p_estimates_prop}).
\begin{proposition}\label{p_estimates_prop}
Suppose $|a|/M$ is sufficiently small. Fix $\delm,\delp>0$ and let $p\in[\delm,2-\delp]$. Then the following estimate holds if $(\phi,\psi)$ decay sufficiently fast as $r\rightarrow\infty$.
\begin{multline*}
\int_{\Sigma_{t_2}}r^p\left[(L\phi)^2+|\sla\nabla\phi|^2+r^{-2}\phi^2 + r^{-2}(\pd_r\phi)^2\right] +\int_{\tilde\Sigma_{t_2}}r^p\left[(L\psi)^2+|\tsla\nabla\psi|^2+r^{-2}\psi^2+r^{-2}(\pd_r\psi)^2\right] \\
+ \int_{t_1}^{t_2}\int_{\Sigma_t}r^{p-1}\left[\chi_{trap}(L\phi)^2+\chi_{trap}|\sla\nabla\phi|^2+r^{-2}\phi^2+r^{-2}(\pd_r\phi)^2\right] \\
+ \int_{t_1}^{t_2}\int_{\tilde\Sigma_t}r^{p-1}\left[\chi_{trap}(L\psi)^2+\chi_{trap}|\tsla\nabla\psi|^2+r^{-2}\psi^2 +r^{-2}(\pd_r\psi)^2\right] \\
\lesssim \int_{\Sigma_{t_1}}r^p\left[(L\phi)^2+|\sla\nabla\phi|^2+r^{-2}\phi^2+r^{-2}(\pd_r\phi)^2\right] +\int_{\tilde\Sigma_{t_1}}r^p\left[(L\psi)^2+|\tsla\nabla\psi|^2+r^{-2}\psi^2+r^{-2}(\pd_r\psi)^2\right] \\
+Err,
\end{multline*}
where $\chi_{trap}=\left(1-\frac{r_{trap}}{r}\right)^2$ and
\begin{align*}
Err &= Err_l+Err_{nl} \\
Err_l &= \int_{t_1}^{t_2}\int_{\Sigma_t\cap\{R<r\}}r^p(|L\phi|+r^{-1}|\phi|)|\mathcal{L}_\phi| + \int_{t_1}^{t_2}\int_{\tilde{\Sigma}_t\cap\{R<r\}}r^p(|L\psi|+r^{-1}|\psi|)|\mathcal{L}_\psi| \\
Err_{nl} &= \int_{t_1}^{t_2}\int_{\Sigma_t}|(2X(\phi)+w\phi+w_{(a)}\psi)(\Box_g\phi-\mathcal{L}_\phi)| \\
&\hspace{2in}+\int_{t_1}^{t_2}\int_{\tilde{\Sigma}_t}|(2X(\psi)+\tilde{w}\psi+\tilde{w}_{(a)}\phi)(\Box_{\tilde{g}}\psi-\mathcal{L}_\psi)|,
\end{align*}
where the vectorfield $X$ and functions $w$, $w_{(a)}$, $\tilde{w}$, and $\tilde{w}_{(a)}$ satisfy the following properties. \\
\bp $X$ is everywhere timelike, but asymptotically null at the rate $X=O(r^p)L+O(r^{p-2})\pd_t$.
\bp $X|_{r=r_H}=-\lambda\pd_r$ for some positive constant $\lambda$. \\
\bp $X|_{r=r_{trap}}=\lambda\pd_t$ for some positive constant $\lambda$. \\
\bp $w$ and $\tilde{w}$ are both $O(r^{p-1})$. \\
\bp $w_{(a)}$ and $\tilde{w}_{(a)}$ are the same functions as defined in Proposition \ref{translated_morawetz_prop}. In particular,
$$|w_{(a)}\psi|\lesssim \frac{|a|^3M\sin^2\theta}{r^5}|A \psi|,$$
$$|\tilde{w}_{(a)}\phi|\lesssim \frac{|a|^3M\sin^2\theta}{r^5}|A^{-1}\phi|.$$
\end{proposition}
\begin{proof}
We start with the Morawetz estimate (Proposition \ref{translated_morawetz_prop}) and add a small constant times the incomplete $r^p$ estimates (Proposition \ref{incomplete_p_estimates_prop}). The small constant can be chosen so that the bulk error term $Err_1$ from Proposition \ref{incomplete_p_estimates_prop} can be absorbed into the bulk in the Morawetz estimate. The result is the following estimate.
\begin{align*}
&\int_{\Sigma_{t_2}}r^p\left[(L\phi)^2+|\sla\nabla\phi|^2+r^{-2}\phi^2\right]+\frac{M^2}{r^2}(\pd_r\phi)^2 +\int_{\tilde{\Sigma}_{t_2}}r^p\left[(L\psi)^2+|\tsla\nabla\psi|^2+r^{-2}\psi^2\right]+\frac{M^2}{r^2}(\pd_r\psi)^2 \\
&\hspace{1in}+\int_{t_1}^{t_2}\int_{\Sigma_t} r^{p-1}\left[\chi_{trap}(L\phi)^2+\chi_{trap}|\sla\nabla\phi|^2+r^{-2}\phi^2\right]+\frac{M^2}{r^3}(\pd_r\phi)^2 \\
&\hspace{1in}+\int_{t_1}^{t_2}\int_{\tilde{\Sigma}_t} r^{p-1}\left[\chi_{trap}(L\psi)^2+\chi_{trap}|\tsla\nabla\psi|^2+r^{-2}\psi^2\right]+\frac{M^2}{r^3}(\pd_r\psi)^2 \\
&\lesssim \int_{\Sigma_{t_1}}r^p\left[(L\phi)^2+|\sla\nabla\phi|^2+r^{-2}\phi^2\right]+\frac{M^2}{r^2}(\pd_r\phi)^2 +\int_{\tilde{\Sigma}_{t_1}}r^p\left[(L\psi)^2+|\tsla\nabla\psi|^2+r^{-2}\psi^2\right]+\frac{M^2}{r^2}(\pd_r\psi)^2 \\
&\hspace{6in}+Err'
\end{align*}
where
\begin{align*}
Err' &= Err'_1+Err'_2+Err'_3+Err'_\Box+Err'_{nl} \\
Err'_1 &= \int_{\Sigma_{t_2}}r^{-1}|\phi L\phi|+\int_{\tilde{\Sigma}_{t_2}}r^{-1}|\psi L\psi|+r^{-2}\psi^2 \\
Err'_2 &= \int_{H_{t_1}^{t_2}}(\pd_t\phi)^2+\int_{H_{t_1}^{t_2}}(\pd_t\psi)^2 +\int_{\Sigma_{t_2}}\frac{M^2}{r^2}\left[\chi_H(\pd_r\phi)^2+(\pd_t\phi)^2\right] \\
&\hspace{1in} +\int_{\tilde{\Sigma}_{t_2}}\frac{M^2}{r^2}\left[\chi_H(\pd_r\psi)^2+(\pd_t\psi)^2\right] \\ 
Err'_3 &= \int_{\Sigma_{t_1}\cap\{r>R\}}a^2r^{p-2}(\pd_r\phi)^2 + \int_{\tilde{\Sigma}_{t_1}\cap\{r>R\}}a^2r^{p-2}(\pd_r\psi)^2 \\
Err'_\Box &= \int_{t_1}^{t_2}\int_{\Sigma_t\cap\{R<r\}}r^p(|L\phi|+r^{-1}|\phi|)|\Box_g\phi| + \int_{t_1}^{t_2}\int_{\tilde{\Sigma}_t\cap\{R<r\}}r^p(|L\psi|+r^{-1}|\psi|)|\Box_{\tilde{g}}\psi| \\
Err'_{nl} &= \int_{t_1}^{t_2}\int_{\Sigma_t}|(2X'(\phi)+w'\phi+w'_{(a)}\psi)(\Box_g\phi-\mathcal{L}_\phi)| \\
&\hspace{1in} +\int_{t_1}^{t_2}\int_{\tilde{\Sigma}_t}|(2X'(\psi)+\tilde{w}'\psi+\tilde{w}'_{(a)}\phi)(\Box_{\tilde{g}}\psi-\mathcal{L}_\psi)|,
\end{align*}
and $X'$, $w'$, $\tilde{w}'$, $w'_{(a)}$, and $\tilde{w}'_{(a)}$ are the vectorfield and functions defined in the Morawetz estimate (Proposition \ref{translated_morawetz_prop}).

The error term $Err'_1$ can in fact be removed due to the following argument.
\begin{align*}
Err'_1 &\lesssim \int_{\Sigma_{t_2}}\epsilon r^p(L\phi)^2+\epsilon^{-1}r^{-p}r^{-2}\phi^2 + \int_{\tilde\Sigma_{t_2}} \epsilon r^p(L\psi)^2+ (\epsilon^{-1}r^{-p}+1)r^{-2}\psi^2 \\
&\lesssim \int_{\Sigma_{t_2}}\epsilon r^p[(L\phi)^2+r^{-2}\phi^2] + \int_{\tilde\Sigma_{t_2}} \epsilon r^p[(L\psi)^2+r^{-2}\psi^2] \\
&\hspace{2in}+\int_{\Sigma_{t_2}\cap\{r\le R_\epsilon\}}\epsilon^{-1}\phi^2 + \int_{\tilde{\Sigma}_{t_2}\cap\{r\le R_\epsilon \}}\epsilon^{-1}\psi^2.
\end{align*}
The radius $R_\epsilon$ should be chosen sufficiently large so that $\epsilon^{-1}r^{-p}\le \epsilon r^p$ and $\epsilon^{-1}r^{-p}+1\le \epsilon r^p$ whenever $r>R_\epsilon$. This critically depends on the fact that $p\ge\delm>0$. Now, the parameter $\epsilon$ can be taken sufficiently small so as to absorb the first two terms into the left hand side of the main estimate and the last two terms can be included with the term $Err'_2$ after applying a Hardy estimate.

We return to the main estimate. Notice that most terms have improved weights near $i^0$ and a few error terms remain on $H_{t_1}^{t_2}$ and $\Sigma_{t_2}$. The next step is to use the $h\pd_t$ estimate (Proposition \ref{translated_h_dt_prop}) to eliminate these error terms and improve the weights near $i^0$ for the remaining $\pd_r\phi$ and $\pd_r\psi$ terms. The result is the following estimate.
\begin{align*}
&\int_{\Sigma_{t_2}}r^p\left[(L\phi)^2+|\sla\nabla\phi|^2+r^{-2}\phi^2+r^{-2}(\pd_r\phi)^2\right] +\int_{\tilde{\Sigma}_{t_2}}r^p\left[(L\psi)^2+|\tsla\nabla\psi|^2+r^{-2}\psi^2+r^{-2}(\pd_r\psi)^2\right] \\
&\hspace{1in}+\int_{t_1}^{t_2}\int_{\Sigma_t} r^{p-1}\left[\chi_{trap}(L\phi)^2+\chi_{trap}|\sla\nabla\phi|^2+r^{-2}\phi^2+c_\epsilon r^{-2}(\pd_r\phi)^2\right] \\
&\hspace{1in}+\int_{t_1}^{t_2}\int_{\tilde{\Sigma}_t} r^{p-1}\left[\chi_{trap}(L\psi)^2+\chi_{trap}|\tsla\nabla\psi|^2+r^{-2}\psi^2+c_{\epsilon}r^2(\pd_r\psi)^2\right] \\
&\lesssim \int_{\Sigma_{t_1}}r^p\left[(L\phi)^2+|\sla\nabla\phi|^2+r^{-2}\phi^2+r^2(\pd_r\phi)^2\right] +\int_{\tilde{\Sigma}_{t_1}}r^p\left[(L\psi)^2+|\tsla\nabla\psi|^2+r^{-2}\psi^2+r^2(\pd_r\psi)^2\right] \\
&\hspace{6in}+Err''
\end{align*}
where
\begin{align*}
Err'' &= Err''_1+Err''_2+Err''_3+Err''_\Box+Err''_{nl} \\
Err''_1 &= \int_{t_1}^{t_2}\int_{\Sigma_t\cap\{5M<r\}} \epsilon r^{-1}((L\phi)^2+r^{-2}\phi^2) + \int_{t_1}^{t_2}\int_{\tilde\Sigma_t\cap\{5M<r\}} \epsilon r^{-1}((L\psi)^2+r^{-2}\psi^2) \\
Err''_2 &= \int_{\Sigma_{t_2}}\frac{|a|}{M}r^{p-4}(\phi^2+A^2\psi^2) \\
Err''_3 &= \int_{t_1}^{t_2}\int_{\Sigma_t\cap\{6M<r\}} c_\epsilon\frac{|a|}{M}r^{p-5}(\phi^2+A^2\psi^2) \\
Err''_\Box &= \int_{t_1}^{t_2}\int_{\Sigma_t\cap\{R<r\}}r^p(|L\phi|+r^{-1}|\phi|)|\Box_g\phi| + \int_{t_1}^{t_2}\int_{\tilde{\Sigma}_t\cap\{R<r\}}r^p(|L\psi|+r^{-1}|\psi|)|\Box_{\tilde{g}}\psi| \\
Err''_{nl} &= \int_{t_1}^{t_2}\int_{\Sigma_t}|(2X'(\phi)+w'\phi+w'_{(a)}\psi)(\Box_g\phi-\mathcal{L}_\phi)| \\
&\hspace{1in} +\int_{t_1}^{t_2}\int_{\tilde{\Sigma}_t}|(2X'(\psi)+\tilde{w}'\psi+\tilde{w}'_{(a)}\phi)(\Box_{\tilde{g}}\psi-\mathcal{L}_\psi)|, \\
&\hspace{1in}+ \int_{t_1}^{t_2}\int_{\Sigma_t}C_\epsilon r^{p-2}|\pd_t\phi(\Box_g\phi-\mathcal{L}_\phi)| +\int_{t_1}^{t_2}\int_{\tilde{\Sigma}_t}C_\epsilon r^{p-2}|\pd_t\psi(\Box_{\tilde{g}}\psi-\mathcal{L}_\psi)|.
\end{align*}
Note that each of the error terms $Err''_1$, $Err''_2$, and $Err''_3$ comes with a smallness parameter. By taking $\epsilon$ and $|a|/M$ sufficiently small, these error terms can be absorbed into the left hand side. 

We are left only with the error terms $Err''_\Box$ and $Err''_{nl}$, which we now combine by replacing $|\Box_g\phi|$ and $|\Box_{\tilde{g}}\psi|$ in $Err''_\Box$ with $|\mathcal{L}_\phi|+|\Box_g\phi-\mathcal{L}_\phi|$ and $|\mathcal{L}_\psi|+|\Box_{\tilde{g}}\psi-\mathcal{L}_\psi|$ respectively. The $|\mathcal{L}_\phi|$ and $|\mathcal{L}_\psi|$ terms are collected into the linear error term $Err_l$, and the $|\Box_g\phi-\mathcal{L}_\phi|$ and $|\Box_{\tilde{g}}\psi-\mathcal{L}_\psi|$ terms are combined with the terms in $Err''_{nl}$ to form the nonlinear error term $Err_{nl}$.
\begin{multline*}
Err''_\Box + Err''_{nl} \\
\lesssim \int_{t_1}^{t_2}\int_{\Sigma_t\cap\{R<r\}}r^p(|L\phi|+r^{-1}|\phi|)|\mathcal{L}_\phi| + \int_{t_1}^{t_2}\int_{\tilde{\Sigma}_t\cap\{R<r\}}r^p(|L\psi|+r^{-1}|\psi|)|\mathcal{L}_\psi| \\
+\int_{t_1}^{t_2}\int_{\Sigma_t}|(2X(\phi)+w\phi+w_{(a)}\psi)(\Box_g\phi-\mathcal{L}_\phi)| \\
+\int_{t_1}^{t_2}\int_{\tilde{\Sigma}_t}|(2X(\psi)+\tilde{w}\psi+\tilde{w}_{(a)}\phi)(\Box_{\tilde{g}}\psi-\mathcal{L}_\psi)| \\
\lesssim Err_l+Err_{nl},
\end{multline*}
where
\begin{align*}
X &= X'+O(r^{p-2})\pd_t+O(r^p)L \\
w &= w'+O(r^{p-1}) \\
\tilde{w} &= \tilde{w}'+O(r^{p-1}) \\
w_{(a)} &= w'_{(a)} \\
\tilde{w}_{(a)} &= \tilde{w}'_{(a)}.
\end{align*}
This concludes the proof. \qed
\end{proof}

\section{The Energy Estimates}\label{ee_sec}

In this section, we prove a few versions of the main energy estimates, which provide all of the necessary information related to the future dynamics of the system. These estimates consist of the classic energy estimate (Proposition \ref{translated_energy_estimate_prop}) and the $p$-weighted energy estimates (Proposition \ref{p_estimates_prop}) proved in \S\ref{phi_psi_sec}.

The main energy estimates are roughly of the form
$$E(t_2)\lesssim E(t_1)+\int_{t_1}^{t_2}N(t)dt$$
$$E_p(t_2)+\int_{t_1}^{t_2}B_p(t)dt\lesssim E_p(t_1)+\int_{t_1}^{t_2}N_p(t)dt,$$
where $p$ ranges from $\delm$ to $2-\delp$ for arbitrarily small $\delm,\delp>0$. The norm $E(t)$ is the classic energy norm, the norm $E_p(t)$ is the $p$-weighted energy norm with a weight of $r^p$ near $i^0$, the norm $B_p(t)$ is the bulk norm that has a degeneracy at the photon sphere and a weight of $r^{p-1}$ near $i^0$. The norms $N(t)$ and $N_p(t)$ are both nonlinear error norms which can be ignored in the linear problem. The estimates are put in this form for the convenience of the reader, who is strongly encouraged to have this form memorized.

By applying certain operators to the wave map system, the main energy estimates are generalized to higher derivatives of the quantities $\phi$ and $\psi$. This will generalize, for example, the classic energy norm $E(t)$ to the homogenous classic energy norm $\mathring{E}^s(t)$ (obtained by applying the only true commutators of the linear system, $\pd_t^s$), as well as to the norms $E^s(t)$ and $E^{s,k}(t)$. The other norms generalize the same way.

The various versions of the main energy estimates are given in Theorems \ref{p_L_thm}, \ref{p_thm}, \ref{p_o_thm}, \ref{p_s_L_thm}, \ref{p_s_thm}, \ref{p_s_k_L_thm}, and \ref{p_s_k_thm}. Theorem \ref{p_L_thm} is an immediate consequence of Propositions \ref{translated_energy_estimate_prop} and \ref{p_estimates_prop}. Theorems \ref{p_s_L_thm} and \ref{p_s_k_L_thm} follow from Theorem \ref{p_L_thm} by applying operators to the wave map system. These theorems are then improved by handling various linear error terms (which mainly arise because the operators do not commute with the entire linear system) and applying the equation. More specifically, Theorem \ref{p_thm} follows from Theorem \ref{p_L_thm}, Theorem \ref{p_s_thm} follows from Theorem \ref{p_s_L_thm}, and Theorem \ref{p_s_k_thm} follows Theorem \ref{p_s_k_L_thm}. Additionally, Theorem \ref{p_o_thm} follows directly from Theorem \ref{p_thm}, since the operators $\pd_t^s$ do in fact commute with the entire linear system.

\subsection{The main energy estimates for $(\phi,\psi)$}

We start by combining Propositions \ref{translated_energy_estimate_prop} and \ref{p_estimates_prop} to obtain Theorem \ref{p_L_thm}. But first, we define the following two expressions to simplify the nonlinear term.
\begin{definition} Let $X$, $w$, $w_{(a)}$, $\tilde{w}$, and $\tilde{w}_{(a)}$ be as defined in Proposition \ref{p_estimates_prop}. Then define
\begin{align*}
\mathcal{X}_\phi &:= 2X(\phi)+w\phi+w_{(a)}\psi \\
\mathcal{X}_\psi &:= 2X(\psi)+\tilde{w}\psi+\tilde{w}_{(a)}\phi.
\end{align*}
\end{definition}
With this definition, we now have the following theorem.
\begin{theorem}\label{p_L_thm}
Suppose $|a|/M$ is sufficiently small and fix $\delm,\delp>0$. The following estimates hold for $p\in [\delm,2-\delp]$.
$$E(t_2)\lesssim E(t_1)+\int_{t_1}^{t_2}N(t)dt,$$
$$E_p(t_2)+\int_{t_1}^{t_2}B_p(t)dt\lesssim E_p(t_1)+\int_{t_1}^{t_2}(L_1)_p(t)+N_p(t)dt,$$
where
\begin{align*}
E(t)=&\int_{\Sigma_t}\chi_H(\pd_r\phi)^2+(\pd_t\phi)^2+|\sla\nabla\phi|^2+r^{-2}\phi^2 \\
&+\int_{\tilde\Sigma_t}\chi_H(\pd_r\psi)^2+(\pd_t\psi)^2+|\tsla\nabla\psi|^2+r^{-2}\psi^2,
\end{align*}
\begin{align*}
E_p(t)=&\int_{\Sigma_t}r^p\left[(L\phi)^2+|\sla\nabla\phi|^2+r^{-2}\phi^2+r^{-2}(\pd_r\phi)^2\right] \\
&+\int_{\tilde\Sigma_t}r^p\left[(L\psi)^2+|\tsla\nabla\psi|^2+r^{-2}\psi^2+r^{-2}(\pd_r\psi)^2\right],
\end{align*}
\begin{align*}
B_p(t)=&\int_{\Sigma_t}r^{p-1}\left[\chi_{trap}(L\phi)^2+\chi_{trap}|\sla\nabla\phi|^2+r^{-2}\phi^2+r^{-2}(\pd_r\phi)^2\right] \\
&+\int_{\tilde\Sigma_t}r^{p-1}\left[\chi_{trap}(L\psi)^2+\chi_{trap}|\tsla\nabla\psi|^2+r^{-2}\psi^2+r^{-2}(\pd_r\psi)^2\right],
\end{align*}
$$(L_1)_p(t)=\int_{\Sigma_t\cap\{r>R\}}r^p(|L\phi|+r^{-1}|\phi|)|\mathcal{L}_\phi|+\int_{\tilde\Sigma_t\cap\{r>R\}}r^p(|L\psi|+r^{-1}|\psi|)|\mathcal{L}_\psi|,$$
$$N(t)=\int_{\Sigma_t}|\pd_t\phi||\Box_g\phi-\mathcal{L}_\phi| +\int_{\tilde{\Sigma}_t}|\pd_t\psi||\Box_{\tilde g}\psi-\mathcal{L}_\psi|,$$
$$N_p(t)=\int_{\Sigma_t}r^p|\mathcal{X}_\phi||\Box_g\phi-\mathcal{L}_\phi| + \int_{\tilde\Sigma_t}r^p|\mathcal{X}_\psi||\Box_{\tilde{g}}\psi-\mathcal{L}_\psi|, $$
where $\chi_H=1-\frac{r_H}r$ and $\chi_{trap}=\left(1-\frac{r_{trap}}r\right)^2$.
\end{theorem}
\begin{proof}
This is a direct application of Propositions \ref{translated_energy_estimate_prop} and \ref{p_estimates_prop}. Note that the linear error term $L_1$ derives from the error term $Err_l$ in Proposition \ref{p_estimates_prop}, which derives from the $p$-weighted estimates near $i^0$ that were proved only for the standard wave equation. This is why the linear error term $(L_1)_p$ is supported far away from the Kerr black hole. \qed
\end{proof}

We will momentarily improve the prevoius theorem by absorbing the linear error term $(L_1)_p$ into the bulk term $B_p(t)$ on the left hand side. To do this, we need the following lemma.
\begin{lemma}\label{p_L_a_lem}
If $(L_1)_p(t)$ and $B_p(t)$ are as defined in the previous theorem, then
$$(L_1)_p(t)\lesssim \frac{|a|}{M}B_p(t).$$
\end{lemma}
\begin{proof}
We have
\begin{multline*}
\int_{\Sigma_t\cap\{r>R\}}r^p(|L\phi|+r^{-1}|\phi|)|\mathcal{L}_\phi|  \\
\lesssim \left(\int_{\Sigma_t\cap\{r>R\}}r^{p-1}[(L\phi)^2+r^{-2}\phi^2]\right)^{1/2}\left(\int_{\Sigma_t\cap\{r>R\}}r^{p+1}(\mathcal{L}_\phi)^2\right)^{1/2} \\
\lesssim (B_p(t))^{1/2}\left(\int_{\Sigma_t\cap\{r>R\}}r^{p+1}(\mathcal{L}_\phi)^2\right)^{1/2} \\
\lesssim \frac{|a|}{M}B_p(t)+\left(\frac{|a|}{M}\right)^{-1}\int_{\Sigma_t\cap\{r>R\}}r^{p+1}(\mathcal{L}_\phi)^2.
\end{multline*}
An analogous estimate also shows that
$$\int_{\tilde\Sigma_t\cap\{r>R\}}r^p(|L\psi|+r^{-1}|\psi|)|\mathcal{L}_\psi| \lesssim \frac{|a|}{M}B_p(t)+\left(\frac{|a|}{M}\right)^{-1}\int_{\tilde\Sigma_t\cap\{r>R\}}r^{p+1}(\mathcal{L}_\psi)^2.$$
Therefore, it suffices to establish the following estimate.
\begin{equation}\label{mathcalL_p_estimate}
\int_{\Sigma_t\cap\{r>R\}}r^{p+1}(\mathcal{L}_\phi)^2 + \int_{\tilde\Sigma_t\cap\{r>R\}}r^{p+1}(\mathcal{L}_\psi)^2 \lesssim \frac{a^2}{M^2}B_p(t).
\end{equation}
Let us look at the $\mathcal{L}_\phi$ term first. Recall that
$$\mathcal{L}_\phi=-2\frac{\pd^\alpha B}{A}A\pd_\alpha \psi + 2\frac{\pd^\alpha B\pd_\alpha B}{A^2}\phi-4\frac{\pd^\alpha A\pd_\alpha B}{A^2} A\psi.$$
It follows that
\begin{multline*}
\int_{\Sigma_t\cap\{r>R\}}r^{p+1}(\mathcal{L}_\phi)^2 \\
\lesssim \int_{\Sigma_t\cap\{r>R\}}r^{p-1}\left[\left(\frac{r\pd^\alpha B}{A}\pd_\alpha\psi\right)^2A^2 +\left(\frac{r\pd^\alpha B\pd_\alpha B}{A^2}\phi\right)^2+\left(\frac{r\pd^\alpha A\pd_\alpha B}{A^2}\psi\right)^2A^2\right] \\
\lesssim \int_{\Sigma_t\cap\{r>R\}}r^{p-1}\left(\frac{r\pd^\alpha B\pd_\alpha B}{A^2}\phi\right)^2 +\int_{\tilde\Sigma_t\cap\{r>R\}}r^{p-1}\left[\left(\frac{r\pd^\alpha B}{A}\pd_\alpha\psi\right)^2 +\left(\frac{r\pd^\alpha A\pd_\alpha B}{A^2}\psi\right)^2\right].
\end{multline*}
Using the identities from Lemma \ref{small_a_quantities_lem}, one can easily check each of the following estimates.
\begin{align*}
\left(\frac{r\pd^\alpha B\pd_\alpha B}{A^2}\phi\right)^2 &\lesssim \frac{a^2}{M^2}r^{-2}\phi^2 \\
\left(\frac{r\pd^\alpha B}{A}\pd_\alpha \psi\right)^2 &\lesssim \frac{a^2}{M^2}|\tsla\nabla\psi|^2+\frac{a^2}{M^2}\frac{M^2}{r^2}(\pd_r\psi)^2 \\
\left(\frac{r\pd^\alpha A\pd_\alpha B}{A^2}\psi\right)^2 &\lesssim \frac{a^2}{M^2}r^{-2}\psi^2.
\end{align*}
Thus,
$$\int_{\Sigma_t\cap\{r>R\}}r^{p+1}(\mathcal{L}_\phi)^2 \lesssim \frac{a^2}{M^2} B_p(t).$$
Let us now look at the $\mathcal{L}_\psi$ term. Recall that
$$\mathcal{L}_\psi=-2\frac{\pd^\alpha A_2}{A_2}\pd_\alpha\psi+2A^{-1}\frac{\pd^\alpha B}{A}\pd_\alpha\phi+2\frac{\pd^\alpha B\pd_\alpha B}{A^2}\psi.$$
It follows that
\begin{multline*}
\int_{\tilde\Sigma_t\cap\{r>R\}}r^{p+1}(\mathcal{L}_\psi)^2 \\
\lesssim \int_{\tilde{\Sigma}_t\cap\{r>R\}}r^{p-1}\left[\left(\frac{r\pd^\alpha A_2}{A_2}\pd_\alpha\psi\right)^2+\left(\frac{r\pd^\alpha B}{A}\pd_\alpha \phi\right)^2A^{-2}+\left(\frac{r\pd^\alpha B\pd_\alpha B}{A^2}\psi\right)^2\right] \\
\lesssim \int_{\tilde{\Sigma}_t\cap\{r>R\}}r^{p-1}\left[\left(\frac{r\pd^\alpha A_2}{A_2}\pd_\alpha\psi\right)^2+\left(\frac{r\pd^\alpha B\pd_\alpha B}{A^2}\psi\right)^2\right] +\int_{\Sigma_t\cap\{r>R\}}r^{p-1}\left(\frac{r\pd^\alpha B}{A}\pd_\alpha \phi\right)^2. 
\end{multline*}
Using again the identities from Lemma \ref{small_a_quantities_lem}, one can easily check each of the following estimates.
\begin{align*}
\left(\frac{r\pd^\alpha A_2}{A_2}\pd_\alpha\psi\right)^2 &\lesssim \frac{a^2}{M^2}|\tsla\nabla\psi|^2+\frac{a^2}{M^2}\frac{M^2}{r^2}(\pd_r\psi)^2 \\
\left(\frac{r\pd^\alpha B\pd_\alpha B}{A^2}\psi\right)^2 &\lesssim \frac{a^2}{M^2}r^{-2}\psi^2\\
\left(\frac{r\pd^\alpha B}{A}\pd_\alpha \phi\right)^2 &\lesssim \frac{a^2}{M^2}|\sla\nabla\phi|^2+\frac{a^2}{M^2}\frac{M^2}{r^2}(\pd_r\phi)^2.
\end{align*}
Thus,
$$\int_{\tilde\Sigma_t\cap\{r>R\}}r^{p+1}(\mathcal{L}_\psi)^2 \lesssim \frac{a^2}{M^2} B_p(t).$$
This completes the proof.\qed
\end{proof}

With Lemma \ref{p_L_a_lem}, Theorem \ref{p_L_thm} can be improved so that the linear error term appearing on the right hand side of the $r^p$ estimate is removed. Also, by assuming the full nonlinear equations, we can simplify the nonlinear error term. This is the purpose of the following theorem.
\begin{theorem}\label{p_thm} (Improved version of Theorem \ref{p_L_thm})
Suppose $|a|/M$ is sufficiently small and fix $\delm,\delp>0$. Suppose furthermore that the pair $(\phi,\psi)$ satisfies the system
\begin{align*}
\Box_g\phi &= \mathcal{L}_\phi+\mathcal{N}_\phi, \\
\Box_{\tilde{g}}\psi &= \mathcal{L}_\psi + \mathcal{N}_\psi.
\end{align*}
Then the following estimates hold for $p\in [\delm,2-\delp]$.
$$E(t_2)\lesssim E(t_1)+\int_{t_1}^{t_2}N(t)dt,$$
$$E_p(t_2)+\int_{t_1}^{t_2}B_p(t)dt\lesssim E_p(t_1)+\int_{t_1}^{t_2}N_p(t)dt,$$
where $E(t)$, $E_p(t)$, and $B_p(t)$ are as defined in Theorem \ref{p_L_thm}, and
$$N(t)=(E(t))^{1/2}\left(||\mathcal{N}_\phi||_{L^2(\Sigma_t)}+||\mathcal{N}_\psi||_{L^2(\tilde{\Sigma}_t)}\right),$$
$$N_p(t)=\int_{\Sigma_t}r^{p+1}(\mathcal{N}_\phi)^2 + \int_{\tilde\Sigma_t}r^{p+1}(\mathcal{N}_\psi)^2 + \int_{\Sigma_t\cap\{r\approx r_{trap}\}}|\pd_t\phi\mathcal{N}_\phi| + \int_{\tilde\Sigma_t\cap\{r\approx r_{trap}\}}|\pd_t\psi\mathcal{N}_\psi|. $$
\end{theorem}

\begin{proof}
By Theorem \ref{p_L_thm}, we have
$$E(t_2)\lesssim E(t_1)+\int_{t_1}^{t_2}N'(t)dt,$$
where
\begin{align*}
N'(t)&=\int_{\Sigma_t}|\pd_t\phi||\Box_g\phi-\mathcal{L}_\phi|+\int_{\tilde{\Sigma}_t}|\pd_t\psi||\Box_{\tilde{g}}\psi-\mathcal{L}_\psi| \\
&=\int_{\Sigma_t}|\pd_t\phi||\mathcal{N}_\phi|+\int_{\tilde{\Sigma}_t}|\pd_t\psi||\mathcal{N}_\psi|.
\end{align*}
Observe that
\begin{align*}
N'(t) &=\int_{\Sigma_t}|\pd_t\phi||\mathcal{N}_\phi|+\int_{\tilde{\Sigma}_t}|\pd_t\psi||\mathcal{N}_\psi| \\
&\lesssim ||\pd_t\phi||_{L^2(\Sigma_t)}||\mathcal{N}_\phi||_{L^2(\Sigma_t)}+||\pd_t\psi||_{L^2(\tilde\Sigma_t)}||\mathcal{N}_\psi||_{L^2(\tilde\Sigma_t)} \\
&\lesssim (E(t))^{1/2}||\mathcal{N}_\phi||_{L^2(\Sigma_t)}+(E(t))^{1/2}||\mathcal{N}_\psi||_{L^2(\tilde\Sigma_t)} \\
&\lesssim N(t).
\end{align*}
This proves the first estimate of the theorem.

By Theorem \ref{p_L_thm} and Lemma \ref{p_L_a_lem}, we have
$$E_p(t_2)+\int_{t_1}^{t_2}B_p(t)dt\lesssim E_p(t_1)+\frac{|a|}{M}\int_{t_1}^{t_2}B_p(t)dt+\int_{t_1}^{t_2}N_p'(t)dt,$$
where
\begin{align*}
N_p'(t)&=\int_{\Sigma_t}r^p|\mathcal{X}_\phi||\Box_g\phi-\mathcal{L}_\phi| + \int_{\tilde\Sigma_t}r^p|\mathcal{X}_\psi||\Box_{\tilde{g}}\psi-\mathcal{L}_\psi| \\
&=\int_{\Sigma_t}r^p|\mathcal{X}_\phi||\mathcal{N}_\phi| + \int_{\tilde\Sigma_t}r^p|\mathcal{X}_\psi||\mathcal{N}_\psi|.
\end{align*}
Observe that
\begin{multline*}
\int_{\Sigma_t}r^p|\mathcal{X}_\phi||\mathcal{N}_\phi| \\
\lesssim \int_{\Sigma_t}r^p(\chi_{trap}|L\phi|+r^{-1}|\phi|+r^{-1}A|\psi|+r^{-2}|\pd_r\phi|)|\mathcal{N}_\phi| +\int_{\Sigma_t\cap\{r\approx r_{trap}\}}|\pd_t||\phi\mathcal{N}_\phi| \\
\lesssim (B_p(t))^{1/2}\left(\int_{\Sigma_t}r^{p+1}|\mathcal{N}_\phi|^2\right)^{1/2} +\int_{\Sigma_t\cap\{r\approx r_{trap}\}}|\pd_t\phi||\mathcal{N}_\phi| \\
\lesssim \epsilon B_p(t) +\epsilon^{-1}\int_{\Sigma_t}r^{p+1}|\mathcal{N}_\phi|^2 +\int_{\Sigma_t\cap\{r\approx r_{trap}\}}|\pd_t\phi||\mathcal{N}_\phi| \\
\lesssim \epsilon B_p(t) +\epsilon^{-1}N_p(t).
\end{multline*}
An analogous estimate also shows that
$$\int_{\tilde\Sigma_t}r^p|\mathcal{X}_\psi||\mathcal{N}_\psi|\lesssim \epsilon B_p(t)+\epsilon^{-1}N_p(t).$$
Therefore, we have
$$E_p(t_2)+\int_{t_1}^{t_2}B_p(t)dt\lesssim E_p(t_1)+(|a|/M+\epsilon)\int_{t_1}^{t_2}B_p(t)dt+\epsilon^{-1}\int_{t_1}^{t_2}N_p(t)dt.$$
By taking $|a|/M$ and $\epsilon$ sufficiently small, the bulk term on the right hand side can be absorbed into the left hand side. The result is the second estimate of the theorem. \qed
\end{proof}

\subsection{The main energy estimates for $(\pd_t^s\phi,\pd_t^s\psi)$}

We now begin to derive higher order estimates analogous to Theorem \ref{p_thm} by commuting with the linear system. The only operator that completely commutes with the linear system is the operator $\pd_t$. Therefore, we immediately have the following homogeneous estimate, which will be important at the highest level of derivatives in the proof of the main theorem.

\begin{theorem}\label{p_o_thm} (Generalization of Theorem \ref{p_thm})
Suppose $|a|/M$ is sufficiently small and fix $\delm,\delp>0$. Suppose furthermore that the pair $(\phi,\psi)$ satisfies the system
\begin{align*}
\Box_g\phi &= \mathcal{L}_\phi+\mathcal{N}_\phi, \\
\Box_{\tilde{g}}\psi &= \mathcal{L}_\psi + \mathcal{N}_\psi.
\end{align*}
Then the following estimates hold for $p\in [\delm,2-\delp]$.
$$\mathring{E}^s(t_2)\lesssim \mathring{E}^s(t_1)+\int_{t_1}^{t_2}\mathring{N}^s(t)dt,$$
$$\mathring{E}_p^s(t_2)+\int_{t_1}^{t_2}\mathring{B}_p^s(t)dt\lesssim \mathring{E}_p^s(t_1)+\int_{t_1}^{t_2}\mathring{N}_p^s(t)dt,$$
where 
$$\mathring{E}^s(t)=\sum_{s'\le s}E[(\pd_t^{s'}\phi,\pd_t^{s'}\psi)](t),$$
$$\mathring{E}_p^s(t)=\sum_{s'\le s}E_p[(\pd_t^{s'}\phi,\pd_t^{s'}\psi)](t),$$
$$\mathring{B}_p^s(t)=\sum_{s'\le s}B_p[(\pd_t^{s'}\phi,\pd_t^{s'}\psi)](t),$$
and
$$\mathring{N}^s(t)=(\mathring{E}^s(t))^{1/2}\left(||\pd_t^s\mathcal{N}_\phi||_{L^2(\Sigma_t)}+||\pd_t^s\mathcal{N}_\psi||_{L^2(\tilde{\Sigma}_t)}\right),$$
\begin{multline*}
\mathring{N}_p^s(t)=\int_{\Sigma_t}r^{p+1}(\pd_t^s\mathcal{N}_\phi)^2 + \int_{\tilde\Sigma_t}r^{p+1}(\pd_t^s\mathcal{N}_\psi)^2 \\
+ \int_{\Sigma_t\cap\{r\approx r_{trap}\}}|\pd_t^{s+1}\phi\pd_t^s\mathcal{N}_\phi| + \int_{\tilde\Sigma_t\cap\{r\approx r_{trap}\}}|\pd_t^{s+1}\psi\pd_t^s\mathcal{N}_\psi|. 
\end{multline*}
\end{theorem}
\begin{proof}
The proof is a direct application of Theorem \ref{p_thm} by making the substitutions
\begin{align*}
\phi &\mapsto \pd_t^{s'}\phi \\
\psi &\mapsto \pd_t^{s'}\psi
\end{align*}
for all values of $s'$ where $s'\le s$, and observing that if
\begin{align*}
\Box_g\phi &= \mathcal{L}_\phi+\mathcal{N}_\phi, \\
\Box_{\tilde{g}}\psi &= \mathcal{L}_\psi + \mathcal{N}_\psi,
\end{align*}
then
\begin{align*}
\Box_g(\pd_t^{s'}\phi) &= \mathcal{L}_{(\pd_t^{s'}\phi)}+\pd_t^{s'}\mathcal{N}_\phi, \\
\Box_{\tilde{g}}(\pd_t^{s'}\psi) &= \mathcal{L}_{(\pd_t^{s'}\psi)} + \pd_t^{s'}\mathcal{N}_\psi,
\end{align*}
where $\mathcal{L}_{(\pd_t^{s'}\phi)}$ and $\mathcal{L}_{(\pd_t^{s'}\psi)}$ are the expressions obtained by substituting $(\phi,\psi)$ with $(\pd_t^{s'}\phi,\pd_t^{s'}\psi)$ in $\mathcal{L}_\phi$ and $\mathcal{L}_\psi$ respectively.
\qed
\end{proof}

\subsection{The main energy estimates for $(\phi^s,\psi^s)$}

Even though $\pd_t$ is the only operator that completely commutes with the linear system, the second order Carter operator $Q$ and its modified version $\tilde{Q}$ commute with the wave operators $q^2\Box_g$ and $q^2\Box_{\tilde{g}}$ respectively.

\begin{definition}
$$Q:=\pd_\theta^2+\cot\theta \pd_\theta +a^2\sin^2\theta \pd_t^2,$$
$$\tilde{Q}:=\pd_\theta^2+5\cot\theta \pd_\theta +a^2\sin^2\theta \pd_t^2.$$
\end{definition}

\begin{lemma}
$$[Q,q^2\Box_g]=0,$$
$$[\tilde{Q},q^2\Box_{\tilde{g}}]=0.$$
\end{lemma}
\begin{proof}
One can check that the operator $q^2\Box_g-Q$ does not depend on $\theta$ or $t$ and has no $\pd_\theta$ operators. Since $Q$ only depends on $\theta$, and has only $\pd_\theta$ and $\pd_t$ operators, that means
$$0=[Q,q^2\Box_g-Q]=[Q,q^2\Box_g].$$
Similarly,
\begin{align*}
q^2\Box_{\tilde{g}}-\tilde{Q} &= q^2\Box_g+2q^2\frac{\pd^\alpha A_1}{A_1}\pd_\alpha-(Q+4\cot\theta\pd_\theta) \\
&=\Box_g-Q+\left(2q^2\frac{\pd^\alpha A_1}{A_1}\pd_\alpha-4\cot\theta\pd_\theta\right) \\
&=\Box_g-Q+\frac{4r}{r^2+a^2}(q^2g^{\alpha r})\pd_\alpha.
\end{align*}
Since $q^2\Box_g-Q$ and $\frac{4r}{r^2+a^2}(q^2g^{\alpha r})\pd_\alpha$ do not depend on $\theta$ or $t$ and have no $\pd_\theta$ operators, and since $\tilde{Q}$ only depends on $\theta$ and has only $\pd_\theta$ and $\pd_t$ operators, that means
$$0=[\tilde{Q},q^2\Box_{\tilde{g}}-\tilde{Q}]=[\tilde{Q},q^2\Box_{\tilde{g}}].$$
This completes the proof. \qed
\end{proof}

We define the $s$-order commutators $\Gamma$ and $\tilde{\Gamma}$.
\begin{definition}
$$\Gamma^su := Q^{l}\pd_t^{s-2l}u$$
$$\tilde\Gamma^su := \tilde{Q}^{l}\pd_t^{s-2l}u,$$
where $0\le 2l\le s$.
\end{definition}
We also define the $s$-order dynamic quantities $\phi^s$ and $\psi^s$.
\begin{definition}
$$\phi^s:=\Gamma^s\phi,$$
$$\psi^s:=\tilde\Gamma^s\psi.$$
\end{definition}
Additionally, we define the $s$-order analogues of $\mathcal{L}_\phi$, $\mathcal{L}_\psi$, $\mathcal{X}_\phi$, and $\mathcal{X}_\psi$.
\begin{definition}
$$\mathcal{L}_{\phi^s}:=-2\frac{\pd^\alpha B}{A}A\pd_\alpha \psi^s + 2\frac{\pd^\alpha B\pd_\alpha B}{A^2}\phi^s-4\frac{\pd^\alpha A\pd_\alpha B}{A^2} A\psi^s$$
$$\mathcal{L}_{\psi^s}:=-2\frac{\pd^\alpha A_2}{A_2}\pd_\alpha\psi^s+2\frac{\pd^\alpha B\pd_\alpha B}{A^2}\psi^s + 2A^{-1}\frac{\pd^\alpha B}{A}\pd_\alpha\phi^s$$
and
\begin{align*}
\mathcal{X}_{\phi^s} &:= 2X(\phi^s)+w\phi^s+w_{(a)}\psi^s \\
\mathcal{X}_{\psi^s} &:= 2X(\psi^s)+\tilde{w}\psi^s+\tilde{w}_{(a)}\phi^s,
\end{align*}
where, in each expression, the exact same operator $\Gamma^s$ or $\tilde{\Gamma}^s$ is to be used in each term on the right hand side, replacing $Q$ with $\tilde{Q}$ where appropriate. For example, the expression
$$-2\frac{\pd^\alpha B}{A}A\pd_\alpha (\tilde{Q}\psi) + 2\frac{\pd^\alpha B\pd_\alpha B}{A^2}(Q\phi)-4\frac{\pd^\alpha A\pd_\alpha B}{A^2} A(\tilde{Q}\psi)$$
belongs to $\mathcal{L}_{\phi^s}$ ($s=2$) while the expression
$$-2\frac{\pd^\alpha B}{A}A\pd_\alpha(\pd_t^2\psi) + 2\frac{\pd^\alpha B\pd_\alpha B}{A^2}(Q\phi)-4\frac{\pd^\alpha A\pd_\alpha B}{A^2} A(\tilde{Q}\psi)$$
does not.
\end{definition}

We take a look again at the equations
$$\Box_g\phi = \mathcal{L}_\phi +\mathcal{N}_\phi,$$
$$\Box_{\tilde{g}}\psi = \mathcal{L}_\psi +\mathcal{N}_\psi.$$
By applying $\Gamma^s$ and $\tilde\Gamma^s$ repsectively, we obtain additional useful equations.
\begin{align*}
\Box_g\phi^s &= q^{-2}\Gamma^s(q^2\mathcal{L}_\phi) + q^{-2}\Gamma^s(q^2\mathcal{N}_\phi) \\
&= \mathcal{L}_{\phi^s} +(q^{-2}\Gamma^s(q^2\mathcal{L}_\phi)-\mathcal{L}_{\phi^s}) +q^{-2}\Gamma^s(q^2\mathcal{N}_\phi),
\end{align*}
\begin{align*}
\Box_{\tilde{g}}\psi^s &= q^{-2}\tilde\Gamma^s(q^2\mathcal{L}_\psi) + q^{-2}\tilde\Gamma^s(q^2\mathcal{N}_\psi) \\
&= \mathcal{L}_{\psi^s} +(q^{-2}\tilde\Gamma^s(q^2\mathcal{L}_\psi)-\mathcal{L}_{\psi^s}) +q^{-2}\tilde\Gamma^s(q^2\mathcal{N}_\psi).
\end{align*}
Therefore, Theorem \ref{p_L_thm} can be generalized to the following theorem. (Note the presence of the additional terms $(L_2)^s(t)$ and $(L_2)_p^s(t)$, which arise from the fact that the $\Gamma$ and $\tilde\Gamma$ operators do not completely commute with the linear system.)
\begin{theorem}\label{p_s_L_thm}
Suppose $|a|/M$ is sufficiently small and fix $\delm,\delp>0$. The following estimates hold for $p\in [\delm,2-\delp]$.
$$E^s(t_2)\lesssim E^s(t_1)+\int_{t_1}^{t_2}(L_2)^s(t)+N^s(t)dt,$$
$$E_p^s(t_2)+\int_{t_1}^{t_2}B_p^s(t)dt\lesssim E_p^s(t_1)+\int_{t_1}^{t_2}(L_1)_p^s(t)+(L_2)_p^s(t)+N_p^s(t)dt,$$
where
\begin{align*}
E^s(t) &= \sum_{s'\le s} E[(\phi^{s'},\psi^{s'})](t), \\
E_p^s(t) &= \sum_{s'\le s} E_p[(\phi^{s'},\psi^{s'})](t), \\
B_p^s(t) &= \sum_{s'\le s} B_p[(\phi^{s'},\psi^{s'})](t),
\end{align*}
and
$$(L_2)^s(t)=\sum_{s'\le s}\int_{\Sigma_t}|\pd_t\phi^{s'} (q^{-2}\Gamma^{s'}(q^2\mathcal{L}_\phi)-\mathcal{L}_{\phi^{s'}})|+\sum_{s'\le s}\int_{\tilde\Sigma_t}|\pd_t\psi^{s'}(q^{-2}\tilde{\Gamma}^{s'}(q^2\mathcal{L}_\psi)-\mathcal{L}_{\psi^{s'}})|,$$
$$N^s(t)=\sum_{s'\le s}\int_{\Sigma_t}|\pd_t\phi^{s'}(\Box_g\phi^{s'}-q^{-2}\Gamma^{s'}(q^2\mathcal{L}_\phi))| +\sum_{s'\le s}\int_{\tilde{\Sigma}_t}|\pd_t\psi^{s'}(\Box_{\tilde g}\psi^{s'}-q^{-2}\tilde\Gamma^{s'}(q^2\mathcal{L}_\psi))|,$$
$$(L_1)_p^s(t) = \sum_{s'\le s}\int_{\Sigma_t\cap\{r>R\}}r^p(|L\phi^{s'}|+r^{-1}|\phi^{s'}|)|\mathcal{L}_{\phi^{s'}}|+\sum_{s'\le s}\int_{\tilde\Sigma_t\cap\{r>R\}}r^p(|L\psi^{s'}|+r^{-1}|\psi^{s'}|)|\mathcal{L}_{\psi^{s'}}|,$$
$$(L_2)_p^s(t) = \sum_{s'\le s}\int_{\Sigma_t}r^p|\mathcal{X}_{\phi^s}||q^{-2}\Gamma^{s'}(q^2\mathcal{L}_\phi)-\mathcal{L}_{\phi^{s'}}| + \sum_{s'\le s}\int_{\tilde\Sigma_t}r^p|\mathcal{X}_{\psi^s}||q^{-2}\tilde\Gamma^{s'}(q^2\mathcal{L}_\psi)-\mathcal{L}_{\psi^{s'}}|,$$
$$N_p^s(t)= \sum_{s'\le s}\int_{\Sigma_t}r^p|\mathcal{X}_{\phi^s}||\Box_g\phi^{s'}-q^{-2}\Gamma^{s'}(q^2\mathcal{L}_\phi)| + \sum_{s'\le s}\int_{\tilde\Sigma_t}r^p|\mathcal{X}_{\psi^s}||\Box_{\tilde{g}}\psi^{s'}-q^{-2}\tilde\Gamma^{s'}(q^2\mathcal{L}_\psi)|.$$
\end{theorem}
\begin{proof}
The proof is a direct application of Theorem \ref{p_L_thm} by making the substitutions
\begin{align*}
\phi&\mapsto\phi^{s'} \\
\psi&\mapsto\psi^{s'}
\end{align*}
for all values of $s'$ (and all commutators represented by $\Gamma^{s'}$ and $\tilde\Gamma^{s'}$) where $s'\le s$. The following estimates are used.
$$|\Box_g\phi^{s'}-\mathcal{L}_{\phi^{s'}}| \le |\Box_g\phi^{s'}-q^{-2}\Gamma^{s'}(q^2\mathcal{L}_\phi)|+|q^{-2}\Gamma^{s'}(q^2\mathcal{L}_\phi)-\mathcal{L}_{\phi^{s'}}|,$$
$$|\Box_{\tilde{g}}\psi^{s'}-\mathcal{L}_{\psi^{s'}}| \le |\Box_{\tilde{g}}\psi^{s'}-q^{-2}\tilde\Gamma^{s'}(q^2\mathcal{L}_\psi)|+|q^{-2}\tilde\Gamma^{s'}(q^2\mathcal{L}_\psi)-\mathcal{L}_{\psi^{s'}}|.$$
The resulting error terms have been grouped into the parts that will either be linear or nonlinear when using the equations from the fully nonlinear system. (See Theorem \ref{p_s_thm} below.) \qed
\end{proof}

As in the proof of Theorem \ref{p_thm}, we would like to absorb the linear error terms into the bulk. But unfortunately, the terms $(L_2)^s(t)$ and $(L_2)_p^s(t)$ are not as straightforward to eliminate. The term $(L_2)^s(t)$ cannot be absorbed, because it belongs to the classic energy estimate, which has no bulk quantity on the left hand side. The term $(L_2)_p^s(t)$, which belongs to the $r^p$ estimate, cannot be completely absorbed into the bulk on the left hand side, because of a complication at the trapping radius.

The strategy is as follows. In Lemma \ref{p_s_L_a_1_lem}, the linear terms $(L_2)^s(t)$, $(L_1)_p^s(t)$, and $(L_2)_p^s(t)$ are estimated by the appropriate bulk norms except near the trapping radius. The trapping radius will need special care, because the factors $\pd_t\phi^s$ and $\pd_t\psi^s$, which appear in each of these linear terms, cannot be estimated by the appropriate bulk norm. Actually, for the particular case where $\pd_t\phi^s$ (resp. $\pd_t\psi^s$) represents $\pd_t^{s+1}\phi$ (resp. $\pd_t^{s+1}\psi$), then this factor can be estimated by the homogeneous bulk norm $\mathring{B}^{s+1}(t)$ with a loss of one derivative. The advantage in this case is that the homogenous bulk norm has already been estimated in Theorem \ref{p_o_thm}. The problem is that $\pd_t\phi^s$ (resp. $\pd_t\psi^s$) also represents terms with the operator $Q$ (resp. $\tilde{Q}$). In Lemma \ref{pd_t_identities_lem}, an approximate identity is given for the factor $\pd_t\phi^s$ (resp. $\pd_t\psi^s$), which expresses it as a sum of $\pd_t^{s+1}\phi$ (resp. $\pd_t^{s+1}\psi$) and other terms. Then in Lemma \ref{p_s_L_a_lem}, this approximate identity is used to refine Lemma \ref{p_s_L_a_1_lem}. The result is that there will be a loss of one derivative (although to a homogeneous norm) as well as nonlinear terms.

\begin{lemma}\label{p_s_L_a_1_lem}
If $(L_2)^s(t)$, $(L_1)_p^s(t)$, $(L_2)_p^s(t)$, and $B_p(t)$ are defined as in the previous theorem (but with the absolute values moved outside the integral--see the remark below), then
$$(L_2)^s(t) \lesssim \frac{|a|}{M}B_1^s(t) +Err_{trap}$$
and
$$(L_1)_p^s(t)+(L_2)_p^s(t) \lesssim \frac{|a|}{M}B_p^s(t) +Err_{trap},$$
where
$$Err_{trap} =\left|\int_{\Sigma_t\cap\{r\approx r_{trap}\}}\pd_t\phi^{s}(q^{-2}\Gamma^{s}(q^2\mathcal{L}_\phi)-\mathcal{L}_{\phi^{s}})\right| +\left|\int_{\tilde\Sigma_t\cap\{r\approx r_{trap}\}}\pd_t\psi^{s}(q^{-2}\tilde\Gamma^{s}(q^2\mathcal{L}_\psi)-\mathcal{L}_{\psi^{s}})\right|.$$
\end{lemma}
\begin{remark}
The quantities $(L_2)^s(t)$, $(L_1)_p^s(t)$, and $(L_2)_p^s(t)$ have been slightly redefined so that the absolute value is moved outside the integral. This small detail will be important in the proof of Lemma \ref{p_s_L_a_lem}, because it allows integration by parts. The concerned reader can easily check that all of the estimates developed so far are also valid with the absolute value outside the integral.
\end{remark}
\begin{proof}
From Lemma \ref{p_L_a_lem}, by replacing $\phi$ and $\psi$ with $\phi^s$ and $\psi^s$, we have
$$(L_1)_p^s(t)\lesssim \frac{|a|}{M}B_p^s(t).$$
The challenge is to estimate the new terms $(L_2)^s(t)$ and $(L_2)_p^s(t)$, which arise from commuting with the operators $Q$ and $\tilde{Q}$.

Observe that
\begin{multline*}
\left|\int_{\Sigma_t}\pd_t\phi^s(q^{-2}\Gamma^s(q^2\mathcal{L}_\phi)-\mathcal{L}_{\phi^s})\right| \\
\lesssim \frac{|a|}{M}\int_{\Sigma_t\setminus \{r\approx r_{trap}\}}r^{-2}(\pd_t\phi)^2 + \left(\frac{|a|}{M}\right)^{-1}\int_{\Sigma_t\setminus \{r\approx r_{trap}\}}r^2(q^{-2}\Gamma^s(q^2\mathcal{L}_\phi)-\mathcal{L}_{\phi^s})^2 + Err_{trap} \\
\lesssim \frac{|a|}{M}B_1(t)+ \left(\frac{|a|}{M}\right)^{-1}\int_{\Sigma_t\setminus \{r\approx r_{trap}\}}r^2(q^{-2}\Gamma^s(q^2\mathcal{L}_\phi)-\mathcal{L}_{\phi^s})^2 + Err_{trap},
\end{multline*}
and
\begin{multline*}
\left|\int_{\Sigma_t}r^pL\phi^s(q^{-2}\Gamma^s(q^2\mathcal{L}_\phi)-\mathcal{L}_{\phi^s})\right| \\
\lesssim \frac{|a|}{M}\int_{\Sigma_t\setminus \{r\approx r_{trap}\}}r^{p-1}(L\phi)^2 + \left(\frac{|a|}{M}\right)^{-1}\int_{\Sigma_t\setminus \{r\approx r_{trap}\}}r^{p+1}(q^{-2}\Gamma^s(q^2\mathcal{L}_\phi)-\mathcal{L}_{\phi^s})^2 + Err_{trap} \\
\lesssim \frac{|a|}{M}B_p(t)+ \left(\frac{|a|}{M}\right)^{-1}\int_{\Sigma_t\setminus \{r\approx r_{trap}\}}r^{p+1}(q^{-2}\Gamma^s(q^2\mathcal{L}_\phi)-\mathcal{L}_{\phi^s})^2 + Err_{trap}.
\end{multline*}
From these two particular example estimates, it should become clear that the lemma reduces to the following estimate
$$\int_{\Sigma_t\setminus\{r\approx r_{trap}\}} r^{p+1}(q^{-2}\Gamma^s(q^2\mathcal{L}_\phi)-\mathcal{L}_{\phi^s})^2 + \int_{\tilde\Sigma_t\setminus\{r\approx r_{trap}\}} r^{p+1}(q^{-2}\tilde\Gamma^s(q^2\mathcal{L}_\psi)-\mathcal{L}_{\psi^s})^2 \lesssim \frac{a^2}{M^2}B_p^s(t),$$
since one can take $p=1$ to estimate the $(L_2)^s(t)$ term.

Given the estimate (\ref{mathcalL_p_estimate}) for $\mathcal{L}_\phi$ and $\mathcal{L}_\psi$ established in Lemma \ref{p_L_a_lem}, it suffices to show
$$\int_{\Sigma_t\setminus\{r\approx r_{trap}\}} r^{p+1}(q^{-2}\Gamma^s(q^2\mathcal{L}_\phi))^2 + \int_{\tilde\Sigma_t\setminus\{r\approx r_{trap}\}} r^{p+1}(q^{-2}\tilde\Gamma^s(q^2\mathcal{L}_\psi))^2 \lesssim \frac{a^2}{M^2}B_p^s(t).$$
This follows from the formalism developed in \S\ref{regularity_sec}. \qed
\end{proof}

In a moment, we will estimate the error terms $Err_{trap}$ from the previous lemma. But in order to do so, we need the approximate identities given by the following lemma.

\begin{lemma}\label{pd_t_identities_lem} In a neighborhood of $r_{trap}$, the following identities hold in the sense that each term on the right hand side is missing a smooth factor.
$$\pd_t\phi^s \approx \pd_t^{s+1}\phi + \pd_r^2\phi^{s-1} +r^{-1}\pd_r\phi^{s-1}+q^{-2}\Gamma^{s-1}(q^2\mathcal{L}_\phi) +q^{-2}\Gamma^{s-1}(q^2\Box_g\phi-q^2\mathcal{L}_\phi),$$
$$\pd_t\psi^s \approx \pd_t^{s+1}\psi + \pd_r^2\psi^{s-1} +r^{-1}\pd_r\psi^{s-1}+q^{-2}\Gamma^{s-1}(q^2\mathcal{L}_\psi) +q^{-2}\Gamma^{s-1}(q^2\Box_{\tilde{g}}\psi-q^2\mathcal{L}_\psi).$$
\end{lemma}
\begin{proof}
By the definition of $\phi^s$ and the fact that $[\pd_t,Q]=0$,
$$\pd_t\phi^s = \pd_t^{s-2i+1}Q^i\phi.$$
We use the approximate identity
$$Q\approx \pd_t^2 +\pd_r^2 +r^{-1}\pd_r +\Box_g.$$
(In reality, there is a factor of $q^2$ missing, but this function is smooth and bounded in a neighborhood of $r_{trap}$.)

We have
\begin{align*}
\pd_t^{s-2i+1}Q^i\phi &\approx \pd_t^{s-2i+1}Q^{i-1}\left(\pd_t^2\phi +\pd_r^2\phi +r^{-1}\pd_r\phi +\Box_g\phi\right) \\
&\approx \pd_t^{s-2(i-1)+1}Q^{i-1}\phi + \pd_r^2\phi^{s-1}+r^{-1}\pd_r\phi^{s-1}+\Gamma^{s-1}(\Box_g\phi-\mathcal{L}_\phi) +\Gamma^{s-1}(\mathcal{L}_\phi)
\end{align*}
Repeating this procedure $i-1$ more times proves the first identity of the lemma. The second identity is proved the same way. \qed
\end{proof}

Now, we improve the estimates from Lemma \ref{p_s_L_a_1_lem} by estimating the error term $Err_{trap}$.
\begin{lemma}\label{p_s_L_a_lem}
If $(L_2)^s(t)$, $(L_1)_p^s(t)$, $(L_2)_p^s(t)$, and $B_p^s(t)$ are as defined in Theorem \ref{p_s_L_thm} (but with the absolute values moved outside of the integral as in Lemma \ref{p_s_L_a_1_lem}), then
$$(L_2)^s(t) \lesssim \frac{|a|}{M}(B_1^s(t)+\mathring{B}_1^{s+1}(t)+N^s(t))$$
and
$$(L_1)_p^s(t)+(L_2)_p^s(t)\lesssim \frac{|a|}{M}(B_p^s(t)+\mathring{B}_p^{s+1}(t)+N_p^s(t)),$$
where
$$N^s(t)=(E^s(t))^{1/2}\left(\sum_{s'\le s}||q^{-2}\Gamma^{s'}(q^2\Box_g\phi-q^2\mathcal{L}_\phi)||_{L^2(\Sigma_t)}+\sum_{s'\le s}||q^{-2}\tilde\Gamma^{s'}(q^2\Box_{\tilde{g}}\psi-q^2\mathcal{L}_\psi)||_{L^2(\tilde{\Sigma}_t)}\right),$$
$$N_p^s(t) = \sum_{s'\le s}\int_{\Sigma_t}r^{p+1}|q^{-2}\Gamma^{s'}(q^2\Box_g\phi-q^2\mathcal{L}_\phi)|^2 + \sum_{s'\le s}\int_{\tilde\Sigma_t}r^{p+1}|q^{-2}\tilde\Gamma^{s'}(q^2\Box_{\tilde{g}}\psi-q^2\mathcal{L}_\psi)|^2.$$
\end{lemma}
\begin{proof}
From Lemma \ref{p_s_L_a_1_lem}, it suffices to show that
\begin{multline*}
\left|\int_{\Sigma_t\cap\{r\approx r_{trap}\}}\pd_t\phi^{s}(q^{-2}\Gamma^{s}(q^2\mathcal{L}_\phi)-\mathcal{L}_{\phi^{s}})\right| +\left|\int_{\tilde\Sigma_t\cap\{r\approx r_{trap}\}}\pd_t\psi^{s}(q^{-2}\tilde\Gamma^{s}(q^2\mathcal{L}_\psi)-\mathcal{L}_{\psi^{s}})\right| \\
\lesssim \frac{|a|}{M}(B_p^s(t) + \mathring{B}_p^{s+1}(t)+N^s(t))
\end{multline*}
and
\begin{multline*}
\left|\int_{\Sigma_t\cap\{r\approx r_{trap}\}}\pd_t\phi^{s}(q^{-2}\Gamma^{s}(q^2\mathcal{L}_\phi)-\mathcal{L}_{\phi^{s}})\right| +\left|\int_{\tilde\Sigma_t\cap\{r\approx r_{trap}\}}\pd_t\psi^{s}(q^{-2}\tilde\Gamma^{s}(q^2\mathcal{L}_\psi)-\mathcal{L}_{\psi^{s}})\right| \\
\lesssim \frac{|a|}{M}(B_p^s(t) + \mathring{B}_p^{s+1}(t)+N_p^s(t)).
\end{multline*}
We prove these estimates by using the approximate identities from Lemma \ref{pd_t_identities_lem}, ignoring the factor of $|a|/M$, which clearly comes from the factors $(q^{-2}\Gamma^s(q^2\mathcal{L}_\phi)-\mathcal{L}_{\phi^s})$ and $(q^{-2}\tilde{\Gamma}^s(q^2\mathcal{L}_\psi)-\mathcal{L}_{\psi^s})$. The following examples illustrate all of the difficulties.
\begin{multline*}
\left|\int_{\Sigma_t\cap\{r\approx r_{trap}\}}(\pd_t^{s+1}\phi)(q^{-2}\Gamma^{s}(q^2\mathcal{L}_\phi)-\mathcal{L}_{\phi^{s}})\right| \\
\lesssim \int_{\Sigma_t\cap\{r\approx r_{trap}\}}(\pd_t^{s+1}\phi)^2 + \int_{\Sigma_t\cap\{r\approx r_{trap}\}}(q^{-2}\Gamma^{s}(q^2\mathcal{L}_\phi)-\mathcal{L}_{\phi^{s}})^2 \\
\lesssim \mathring{B}_p^{s+1}(t) + B_p^s(t).
\end{multline*}
\begin{multline*}
\left|\int_{\Sigma_t\cap\{r\approx r_{trap}\}}(\pd_r^2\Gamma^{s-1}\phi)(q^{-2}\Gamma^{s}(q^2\mathcal{L}_\phi)-\mathcal{L}_{\phi^{s}})\right| \\
\lesssim \left|\int_{\Sigma_t\cap\{r\approx r_{trap}\}}(\pd_r\Gamma^{s-1}\phi)\pd_r(q^{-2}\Gamma^{s}(q^2\mathcal{L}_\phi)-\mathcal{L}_{\phi^{s}})\right| \\
\lesssim \int_{\Sigma_t\cap\{r\approx r_{trap}\}}(\pd_r\Gamma^{s-1}\phi)^2 + \int_{\Sigma_t\cap\{r\approx r_{trap}\}}(\pd_r(q^{-2}\Gamma^{s}(q^2\mathcal{L}_\phi)-\mathcal{L}_{\phi^{s}}))^2 \\
\lesssim B_p^s(t).
\end{multline*}
In particular, note in the previous estimate the need to integrate by parts. The expression $q^{-2}\Gamma^s(q^2\mathcal{L}_\phi)-\mathcal{L}_{\phi^s}$ is a commutator, so it has at most $s$ derivatives. We can assume that none of these derivatives is $\pd_r$, because otherwise, instead of using the approximate identities from Lemma \ref{pd_t_identities_lem}, we could have simply integrated by parts to move one of the angular derivatives in $\pd_t\phi^s$ to the other factor. We continue with more examples.
\begin{multline*}
\left|\int_{\Sigma_t\cap\{r\approx r_{trap}\}}(q^{-2}\Gamma^{s-1}(q^2\mathcal{L}_\phi))(q^{-2}\Gamma^{s}(q^2\mathcal{L}_\phi)-\mathcal{L}_{\phi^{s}})\right|  \\
\lesssim \int_{\Sigma_t\cap\{r\approx r_{trap}\}}(q^{-2}\Gamma^{s-1}(q^2\mathcal{L}_\phi))^2 + \int_{\Sigma_t\cap\{r\approx r_{trap}\}}(q^{-2}\Gamma^{s}(q^2\mathcal{L}_\phi)-\mathcal{L}_{\phi^{s}})^2 \\
\lesssim B_p^s(t).
\end{multline*}

\begin{multline*}
\left|\int_{\Sigma_t\cap\{r\approx r_{trap}\}}q^{-2}\Gamma^{s-1}(q^2\Box_g\phi-q^2\mathcal{L}_\phi)(q^{-2}\Gamma^{s}(q^2\mathcal{L}_\phi)-\mathcal{L}_{\phi^{s}})\right| \\
\lesssim \int_{\Sigma_t\cap\{r\approx r_{trap}\}}(q^{-2}\Gamma^{s-1}(q^2\Box_g\phi-q^2\mathcal{L}_\phi))^2 + \int_{\Sigma_t\cap\{r\approx r_{trap}\}}(q^{-2}\Gamma^{s}(q^2\mathcal{L}_\phi)-\mathcal{L}_{\phi^{s}})^2 \\
\lesssim N_p^s(t)+B_p^s(t).
\end{multline*}
Also,
\begin{multline*}
\left|\int_{\Sigma_t\cap\{r\approx r_{trap}\}}q^{-2}\Gamma^{s-1}(q^2\Box_g\phi-q^2\mathcal{L}_\phi)(q^{-2}\Gamma^{s}(q^2\mathcal{L}_\phi)-\mathcal{L}_{\phi^{s}})\right| \\
\lesssim \left(\int_{\Sigma_t\cap\{r\approx r_{trap}\}}(q^{-2}\Gamma^{s-1}(q^2\Box_g\phi-q^2\mathcal{L}_\phi))^2\right)^{1/2}\left(\int_{\Sigma_t\cap\{r\approx r_{trap}\}}(q^{-2}\Gamma^{s}(q^2\mathcal{L}_\phi)-\mathcal{L}_{\phi^{s}})^2\right)^{1/2} \\
\lesssim ||q^{-2}\Gamma^{s-1}(q^2\Box_g\phi-q^2\mathcal{L}_\phi)||_{L^2(\Sigma_t)}(E^s(t))^{1/2} \\
\lesssim N^s(t).
\end{multline*}
These example estimates are sufficient to verify the lemma. \qed
\end{proof}

By using Lemma \ref{p_s_L_a_lem} and assuming the full nonlinear equations, Theorem \ref{p_s_L_thm} can be improved slightly. This is stated in the following theorem.
\begin{theorem}\label{p_s_thm}(Improved version of Theorem \ref{p_s_L_thm}) Suppose $|a|/M$ is sufficiently small and fix $\delm,\delp>0$. Suppose furthermore that the pair $(\phi,\psi)$ satisfies the system
\begin{align*}
\Box_g\phi &= \mathcal{L}_\phi+\mathcal{N}_\phi, \\
\Box_{\tilde{g}}\psi &= \mathcal{L}_\psi + \mathcal{N}_\psi.
\end{align*}
Then the following estimates hold for $p\in [\delm,2-\delp]$.
$$E^s(t_2)\lesssim E^s(t_1)+\int_{t_1}^{t_2}\mathring{B}_1^{s+1}(t)+B_1^s(t)+N^s(t)dt$$
$$E_p^s(t_2)+\int_{t_1}^{t_2}B_p^s(t)dt\lesssim E_p^s(t_1)+\int_{t_1}^{t_2}\mathring{B}_p^{s+1}(t)+N_p^s(t)dt,$$
where $E^s(t)$, $E_p^s(t)$, and $B_p^s(t)$ are as defined in Theorem \ref{p_s_L_thm}, and
$$N^s(t)=(E^s(t))^{1/2}\left(\sum_{s'\le s}||q^{-2}\Gamma^{s'}(q^2\mathcal{N}_\phi)||_{L^2(\Sigma_t)}+\sum_{s'\le s}||q^{-2}\tilde\Gamma^{s'}(q^2\mathcal{N}_\psi)||_{L^2(\tilde{\Sigma}_t)}\right),$$
\begin{align*}
N_p^s(t) =& \sum_{s'\le s}\int_{\Sigma_t}r^{p+1}|q^{-2}\Gamma^{s'}(q^2\mathcal{N}_\phi)|^2 + \sum_{s'\le s}\int_{\tilde\Sigma_t}r^{p+1}|q^{-2}\tilde\Gamma^{s'}(q^2\mathcal{N}_\psi)|^2.
\end{align*}
\end{theorem}
\begin{proof}
By Theorem \ref{p_s_L_thm} and Lemma \ref{p_s_L_a_lem}, we have
$$E^s(t_2)\lesssim E^s(t_1)+\int_{t_1}^{t_2}\mathring{B}_1^{s+1}(t)+B_1^s(t)+N^s(t)+(N')^s(t)dt,$$
where
\begin{align*}
(N')^s(t)&=\sum_{s'\le s}\int_{\Sigma_t}|\pd_t\phi^{s'}(\Box_g\phi^{s'}-q^{-2}\Gamma^{s'}(q^2\mathcal{L}_\phi))| +\sum_{s'\le s}\int_{\tilde{\Sigma}_t}|\pd_t\psi^{s'}(\Box_{\tilde g}\psi^{s'}-q^{-2}\tilde\Gamma^{s'}(q^2\mathcal{L}_\psi))| \\
&=\sum_{s'\le s}\int_{\Sigma_t}|\pd_t\phi^{s'}q^{-2}\Gamma^{s'}(q^2\mathcal{N}_{\phi})| +\sum_{s'\le s}\int_{\tilde{\Sigma}_t}|\pd_t\psi^{s'}q^{-2}\tilde{\Gamma}^{s'}(q^2\mathcal{N}_{\psi})|.
\end{align*}
Observe that
\begin{align*}
(N')^s(t) &=\sum_{s'\le s}\int_{\Sigma_t}|\pd_t\phi^{s'}q^{-2}\Gamma^{s'}(q^2\mathcal{N}_{\phi})| +\sum_{s'\le s}\int_{\tilde{\Sigma}_t}|\pd_t\psi^{s'}q^{-2}\tilde{\Gamma}^{s'}(q^2\mathcal{N}_{\psi})| \\
&\lesssim \sum_{s'\le s}||\pd_t\phi^{s'}||_{L^2(\Sigma_t)}||q^{-2}\Gamma^{s'}(q^2\mathcal{N}_\phi)||_{L^2(\Sigma_t)}+ \sum_{s'\le s}||\pd_t\psi^{s'}||_{L^2(\tilde\Sigma_t)}||q^{-2}\tilde\Gamma^{s'}(q^2\mathcal{N}_\psi)||_{L^2(\tilde\Sigma_t)} \\
&\lesssim \sum_{s'\le s}(E^{s'}(t))^{1/2}||q^{-2}\Gamma^{s'}(q^2\mathcal{N}_\phi)||_{L^2(\Sigma_t)}+ \sum_{s'\le s}(E^{s'}(t))^{1/2}||q^{-2}\tilde\Gamma^{s'}(q^2\mathcal{N}_\psi)||_{L^2(\tilde\Sigma_t)} \\
&\lesssim N^s(t).
\end{align*}
This proves the first estimate of the theorem.

By Theorem \ref{p_s_L_thm} and Lemma \ref{p_s_L_a_lem}, we have
$$E_p^s(t_2)+\int_{t_1}^{t_2}B_p^s(t)dt\lesssim E_p^s(t_1)+\int_{t_1}^{t_2}\mathring{B}_p^s(t) +\frac{|a|}{M}B_p^s(t)+N_p^s(t)+(N')_p^s(t)dt,$$
where
\begin{align*}
(N')_p^s(t)=& \sum_{s'\le s}\int_{\Sigma_t}r^p|\mathcal{X}_{\phi^s}||\Box_g\phi^{s'}-q^{-2}\Gamma^{s'}(q^2\mathcal{L}_\phi)| + \sum_{s'\le s}\int_{\tilde\Sigma_t}r^p|\mathcal{X}_{\psi^s}||\Box_{\tilde{g}}\psi^{s'}-q^{-2}\tilde\Gamma^{s'}(q^2\mathcal{L}_\psi)| \\
=& \sum_{s'\le s}\int_{\Sigma_t}r^p|\mathcal{X}_{\phi^s}||q^{-2}\Gamma^{s'}(q^2\mathcal{N}_\phi)| + \sum_{s'\le s}\int_{\tilde\Sigma_t}r^p|\mathcal{X}_{\psi^s}||q^{-2}\tilde\Gamma^{s'}(q^2\mathcal{N}_\psi)|.
\end{align*}
Observe that
\begin{multline*}
\sum_{s'\le s}\int_{\Sigma_t}r^p|\mathcal{X}_{\phi^s}||q^{-2}\Gamma^{s'}(q^2\mathcal{N}_\phi)| \\
\lesssim \sum_{s'\le s} \left(\int_{\Sigma_t}r^{p-1}(\mathcal{X}_{\phi^s})^2\right)^{1/2}\left(\int_{\Sigma_t}r^{p+1}(q^{-2}\Gamma^{s'}(q^2\mathcal{N}_\phi))^2\right)^{1/2} \\
\lesssim \sum_{s'\le s}\left(\mathring{B}_p^{s'+1}(t)+B_p^{s'}(t)\right)^{1/2}\left(N_p^{s'}(t)\right)^{1/2} \\
\lesssim \epsilon\mathring{B}_p^{s+1}(t)+\epsilon B_p^s(t) +\epsilon^{-1}N_p^s(t).
\end{multline*}
An analogous estimate also shows that
$$\sum_{s'\le s}\int_{\tilde\Sigma_t}r^p|\mathcal{X}_{\psi^s}||q^{-2}\tilde\Gamma^{s'}(q^2\mathcal{N}_\psi)| \lesssim \epsilon\mathring{B}_p^{s+1}(t)+\epsilon B_p^s(t)+\epsilon^{-1}N_p^s(t).$$
Therefore, we have
$$E_p^s(t_2)+\int_{t_1}^{t_2}B_p^s(t)dt\lesssim E_p^s(t_1)+(|a|/M+\epsilon)\int_{t_1}^{t_2}B_p^s(t)dt+\int_{t_1}^{t_2}\mathring{B}_p^s(t)+\epsilon^{-1}N_p^s(t)dt.$$
By taking $|a|/M$ and $\epsilon$ sufficiently small, the bulk term on the right hand side can be absorbed into the left hand side. The result is the second estimate of the theorem. \qed
\end{proof}

\subsection{The main energy estimates for $(\phi^{s,k},\psi^{s,k})$}

To handle nonlinear terms with a factor of $\pd_r\phi$ or $\pd_r\psi$ near the event horizon, it will be necessary to use an additional commutator, which we call $\tg$. (See Proposition \ref{infinity_prop} and the remark that follows its proof.)
\begin{definition}
$$\tg:=1_H(r)\pd_r,$$
where $1_H$ is a nonnegative smooth function supported near the event horizon.
\end{definition}

Unfortunately, $\tg$ does not even commute with either of the wave operators $\Box_g$ or $\Box_{\tilde{g}}$. But we will see that one particular term in the commutator with either wave operator has an appropriate sign near the event horizon, and that is what allows us to use $\tg$. We now compute the commutators.
\begin{lemma}\label{dr_comm_1_lem}
$$[\tg,q^2\Box_g]u=\Delta'\pd_r\tg u -21_H'q^2\Box_gu+\{\pd_r\Gamma^{\le 1}u,\Gamma^{\le 2}u\} $$
$$[\tg,q^2\Box_{\tilde{g}}]u=\Delta'\pd_r\tg u -21_H'q^2\Box_{\tilde{g}}u+\{\pd_r\tilde{\Gamma}^{\le 1}u,\tilde{\Gamma}^{\le 2}u\},$$
where the expression $\{\pd_r\Gamma^{\le 1}u,\Gamma^{\le 2}u\}$ represents any terms of the form $f(r,\theta)\pd_r\Gamma^{\le 1}u$ or $f(r,\theta)\Gamma^{\le 2}u$ for smooth $f$ with compact support.
\end{lemma}
\begin{proof}
We expand
$$q^2\Box_g= q^2g^{tt}\pd_t^2 + 2q^2g^{tr}\pd_r\pd_t +q^2g^{rr}\pd_r^2+\pd_r(q^2g^{tr})\pd_t+\pd_r(q^2g^{rr})\pd_r+\frac1{\sin\theta}\pd_\alpha\left(\sin\theta Q^{\alpha\beta}\pd_\beta\cdot\right).$$
Note that the $Q$ term commutes with $\tg$. We compute terms arising in the commutator by neglecting the highest order derivative terms and the $Q$ terms (both represented by the ellipsis in what follows).
\begin{multline*}
\tg (q^2\Box_g u)= \\
\tg(q^2g^{tt})\pd_t^2u+2\tg(q^2g^{tr})\pd_r\pd_tu+\tg(q^2g^{rr})\pd_r^2u+\tg(\pd_r(q^2g^{tr}))\pd_tu+\tg(\pd_r(q^2g^{rr}))\pd_ru+... \\
=1_H\Delta'\pd_r^2u+\{\pd_t^2u,\pd_r\pd_tu,\pd_tu,\pd_ru\}+...
\end{multline*}
\begin{align*}
q^2\Box_g(\tg u)&=2q^2g^{tr}1_H'\pd_r\pd_tu+2q^2g^{rr}1_H'\pd_r^2u+q^2g^{rr}1_H''\pd_ru+\pd_r(q^2g^{rr})1_H'\pd_ru+... \\
&=2q^2g^{rr}1_H' \pd_r^2u+\{\pd_r\pd_tu,\pd_ru\}+... \\
&=21_H' q^2\Box_g u+\{\pd_t^2u,\pd_r\pd_tu,Qu,\pd_tu,\pd_ru\}+...
\end{align*}
Taking the difference,
\begin{align*}
[\tg,q^2\Box_g]u &= \tg(q^2\Box_g u)-q^2\Box_g(\tg u) \\
&=1_H\Delta'\pd_r^2u - 21_H' q^2\Box_g u+\{\pd_t^2u,\pd_r\pd_tu,Qu,\pd_tu,\pd_ru\} \\
&=\Delta'\pd_r\tg u - 21_H'q^2\Box_g u+\{\pd_t^2u,\pd_r\pd_tu,Qu,\pd_tu,\pd_ru\}.
\end{align*}
The proof for $[\tg,q^2\Box_{\tilde{g}}]$ is identical. \qed
\end{proof}

Now we generalize the prevous lemma by commuting with $\tg$ arbitrarily many times.
\begin{lemma}\label{dr_comm_2_lem}
$$[\tg^k,q^2\Box_g]u=k\Delta'\pd_r\tg^ku+\{\pd_r\tg^{\le k-1}\Gamma^{\le 1}u,\tg^{\le k-1}\Gamma^{\le 2}u,\tg^{\le k-1}(q^2\Box_gu)\},$$
$$[\tg^k,q^2\Box_{\tilde{g}}]u=k\Delta'\pd_r\tg^ku+\{\pd_r\tg^{\le k-1}\tilde\Gamma^{\le 1}u,\tg^{\le k-1}\tilde\Gamma^{\le 2}u,\tg^{\le k-1}(q^2\Box_{\tilde{g}}u)\},$$
where the $\{...\}$ notation is the same as in Lemma \ref{dr_comm_1_lem}.
\end{lemma}
\begin{proof}
We only prove the first identity since the proof of the second is identical. The proof is an induction argument on $k$. The case $k=1$ corresponds to Lemma \ref{dr_comm_1_lem}. Assuming that the statement of the lemma holds at the level $k$, we prove the analogous statement at the level $k+1$. We have
\begin{align*}
[\tg^{k+1},q^2\Box_g]u &= \tg^{k+1}(q^2\Box_gu)-q^2\Box_g(\tg^{k+1}u) \\
&= \tg (\tg^k(q^2\Box_gu)-q^2\Box_g(\tg^k u))+(\tg(q^2\Box_g(\tg^ku))-q^2\Box_g(\tg^{k+1}u)) \\
&= \tg[\tg^k,q^2\Box_g]u+[\tg,q^2\Box_g](\tg^k u).
\end{align*}
Now, by the inductive hypothesis,
\begin{align*}
\tg[\tg^k,q^2\Box_g]u &= \tg\left(k\Delta'\pd_r\tg^ku+\{\pd_r\tg^{\le k-1}\Gamma^{\le 1}u,\tg^{\le k-1}\Gamma^{\le 2}u,\tg^{\le k-1}(q^2\Box_gu)\}\right) \\
&= k\Delta'\pd_r\tg^{k+1}u +\{\pd_r\tg^{\le k}\Gamma^{\le 1}u,\tg^{\le k}\Gamma^{\le 2}u,\tg^{\le k}(q^2\Box_gu)\},
\end{align*}
and by the base case,
\begin{align*}
[\tg,q^2\Box_g](\tg^k u) &= \Delta'\pd_r\tg^{k+1}u +\{\pd_r\Gamma^{\le 1}\tg^ku,\Gamma^{\le 2}\tg^ku,q^2\Box_g(\tg^ku)\} \\
&= \Delta'\pd_r\tg^{k+1}u +\{\pd_r\tg^k\Gamma^{\le 1}u,\tg^k\Gamma^{\le 2}u,\tg^{\le k}(q^2\Box_gu)\},
\end{align*}
where in the last step we used the inductive hypothesis a second time.

Summing these two yields
$$[\tg^{k+1},q^2\Box_g]u = (k+1)\Delta'\pd_r\tg^{k+1}u+ \{\pd_r\tg^{\le k}\Gamma^{\le 1}u,\tg^{\le k}\Gamma^{\le 2}u,\tg^{\le k}(q^2\Box_gu)\}.$$
This completes the inductive argument.
\qed
\end{proof}

Since $\tg$ doesn't commute with $\Box_g$ or $\Box_{\tilde{g}}$, it is treated seperately than the previous commutators. We define the $s,k$-order dynamic quantities $\phi^{s,k}$ and $\psi^{s,k}$.
\begin{definition}
$$\phi^{s,k} := \tg^k\Gamma^s\phi,$$
$$\psi^{s,k} := \tg^k\tilde\Gamma^s\psi.$$
\end{definition}
We also define the $s,k$-order analogues of $\mathcal{L}_\phi$, $\mathcal{L}_\psi$, $\mathcal{X}_\phi$, and $\mathcal{X}_\psi$.
\begin{definition}
$$\mathcal{L}_{\phi^{s,k}}:=-2\frac{\pd^\alpha B}{A}A\pd_\alpha \psi^{s,k} + 2\frac{\pd^\alpha B\pd_\alpha B}{A^2}\phi^{s,k}-4\frac{\pd^\alpha A\pd_\alpha B}{A^2} A\psi^{s,k}$$
$$\mathcal{L}_{\psi^{s,k}}:=-2\frac{\pd^\alpha A_2}{A_2}\pd_\alpha\psi^{s,k}+2\frac{\pd^\alpha B\pd_\alpha B}{A^2}\psi^{s,k} + 2A^{-1}\frac{\pd^\alpha B}{A}\pd_\alpha\phi^{s,k}$$
and
\begin{align*}
\mathcal{X}_{\phi^{s,k}} &:= 2X(\phi^{s,k})+w\phi^{s,k}+w_{(a)}\psi^{s,k} \\
\mathcal{X}_{\psi^{s,k}} &:= 2X(\psi^{s,k})+\tilde{w}\psi^{s,k}+\tilde{w}_{(a)}\phi^{s,k},
\end{align*}
where, in each expression, the exact same operator $\tg^k\Gamma^s$ or $\tg^k\tilde{\Gamma}^s$ is to be used in each term on the right hand side, replacing $Q$ with $\tilde{Q}$ where appropriate. For example, the expression
$$-2\frac{\pd^\alpha B}{A}A\pd_\alpha (\tg\tilde{Q}\psi) + 2\frac{\pd^\alpha B\pd_\alpha B}{A^2}(\tg Q\phi)-4\frac{\pd^\alpha A\pd_\alpha B}{A^2} A (\tg\tilde{Q}\psi)$$
belongs to $\mathcal{L}_{\phi^{s,k}}$ ($s=2,k=1$) while the expression
$$-2\frac{\pd^\alpha B}{A}A\pd_\alpha(\tg\pd_t^2\psi) + 2\frac{\pd^\alpha B\pd_\alpha B}{A^2}(\tg Q\phi)-4\frac{\pd^\alpha A\pd_\alpha B}{A^2} A (\tg\tilde{Q}\psi)$$
does not.
\end{definition}

We take a look again at the equations
$$\Box_g\phi = \mathcal{L}_\phi +\mathcal{N}_\phi,$$
$$\Box_{\tilde{g}}\psi = \mathcal{L}_\psi +\mathcal{N}_\psi.$$
By applying $\tg^k\Gamma^s$ and $\tg^k\tilde\Gamma^s$ respectively, we obtain additional useful equations.
\begin{align*}
\Box_g\phi^{s,k} &= (\Box_g\phi^{s,k}-q^{-2}\tg^k(q^2\Box_g\phi^s)) +q^{-2}\tg^k\Gamma^s(q^2\mathcal{L}_\phi) + q^{-2}\tg^k\Gamma^s(q^2\mathcal{N}_\phi) \\
&= \mathcal{L}_{\phi^{s,k}} + q^{-2}[q^2\Box_g,\tg^k]\phi^s + (q^{-2}\tg^k\Gamma^s(q^2\mathcal{L}_\phi)-\mathcal{L}_{\phi^{s,k}}) + q^{-2}\tg^k\Gamma^s(q^2\mathcal{N}_\phi),
\end{align*}
\begin{align*}
\Box_{\tilde{g}}\psi^{s,k} &= (\Box_{\tilde{g}}\psi^{s,k}-q^{-2}\tg^k(q^2(\Box_{\tilde{g}}\psi^s)) +q^{-2}\tg^k\tilde{\Gamma}^s(q^2\mathcal{L}_\psi) + q^{-2}\tg^k\tilde{\Gamma}^s(q^2\mathcal{N}_\psi) \\
&= \mathcal{L}_{\psi^{s,k}} + q^{-2}[q^2\Box_{\tilde{g}},\tg^k]\psi^s + (q^{-2}\tg^k\tilde{\Gamma}^s(q^2\mathcal{L}_\psi)-\mathcal{L}_{\psi^{s,k}}) + q^{-2}\tg^k\tilde{\Gamma}^s(q^2\mathcal{N}_\psi).
\end{align*}
Therefore, Theorem \ref{p_s_L_thm} (which generalized Theorem \ref{p_L_thm}) can be generalized to the following theorem. (Note the presence of the additional terms $(L_3)^s(t)$ and $(L_3)_p^s(t)$, which arise from the fact that the $\tg$ operators do not commute with the wave operators.)
\begin{theorem}\label{p_s_k_L_thm}
Suppose $|a|/M$ is sufficiently small and fix $\delm,\delp>0$. The following estimates hold for $p\in [\delm,2-\delp]$.
$$E^{s,k}(t_2)\lesssim E^{s,k}(t_1)+\int_{t_1}^{t_2}(L_2)^{s,k}(t)+(L_3)^{s,k}(t)+N^{s,k}(t)dt,$$
$$E_p^{s,k}(t_2)+\int_{t_1}^{t_2}B_p^{s,k}(t)dt\lesssim E_p^{s,k}(t_1)+\int_{t_1}^{t_2}(L_1)_p^s(t)+(L_2)_p^{s,k}(t)+(L_3)_p^{s,k}(t)+N_p^{s,k}(t)dt,$$
where
\begin{align*}
E^{s,k}(t) &= \sum_{\substack{s'\le s \\ k'\le k}} E[(\phi^{s',k'},\psi^{s',k'})](t), \\
E_p^{s,k}(t) &= \sum_{\substack{s'\le s \\ k'\le k}} E_p[(\phi^{s',k'},\psi^{s',k'})](t), \\
B_p^{s,k}(t) &= \sum_{\substack{s'\le s \\ k'\le k}} B_p[(\phi^{s',k'},\psi^{s',k'})](t),
\end{align*}
and
\begin{align*}
(L_2)^{s,k}(t)=&\sum_{\substack{s'\le s \\ k'\le k}}\int_{\Sigma_t}|\pd_t\phi^{s',k'} (q^{-2}\tg^{k'}\Gamma^{s'}(q^2\mathcal{L}_\phi)-\mathcal{L}_{\phi^{s',k'}})| \\
&+\sum_{\substack{s'\le s \\ k'\le k}}\int_{\tilde\Sigma_t}|\pd_t\psi^{s',k'}(q^{-2}\tg^{k'}\tilde{\Gamma}^{s'}(q^2\mathcal{L}_\psi)-\mathcal{L}_{\psi^{s',k'}})|,
\end{align*}
\begin{align*}
(L_3)^{s,k}(t)=&\sum_{\substack{s'\le s \\ k'\le k}}\int_{\Sigma_t}|\pd_t\phi^{s',k'} (q^{-2}[q^2\Box_g,\tg^{k'}]\phi^{s'})|
+\sum_{\substack{s'\le s \\ k'\le k}}\int_{\tilde\Sigma_t}|\pd_t\psi^{s',k'} (q^{-2}[q^2\Box_{\tilde{g}},\tg^{k'}]\psi^{s'})|,
\end{align*}
\begin{align*}
N^{s,k}(t) =& \sum_{\substack{s'\le s \\ k'\le k}}\int_{\Sigma_t}|\pd_t\phi^{s',k'}(q^{-2}\tg^{k'}(q^2\Box_g\phi^{s'}-\Gamma^{s'}(q^2\mathcal{L}_\phi)))| \\
&+\sum_{\substack{s'\le s \\ k'\le k}}\int_{\tilde{\Sigma}_t}|\pd_t\psi^{s',k'}(q^{-2}\tg^{k'}(q^2\Box_{\tilde g}\psi^{s'}-\tilde\Gamma^{s'}(q^2\mathcal{L}_\psi)))|,
\end{align*}
\begin{align*}
(L_1)_p^s(t) =& \sum_{s'\le s}\int_{\Sigma_t\cap\{r>R\}}r^p(|L\phi^{s'}|+r^{-1}|\phi^{s'}|)|\mathcal{L}_{\phi^{s'}}| +\sum_{s'\le s}\int_{\tilde\Sigma_t\cap\{r>R\}}r^p(|L\psi^{s'}|+r^{-1}|\psi^{s'}|)|\mathcal{L}_{\psi^{s'}}|,
\end{align*}
\begin{align*}
(L_2)_p^{s,k}(t) =& \sum_{\substack{s'\le s \\ k'\le k}}\int_{\Sigma_t}r^p|\mathcal{X}_{\phi^{s,k}}||q^{-2}\tg^{k'}\Gamma^{s'}(q^2\mathcal{L}_\phi)-\mathcal{L}_{\phi^{s',k'}}|  \\
&+ \sum_{\substack{s'\le s \\ k'\le k}}\int_{\tilde\Sigma_t}r^p|\mathcal{X}_{\psi^{s,k}}||q^{-2}\tg^{k'}\tilde\Gamma^{s'}(q^2\mathcal{L}_\psi)-\mathcal{L}_{\psi^{s',k'}}|,
\end{align*}
\begin{align*}
(L_3)_p^{s,k}(t) =& \sum_{\substack{s'\le s \\ k'\le k}}\int_{\Sigma_t}r^p|\mathcal{X}_{\phi^{s,k}}||q^{-2}[q^2\Box_g,\tg^{k'}]\phi^{s'}|  
+ \sum_{\substack{s'\le s \\ k'\le k}}\int_{\tilde\Sigma_t}r^p|\mathcal{X}_{\psi^{s,k}}||q^{-2}[q^2\Box_{\tilde{g}},\tg^{k'}]\psi^{s'}|,
\end{align*}
\begin{align*}
N_p^{s,k}(t)=& \sum_{\substack{s'\le s \\ k'\le k}}\int_{\Sigma_t}r^p|\mathcal{X}_{\phi^{s,k}}||q^{-2}\tg^{k'}(q^2\Box_g\phi^{s'}-\Gamma^{s'}(q^2\mathcal{L}_\phi))| \\
&+ \sum_{\substack{s'\le s \\ k'\le k}}\int_{\tilde\Sigma_t}r^p|\mathcal{X}_{\psi^{s,k}}||q^{-2}\tg^{k'}(q^2\Box_{\tilde{g}}\psi^{s'}-\tilde\Gamma^{s'}(q^2\mathcal{L}_\psi))|.
\end{align*}
\end{theorem}
\begin{proof}
The proof is a direct application of Theorem \ref{p_L_thm} by making the substitutions
\begin{align*}
\phi&\mapsto\phi^{s',k'} \\
\psi&\mapsto\psi^{s',k'}
\end{align*}
for all values of $s'$ and $k'$ (and all commutators represented by $\tg^k\Gamma^{s'}$ and $\tg^k\tilde\Gamma^{s'}$) where $s'\le s$ and $k'\le k$. The following estimates are used.
\begin{multline*}
|\Box_g\phi^{s',k'}-\mathcal{L}_{\phi^{s',k'}}| \\
\le |q^{-2}[q^2\Box_g,\tg^{k'}]\phi^{s'}|+ |q^{-2}\tg^{k'}(q^2\Box_g\phi^{s'}-\Gamma^{s'}(q^2\mathcal{L}_\phi)|+|q^{-2}\tg^{k'}\Gamma^{s'}(q^2\mathcal{L}_\phi)-\mathcal{L}_{\phi^{s',k'}}|,
\end{multline*}
\begin{multline*}
|\Box_{\tilde{g}}\psi^{s',k'}-\mathcal{L}_{\psi^{s',k'}}| \\
\le |q^{-2}[q^2\Box_{\tilde{g}},\tg^{k'}]\psi^{s'}|+ |q^{-2}\tg^{k'}(q^2\Box_{\tilde{g}}\psi^{s'}-\tilde\Gamma^{s'}(q^2\mathcal{L}_\psi)|+|q^{-2}\tg^{k'}\tilde\Gamma^{s'}(q^2\mathcal{L}_\psi)-\mathcal{L}_{\psi^{s',k'}}|.
\end{multline*}
The resulting error terms have been grouped into the parts that will either be linear or nonlinear when using the equations from the fully nonlinear system. (See Theorem \ref{p_s_k_thm} below.) \qed
\end{proof}

Once again, the plan is to improve the prevoius theorem after proving a lemma that handles the linear error terms on the right hand side. In this case, the new linear error terms are $(L_3)^{s,k}(t)$ and $(L_3)_p^{s,k}(t)$. These terms arise from the fact that the operator $\tg$ does not commute with the wave operators. Since $\tg$ is supported near the event horizon, there are no new issues related to trapping. The important observation to make is that one of the terms in the commutator has an appropriate sign on the event horizon. This is the point of the following lemma.
\begin{lemma}\label{p_s_k_L_a_lem}
If $(L_2)^{s,k}(t)$, $(L_3)^{s,k}(t)$, $(L_1)_p^s(t)$, $(L_2)_p^{s,k}(t)$, $(L_3)_p^{s,k}(t)$, and $B_p^{s,k}(t)$ are as defined in Theorem \ref{p_s_k_L_thm}, then
$$(L_2)^{s,k}(t)+(L_3)^{s,k}(t) \lesssim B_{p'}^{s+2,k-1}(t)+\frac{|a|}{M}(B_1^{s,k}(t)+\mathring{B}_1^{s+1}(t))+N^{s,k}(t)$$
and
$$(L_1)_p^s(t)+(L_2)_p^{s,k}(t)+(L_3)_p^{s,k}(t)\lesssim B_{p'}^{s+2,k-1}(t) + \frac{|a|}{M}(B_p^{s,k}(t)+\mathring{B}_p^{s+1}(t))+N_p^{s,k}(t),$$
where
\begin{multline*}
N^{s,k}(t)=(E^{s,k}(t))^{1/2}\left(\sum_{\substack{s'\le s \\ k'\le k}}||q^{-2}\tg^{k'}\Gamma^{s'}(q^2\Box_g\phi-q^2\mathcal{L}_\phi)||_{L^2(\Sigma_t)}\right. \\
+\left.\sum_{\substack{s'\le s \\ k'\le k}}||q^{-2}\tg^{k'}\tilde\Gamma^{s'}(q^2\Box_{\tilde{g}}\psi-q^2\mathcal{L}_\psi)||_{L^2(\tilde{\Sigma}_t)}\right),
\end{multline*}
$$N_p^{s,k}(t) = \sum_{\substack{s'\le s \\ k'\le k}}\int_{\Sigma_t}r^{p+1}|q^{-2}\tg^{k'}\Gamma^{s'}(q^2\Box_g\phi-q^2\mathcal{L}_\phi)|^2 + \sum_{\substack{s'\le s \\ k'\le k}}\int_{\tilde\Sigma_t}r^{p+1}|q^{-2}\tg^{k'}\tilde\Gamma^{s'}(q^2\Box_{\tilde{g}}\psi-q^2\mathcal{L}_\psi)|^2,$$
and $p'$ is arbitrary.
\end{lemma}
\begin{remark}
The reason for the arbitrary $p'$ on the right hand side is that the bulk norms with the arbitrary $p'$ are only used to control the new terms related to commuting with $\tg$. These terms are all supported on a compact radial interval, so the factor $r^{p'-1}$ that appears in $B_{p'}^{s,k}(t)$ can be approximated by a constant.
\end{remark}
\begin{proof}
The proof of Lemma \ref{p_s_L_a_lem} can be adapted to show that
$$(L_2)^{s,k}(t)\lesssim \frac{|a|}{M}(B_1^{s,k}(t)+\mathring{B}_1^{s+1}(t)+N^{s,k}(t)),$$
$$(L_1)_p^{s,k}(t)+(L_2)^{s,k}(t)\lesssim \frac{|a|}{M}(B_p^{s,k}(t)+\mathring{B}_p^{s+1}(t)+N_p^{s,k}(t)).$$
It suffices to show that
$$(L_3)^{s,k}(t)\lesssim B_{p'}^{s+2,k-1}(t) + N^{s,k}(t),$$
$$(L_3)_p^{s,k}(t)\lesssim B_{p'}^{s+2,k-1}(t) + N_p^{s,k}(t).$$
The key observation is to recognize that the term represented by $L_3$ actually has a good sign near the event horizon. That is, according to Lemma \ref{dr_comm_2_lem},
\begin{align*}
\int_{\Sigma_t}-2X(\tg^k\phi^s)q^{-2}[q^2\Box_g,\tg^k]\phi^s &= \int_{\Sigma_t}-2(-X^r)\pd_r(\tg^k\phi^s)q^{-2}k\Delta'\pd_r(\tg^k\phi^s)+err \\
&=\int_{\Sigma_t}-2(-X^r)q^{-2}\Delta'(\pd_r\tg^k\phi^s)^2+err
\end{align*}
Since $X^r<0$ near the event horizon and $\Delta'>0$, the principal term becomes minus a square, so it can be ignored, or used to control small error terms.

We now investigate the error terms, which come from the remaining part of the vectorfield $X$ and functions $w$, $\tilde{w}$, $w_{(a)}$, and $\tilde{w}_{(a)}$ in $\mathcal{X}_{\phi^{s,k}}$ and $\mathcal{X}_{\psi^{s,k}}$, as well as the remainder in Lemma \ref{dr_comm_2_lem}.
\begin{multline*}
err = \int_{\Sigma_t\cap \{r_H\le r\le r_H+\delh\}}|\mathcal{X}_{\phi^{s,k}}-X^r\pd_r\phi^{s,k}||q^{-2}k\Delta'\pd_r(\tg^k\phi^s)| \\
+\int_{\Sigma_t\cap\{r_H\le r\le r_H+\delh\}}|\mathcal{X}_{\phi^{s,k}}|(|\pd_r\psi^{k-1,s+1}|+|\psi^{k-1,s+2}|+|q^{-2}\tg^{\le k-1}(q^2\Box_g\phi^s)|).
\end{multline*}
Since there is no product of principal factors (ie. factors of the form $\pd_r\tg^k\phi^s$) each term can be estimated in such a way that at least one factor is estimated by one of $B_{p'}^{s+2,k-1}(t)$ or $N^{s,k}(t)$ or $N_p^{s,k}(t)$, and at most one factor is estimated by $\epsilon (\pd_r\tg^k\phi^s)^2$ after separating the factors. This latter term can be absorbed into the term with the good sign. The same observation and procedure can be repeated for the $\psi$ integral as well.\qed
\end{proof}

Finally, we arrive at the following theorem, which is the most general form required by the proof of the main theorem.
\begin{theorem}\label{p_s_k_thm}Suppose $|a|/M$ is sufficiently small and fix $\delm,\delp>0$. Suppose furthermore that the pair $(\phi,\psi)$ satisfies the system
\begin{align*}
\Box_g\phi &= \mathcal{L}_\phi+\mathcal{N}_\phi, \\
\Box_{\tilde{g}}\psi &= \mathcal{L}_\psi + \mathcal{N}_\psi.
\end{align*}
Then the following estimates hold for $p\in [\delm,2-\delp]$, arbitrary $p'$, and integers $s\ge 0$ and $k\ge 1$.
$$E^{s,k}(t_2)\lesssim E^{s,k}(t_1)+\int_{t_1}^{t_2}B_{p'}^{s+2,k-1}(t)+\mathring{B}_1^{s+1}(t)+B_1^{s,k}(t)+N^{s,k}(t)dt,$$
$$E_p^{s,k}(t_2)+\int_{t_1}^{t_2}B_p^{s,k}(t)dt\lesssim E_p^{s,k}(t_1)+\int_{t_1}^{t_2}B_{p'}^{s+2,k-1}(t)+\mathring{B}_p^{s+1}(t)+N_p^{s,k}(t)dt,$$
where $E^{s,k}(t)$, $E_p^{s,k}(t)$, and $B_p^{s,k}(t)$ are as defined in Theorem \ref{p_s_k_L_thm}, and
$$N^{s,k}(t)=(E^{s,k}(t))^{1/2}\left(\sum_{\substack{s'\le s \\ k'\le k}}||q^{-2}\tg^{k'}\Gamma^{s'}(q^2\mathcal{N}_\phi)||_{L^2(\Sigma_t)}+\sum_{\substack{s'\le s \\ k'\le k}}||q^{-2}\tg^{k'}\tilde\Gamma^{s'}(q^2\mathcal{N}_\psi)||_{L^2(\tilde{\Sigma}_t)}\right),$$
\begin{align*}
N_p^{s,k}(t) =& \sum_{\substack{s'\le s \\ k'\le k}}\int_{\Sigma_t}r^{p+1}|q^{-2}\tg^{k'}\Gamma^{s'}(q^2\mathcal{N}_\phi)|^2 + \sum_{\substack{s'\le s \\ k'\le k}}\int_{\tilde\Sigma_t}r^{p+1}|q^{-2}\tg^{k'}\tilde\Gamma^{s'}(q^2\mathcal{N}_\psi)|^2.
\end{align*}
\end{theorem}

\begin{proof}
The proof is very similar to the proof of Theorem \ref{p_s_thm}. An outline of the proof is given here.

By Theorem \ref{p_s_k_L_thm} and Lemma \ref{p_s_k_L_a_lem}, we have
$$E^{s,k}(t_2)\lesssim E^{s,k}(t_1)+\int_{t_1}^{t_2}B_{p'}^{s+2,k-1}(t)+\mathring{B}_1^{s+1}(t)+B_1^{s,k}(t)+N^{s,k}(t)+(N')^{s,k}(t),$$
where $(N')^{s,k}(t)$ is the quantity $N^{s,k}(t)$ from Theorem \ref{p_s_k_L_thm}. By the same argument as in the proof of Theorem \ref{p_s_thm}, 
$$(N')^{s,k}(t)\lesssim N^{s,k}(t).$$
This proves the first estimate of the theorem.

By Theorem \ref{p_s_k_L_thm} and Lemma \ref{p_s_k_L_a_lem}, we have
$$E_p^{s,k}(t_2)+\int_{t_1}^{t_2}B_p^{s,k}(t)dt \lesssim E_p^{s,k}(t_1)+\int_{t_1}^{t_2}B_{p'}^{s+2,k-1}(t)+\frac{|a|}{M}B_p^{s,k}(t)+\mathring{B}_p^{s+1}(t)+N_p^{s,k}(t)+(N')_p^{s,k}(t)dt,$$
where $(N')_p^{s,k}(t)$ is the quantity $N_p^{s,k}(t)$ from Theorem \ref{p_s_k_L_thm}. By the same argument as in the proof of Theorem \ref{p_s_thm}, 
$$(N')_p^{s,k}(t)\lesssim \epsilon\mathring{B}_p^{s+1}(t)+\epsilon B_p^{s,k}(t)+\epsilon^{-1}N_p^{s,k}(t).$$
Proceeding as in the proof of Theorem \ref{p_s_thm}, we conclude that if $\epsilon$ and $\frac{|a|}{M}$ are sufficiently small, then the bulk term
$$(|a|/M+\epsilon)\int_{t_1}^{t_2}B_p^{s,k}(t)dt$$
can be absorbed into the left hand side. The result is the second estimate of the theorem. \qed
\end{proof}

\section{The Pointwise Estimates}\label{pointwise_sec}

In this section, we prove pointwise estimates for certain derivatives of $\phi$ and $\psi$ that will appear in the nonlinear quantities $q^{-2}\tg^k\Gamma^s(q^2\mathcal{N}_\phi)$ and $q^{-2}\tg^k\tilde{\Gamma}^s(q^2\mathcal{N}_\psi)$. In \S\ref{nonlinear_sec}, we will see that these quantities are not simply sums of products of single derivatives of $\phi^{s,k}$ and $\psi^{s,k}$. Instead, they will take a slightly more general form consisting of sums of products of single derivatives of $\fd^l\phi^{s-l,k}$ and $\fd^l\psi^{s-l,k}$, where the $\fd^l$ operators are defined in \S\ref{regularity_sec}. For this reason, the estimates established in this section will apply to $\fd^l\phi^{s-l,k}$ and $\fd^l\psi^{s-l,k}$ and their first derivatives. However, for simplicity the reader is welcome to think of these quantities as simply $\phi^{s,k}$ and $\psi^{s,k}$ or even more simply as $\phi$ and $\psi$.

We begin with Lemma \ref{infty_base_lem}, which estimates an arbitrary function in $L^\infty(\Sigma_t)$ (which is the same space as $L^\infty(\tilde{\Sigma}_t)$) by certain Sobolev norms in $\Sigma_t$ and $\tilde{\Sigma}_t$. This lemma is then repeatedly applied to single derivatives of $\fd^l\phi^{s-l,k}$ and $\fd^l\psi^{s-l,k}$, resulting in Sobolev norms that can be estimated by the energy norms. (See Lemmas \ref{low_order_infty_lem}-\ref{gh_infty_lem}.)  These estimates are all summarized at the end in Proposition \ref{infinity_prop}.

It is important to pay special attention to the $r$ weights in the lemmas that follow. In the main theorem (Theorem \ref{main_thm}), we will see that the energies $E_p^{s,k}(t)$ behave like $t^{p-2+\delp}$ for late times and that $E^{s,k}(t)$ will remaind bounded in time. In this section, we will see that multiplying a derivative of $\fd^l\phi^{s-l,k}$ or $\fd^l\psi^{s-l,k}$ by $r$ changes which energy norm can be used to estimate the Sobolev norm provided by Lemma \ref{infty_base_lem}. Some derivatives of $\fd^l\phi^{s-l,k}$ or $\fd^l\psi^{s-l,k}$ can have more $r$ factors than others. These derivatives will eventually be shown to decay better in time. See the statement of Theorem \ref{main_thm}.

\subsection{A Sobolev-type estimate}

First, we prove the following lemma, which is a Sobolev-type estimate. In particular, this lemma uses the fact that the commutators $Q$ and $\tilde{Q}$ grow at a rate of $r^2$. But the weight that is gained in the following estimate depends on the volume form for the associated space, so $L^\infty$ estimates on $\Sigma_t$ gain two factors of $r$ and $L^\infty$ estimates on $\tilde{\Sigma}_t$ gain six factors of $r$. Also, the fact that the volume form for $\tilde{\Sigma}_t$ has additional factors of $\sin\theta$ means that more derivatives are required in the Sobolev estimate.

\begin{lemma}\label{infty_base_lem}
Let $u$ be an arbitrary function decaying sufficiently fast as $r\rightarrow\infty$. Then for any $r_0\ge r_H$,
$$||\fd^lu||^2_{L^\infty(\Sigma_t\cap\{r>r_0\})}\lesssim \int_{\Sigma_t\cap\{r>r_0\}}r^{-2}\left[(\pd_r\Gamma^{\le l+3}u)^2+(\Gamma^{\le l+3}\phi)^2\right]$$
and
$$||\fd^lu||^2_{L^\infty(\tilde\Sigma_t\cap\{r>r_0\})}\lesssim \int_{\tilde{\Sigma}_t\cap\{r>r_0\}}r^{-6}\left[(\pd_r\tilde{\Gamma}^{\le l+5} u)^2+(\tilde{\Gamma}^{\le l+5} u)^2\right].$$
If $l$ is even, the same results hold with only $\Gamma^{\le l+2}$ and $\tilde{\Gamma}^{\le l+4}$ respectively.
\end{lemma}
\begin{proof}
For a fixed $r$, denote by $\bar{u}:S^2(1)\rightarrow\R{}$ the pullback of the function $u:S^2(r)\rightarrow\R{}$ via the canonical map from $S^2(1)$ to $S^2(r)$. It is straightforward to show that
$$\overline{\fd^l u}=\fd^l\bar{u}$$
and
$$\overline{(q^2\sla\triangle)^lu}=\sla\triangle^l \bar{u}.$$
Also, denote by $d\omega$ the measure on $S^2(1)$.

\textbf{If $l$ is even, then}
\begin{multline*}
||\fd^lu||^2_{L^\infty(S^2(r))}=||\overline{\fd^l u}||^2_{L^\infty(S^2(1))}=||\fd^l\bar{u}||^2_{L^\infty(S^2(1))} \lesssim ||\fd^{\le l+2}\bar{u}||_{L^2(S^2(1))} \\
\lesssim ||\sla\triangle^{\le (l+2)/2} \bar{u}||_{L^2(S^2(1))}^2 =\int_{S^2(r)}((q^2\sla\triangle)^{\le (l+2)/2}u)^2d\omega \lesssim \int_{S^2(r)}(\Gamma^{\le l+2}u)^2d\omega.
\end{multline*}
We have used Lemma \ref{spherical_infty_lem}, Theorem \ref{cl_embedding_thm}, and Lemma \ref{laplacian_le_commutators_lem} in the three $\lesssim$ steps.

\textbf{If instead $l$ is odd, then}
$$||\fd^lu||_{L^\infty(S^2(r))}\lesssim ... \lesssim ||\fd^{\le l+3}\bar{u}||_{L^2(S^2(1))} \lesssim ... \lesssim \int_{S^2(r)}\left(\Gamma^{\le l+3}u\right)^2d\omega.$$
That is, the calculation is the exact same as for the even case, except in the application of Lemma \ref{spherical_infty_lem}, which requires $\fd^{\le l+3}$ instead of $\fd^{\le l+2}$.

\textbf{Thus, in both cases},
$$||\fd^lu||_{L^\infty(S^2(r))}\lesssim \int_{S^2(r)}\left(\Gamma^{\le l+3} u\right)^2d\omega.$$
Now, set $f(r)=\int_{S^2(r)}(\Gamma^{\le l+3}u)^2d\omega$. Note that
\begin{multline*}
|f'(r)|\lesssim \int_{S^2(r)}|\Gamma^{\le l+3}u\pd_r\Gamma^{\le l+3}u|d\omega\lesssim \int_{S^2(r)}\left[(\pd_r \Gamma^{\le l+3}u)^2+(\Gamma^{\le l+3}u)^2\right]d\omega \\
\lesssim\int_{S^2(r)}r^{-2}\left[(\pd_r \Gamma^{\le l+3}u)^2+(\Gamma^{\le l+3}u)^2\right]q^2d\omega.
\end{multline*}
Then, assuming $\lim_{r\rightarrow\infty}f(r)=0$,
$$|\fd^l u(r_0)|^2\lesssim f(r_0)\lesssim \int_{r_0}^\infty |f'(r)|dr\lesssim \int_{\Sigma_t\cap\{r\ge r_0\}}r^{-2}\left[(\pd_r \Gamma^{\le l+3}u)^2+(\Gamma^{\le l+3}u)^2\right].$$
The same procedure can be used to prove the analogous estimate on $\tilde\Sigma$, but there are two differences. The first is that there will be a loss of two extra derivatives since the $L^\infty$ estimate is now applied on $S^6$ instead of $S^2$. The second is that there will be a factor of $r^{-6}$ instead of a factor of $r^{-2}$, since the volume form for $\tilde{\Sigma}_t$ is $q^2A^2=O(r^6)$. This explains the differences between the two estimates of the lemma.
\qed
\end{proof}

\subsection{Estimating derivatives using the Sobolev-type estimate}

Now we use Lemma \ref{infty_base_lem} to prove estimates for various quantities. For simplicity, let us momentarily set $s=k=l=0$ and focus on $\phi$ only. We will estimate the quantities $\phi$, $\pd_t\phi$, $(r^{-1}\pd_\theta)\phi$, $L\phi$, and $\tg\phi$ in Lemmas \ref{low_order_infty_lem}, \ref{pd_t_infty_lem}, \ref{pd_theta_infty_lem}, \ref{L_infty_lem}, and \ref{gh_infty_lem} respectively. The operators $\pd_t$, $r^{-1}\pd_\theta$, and $L$ form an approximately normalized basis except near the event horizon, since $L$ coincides with $\pd_t$ on the event horizon. This is why we also need the operator $\tg$. These estimates from Lemmas \ref{low_order_infty_lem}-\ref{gh_infty_lem} are summarized in Proposition \ref{infinity_prop}.

\subsubsection{Estimating $\fd^l\phi^{s-l,k}$ and $\fd^l\psi^{s-l,k}$}
The following lemma estimates $\phi$ and $\psi$, as well as the higher order analogues $\fd^l\phi^{s-l,k}$ and $\fd^l\psi^{s-l,k}$.
\begin{lemma}\label{low_order_infty_lem}
For $r\ge r_H$,
$$|r^p\fd^l\phi^{s-l,k}|^2+|r^{p+2}\fd^l\psi^{s-l,k}|^2\lesssim E_{2p}^{s+5,k}(t)$$
and for $r\ge r_0>r_H$,
$$|\fd^l\phi^{s-l,k}|^2+|r^2\fd^l\psi^{s-l,k}|^2\lesssim E^{s+5,k}(t).$$
\end{lemma}
\begin{proof}
First, we apply Lemma \ref{infty_base_lem} with $u=r^p\phi^{s-l,k}$.
\begin{align*}
|r^p\fd^l\phi^{s-l,k}|^2 &=|\fd^l(r^p\phi^{s-l,k})|^2 \\
&\lesssim \int_{\Sigma_t}r^{-2}\left[(\pd_r\Gamma^{\le l+3}(r^p\phi^{s-l,k}))^2+(\Gamma^{\le l+3}(r^p\phi^{s-l,k}))^2\right] \\
&\lesssim \int_{\Sigma_t}r^{2p-2}\left[(\pd_r\Gamma^{\le l+3}\phi^{s-l,k})^2+(\Gamma^{\le l+3}\phi^{s-l,k})^2\right] \\
&\lesssim E_{2p}^{s+3}(t).
\end{align*}
Then, we apply Lemma \ref{infty_base_lem} with $u=r^{p+2}\psi^{s-l,k}$.
\begin{align*}
|r^{p+2}\fd^l\psi^{s-l,k}|^2 &=|\fd^l(r^{p+2}\psi^{s-l,k})|^2 \\
&\lesssim \int_{\tilde\Sigma_t}r^{-6}\left[(\pd_r\Gamma^{\le l+5}(r^{p+2}\psi^{s-l,k}))^2+(\Gamma^{\le l+5}(r^{p+2}\psi^{s-l,k}))^2\right] \\
&\lesssim \int_{\tilde\Sigma_t}r^{2p-2}\left[(\pd_r\Gamma^{\le l+5}\psi^{s-l,k})^2+(\Gamma^{\le l+5}\psi^{s-l,k})^2\right] \\
&\lesssim E_{2p}^{s+5}(t).
\end{align*}
Together, these estimates prove the first estimate of the lemma. The second estimate follows from the same exact argument in the special case $p=0$, and the observation that as long as $r\ge r_0>r_H$, then $E^{s,k}(t)$ can be used in place of $E_0^{s,k}(t)$. \qed
\end{proof}

\subsubsection{Estimating $\pd_t\fd^l\phi^{s-l,k}$ and $\pd_t\fd^l\psi^{s-l,k}$}
The following lemma estimates $\pd_t\phi$ and $\pd_t\psi$ as well as the higher order analogues $\pd_t\fd^l\phi^{s-l,k}$ and $\pd_t\fd^l\psi^{s-l,k}$.
\begin{lemma}\label{pd_t_infty_lem}
For $r\ge r_H$,
$$|r^p\pd_t\fd^l\phi^{s-l,k}|^2+|r^{p+2}\pd_t\fd^l\psi^{s-l,k}|^2 \lesssim E_{2p}^{s+6,k}(t)$$
and for $r\ge r_0>r_H$,
$$|r\pd_t\fd^l\phi^{s-l,k}|^2+|r^3\pd_t\fd^l\psi^{s-l,k}|^2 \lesssim E^{s+6,k}(t).$$
\end{lemma}
\begin{proof}
The first estimate reduces to Lemma \ref{low_order_infty_lem} by observing that $\pd_t\fd^l\phi^{s-l,k}=\fd^l\pd_t\phi^{s-l,k}=\fd^l\phi^{s+1-l,k}$ and likewise $\pd_t\fd^l\psi^{s-l,k}=\fd^l\psi^{s+1-l,k}$. We now prove the second estimate.

First, we apply Lemma \ref{infty_base_lem} with $u=r\pd_t\phi^{s-l,k}$.
\begin{align*}
|r\pd_t\fd^l\phi^{s-l,k}|^2 &= |\fd^l(r\pd_t\phi^{s-l,k})|^2 \\
&\lesssim \int_{\Sigma_t\cap\{r>r_0\}}r^{-2}\left[(\pd_r\Gamma^{\le l+3}(r\pd_t\phi^{s-l,k}))^2+(\Gamma^{\le l+3}(r\pd_t\phi^{s-l,k}))^2\right] \\
&\lesssim \int_{\Sigma_t\cap\{r>r_0\}}\left[(\pd_r\Gamma^{\le l+3}\pd_t\phi^{s-l,k})^2+(\Gamma^{\le l+3}\pd_t\phi^{s-l,k})^2\right] \\
&\lesssim \int_{\Sigma_t\cap\{r>r_0\}}\left[(\pd_r\Gamma^{\le l+3}\phi^{s+1-l,k})^2+(\pd_t\Gamma^{\le l+3}\phi^{s-l,k})^2\right] \\
&\lesssim E^{s+4,k}(t).
\end{align*}
Next, by applying Lemma \ref{infty_base_lem} with $u=r^3\pd_t\psi^{s-l,k}$ and repeating the same procedure, we arrive at the following estimate.
$$|r^3\pd_t\fd^l\psi^{s-l,k}|^2\lesssim E^{s+6,k}(t).$$
Together, these estimates prove the second estimate of the lemma. \qed
\end{proof}

\subsubsection{Estimating $(r^{-1}\pd_\theta)\fd^l\phi^{s-l,k}$ and $(r^{-1}\pd_\theta)\fd^l\psi^{s-l,k}$}
The following lemma estimates $(r^{-1}\pd_\theta)\phi$ and $(r^{-1}\pd_\theta)\psi$ as well as the higher order analogues $(r^{-1}\pd_\theta)\fd^l\phi^{s-l,k}$ and $(r^{-1}\pd_\theta)\fd^l\psi^{s-l,k}$.
\begin{lemma}\label{pd_theta_infty_lem}
For $r\ge r_H$,
$$|r^{p+1}(r^{-1}\pd_\theta)\fd^l\phi^{s-l,k}|^2+|r^{p+3}(r^{-1}\pd_\theta)\fd^l\psi^{s-l,k}|^2\lesssim E_{2p}^{s+6,k}(t)$$
and for $r\ge r_0>r_H$,
$$|r(r^{-1}\pd_\theta)\fd^l\phi^{s-l,k}|^2+|r^3(r^{-1}\pd_\theta)\fd^l\psi^{s-l,k}|^2\lesssim E^{s+6,k}(t).$$
\end{lemma}
\begin{proof}
This lemma reduces to Lemma \ref{low_order_infty_lem} by observing that
$$r^{p+1}(r^{-1}\pd_\theta)\fd^l\phi^{s-l,k}=r^p\fd\fd^l\phi^{s-l,k}=r^p\fd^{l+1}\phi^{s-l,k}\subset r^p\fd^l\phi^{s+1-l,k},$$
and likewise
$$r^{p+3}(r^{-1}\pd_\theta)\fd^l\psi^{s-l,k}\subset r^{p+2}\fd^l\psi^{s+1-l,k}.$$
\qed
%The second estimate follows from the first by setting $p=0$ and observing that $E_0^{s,k}(t)\lesssim E^{s,k}(t)$. \qed
\iffalse
We now prove the second estimate.

First, we apply Lemma \ref{infty_base_lem} to estimate $\fd^{l+1}\phi^{s-l,k}$.
\begin{align*}
|r(r^{-1}\pd_\theta)\fd^l\phi^{s-l,k}|^2 &= |\fd^{l+1}\phi^{s-l,k}|^2 \\
&\lesssim \int_{\Sigma_t}r^{-2}\left[(\pd_r\Gamma^{\le l+4}\phi^{s-l,k})^2+(\Gamma^{\le l+4}\phi^{s-l,k})^2\right] \\
&\lesssim E^{s+4,k}(t),
\end{align*}
where the last step required the use of a Hardy estimate. Next, by applying Lemma \ref{infty_base_lem} again to estimate $r^2\fd^{l+1}\psi^{s-l,k}$ and repeating the same procedure, we arrive at the following estimate.
$$|r^3(r^{-1}\pd_\theta)\fd^l\psi^{s-l,k}|^2 \lesssim E^{s+6,k}(t).$$
Together, these estimates prove the second estimate of the lemma. \qed
\fi
\end{proof}

\subsubsection{Estimating $L\fd^l\phi^{s-l,k}$ and $L\fd^l\psi^{s-l,k}$}
The following lemma estimates $L\phi$ and $L\psi$ as well as the higher order analogues $L\fd^l\phi^{s-l,k}$ and $L\fd^l\psi^{s-l,k}$.
\begin{lemma}\label{L_infty_lem}
Letting $L=\alpha\pd_r+\pd_t$, where $\alpha=\frac{\Delta}{r^2+a^2}$, we have that for $r\ge r_H$,
$$|r^{p+1}L\fd^l\phi^{s-l,k}|^2+|r^{p+3}L\fd^l\psi^{s-l,k}|^2\lesssim E_{2p}^{s+7,k}(t)+\int_{\Sigma_t}r^{2p}(\Box_g\phi^{s+3,k})^2+\int_{\tilde\Sigma_t}r^{2p}(\Box_g\psi^{s+5,k})^2$$
and for $r\ge r_0>r_H$,
$$|rL\fd^l\phi^{s-l,k}|^2+|r^3L\fd^l\psi^{s-l,k}|^2\lesssim E^{s+7,k}(t)+\int_{\Sigma_t\cap\{r>r_0\}}(\Box_g\phi^{s+3,k})^2+\int_{\tilde\Sigma_t\cap\{r>r_0\}}(\Box_g\psi^{s+5,k})^2.$$
\end{lemma}
\begin{proof}
Before beginning the estimates stated by the lemma, it is important to establish
$$(\pd_rLu)^2\lesssim (\Box_gu)^2+(L\pd_tu)^2+r^{-2}(\pd_t^2u)^2+r^{-2}(\pd_r\pd_tu)^2+r^{-2}(\pd_ru)^2+r^{-2}(Qu)^2,$$
$$(\pd_rLu)^2\lesssim (\Box_{\tilde g}u)^2+(L\pd_tu)^2+r^{-2}(\pd_t^2u)^2+r^{-2}(\pd_r\pd_tu)^2+r^{-2}(\pd_ru)^2+r^{-2}(\tilde{Q}u)^2.$$
To verify these, we expand
$$q^2\Box_g=q^2g^{tt}\pd_t^2+q^2g^{rr}\pd_r^2+\pd_r(q^2g^{rr})\pd_r+q^2g^{rt}\pd_r\pd_t+\pd_r(q^2g^{rt})\pd_t+Q$$
and observe that for $r>r_H+\delh$, $g^{rt}=0$ and
$$q^2g^{tt}\pd_t^2u+q^2g^{rr}\pd_r^2u+\pd_r(q^2g^{rr})\pd_ru=-\frac{(r^2+a^2)^2}{\Delta}\pd_t^2u+\Delta\pd_r^2u+\Delta'\pd_ru.$$
We also expand
\begin{align*}
(r^2+a^2)\pd_rLu &= (r^2+a^2)\pd_r\left(\frac{\Delta}{r^2+a^2}\pd_ru+\pd_tu\right) \\
&= \Delta\pd_r^2u+\Delta'\pd_ru-\frac{2r\Delta}{r^2+a^2}\pd_ru+(r^2+a^2)\pd_r\pd_t
\end{align*}
Note that both of these expressions share the terms $\Delta\pd_r^2u$ and $\Delta'\pd_ru$. It follows that for $r\ge r_H+\delh$,
\begin{multline*}
q^2g^{tt}\pd_t^2u+q^2g^{rr}\pd_r^2u+\pd_r(q^2g^{rr})\pd_ru - (r^2+a^2)\pd_rLu \\
=-\frac{(r^2+a^2)^2}{\Delta}\pd_t^2u-(r^2+a^2)\pd_r\pd_tu + \frac{2r\Delta}{r^2+a^2}\pd_ru \\
=-\frac{(r^2+a^2)^2}{\Delta}\left(\pd_t^2u+\frac{\Delta}{r^2+a^2}\pd_r\pd_tu\right)+\frac{2r\Delta}{r^2+a^2}\pd_ru \\
=-\frac{(r^2+a^2)^2}{\Delta}L\pd_tu+\frac{2r\Delta}{r^2+a^2}\pd_ru.
\end{multline*}
Keeping in mind the additional terms that show up for $r\le r_H+ \delh$, we arrive at the first estimate for $(\pd_rLu)^2$. The argument for the second estimate is basically the same. With these two estimates in mind, we begin to prove the estimates stated by the lemma.

We now apply Lemma \ref{infty_base_lem} with $u=r^{p+1}L\phi^{s-l,k}$.
\begin{align*}
|r^{p+1}L\fd^l\phi^{s-l,k}|^2 &= |\fd^l(r^{p+1}L\phi^{s-l,k})|^2 \\
&\lesssim \int_{\Sigma_t}r^{-2}\left[(\pd_r\Gamma^{\le l+3}(r^{p+1}L\phi^{s-l,k}))^2+(\Gamma^{\le l+3}(r^{p+1}L\phi^{s-l,k}))^2\right] \\
&\lesssim \int_{\Sigma_t}r^{2p}\left[(\pd_r\Gamma^{\le l+3}L\phi^{s-l,k})^2+(\Gamma^{\le l+3}L\phi^{s-l,k})^2\right] \\
&\lesssim \int_{\Sigma_t}r^{2p}\left[(\pd_rL\phi^{s+3,k})^2+(L\phi^{s+3,k})^2\right] \\
&\lesssim E_{2p}^{s+3,k}(t) +\int_{\Sigma_t}r^{2p}(\pd_rL\phi^{s+3,k})^2.
\end{align*}
Now, according to the estimate we previously established,
\begin{align*}
\int_{\Sigma_t}r^{2p}&(\pd_rL\phi^{s+3,k})^2 \\
&\lesssim \int_{\Sigma_t}r^{2p}\left[(\Box_g\phi^{s+3,k})^2+(L\pd_t\phi^{s+3,k})^2+r^{-2}(\pd_t^2\phi^{s+3,k})^2+r^{-2}(\pd_r\pd_t\phi^{s+3,k})^2\right.\\
&\hspace{3in}\left.+r^{-2}(\pd_r\phi^{s+3,k})^2+r^{-2}(Q\phi^{s+3,k})^2\right] \\
&\lesssim \int_{\Sigma_t}r^{2p}\left[(\Box_g\phi^{s+3,k})^2+(L\phi^{s+4,k})^2+r^{-2}(\phi^{s+5,k})^2+r^{-2}(\pd_r\phi^{s+4,k})^2\right].
\end{align*}
It follows that
$$|r^{p+1}L\fd^l\phi^{s-l,k}|^2\lesssim E_{2p}^{s+5,k}(t)+\int_{\Sigma_t}r^{2p}(\Box_g\phi^{s+3,k})^2.$$
By a similar argument,
$$|r^{p+3}L\fd^l\psi^{s-l,k}|^2\lesssim E_{2p}^{s+7,k}(t)+\int_{\tilde{\Sigma}_t}r^{2p}(\Box_{\tilde{g}}\psi^{s+5,k})^2.$$
These two estimates together establish the first estimate of the lemma. The second estimate follows from the same exact argument in the special case $p=0$, and the observation that as long as $r\ge r_0>r_H$, then $E^{s,k}(t)$ can be used in place of $E_0^{s,k}(t)$. \qed
\end{proof}

\subsubsection{Estimating $\tg\fd^l\phi^{s-l,k}$ and $\tg\fd^l\psi^{s-l,k}$}
The following lemma estimates $\tg\phi$ and $\tg\psi$ as well as the higher order analogues $\tg\fd^l\phi^{s-l,k}$ and $\tg\fd^l\psi^{s-l,k}$.
\begin{lemma}\label{gh_infty_lem}
Keeping in mind that $\tg$ is supported in a neighborhood of the event horizion, for arbitrary $p'$, we have for $r\ge r_H$,
$$|\tg \fd^l\phi^{s-l,k}|^2+|\tg \fd^l\psi^{s-l,k}|^2 \lesssim E_{p'}^{s+5,k+1}(t)$$
and for $r\ge r_0>r_H$,
$$|\tg \fd^l\phi^{s-l,k}|^2+|\tg \fd^l\psi^{s-l,k}|^2 \lesssim E^{s+5,k+1}(t).$$
\end{lemma}
\begin{proof}
We apply Lemma \ref{infty_base_lem} with $u=\tg\phi^{s-l,k}$, and freely introduce a factor of $r^{p'}$ since $\tg$ is supported on a compact interval in $r$.
\begin{align*}
|\tg\fd^l\phi^{s-l,k}|^2 &= |\fd^l\phi^{s-l,k+1}|^2 \\
&\lesssim \int_{\Sigma_t}r^{-2}\left[(\pd_r\Gamma^{\le l+3}\phi^{s-l,k+1})^2+(\Gamma^{\le l+3}\phi^{s-l,k+1})^2\right] \\
&\lesssim \int_{\Sigma_t}r^{p'-2}\left[(\pd_r\Gamma^{\le l+3}\phi^{s-l,k+1})^2+(\Gamma^{\le l+3}\phi^{s-l,k+1})^2\right] \\
&\lesssim E_{p'}^{s+3,k+1}(t).
\end{align*}
A similar argument shows that
$$|\tg \fd^l\psi^{s-l,k}|^2\lesssim E_{p'}^{s+5,k+1}(t).$$
Together, these estimates prove the first estimate of the lemma. The second estimate follows from the same argument, and the observation that as long as $r\ge r_0>r_H$, then $E^{s,k}(t)$ can be used in place of $E_{p'}^{s,k}(t)$. \qed
\end{proof}

\subsection{Summarizing the pointwise estimates}
To conclude this section, we summarize the previous lemmas in a single proposition.
\begin{definition} We define two families of operators.
$$\bar{D}=\{L,r^{-1}\pd_\theta\},$$
$$D=\{L,\pd_r,r^{-1}\pd_\theta\}.$$
\end{definition}

\begin{proposition}\label{infinity_prop}
For $r\ge r_H$,
\begin{multline*}
|r^{p+1}\bar{D}\fd^l\phi^{s-l,k}|^2+|r^pD\fd^l\phi^{s-l,k}|^2+|r^p\fd^l\phi^{s-l,k}|^2 \\
+|r^{p+3}\bar{D}\fd^l\psi^{s-l,k}|^2+|r^{p+2}D\fd^l\psi^{s-l,k}|^2+|r^{p+2}\fd^l\psi^{s-l,k}|^2 \\
\lesssim E_{2p}^{s+5,k+1}(t)+E_{2p}^{s+7,k}(t)+\int_{\Sigma_t}r^{2p}(\Box_g\phi^{s+5,k})^2+\int_{\tilde\Sigma_t}r^{2p}(\Box_{\tilde{g}}\psi^{s+5,k})^2
\end{multline*}
and for $r\ge r_0>r_H$,
\begin{multline*}
|rD\fd^l\phi^{s-l,k}|^2+|r^3D\fd^l\psi^{s-l,k}|^2 \\
\lesssim E^{s+5,k+1}(t)+E^{s+7,k}(t)+\int_{\Sigma_t\cap\{r>r_0\}}(\Box_g\phi^{s+5,k})^2+\int_{\tilde\Sigma_t\cap\{r>r_0\}}(\Box_{\tilde{g}}\psi^{s+5,k})^2.
\end{multline*}
\end{proposition}
\begin{proof}
With the exception of the operator $\pd_r$, all of the cases have been proved in Lemmas \ref{low_order_infty_lem}-\ref{gh_infty_lem}. Finally, observe that
$$|r^p\pd_r\fd^l\phi^{s-l,k}|^2\lesssim |r^pL\fd^l\phi^{s-l,k}|^2+|r^p\pd_t\fd^l\phi^{s-l,k}|^2+|\tg\fd^l\phi^{s-l,k}|^2$$
and
$$|r^p\pd_r\fd^l\phi^{s-l,k}|^2\lesssim |r^pL\fd^l\phi^{s-l,k}|^2+|r^p\pd_t\fd^l\phi^{s-l,k}|^2+|\tg\fd^l\phi^{s-l,k}|^2.$$
Thus, even the case of the operator $\pd_r$ can be reduced to Lemmas \ref{low_order_infty_lem}-\ref{gh_infty_lem}. \qed
\end{proof}

\begin{remark}
In the above proof, we see the reason for the need for the commutator $\tg$. Note that Lemmas \ref{low_order_infty_lem}-\ref{L_infty_lem} do not have an increase in $k$ on the right hand side, but Lemma \ref{gh_infty_lem} does. Lemma \ref{gh_infty_lem} was needed in order to estimate the $\pd_r$ derivative near the event horizon, because $L$ coincides with $\pd_t$ on the event horizon. Excluding this issue, there would be no need to commute with $\tg$ and introduce the $k$ index.
\end{remark}

\section{The Structures of $\mathcal{N}_\phi$ and $\mathcal{N}_\psi$}\label{nonlinear_sec}

% TODO
% the last part
% check that the introduction is accurate after working out the last part of the section
% 
% A note of where all the references to SC are.
%   references in Lemma 5.1
%   a reference in Lemma 6.2
%   a reference in Lemma 6.6
%   a reference in Lemma 6.7

In this section, we carefully examine the nonlinear terms $\mathcal{N}_\phi$ and $\mathcal{N}_\psi$ as well as their higher order analogues $q^{-2}\tg^k\Gamma^s(q^2\mathcal{N}_\phi)$ and $q^{-2}\tg^k\tilde\Gamma^s(q^2\mathcal{N}_\psi)$ and determine a procedure for estimating them in the proof of the main theorem. In particular, we note two important structural conditions that these terms satisfy. The first is the well-known \textit{null condition} and the second is a condition on the axis that allows us to use the formalism provided in \S\ref{regularity_sec}.

We start in \S\ref{example_terms_sec} by looking at a few example terms to illustrate the procedure that will be used. In particular, we will see how to estimate terms that contain products of both $\phi$ and $\psi$ as well as the role of the null condition. Then in \S\ref{nl_structure_sec}, we define and prove the precise structures of $\mathcal{N}_\phi$ and $\mathcal{N}_\psi$. These structures will then be used in \S\ref{nl_s_k_structure_sec} to define and prove the precise structures of $q^{-2}\tg^k\Gamma^s(q^2\mathcal{N}_\phi)$ and $q^{-2}\tg^k\tilde\Gamma^s(q^2\mathcal{N}_\psi)$. Finally, in \S\ref{nl_strategy_revisited_sec} we prove a proposition that uses these structures to estimate nonlinear quantities that will show up in the proof of the main theorem (Theorem \ref{main_thm}).

In order to proceed, we first calculate the nonlinear terms $\mathcal{N}_\phi$ and $\mathcal{N}_\psi$.
\begin{proposition}\label{N_identities_prop}
If one makes the substitutions
\begin{align*}
X&= A+A\phi \\
Y&=B+A^2\psi
\end{align*}
and requires that $\phi$ and $\psi$ are axisymmetric functions, then the wave map system (\ref{wm_X_eqn}-\ref{wm_Y_eqn}) reduces to the following system of equations for $\phi$ and $\psi$.
$$\Box_g \phi = \mathcal{L}_\phi + \mathcal{N}_\phi$$
$$\Box_{\tilde{g}} \psi = \mathcal{L}_\psi + \mathcal{N}_\psi,$$
where
$$\mathcal{L}_\phi=-2\frac{\pd^\alpha B}{A}A\pd_\alpha \psi + 2\frac{\pd^\alpha B\pd_\alpha B}{A^2}\phi-4\frac{\pd^\alpha A\pd_\alpha B}{A^2} A\psi$$
$$\mathcal{L}_\psi=-2\frac{\pd^\alpha A_2}{A_2}\pd_\alpha\psi+2\frac{\pd^\alpha B\pd_\alpha B}{A^2}\psi + 2A^{-1}\frac{\pd^\alpha B}{A}\pd_\alpha\phi$$
are the linear terms that first appeared in \S\ref{phi_psi_sec}, and
\begin{align*}
(1+\phi)\mathcal{N}_\phi &= \pd^\alpha \phi \pd_\alpha \phi - A\pd^\alpha \psi A\pd_\alpha \psi +2\frac{\pd^\alpha B}{A} \phi A\pd_\alpha \psi -4\frac{\pd^\alpha A}{A}A\psi A\pd_\alpha \psi \\
&\hspace{.5in}-\frac{\pd^\alpha B\pd_\alpha B}{A^2}\phi^2+4\frac{\pd^\alpha A\pd_\alpha B}{A^2}\phi A\psi -4\frac{\pd^\alpha A\pd_\alpha A}{A^2}(A\psi)^2
\end{align*}
$$(1+\phi)\mathcal{N}_\psi = 2\pd^\alpha \phi \pd_\alpha\psi  +4\frac{\pd^\alpha A}{A}\psi\pd_\alpha \phi -2\frac{\pd^\alpha B}{A}A^{-1}\phi\pd_\alpha\phi$$
are the nonlinear terms.
\end{proposition}
\begin{proof}
The first equation of the wave map system (\ref{wm_X_eqn}) is
$$\Box_gX = \frac{\pd^\alpha X\pd_\alpha X}{X}-\frac{\pd^\alpha Y\pd_\alpha Y}{X}.$$
We substitute $X=A(1+\phi)$ and $Y=B+A^2\psi$.
$$\Box_g(A(1+\phi))=\frac{\pd^\alpha (A(1+\phi))\pd_\alpha (A(1+\phi))}{A(1+\phi)}-\frac{\pd^\alpha (B+A^2\psi)\pd_\alpha (B+A^2\psi)}{A(1+\phi)}.$$
Now, we expand each term as follows.
$$\Box_g(A(1+\phi)) = (1+\phi)\Box_gA +2\pd^\alpha A\pd_\alpha \phi+ A\Box_g\phi$$
$$\frac{\pd^\alpha (A(1+\phi))\pd_\alpha (A(1+\phi))}{A(1+\phi)} = (1+\phi)\frac{\pd^\alpha A\pd_\alpha A}{A} + 2\pd^\alpha A\pd_\alpha\phi+A\frac{\pd^\alpha \phi\pd_\alpha\phi}{(1+\phi)}$$
\begin{multline*}
-\frac{\pd^\alpha (B+A^2\psi)\pd_\alpha (B+A^2\psi)}{A(1+\phi)} = -(1+\phi)\frac{\pd^\alpha B\pd_\alpha B}{A}-\left(\frac{1}{1+\phi}-(1+\phi)\right)\frac{\pd^\alpha B\pd_\alpha B}{A} \\
-2\frac{\pd^\alpha B\pd_\alpha (A^2\psi)}{A(1+\phi)}-\frac{\pd^\alpha (A^2\psi)\pd_\alpha(A^2\psi)}{A(1+\phi)}
\end{multline*}
Using the fact that $(X,Y)=(A,B)$ also solves equation (\ref{wm_X_eqn}), we determine that
$$A\Box_g\phi = A\frac{\pd^\alpha \phi\pd_\alpha\phi}{(1+\phi)} -\left(\frac{1}{1+\phi}-(1+\phi)\right)\frac{\pd^\alpha B\pd_\alpha B}{A} -2\frac{\pd^\alpha B\pd_\alpha (A^2\psi)}{A(1+\phi)}-\frac{\pd^\alpha (A^2\psi)\pd_\alpha(A^2\psi)}{A(1+\phi)}.$$
Note the following identities.
$$-\frac{1}{1+\phi}+1+\phi = \frac{-1+(1+\phi)^2}{1+\phi}=\frac{2\phi+\phi^2}{1+\phi}=\frac{\phi(2(1+\phi)-\phi)}{1+\phi}=2\phi-\frac{\phi^2}{1+\phi}$$
$$\frac{1}{1+\phi} = 1+\left(\frac1{1+\phi}-1\right) = 1+\frac{1-(1+\phi)}{1+\phi} = 1-\frac{\phi}{1+\phi}.$$
We conclude that
\begin{align*}
\mathcal{L}_\phi &=\frac1A\left[ 2\phi\frac{\pd^\alpha B\pd_\alpha B}{A}-2\frac{\pd^\alpha B\pd_\alpha (A^2\psi)}{A}\right] \\
&= 2\frac{\pd^\alpha B\pd_\alpha B}{A^2}\phi -4\frac{\pd^\alpha A\pd_\alpha B}{A^2}A\psi  -2\frac{\pd^\alpha B}{A}A\pd_\alpha \psi
\end{align*}
and
\begin{align*}
(1+\phi)\mathcal{N}_\phi &= \frac{1+\phi}{A}\left[ A\frac{\pd^\alpha\phi\pd_\alpha\phi}{1+\phi}-\frac{\phi^2}{1+\phi}\frac{\pd^\alpha B\pd_\alpha B}{A}+2\frac{\phi}{1+\phi}\frac{\pd^\alpha B\pd_\alpha (A^2\psi)}{A}-\frac{\pd^\alpha(A^2\psi)\pd_\alpha (A^2\psi)}{A(1+\phi)} \right]  \\
&= \pd^\alpha\phi\pd_\alpha\phi -\frac{\pd^\alpha B\pd_\alpha B}{A^2}\phi^2 +4\frac{\pd^\alpha A\pd_\alpha B}{A^2}\phi A\psi+2\frac{\pd^\alpha B}{A}\phi A\pd_\alpha\psi \\
&\hspace{2in}-4\frac{\pd^\alpha A\pd_\alpha A}{A^2}(A\psi)^2-4\frac{\pd^\alpha A}{A}A\psi A\pd_\alpha\psi-A\pd^\alpha \psi A\pd_\alpha \psi.
\end{align*}

The second equation of the wave map system (\ref{wm_Y_eqn}) is
$$\Box_gY = 2\frac{\pd^\alpha X\pd_\alpha Y}{X}.$$
We substitute $X=A(1+\phi)$ and $Y=B+A^2\psi$.
$$\Box_g(B+A^2\psi) =2\frac{\pd^\alpha (A(1+\phi))\pd_\alpha (B+A^2\psi)}{A(1+\phi)}.$$
The left hand side simplifies as follows (again using the equation for $\Box_gA$).
\begin{align*}
\Box_g(B+A^2\psi) &= \Box_g B+2A\psi\Box_gA+2\pd^\alpha A\pd_\alpha A\psi +4A\pd^\alpha A\pd_\alpha \psi+A^2\Box_g\psi \\
&= \Box_g B-2\pd^\alpha B\pd_\alpha B\psi+4\pd^\alpha A\pd_\alpha A\psi +4A\pd^\alpha A\pd_\alpha \psi+A^2\Box_g\psi
\end{align*}
The right hand side simplifies as follows.
\begin{multline*}
2\frac{\pd^\alpha (A(1+\phi))\pd_\alpha (B+A^2\psi)}{A(1+\phi)} \\
= 2\frac{\pd^\alpha A\pd_\alpha B}{A}+2\frac{\pd^\alpha\phi\pd_\alpha B}{(1+\phi)}+2\frac{\pd^\alpha A\pd_\alpha(A^2\psi)}{A}+2\frac{\pd^\alpha\phi\pd_\alpha(A^2\psi)}{1+\phi} \\
= 2\frac{\pd^\alpha A\pd_\alpha B}{A}+2\frac{\pd^\alpha\phi\pd_\alpha B}{(1+\phi)}+4\pd^\alpha A\pd_\alpha A\psi+2A\pd^\alpha A \pd_\alpha \psi+2\frac{\pd^\alpha\phi\pd_\alpha(A^2\psi)}{1+\phi} 
\end{multline*}
Note that the term $4\pd^\alpha A\pd_\alpha A\psi$ appears on both sides and therefore cancels out. This cancellation is the motivation for the choice of linearization $Y=B+A^2\psi$. Using the fact that $(X,Y)=(A,B)$ also solves equation (\ref{wm_Y_eqn}), we determine that
$$A^2\Box_g\psi +2A\pd^\alpha A\pd_\alpha \psi = 2\pd^\alpha B\pd_\alpha B \psi +2\frac{\pd^\alpha\phi\pd_\alpha B}{(1+\phi)}+2\frac{\pd^\alpha\phi\pd_\alpha(A^2\psi)}{1+\phi}.$$
Dividing by $A^2$ and using the fact that
$$\Box_g+2\frac{\pd^\alpha A}{A}\pd_\alpha = \Box_g+2\frac{\pd^\alpha A_1}{A_1}\pd_\alpha+2\frac{\pd^\alpha A_2}{A_2}\pd_\alpha = \Box_{\tilde{g}}+2\frac{\pd^\alpha A_2}{A_2}\pd_\alpha,$$
we obtain
$$\Box_{\tilde{g}}\psi+2\frac{\pd^\alpha A_2}{A_2}\pd_\alpha \psi = 2\frac{\pd^\alpha B\pd_\alpha B}{A^2}\psi +\frac{2}{1+\phi}A^{-1}\frac{\pd^\alpha B}{A}\pd_\alpha\phi +\frac{4}{1+\phi} \frac{\pd^\alpha A}{A}\psi \pd_\alpha\phi +\frac{2}{1+\phi}\pd^\alpha\phi \pd_\alpha\psi.$$
Again, since
$$\frac{1}{1+\phi} = 1-\frac{\phi}{1+\phi},$$
we conclude that
$$\mathcal{L}_\psi = -\frac{\pd^\alpha A_2}{A_2}\pd_\alpha \psi +2A^{-1}\frac{\pd^\alpha B}{A}\pd_\alpha\phi+2\frac{\pd^\alpha B\pd_\alpha B}{A^2}\psi$$
and
$$(1+\phi)\mathcal{N}_\psi = -2\frac{\pd^\alpha B}{A}A^{-1}\phi\pd_\alpha\phi +4\frac{\pd^\alpha A}{A}\psi\pd_\alpha \phi +2\pd^\alpha\phi\pd_\alpha\psi.$$
\qed
\end{proof}

\subsection{The strategy illustrated by two example terms ($a=0$)}\label{example_terms_sec}

There are many terms in the nonlinear parts of the equations. Let us illustrate how to handle them by looking at two particular examples in the simpler case $a=0$.

First, we examine the term $r^4\sin^4\theta L\psi\lbar\psi$, which arises from the term $A\pd^\alpha\psi A\pd_\alpha\psi$, which shows up in $\mathcal{N}_\phi$.
\begin{lemma}\label{example_Nphi_term_lem} (example $\mathcal{N}_\phi$ term)
Assume the following estimates.
$$||r^{(p-1)/2+3}L\fd^l\psi^{s-l}||_{L^\infty(t)}^2\lesssim B_p^{s+7}(t),$$
$$||r^3\lbar\fd^l\psi^{s-l}||_{L^\infty(t)}^2\lesssim E^{s+7}(t).$$
Then
$$\int_{\Sigma_t\cap\{r>R\}}r^{p+1}\left(\Gamma^s\left(r^4\sin^4\theta L\psi\lbar\psi\right)\right)^2\lesssim B_p^s(t)E^{s/2+7}(t)+E^s(t)B_p^{s/2+7}(t).$$
\end{lemma}
\begin{proof}
First, we compute
\begin{align*}
\int_{\Sigma\cap\{r>R\}}r^{p+1}r^8\sin^8\theta (L\psi^s)^2(\lbar\psi^{s/2})^2
&\lesssim \int_{\Sigma_t\cap\{r>R\}}r^{p-1}(L\psi^s)^2r^4\sin^4\theta (r^3\lbar\psi^{s/2})^2 \\
&\lesssim \int_{\Sigma_t\cap\{r>R\}}r^{p-1}(L\psi^s)^2r^4\sin^4\theta ||r^3\lbar\psi^{s/2}||_{L^\infty(t)}^2 \\
&\lesssim \int_{\tilde{\Sigma}_t\cap\{r>R\}}r^{p-1}(L\psi^s)^2||r^3\lbar\psi^{s/2}||_{L^\infty(t)}^2 \\
&\lesssim B_p^s(t)E^{s/2+7}(t).
\end{align*}
One important step in this calculation is the gain of $r^4$ by passing to the volume form on $\tilde{\Sigma}_t$. However, this same step also loses a factor of $\sin^4\theta$.

Likewise,
\begin{align*}
\int_{\Sigma_t\cap\{r>R\}}r^{p+1}r^8\sin^8\theta (\lbar\psi^s)^2(L\psi^{s/2})^2
&\lesssim \int_{\Sigma_t\cap\{r>R\}}(\lbar\psi^s)^2r^4\sin^4\theta (r^{(p-1)/2+3}L\psi^{s/2})^2 \\
&\lesssim \int_{\Sigma_t\cap\{r>R\}}(\lbar\psi^s)^2r^4\sin^4\theta ||r^{(p-1)/2+3}L\psi^{s/2}||_{L^\infty(t)}^2 \\
&\lesssim \int_{\tilde{\Sigma}_t\cap\{r>R\}}(\lbar\psi^s)^2 ||r^{(p-1)/2+3}L\psi^{s/2}||_{L^\infty(t)}^2 \\
&\lesssim E^s(t)B_p^{s/2+7}(t).
\end{align*}

So far, we have assumed that each of the $\Gamma$ operators in the expression $\Gamma^{\le s}(r^4\sin^4\theta L\psi\lbar\psi)$ acted on either $L\psi$ or $\lbar\psi$. But more generally,
$$\Gamma^s(r^4\sin^4\theta L\psi\lbar\psi) \approx \sum_{i+j+k+s_1+s_2\le s} r^4\fd^i(\sin^4\theta)L\fd^j\psi^{s_1}\lbar\fd^k\psi^{s_2}.$$
At first glance, a few of these terms may be concerning, because a factor of $\sin^2\theta$ was required for the conversion from an integral over $\Sigma_t$ to an integral over $\tilde{\Sigma}_t$. Note, however, that $\max(j,k)\le s-i$ so that Lemma \ref{gain_sin_lem} can be applied to the high order term. This is illustrated in the example below.
\begin{align*}
\int_{\Sigma_t\cap\{r>R\}}&r^{p+1}r^8(L\fd^l\psi^{s-2-l})^2(\lbar\fd^{l'}\psi^{s/2-{l'}})^2 \\
&\lesssim E^{s/2+7}(t)\int_{\Sigma_t\cap\{r>R\}}r^{p-1}(L\fd^l\psi^{s-2-l})^2r^4 \\
&\lesssim E^{s/2+7}(t)\int_{\Sigma_t\cap\{r>R\}}r^{p-1}(L\fd^{l+2}\psi^{s-2-l})^2r^4\sin^4\theta \\
&\lesssim B_p^s(t)E^{s/2+7}(t).
\end{align*}
\qed
\end{proof}

Now, we examine the terms $L\phi\lbar\psi$ and $\lbar\phi L\psi$, which arise from the term $\pd^\alpha \phi\pd_\alpha \psi$, which shows up in $\mathcal{N}_\psi$.
\begin{lemma} (example $\mathcal{N}_\psi$ term)
Assume the following estimates.
$$||r^{(p-1)/2+1}L\fd^l\phi^{s-l}||^2_{L^\infty(t)}+||r^{(p-1)/2+3}L\fd^l\psi^{s-l}||^2_{L^\infty(t)}\lesssim B_p^{s+7}(t),$$
$$||r\lbar\fd^l\phi^{s-l}||^2_{L^\infty(t)}+||r^{3}\lbar\fd^l\psi^{s-l}||^2_{L^\infty(t)}\lesssim E^{s+7}(t).$$
Then
$$\int_{\tilde{\Sigma}_t\cap\{r>R\}}r^{p+1}\left(\Gamma^s(L\phi\lbar\psi+\lbar\phi L\psi)\right)^2\lesssim B_p^s(t)E^{s/2+7}(t)+E^s(t)B_p^{s/2+7}(t).$$
\end{lemma}
\begin{proof}
We examine four separate cases, each one depending on whether the $\phi$ factor or the $\psi$ factor has the highest number of derivatives and also whether $L$ acts on $\phi$ or $\psi$.
\begin{align*}
\int_{\tilde{\Sigma}_t\cap\{r>R\}}r^{p+1}(L\fd^l\psi^{s-l})^2(\lbar\fd^l\phi^{s/2-l})^2
&\lesssim \int_{\tilde{\Sigma}_t\cap\{r>R\}}r^{p-1}(L\fd^l\psi^{s-l})^2(r\lbar\fd^l\phi^{s/2-l})^2 \\
&\lesssim \int_{\tilde{\Sigma}_t\cap\{r>R\}}r^{p-1}(L\fd^l\psi^{s-l})^2||r\lbar\fd^l\phi^{s/2-l}||_{L^\infty(t)}^2 \\
&\lesssim B_p^s(t)E^{s/2+7}(t).
\end{align*}
\begin{align*}
\int_{\tilde{\Sigma}_t\cap\{r>R\}}r^{p+1}(L\fd^l\phi^{s-l})^2(\lbar\fd^l\psi^{s/2-l})^2
&\lesssim \int_{\tilde{\Sigma}_t\cap\{r>R\}}r^{p-1}(L\fd^l\phi^{s-l})^2r^{-4}(r^3\lbar\fd^l\psi^{s/2-l})^2 \\
&\lesssim \int_{\tilde{\Sigma}_t\cap\{r>R\}}r^{p-1}(L\fd^l\phi^{s-l})^2r^{-4}||r^3\lbar\fd^l\psi^{s/2-l}||_{L^\infty(t)}^2 \\
&\lesssim \int_{\Sigma_t\cap\{r>R\}}r^{p-1}(L\fd^l\phi^{s-l})^2||r^3\lbar\fd^l\psi^{s/2-l}||_{L^\infty(t)}^2 \\
&\lesssim B_p^s(t)E^{s/2+7}(t).
\end{align*}

\begin{align*}
\int_{\tilde{\Sigma}_t\cap\{r>R\}}r^{p+1}(\lbar\fd^l\psi^{s-l})^2(L\fd^l\phi^{s/2-l})^2
&\lesssim \int_{\tilde{\Sigma}_t\cap\{r>R\}}(\lbar\fd^l\psi^{s-l})^2(r^{(p-1)/2+1}L\fd^l\phi^{s/2-l})^2 \\
&\lesssim \int_{\tilde{\Sigma}_t\cap\{r>R\}}(\lbar\fd^l\psi^{s-l})^2||r^{(p-1)/2+1}L\fd^l\phi^{s/2-l}||_{L^\infty(t)}^2 \\
&\lesssim E^s(t)B_p^{s/2+7}(t).
\end{align*}
\begin{align*}
\int_{\tilde{\Sigma}_t\cap\{r>R\}}r^{p+1}(\lbar\fd^l\phi^{s-l})^2(L\fd^l\psi^{s/2-l})^2
&\lesssim \int_{\tilde{\Sigma}_t\cap\{r>R\}}(\lbar\fd^l\phi^{s-l})^2r^{-4}(r^{(p-1)/2+3}L\fd^l\psi^{s/2-l})^2 \\
&\lesssim \int_{\tilde{\Sigma}_t\cap\{r>R\}}(\lbar\fd^l\phi^{s-l})^2r^{-4}||r^{(p-1)/2+3}L\fd^l\psi^{s/2-l}||_{L^\infty(t)}^2 \\
&\lesssim \int_{\Sigma_t\cap\{r>R\}}(\lbar\fd^l\phi^{s-l})^2||r^{(p-1)/2+3}L\fd^l\psi^{s/2-l}||_{L^\infty(t)}^2 \\
&\lesssim E^s(t)B_p^{s/2+7}(t).
\end{align*}
These four cases illustrate the entirety of the proof. \qed
\end{proof}

\subsection{The structures of $\mathcal{N}_\phi$ and $\mathcal{N}_\psi$}\label{nl_structure_sec}

We now categorize each term in $\mathcal{N}_\phi$ and $\mathcal{N}_\psi$ so the previous examples can be generalized systematically. We begin with the definition of the null condition.

\begin{definition} We define two families of terms.
$$\alpha = \{L\phi,r^{-1}\pd_r\phi,r^{-1}\pd_\theta\phi,r^{-1}\phi,r^2L\psi,r\pd_r\psi,r\pd_\theta\psi,r\psi\}$$
$$\beta = \{L\phi,\pd_r\phi,r^{-1}\pd_\theta\phi,r^{-1}\phi,r^2L\psi,r^2\pd_r\psi,r\pd_\theta\psi,r\psi\}$$
The null condition states that any nonlinear term must be a product with at least one $\alpha$ factor.
\end{definition}

We now prove two lemmas (one for $\mathcal{N}_\phi$ and one for $\mathcal{N}_\psi$) which guarantee the null condition as well as a structural condition on the axis.

\begin{lemma}\label{Nphi_structure_lem}(Structure of $\mathcal{N}_\phi$)
The nonlinear term $\mathcal{N}_\phi$ can be expressed as a sum of terms of the form
$$\frac{f\alpha\beta}{1+\phi},$$
satisfying the following additional rules. \\
i) The factor $f$ is smooth and bounded. \\
ii) The term belongs to $\tau_{(\le 1)}(P_\theta^\infty\cup \{\phi,\pd_t\phi,\pd_r\phi,(1+\phi)^{-1},\psi,\pd_t\psi,\pd_r\psi\})$. \\
iii) If $\psi$ appears at least once in the term, then $f$ has a factor of $\sin^2\theta$.
\end{lemma}
\begin{proof}
In this proof, we use the sign $\approx$ to emphasize that smooth and bounded factors (including $M/r$ and $a/r$) are neglected. One can check (using Lemma \ref{small_a_quantities_lem} as a guide) that
$$ \frac{\pd_\theta A}{A} \approx \frac1{\sin\theta}, \hspace{.5in}
\frac{\pd_r A}{A} \approx r^{-1}, \hspace{.5in}
\frac{\pd_\theta B}{A} \approx \sin\theta, \hspace{.5in}
\frac{\pd_r B}{A} \approx r^{-1}\sin^2\theta $$
$$ \frac{\pd^\alpha A\pd_\alpha A}{A^2} \approx \frac1{r^2\sin^2\theta}, \hspace{.5in}
\frac{\pd^\alpha A\pd_\alpha B}{A^2} \approx r^{-2}, \hspace{.5in}
\frac{\pd^\alpha B\pd_\alpha B}{A^2} \approx r^{-2}\sin^2\theta. $$
Using these, we investigate each term in $(1+\phi)\mathcal{N}_\phi$.
\begin{align*}
\pd^\alpha \phi\pd_\alpha \phi &\approx L\phi\lbar \phi + r^{-2}\pd_\theta\phi\pd_\theta \phi \\
&\approx (L\phi)(\lbar\phi)+(r^{-1}\pd_\theta\phi)(r^{-1}\pd_\theta\phi) \\
&\approx (L\phi)(\pd_r\phi)+(L\phi)(L\phi)+(r^{-1}\pd_\theta\phi)(r^{-1}\pd_\theta\phi).
\end{align*}
\begin{align*}
A\pd^\alpha\psi A\pd_\alpha \psi &\approx r^4\sin^4\theta (L\psi \lbar\psi+r^{-2}\pd_\theta\psi\pd_\theta\psi) \\
&\approx \sin^4\theta (r^2L\psi)(r^2\lbar\psi)+\sin^4\theta(r\pd_\theta\psi)(r\pd_\theta\psi) \\
&\approx \sin^4\theta (r^2L\psi)(r^2\pd_r\psi)+\sin^4\theta (r^2L\psi)(r^2L\psi)+\sin^4\theta(r\pd_\theta\psi)(r\pd_\theta\psi).
\end{align*}
\begin{align*}
\frac{\pd^\alpha B}{A}\phi A\pd_\alpha \psi &\approx r^{-2}\frac{\pd_\theta B}{A}\phi A\pd_\theta\psi +\frac{\pd_r B}{A}\phi A\pd_r\psi \\
&\approx \sin^3\theta \phi\pd_\theta\psi +r^{-1}\sin^4\theta \phi\pd_r\psi \\
&\approx \sin^3\theta (r^{-1}\phi)(r\pd_\theta\psi) + \sin^4\theta (r^{-1}\phi)(\pd_r\psi).
\end{align*}
\begin{align*}
\frac{\pd^\alpha A}{A} A\psi A\pd_\alpha \psi &\approx r^{-2}\frac{\pd_\theta A}{A}A\psi A\pd_\theta\psi +\frac{\pd_rA}{A}A\psi A\pd_r\psi \\
&\approx r^2\sin^3\theta\psi\pd_\theta\psi + r^3\sin^4\theta\psi\pd_r\psi \\
&\approx \sin^3\theta (r\psi)(r\pd_\theta\psi) +\sin^4\theta (r\psi)(r^2\pd_r\psi).
\end{align*}
$$\frac{\pd^\alpha B\pd_\alpha B}{A^2}\phi^2 \approx r^{-2}\sin^2\theta\phi^2 \approx \sin^2\theta (r^{-1}\phi)(r^{-1}\phi).$$
$$\frac{\pd^\alpha A\pd_\alpha A}{A^2}\phi A\psi \approx \sin^2\theta\phi\psi \approx \sin^2\theta(r^{-1}\phi)(r\psi).$$
$$\frac{\pd^\alpha A\pd_\alpha A}{A^2}A\psi A\psi \approx r^2\sin^2\theta\psi^2 \approx \sin^2\theta (r\psi)(r\psi).$$
It is now straightforward to check for each of the above calculations that the first factor in parentheses is an $\alpha$ term, and the second factor in parentheses is a $\beta$ term. It is also straightforward to check that any term containing a $\psi$ factor also has a factor of $\sin^2\theta$. \qed
\end{proof}

\begin{lemma}(Structure of $\mathcal{N}_\psi$)
The nonlinear term $\mathcal{N}_\psi$ can be expressed as a sum of terms of the form
$$\frac{f\alpha\beta}{1+\phi}\text{ or }\frac{f\alpha (r^{-1}\fb \phi)}{1+\phi},$$
satisfying the following additional rules. \\
i) The factor $f$ is smooth and bounded. \\
ii) The term belongs to $\tau_{(\le 1)}(P_\theta^\infty\cup \{\phi,\pd_t\phi,\pd_r\phi,(1+\phi)^{-1},\psi,\pd_t\psi,\pd_r\psi\})$. \\
iii) If $\phi$ appears at least once in the term, then $f$ has a factor of $r^{-2}$.
\end{lemma}
\begin{proof}
We proceed as in the proof of the previous lemma, this time investigating each term in $(1+\phi)\mathcal{N}_\psi$.
\begin{align*}
\pd^\alpha\phi\pd_\alpha\psi &\approx L\phi\lbar\psi +\lbar\phi L\psi +r^{-2}\pd_\theta\phi\pd_\theta\psi \\
&\approx r^{-2}(L\phi)(r^2\lbar\psi)+r^{-2}(r^2L\psi)(\lbar\phi)+r^{-2}(r^{-1}\pd_\theta\phi)(r\pd_\theta\psi) \\
&\approx r^{-2}(L\phi)(r^2\pd_r\psi)+r^{-2}(L\phi)(r^2L\psi)+r^{-2}(r^2L\psi)(\pd_r\phi)+r^{-2}(r^2L\psi)(L\phi) \\
&\hspace{4in}+r^{-2}(r^{-1}\pd_\theta\phi)(r\pd_\theta\psi).
\end{align*}
\begin{align*}
\frac{\pd^\alpha A}{A}\psi\pd_\alpha \phi &\approx r^{-2}\frac{\pd_\theta A}{A}\psi\pd_\theta\phi + \frac{\pd_r A}{A}\psi\pd_r\phi \\
&\approx r^{-2}\psi\fb\phi +r^{-1}\psi\pd_r\phi \\
&\approx r^{-2}(r\psi)(r^{-1}\fb\phi) + r^{-2}(r\psi)(\pd_r\phi).
\end{align*}
\begin{align*}
\frac{\pd^\alpha B}{A^2}\phi\pd_\alpha \phi &\approx r^{-2}\frac{\pd_\theta B}{A^2}\phi\pd_\theta\phi +\frac{\pd_r B}{A^2}\phi\pd_r\phi \\
&\approx r^{-4}\phi \fb\phi +r^{-4}\sin\theta \phi\pd_\theta\phi +r^{-3}\phi\pd_r\phi \\
&\approx r^{-2}(r^{-1}\phi)(r^{-1}\fb\phi)+r^{-2}\sin\theta(r^{-1}\phi)(r^{-1}\pd_\theta\phi)+r^{-2}(r^{-1}\phi)(\pd_r\phi).
\end{align*}
It is now straightforward to check for each of the above calculations that the first factor in parentheses is an $\alpha$ term, and the second factor in parentheses is either a $\beta$ term or $r^{-1}\fb\phi$. Finally, since $\phi$ appears in every term, there is an additional factor of $r^{-2}$ accompanying each term as required. \qed
\end{proof}

\subsection{The structures of $q^{-2}\tg^k\Gamma^s(q^2\mathcal{N}_\phi)$ and $q^{-2}\tg^k\tilde\Gamma^s(q^2\mathcal{N}_\psi)$}\label{nl_s_k_structure_sec}

We now generalize the previous two lemmas by applying the commutators.

\begin{definition} We generalize the previous families of terms to the following.
\begin{multline*}
\alpha^{s,k} = \{L\fd^l\phi^{s-l,k},r^{-1}\pd_r\fd^l\phi^{s-l,k},r^{-1}\pd_\theta\fd^l\phi^{s-l,k},r^{-1}\fd^l\phi^{s-l,k}, \\
r^2L\fd^l\psi^{s-l,k},r\pd_r\fd^l\psi^{s-l,k},r\pd_\theta\fd^l\psi^{s-l,k},r\fd^l\psi^{s-l,k}\}
\end{multline*}
\begin{multline*}
\beta^{s,k} = \{L\fd^l\phi^{s-l,k},\pd_r\fd^l\phi^{s-l,k},r^{-1}\pd_\theta\fd^l\phi^{s-l,k},r^{-1}\fd^l\phi^{s-l,k}, \\
r^2L\fd^l\psi^{s-l,k},r^2\pd_r\fd^l\psi^{s-l,k},r\pd_\theta\fd^l\psi^{s-l,k},r\fd^l\psi^{s-l,k}\}
\end{multline*}
The null condition still states that any nonlinear term must be a product with at least one $\alpha$ factor.
\end{definition}

%Note that the terms $r^{-1}\pd_\theta\fd^l\phi^{s-l,k}$ and $r\pd_\theta\fd^l\psi^{s-l,k}$ do not belong to $\alpha^{s,k}$ and $\beta^{s,k}$, but instead belong to $\alpha^{s+1,k}$ and $\beta^{s+1,k}$, since $r^{-1}\pd_\theta\fd^l\phi^{s-l,k}=r^{-1}\fd^{l+1}\phi^{s-l,k}$

The following two lemmas (one for $q^{-2}\tg^k\Gamma^s(q^2\mathcal{N}_\phi)$ and one for $q^{-2}\tg^k\tilde{\Gamma}^s(q^2\mathcal{N}_\psi)$) generalize the previous two lemmas.

\begin{lemma}\label{Nphi_s_k_structure_lem}
(Structure of $q^{-2}\tg^k\Gamma^s(q^2\mathcal{N}_\phi)$) The nonlinear term $q^{-2}\tg^k\Gamma^s(q^2\mathcal{N}_\phi)$ can be expressed as a sum of terms of the following form. (The index $j$ represents the number of times a differential operator acts on the denominator.) \\
$$\frac{f\alpha^{s_1,k_1}\beta^{s_2,k_2}(r\beta^{s_3,k_3})...(r\beta^{s_{2+j},k_{2+j}})}{(1+\phi)^{j+1}},$$
where $0\le j\le s+k$, $s_1+...+s_{2+j}\le s+2$, $\max_is_i\le s$, $k_1+...+k_{2+j}\le k$, and the following rules apply. \\
\bp The factor $f$ is smooth and bounded. \\
\bp If $\psi$ appears at least once in the term, then $f$ has a factor of $(\sin\theta)^{\max(0,2-\max_is_i)}$.
\end{lemma}
\begin{proof}
We start with the case $s=k=0$, which was proved in Lemma \ref{Nphi_structure_lem}. In particular, since $r^{-1}\pd_\theta\phi=\pd_\theta (r^{-1}\phi)$ and $r\pd_\theta\psi=\pd_\theta (r\psi)$, each term in $\mathcal{N}_\phi$ can be written in one of the following forms.
$$\frac{f\alpha_0\beta_0}{1+\phi}\text{ or }
\frac{f_1\pd_\theta f_2\pd_\theta\alpha_0\beta_0}{1+\phi}\text{ or }
\frac{f_1\pd_\theta f_2\alpha_0\pd_\theta\beta_0}{1+\phi}\text{ or }
\frac{f\pd_\theta\alpha_0\pd_\theta\beta_0}{1+\phi},$$
where
$$\alpha_0\in \alpha\setminus \{r^{-1}\pd_\theta\phi,r\pd_\theta\psi\}\text{ and }
\beta_0\in \beta\setminus \{r^{-1}\pd_\theta\phi,r\pd_\theta\psi\}.$$
The effect of applying $\pd_t^{s'}$ to any of these terms is to obtain terms of the form
$$\frac{f\pd_t^{s_1}\alpha_0\pd_t^{s_2}\beta_0\pd_t^{s_3}\phi...\pd_t^{s_{2+j}}\phi}{(1+\phi)^{1+j}}\text{ or }
\frac{f_1\pd_\theta f_2\pd_\theta\pd_t^{s_1}\alpha_0\pd_t^{s_2}\beta_0\pd_t^{s_3}\phi...\pd_t^{s_{2+j}}\phi}{(1+\phi)^{1+j}}\text{ etc.},$$
where $s_1+...+s_{2+j}\le s'$ and $j\le s'$. The additional factors of $\pd_t^{s_i}\phi$ appear each time one of the $\pd_t$ operators acts on the denominator.

Recall that $\tg^k\Gamma^s$ is composed not only of $\pd_t$, but also of $Q$ and $\tg$. Since these operators commute with each other, the order in which they are applied is not important. For the sake of simplicity, we first apply the $\pd_t^{s'}$ operators (which has already been done) and then the $\tg^k$ operators. Since these are both first order operators, it should be clear that the resulting terms are of the form
$$\frac{f\alpha_0^{s_1,k_1}\beta_0^{s_2,k_2}\gamma^{s_3,k_3}...\gamma^{s_{2+j},k_{2+j}}}{(1+\phi)^{1+j}}\text{ or }
\frac{f_1\pd_\theta f_2\pd_\theta\alpha_0^{s_1,k_1}\beta_0^{s_2,k_2}\gamma^{s_3,k_3}...\gamma^{s_{2+j},k_{2+j}}}{(1+\phi)^{1+j}}\text{ etc.},$$
where $j\le s'+k$, $s_1+...+s_{2+j}\le s'$, $k_1+...+k_{2+j}\le k$, and $\alpha_0^{s_1,k_1}=\tg^{k_1}\pd_t^{s_1}\alpha_0$, $\beta_0^{s_2,k_2}=\tg^{k_2}\pd_t^{s_2}\beta_0$, and $\gamma^{s_i,k_i}=\tg^{k_i}\pd_t^{s_i}\phi$.
This is the point at which we apply the operator $Q$. Recall that
$$Q=\fa+\fb+a^2\sin^2\theta\pd_t^2,$$
where the operators $\fa$ and $\fb$ are defined in \S\ref{regularity_sec}. We handle $Q$ as if applying these operators seperately. The operator $a^2\sin^2\theta\pd_t^2$ is equivalent to applying $\pd_t$ twice and multiplying by a smooth bounded function. However, the operators $\fa$ and $\fb$ require the formalism from Lemma \ref{cl_raise_degree_lem}. After applying $Q^{s''}$, the terms should take the following form.
$$
\frac{f_1\fd^{l_0} f_2 \fd^{l_1}\alpha_0^{s_1,k_1}\fd^{l_2}\beta_0^{s_2,k_2}\fd^{l_3}\gamma^{s_3,k_3}...\fd^{l_{2+j}}\gamma^{s_{2+j},k_{2+j}}}{(1+\phi)^{1+j}}
$$
where $l_0=1$ or $\fd^{l_0}f_2=1$ so that $l_0+l_1+...l_{2+j}$ is even. We have in addition that $l_0+l_1+...+l_{2+j}\le s''+2$, since there were initially zero or two $\fd$s and each application of $Q$ contributes at most two more. Finally, we observe that the $\gamma$ factors are all of the form $r\beta$. From here, the lemma follows.\qed
\end{proof}

\begin{lemma}\label{Npsi_s_k_structure_lem}
(Structure of $q^{-2}\tg^k\tilde\Gamma^s(q^2\mathcal{N}_\psi)$) The nonlinear term $q^{-2}\tg^k\tilde\Gamma^s(q^2\mathcal{N}_\psi)$ can be expressed as a sum of terms, each taking one of the following forms. (The index $j$ represents the number of times a differential operator acts on the denominator.) \\
$$\frac{f\alpha^{s_1,k_1}\beta^{s_2,k_2}(r\beta^{s_3,k_3})...(r\beta^{s_{2+j},k_{2+j}})}{(1+\phi)^{j+1}}\text{ or }\frac{f\alpha^{s_1,k_1}(r^{-1}\fb\fc^l\phi^{s_2-2l,k_2})(r\beta^{s_3,k_3})...(r\beta^{s_{2+j},k_{2+j}})}{(1+\phi)^{j+1}},$$
where $0\le j\le s+k$, $s_1+...+s_{2+j}\le s+2$, $\max_is_i\le s$, $k_1+...+k_{2+j}\le k$, and the following rules apply. \\
\bp The factor $f$ is smooth and bounded. \\
\bp If $\phi$ appears at least once in the term, then $f$ has a factor of $r^{-2}$.
\end{lemma}
\begin{proof}
The proof is essentially the same as the proof of the previous lemma, with the one difference being the possible presence of the factor $r^{-1}\fb\phi$. If this is treated like the other factors, the result after applying the commutators is a factor of the form 
$$r^{-1}\fd^i\fb\phi^{s_2-i,k_2}.$$
If $i$ is odd, this term can be rewritten as $r^{-1}\pd_\theta \fd^{i+1}\phi^{s_2-i,k_2}=\beta^{s_1+1,k_2}$. This is at the level $s_2+1$, which is less than or equal to $s$, because $i$ is odd.

If instead $i$ is even, this term can be rewritten as $r^{-1}\fc^l\fb\phi^{s_2-2l,k_2}$, where $i=2l$. By Lemma \ref{fb_fcl_commutation_lem}, this is (up to lower order terms in $s$) equal to $r^{-1}\fb\fc^l\phi^{s_2-2l,k_2}$.

With this exceptional case having been addressed, the proof is complete. \qed
\end{proof}

\subsection{The strategy revisited}\label{nl_strategy_revisited_sec}

Now that the structures of $q^{-2}\tg^k\Gamma^s(q^2\mathcal{N}_\phi)$ and $q^{-2}\tg^k\tilde\Gamma^s(q^2\mathcal{N}_\psi)$ have been determined, the strategy employed by the few example lemmas in \S\ref{example_terms_sec} can be summarized in the following proposition, which explains how to establish the key estimate in a bootstrap argument.

\begin{proposition}\label{strategy_revisited_prop}
Suppose the following estimates hold.
$$(1+\phi)^{-1}\lesssim 1,$$
$$\int_R^\infty\int_0^\pi (r^{(p-1)/2}\alpha^s)^2 r^2\sin^5\theta d\theta dr \lesssim E_{p-1}^s(t),$$
$$||r\beta^s||_{L^\infty(\Sigma_t\{r>R\})}^2\lesssim E^{s+7}(t),$$
$$\int_R^\infty\int_0^\pi (\beta^s)^2r^2\sin^5\theta d\theta dr\lesssim E^s(t),$$
$$||r^{(p-1)/2+1}\alpha^s||_{L^\infty(\Sigma_t\cap\{r>R\})}^2\lesssim E_{p-1}^{s+7}(t).$$
Then
\begin{multline*}
\int_{\Sigma_t\cap\{r>R\}}r^{p+1}|q^{-2}\tg^k\Gamma^s(q^2\mathcal{N}_\phi)|^2 + \int_{\tilde{\Sigma}_t\cap\{r>R\}}r^{p+1}|q^{-2}\tg^k\tilde{\Gamma}^s(q^2\mathcal{N}_\psi)|^2 \\
\lesssim \left(E_{p-1}^s(t)E^{s/2+8}(t)+E^s(t)E_{p-1}^{s/2+8}(t)\right)\sum_{j\le s}(E^{s/2+8}(t))^j.
\end{multline*}
\end{proposition}
\begin{remark}
The second and fourth assumptions are automatically true. The first, third, and fifth assumptions will be true in the context of the proof of the main theorem (Theorem \ref{main_thm}).
\end{remark}
\begin{proof}
Note that Lemmas \ref{Nphi_s_k_structure_lem} and \ref{Npsi_s_k_structure_lem} both had the requirements $s_1+...+s_{2+j}\le s+2$ and $\max_i s_i\le s$. This means there can be at most one  factor with up to $s$ derivatives and all of the remaining factors have at most $(s+2)/2=s/2+1$ derivatives. Since the integrals are taken over the range $r>R$, where $\tg=0$, we are free to ignore all $k$ indices. 

The strategy in this proof is rather simple. All of the factors that have at most $s/2+1$ derivatives are estimated in $L^\infty$, while the remaining factor with at most $s$ derivatives is estimated in $L^2$. The procedure then takes one of two possible directions, depending on whether the high derivative term is an $\alpha$ term or a $\beta$ term.

\textbf{First, we estimate the integral over $\Sigma_t$, using the form given in Lemma \ref{Nphi_s_k_structure_lem}.}
\begin{multline*}
\int_{\Sigma_t\cap\{r>R\}}r^{p+1}|q^{-2}\tg^k\Gamma^s(q^2\mathcal{N}_\phi)|^2 \\
\lesssim \sum_{0\le j\le s} \int_{\Sigma_t\cap\{r>R\}}r^{p+1}\left(f\alpha^s\beta^{s/2+1}(r\beta^{s/2+1})^j\right)^2 + \sum_{0\le j\le s} \int_{\Sigma_t\cap\{r>R\}}r^{p+1}\left(f\beta^s\alpha^{s/2+1}(r\beta^{s/2+1})^j\right)^2 .
\end{multline*}
Now,
\begin{align*}
\int_{\Sigma_t\cap\{r>R\}}r^{p+1}\left(f\alpha^s\beta^{s/2+1}(r\beta^{s/2+1})^j\right)^2
&\lesssim \int_{\Sigma_t\cap\{r>R\}}r^{p-1}\left(f\alpha^s(r\beta^{s/2+1})(r\beta^{s/2+1})^j\right)^2 \\
&\lesssim ||r\beta^{s/2+1}||_{L^\infty}^{2(j+1)}\int_{\Sigma_t\cap\{r>R\}}r^{p-1}(f\alpha^s)^2 \\
&\lesssim (E^{s/2+8}(t))^{j+1}\int_{\Sigma_t\cap\{r>R\}}r^{p-1}(f\alpha^s)^2 \\
&\lesssim E_{p-1}^s(t)(E^{s/2+8}(t))^{j+1}.
\end{align*}
The last step requires further justification, because the $\alpha$ term could possibly be a $\psi$ term. But in that case, Lemma \ref{Nphi_s_k_structure_lem} also states that $f$ has an additional factor of $\sin^2\theta$ (or in the case where $\alpha$ does not have exactly $s$ derivatives, this factor might be either $\sin\theta$ or $1$, but we can apply Lemma \ref{gain_sin_lem}--see the end of the proof of the example Lemma \ref{example_Nphi_term_lem}).

Also,
\begin{align*}
\int_{\Sigma_t\cap\{r>R\}}r^{p+1}\left(f\beta^s\alpha^{s/2+1}(r\beta^{s/2+1})^j\right)^2 
&\lesssim \int_{\Sigma_t\cap\{r>R\}}\left(f\beta^s(r^{(p+1)/2}\alpha^{s/2+1})(r\beta^{s/2+1})^j\right)^2 \\
&\lesssim ||r^{(p+1)/2}\alpha^{s/2+1}||_{L^\infty}^2||r\beta||_{L^\infty}^{2j}\int_{\Sigma_t\cap\{r>R\}}(f\beta^s)^2 \\
&\lesssim E_{p-1}^{s/2+8}(t)(E^{s/2+8}(t))^j\int_{\Sigma_t\cap\{r>R\}}(f\beta^s)^2 \\
&\lesssim E^s(t)E_{p-1}^{s/2+8}(t)(E^{s/2+8}(t))^j.
\end{align*}
The last step requires the same justification given in the preceeding calculation.

Combining both estimates, we conclude that 
$$\int_{\Sigma_t\cap\{r>R\}}r^{p+1}|q^{-2}\tg^k\Gamma^s(q^2\mathcal{N}_\phi)|^2 \lesssim \left(E_{p-1}^s(t)E^{s/2+8}(t)+E^s(t)E_{p-1}^{s/2+8}(t)\right)\sum_{j\le s}(E^{s/2+8}(t))^j.$$

\textbf{Next, we estimate the integral over $\tilde{\Sigma}_t$, using the forms given in Lemma \ref{Npsi_s_k_structure_lem}.} The first of the two forms given in Lemma \ref{Npsi_s_k_structure_lem} is very similar to the form given in Lemma \ref{Nphi_s_k_structure_lem}, and so we will not repeat the estimates in detail. The only difference is in the two final steps when concluding either 
$$\int_{\tilde{\Sigma}_t\cap\{r>R\}}r^{p-1}(f\alpha^s)^2\lesssim E_{p-1}^s(t)$$
or
$$\int_{\tilde{\Sigma}_t\cap\{r>R\}}(f\beta^s)^2\lesssim E^s(t).$$
In either case, if $\alpha^s$ or $\beta^s$ represents a $\phi$ term, then Lemma \ref{Npsi_s_k_structure_lem} also states that $f$ has an additional factor of $r^{-2}$.

The second of the two forms has the factor $r^{-1}\fb\fc^l\phi^{s_2-2l}$, which takes the place of the $\beta^{s_2}$ factor from the first form. We now show that it can be estimated the same way the $\beta^{s_2}$ factor was estimated.

For the case $s_2\le s/2+1$, we must estimate $r$ times the factor $r^{-1}\fb\fc^l\phi^{s_2-2l}$ in $L^\infty$. That is, using the Lemma \ref{low_order_infty_lem},
$$||\fb\fc^l\phi^{s/2+1-2l}||_{L^\infty}^2=||\fd^{2l+2}\phi^{s/2+1-2l}||_{L^\infty}^2 \lesssim E^{s/2+8}(t).$$
For the case where $s_2$ could be at most $s$, we use the fact that there are additional factors of $\sin\theta$ in the volume form for $\tilde{\Sigma}_t$ and again the fact that the function $f$ has a factor of $r^{-2}$. That is,
\begin{align*}
\int_{\tilde{\Sigma}_t\cap\{r>R\}}(fr^{-1}\fb\fc^l\phi^{s-2l})^2
&\lesssim \int_R^\infty\int_0^\pi(r^{-3}\sin^{-1}\theta\pd_\theta\fc^l\phi^{s-2l})^2r^6\sin^5\theta d\theta dr\\
&\lesssim \int_R^\infty\int_0^\pi(r^{-1}\pd_\theta \fc^l\phi^{s-2l})^2r^2\sin^3\theta d\theta dr\\
&\lesssim \int_R^\infty\int_0^\pi(r^{-1}\pd_\theta \phi^s)^2r^2\sin\theta d\theta dr \\
&\lesssim E^s(t).
\end{align*}
(In the second-to-last step, we used an estimate from Theorem \ref{cl_embedding_thm}.) Both of these estimates show that the $r^{-1}\fb\fc^l\phi^{s_2-2l}$ factor can be treated the same way as the $\beta^{s_2}$ factor.

We conclude that
$$\int_{\tilde{\Sigma}_t\cap\{r>R\}}r^{p+1}|q^{-2}\tg^k\tilde{\Gamma}^s(q^2\mathcal{N}_\psi)|^2 \lesssim \left(E_{p-1}^s(t)E^{s/2+8}(t)+E^s(t)E_{p-1}^{s/2+8}(t)\right)\sum_{j\le s}(E^{s/2+8}(t))^j.$$
This completes the proof. \qed
\end{proof}

\section{Statement and Proof of the Main Theorem}\label{main_thm_sec}

In this final section, we state and prove the main theorem. We begin in \S\ref{main_thm_statement_sec} with the precise statement of the main theorem, which completely describes the future asymptotic behavior of various energy norms as well as weighted derivatives of $\phi$ and $\psi$. Then the proof follows.

The proof is a bootstrap argument with the main bootstrap assumptions stated in \S\ref{bootstrap_assumptions_sec}. Then in \S\ref{improved_pointwise_estimates_sec}, the bootstrap assumptions are used to improve the pointwise estimates in Proposition \ref{infinity_prop} by removing the error terms from the right hand side. The resulting improved pointwise estimates play a crucial role in the remainder of the proof. Next, in \S\ref{recover_phi_ba_sec} the pointwise estimates are used to recover the first bootstrap assumption, which gives an upper bound on the factor $(1+\phi)^{-1}$ that appears in the nonlinear terms. The remainder of the proof is split into two parts. 

The first part (\S\ref{homogeneous_case_sec}) recovers the bootstrap assumptions only for the homogeneous norms $\mathring{E}^s(t)$, $\mathring{E}_p^s(t)$, etc. These norms satisfy slightly simpler estimates, because they are based on the commutator $\pd_t$, which commutes with the entire linear system. As a reminder, these estimates, which are given in Theorem \ref{p_o_thm}, are
$$\mathring{E}^s(t_2)\lesssim \mathring{E}^s(t_1)+\int_{t_1}^{t_2}\mathring{N}^s(t)dt,$$
$$\mathring{E}_p^s(t_2)+\int_{t_1}^{t_2}\mathring{B}_p^s(t)dt\lesssim \mathring{E}_p^s(t_1)+\int_{t_1}^{t_2}\mathring{N}_p^s(t)dt.$$

The second part (\S\ref{general_case_s_k_sec}) recovers the bootstrap assumptions for the norms $E^{s,k}(t)$, $E_p^{s,k}(t)$, etc. These norms satisfy slightly more complicated estimates, because they use other commutators that do not completely commute with the linear system. As a reminder, these estimates, which are given in Theorem \ref{p_s_k_thm}, are
$$E^{s,k}(t_2)\lesssim E^{s,k}(t_1)+\int_{t_1}^{t_2}B_{p'}^{s+2,k-1}(t)+\mathring{B}_1^{s+1}(t)+B_1^{s,k}(t)+N^{s,k}(t)dt,$$
$$E_p^{s,k}(t_2)+\int_{t_1}^{t_2}B_p^{s,k}(t)dt\lesssim E_p^{s,k}(t_1)+\int_{t_1}^{t_2}B_{p'}^{s+2,k-1}(t)+\mathring{B}_p^{s+1}(t)+N_p^{s,k}(t)dt.$$
In particular, the fact that these estimates depend on $\mathring{B}_p^{s+1}(t)$ means that this second part depends on the results proved in the first part.

\subsection{The main theorem statement}\label{main_thm_statement_sec}
We begin with the statement of the main theorem.
\begin{theorem}\label{main_thm}
Let
\begin{align*}
X&=A+A\phi, \\
Y&=B+A^2\psi.
\end{align*}
(In particular, the assumption $Y-B=O(\sin^4\theta)$ near the axis excludes any perturbations corresponding to a change in angular momentum.) 

Suppose the pair $(X,Y)$ is axisymmetric and satisfies the wave map system
\begin{align*}
X\Box_g X&=\pd^\alpha X\pd_\alpha X-\pd^\alpha Y\pd_\alpha Y, \\
X\Box_g Y&=2\pd^\alpha X\pd_\alpha Y,
\end{align*}
where $g$ is a Kerr metric with sufficiently small angular momentum $|a|/M$.

Define the energies
$$E^{\ul{n}}(t)=\mathring{E}^n(t)+\sum_{s+2k=n-1}E^{s,k}(t),$$
$$E_p^{\ul{n}}(t)=\mathring{E}_p^{n}(t)+\sum_{s+2k=n-1}E_p^{s,k}(t).$$
Then for $\delp,\delm>0$ sufficiently small, if the initial data for $(\phi,\psi)$ decay sufficiently fast as $r\rightarrow \infty$ and have size
\begin{equation}
I_0=E^{\ul{29}}(0)+E^{\ul{29}}_{2-\delp}(0)
\end{equation}
sufficiently small, then the following estimates hold for $t\ge 0$ (with $T=1+t$).

I) The energies satisfy
$$E^{\ul{29}}(t)\lesssim I_0$$
$$E^{\ul{29}}_{p\in[\delm,2-\delp]}(t)\lesssim I_0$$
$$E^{\ul{27}}_{p\in[1-\delp,2-\delp]}(t)\lesssim T^{p-2+\delp}I_0$$
$$E^{\ul{25}}_{p\in[\delm,2-\delp]}(t)\lesssim T^{p-2+\delp}I_0$$
$$\int_t^{\infty}E^{\ul{23}}_{p\in[\delm-1,\delm]}(\tau)d\tau\lesssim T^{p-2+\delp+1}I_0$$

II) For all $s,k$ such that $s+2k\le 28$, the following pointwise estimates hold.
\begin{multline*}
|r^{p+1}\bar{D}\fd^l\phi^{s-l,k}|^2+|r^pD\fd^l\phi^{s-l,k}|^2+|r^p\fd^l\phi^{s-l,k}|^2 \\
+|r^{p+3}\bar{D}\fd^l\psi^{s-l,k}|^2+|r^{p+2}D\fd^l\psi^{s-l,k}|^2+|r^{p+2}\fd^l\psi^{s-l,k}|^2 \\
\lesssim E_{2p}^{s+5,k+1}(t)+E_{2p}^{s+7,k}(t)
\end{multline*}

III) Together, (I) and (II) imply that if $s+2k\le 15$, for all $p\in [\delm/2,(2-\delp)/2]$,
\begin{multline*}
|r^{p+1}\bar{D}\fd^l\phi^{s-l,k}|+|r^pD\fd^l\phi^{s-l,k}|+|r^p\fd^l\phi^{s-l,k}| \\
+|r^{p+3}\bar{D}\fd^l\psi^{s-l,k}|+|r^{p+2}D\fd^l\psi^{s-l,k}|+|r^{p+2}\fd^l\psi^{s-l,k}| \\
\lesssim T^{(2p-2+\delp)/2}I_0^{1/2}
\end{multline*}
and additionally for $p\in [(\delm-1)/2,\delm/2]$,
\begin{multline*}
\int_t^\infty |r^{p+1}\bar{D}\fd^l\phi^{s-l,k}|+|r^pD\fd^l\phi^{s-l,k}|+|r^p\fd^l\phi^{s-l,k}| \\
+\int_t^\infty |r^{p+3}\bar{D}\fd^l\psi^{s-l,k}|+|r^{p+2}D\fd^l\psi^{s-l,k}|+|r^{p+2}\fd^l\psi^{s-l,k}| \\
\lesssim T^{(2p-2+\delp)/2+1}I_0^{1/2}.
\end{multline*}
The final estimate should be interpreted as saying that $|r^{(\delm-1)/2+1}\bar{D}\fd^l\phi^{s-l,k}|$, $|r^{(\delm-1)/2}D\fd^l\phi^{s-l,k}|$, $|r^{(\delm-1)/2}\fd^l\phi^{s-l,k}|$, $|r^{(\delm-1)/2+3}\bar{D}\fd^l\psi^{s-l,k}|$, $|r^{(\delm-1)/2+2}D\fd^l\psi^{s-l,k}|$, and $|r^{(\delm-1)/2+2}\fd^l\psi^{s-l,k}|$ decay like $T^{(\delm-3+\delp)/2}$ in a weak sense.

\end{theorem}

\subsection{Bootstrap assumptions}\label{bootstrap_assumptions_sec}

We begin the proof of Theorem \ref{main_thm} by making the following bootstrap assumptions.
\begin{align}
|\phi|&\le 1/2, \label{phi_ba} \\
E^{\ul{29}}(t)&\le C_bI_0, \label{energy_ba} \\
E^{\ul{29}}_{p\in[\delm,2-\delp]}(t)&\le C_b I_0, \label{_29_ba} \\
E^{\ul{25}}_{\delm}(t)&\le C_b T^{\delm-2+\delp}I_0, \label{_25_ba} \\
\int_t^\infty E^{\ul{23}}_{\delm-1}(\tau)d\tau &\le C_b T^{(\delm-1)-2+\delp+1}I_0. \label{_23_ba}
\end{align}
Note that, with the exception of the highest order energies, these bootstrap assumptions are consistent with the general principle that $\mathring{E}_p^s(t)$ and $E_p^{s,k}(t)$ behave like $T^{p-2+\delp}$, which the reader should keep in mind throughout the proof of the main theorem.

\subsection{Improved pointwise estimates}\label{improved_pointwise_estimates_sec}

The pointwise estimates from Proposition \ref{infinity_prop} are essential to the argument of the proof of the main theorem. But for the sake of clarity, we first remove the error terms from these estimates and summarize them in the following lemma.
\begin{lemma}\label{simplified_pointwise_lemma}
In the context of the bootstrap assumptions provided in \S\ref{bootstrap_assumptions_sec}, the following pointwise estimates hold for $s+2k\le 28$ and all $p$ in any bounded range.

For $r\ge r_H$,
\begin{multline}
|r^{p+1}\bar{D}\fd^l\phi^{s-l,k}|^2+|r^pD\fd^l\phi^{s-l,k}|^2+|r^p\fd^l\phi^{s-l,k}|^2 \\
+|r^{p+3}\bar{D}\fd^l\psi^{s-l,k}|^2+|r^{p+2}D\fd^l\psi^{s-l,k}|^2+|r^{p+2}\fd^l\psi^{s-l,k}|^2 \\
\lesssim E_{2p}^{s+5,k+1}(t)+E_{2p}^{s+7,k}(t)
\end{multline}
and for $r\ge r_0>r_H$,
\begin{equation}
|rD\fd^l\phi^{s-l,k}|^2+|r^3D\fd^l\psi^{s-l,k}|^2\lesssim E^{s+5,k+1}(t)+E^{s+7,k}(t)
\end{equation}
These are the same as the estimates from Proposition \ref{infinity_prop}, except that the error terms have been removed.
\end{lemma}
\begin{proof} 
Recall that, according to Proposition \ref{infinity_prop}, for $r\ge r_H$,
\begin{multline*}
|r^{p+1}\bar{D}\fd^l\phi^{s-l,k}|^2+|r^pD\fd^l\phi^{s-l,k}|^2+|r^p\fd^l\phi^{s-l,k}|^2 \\
+|r^{p+3}\bar{D}\fd^l\psi^{s-l,k}|^2+|r^{p+2}D\fd^l\psi^{s-l,k}|^2+|r^{p+2}\fd^l\psi^{s-l,k}|^2 \\
\lesssim E_{2p}^{s+5,k+1}(t)+E_{2p}^{s+7,k}(t)+\int_{\Sigma_t}r^{2p}(\Box_g\phi^{s+5,k})^2+\int_{\tilde\Sigma_t}r^{2p}(\Box_{\tilde{g}}\psi^{s+5,k})^2
\end{multline*}
and for $r\ge r_0>r_H$,
\begin{multline*}
|rD\fd^l\phi^{s-l,k}|^2+|r^3D\fd^l\psi^{s-l,k}|^2 \\
\lesssim E^{s+5,k+1}(t)+E^{s+7,k}(t)+\int_{\Sigma_t\cap\{r>r_0\}}(\Box_g\phi^{s+5,k})^2+\int_{\tilde\Sigma_t\cap\{r>r_0\}}(\Box_{\tilde{g}}\psi^{s+5,k})^2.
\end{multline*}
To prove the lemma, it suffices to show that
$$\int_{\Sigma_t}r^{2p}(\Box_g\phi^{s+5,k})^2+\int_{\tilde\Sigma_t}r^{2p}(\Box_{\tilde{g}}\psi^{s+5,k})^2 \lesssim E_{2p}^{s+5,k+1}(t)+E_{2p}^{s+7,k}(t)$$
and
$$\int_{\Sigma_t\cap\{r>r_0\}}(\Box_g\phi^{s+5,k})^2+\int_{\tilde\Sigma_t\cap\{r>r_0\}}(\Box_{\tilde{g}}\psi^{s+5,k})^2 \lesssim E^{s+5,k+1}(t)+E^{s+7,k}(t).$$
To do this, we use the equations for $\Box_g\phi^{s+5,k}$ and $\Box_{\tilde{g}}\psi^{s+5,k}$. This will result in a number of linear and nonlinear terms. While the linear terms can be estimated directly by the energy norms, the nonlinear terms will again require the use of pointwise estimates. For this reason, the proof of the lemma requires a nested bootstrap argument.

The bootstrap assumptions for this argument are that for $r\ge r_H$,
\begin{multline}\label{infinity_inner_Ep_ba}
|r^{p+1}\bar{D}\fd^l\phi^{s-l,k}|^2+|r^pD\fd^l\phi^{s-l,k}|^2+|r^p\fd^l\phi^{s-l,k}|^2 \\
+|r^{p+3}\bar{D}\fd^l\psi^{s-l,k}|^2+|r^{p+2}D\fd^l\psi^{s-l,k}|^2+|r^{p+2}\fd^l\psi^{s-l,k}|^2 \\
\le C_b\left(E_{2p}^{s+5,k+1}(t)+E_{2p}^{s+7,k}(t)\right)
\end{multline}
and for $r\ge r_0>r_H$,
\begin{equation}\label{infinity_inner_E_ba}
|rD\fd^l\phi^{s-l,k}|^2+|r^3D\fd^l\psi^{s-l,k}|^2\le C_b \left(E^{s+5,k+1}(t)+E^{s+7,k}(t) \right).
\end{equation}

We use one representative example to illustrate the bootstrap argument. Consider the term 
$$\frac{L\phi^s\pd_r\phi+L\phi\pd_r\phi^s}{1+\phi},$$
which shows up in the equation for $\Box_g\phi^s$ (see Proposition \ref{N_identities_prop}). We estimate this term as follows.
\begin{align*}
\int_{\Sigma_t}r^{2p}\left(\frac{L\phi^s\pd_r\phi+L\phi\pd_r\phi^s}{1+\phi}\right)^2
\lesssim \int_{\Sigma_t}r^{2p}(L\phi^s)^2||\pd_r\phi||^2_{L^\infty(\Sigma_t)} +\int_{\Sigma_t}r^{2p-2}(\pd_r\phi^s)^2||rL\phi||^2_{L^\infty(\Sigma_t)}.
\end{align*}
Then  we apply the bootstrap assumption (\ref{infinity_inner_Ep_ba}) with $p=0$ to estimate the $L^\infty$ norms.
\begin{multline*}
\int_{\Sigma_t}r^{2p}\left(\frac{L\phi^s\pd_r\phi+L\phi\pd_r\phi^s}{1+\phi}\right)^2 \\
\lesssim \int_{\Sigma_t}r^{2p}(L\phi^s)^2C_b\left(E_0^{5,1}(t)+E_0^{7,0}(t)\right) +\int_{\Sigma_t}r^{2p-2}(\pd_r\phi^s)^2C_b\left(E_0^{5,1}(t)+E_0^{7,0}(t)\right).
\end{multline*}
(In particular, since we only needed to use the $p=0$ norms, this estimate is not particularly delicate.) Next, we use the bootstrap assumption (\ref{_29_ba}) for the main theorem to estimate the energy norms.
$$\int_{\Sigma_t}r^{2p}\left(\frac{L\phi^s\pd_r\phi+L\phi\pd_r\phi^s}{1+\phi}\right)^2 
\lesssim \int_{\Sigma_t}r^{2p}(L\phi^s)^2C_b(C_bI_0) +\int_{\Sigma_t}r^{2p-2}(\pd_r\phi^s)^2C_b(C_bI_0).$$
And finally, we estimate the remaining integrated quantities by the energy norm $E_{2p}^{s,0}(t)$.
$$\int_{\Sigma_t}r^{2p}\left(\frac{L\phi^s\pd_r\phi+L\phi\pd_r\phi^s}{1+\phi}\right)^2 \lesssim C_b^2I_0E_{2p}^{s,0}(t).$$
As a side note, this procedure will be repeatedly used in the remainder of the proof of the main theorem.

By repeating the procedure in the above example for all the different terms that arise in the equations for $\Box_g\phi^{s+5,k}$ and $\Box_{\tilde{g}}\psi^{s+5,k}$, we eventually conclude that
$$\int_{\Sigma_t}r^{2p}(\Box_g\phi^{s+5,k})^2+\int_{\tilde\Sigma_t}r^{2p}(\Box_{\tilde{g}}\psi^{s+5,k})^2 \lesssim \sum_{j=0}^{s+5+k}(C_b^2I_0)^j\left( E_{2p}^{s+5,k+1}(t)+E_{2p}^{s+7,k}(t)\right)$$
and
$$\int_{\Sigma_t\cap\{r>r_0\}}(\Box_g\phi^{s+5,k})^2+\int_{\tilde\Sigma_t\cap\{r>r_0\}}(\Box_{\tilde{g}}\psi^{s+5,k})^2 \lesssim \sum_{j=0}^{s+5+k}(C_b^2I_0)^j\left(E^{s+5,k+1}(t)+E^{s+7,k}(t)\right).$$
(The exponent $j$ corresponds to the number of times a differential operator acts on the denominator $1+\phi$--see Lemmas \ref{Nphi_s_k_structure_lem} and \ref{Npsi_s_k_structure_lem}. It can also be zero because of the presence of linear terms.) By taking $I_0$ sufficiently small so that $C_b^2I_0\lesssim 1$, the dependence on the bootstrap constant $C_b$ can be removed. The resulting estimates can be used together with Proposition \ref{infinity_prop} to recover the bootstrap assumptions (\ref{infinity_inner_Ep_ba}-\ref{infinity_inner_E_ba}). This completes the proof of the lemma. \qed
\end{proof}
The conclusion of this lemma is the same as the statement of part (II) of the main theorem.

\begin{remark}
Lemma \ref{simplified_pointwise_lemma}, which is a simplified version of Proposition \ref{infinity_prop} in the sense that there are no error terms on the right hand side of any of the estimates in Lemma \ref{simplified_pointwise_lemma}, will be used in the remainder of the proof of the main theorem as a replacement for Proposition \ref{infinity_prop}.
\end{remark}

\subsection{Recovering the bootstrap assumption for $|\phi|$}\label{recover_phi_ba_sec}

With the improved pointwise estimates, we can now recover the bootstrap assumption (\ref{phi_ba}) for $|\phi|$. The pointwise estimates imply that
$$|\phi|^2\lesssim E_0^{5,1}(t)+E_0^{7,0}(t).$$
The bootstrap assumption (\ref{_29_ba}) implies that
$$E_0^{5,1}(t)+E_0^{7,0}(t)\lesssim C_b I_0.$$
Thus,
$$|\phi|^2\lesssim C_b I_0.$$
So as long as $I_0$ is sufficiently small, this guarantees that $|\phi|$ is also sufficiently small. This fact allows us to recover the bootstrap assumtion (\ref{phi_ba}).

\subsection{The highest order (homogeneous) case}\label{homogeneous_case_sec}

As explained in the introduction of this section, this is the point where we turn to the first of two parts of the remainder of the proof. In particular, we will recover the bootstrap assumptions (\ref{energy_ba}-\ref{_23_ba}) for the homogeneous norms $\mathring{E}^s(t)$, $\mathring{E}_p^s(t)$, etc.

To begin, in \S\ref{part_1_refined_estimates_sec} we prove refined estimates for the nonlinear norms $\mathring{N}^s(t)$ and $\mathring{N}_p^s(t)$. These estimates constitute the crucial step of the proof, because they handle the nonlinear terms near $i^0$ using pointwise estimates sharply. The remainder of \S\ref{homogeneous_case_sec} simply applies these estimates to recover the bootstrap assumptions for the homogeneous norms. In \S\ref{part_1_E_boundedness_sec} and \S\ref{part_1_Ep_boundedness_sec}, the bootstrap assumptions (\ref{energy_ba}) and (\ref{_29_ba}) at the highest level $s=29$ are recovered. Then in \S\ref{part_1_decay_sec}, a decay lemma is proved and used to recover the bootstrap assumption (\ref{_25_ba}) at the level $s=25$. Finally, in \S\ref{part_1_weak_decay_sec}, the bootstrap assumption (\ref{_23_ba}) at the level $s=23$, which assumes a weaker form of decay, is recovered.

\subsubsection{Refined estimates for $\mathring{N}^s(t)$ and $\mathring{N}_p^s(t)$}\label{part_1_refined_estimates_sec}

The pointwise estimates given in Lemma \ref{simplified_pointwise_lemma} allow us to provide refined estimates for the nonlinear error terms. \textbf{This is the crucial step of the proof.}

\begin{lemma}\label{mr_refined_nl_lem}
In the context of the bootstrap assumptions provided in \S\ref{bootstrap_assumptions_sec}, if $s\le 29$, then
$$\mathring{N}^s(t) \lesssim (\mathring{E}^s(t))^{1/2}\left((\mathring{E}^s(t))^{1/2}(E_{\delm-1}^{\ul{22}}(t))^{1/2}+(\mathring{E}^s_{1-\delm}(t))^{1/2}(E^{\ul{22}}_{\delm-1}(t))^{1/2}\right),$$
$$\mathring{N}_p^s(t) \lesssim \mathring{E}^s(t)B_p^{\ul{22}}(t)+\mathring{B}_p^s(t)E^{\ul{22}}(t)+\mathring{E}^s_{p'}(t)(E_{p''}^{\ul{22}}(t))^{1/2}.$$
\end{lemma}
\begin{proof}
First, we recall the definitions of $\mathring{N}^s(t)$ and $\mathring{N}_p^s(t)$ from Theorem \ref{p_o_thm}.
$$\mathring{N}^s(t)=(\mathring{E}^s(t))^{1/2}\left(||\pd_t^s\mathcal{N}_\phi||_{L^2(\Sigma_t)}+||\pd_t^s\mathcal{N}_\psi||_{L^2(\tilde{\Sigma}_t)}\right),$$
\begin{multline*}
\mathring{N}_p^s(t)=\int_{\Sigma_t}r^{p+1}(\pd_t^s\mathcal{N}_\phi)^2 + \int_{\tilde\Sigma_t}r^{p+1}(\pd_t^s\mathcal{N}_\psi)^2 \\
+ \int_{\Sigma_t\cap\{r\approx r_{trap}\}}|\pd_t^{s+1}\phi\pd_t^s\mathcal{N}_\phi| + \int_{\tilde\Sigma_t\cap\{r\approx r_{trap}\}}|\pd_t^{s+1}\psi\pd_t^s\mathcal{N}_\psi|. 
\end{multline*}
Therefore, it suffices to prove the following three estimates.
\begin{equation*}
||\pd_t^s\mathcal{N}_\phi||_{L^2(\Sigma_t)}+||\pd_t^s\mathcal{N}_\psi||_{L^2(\tilde{\Sigma}_t)}
\lesssim (\mathring{E}^s(t))^{1/2}(E_{\delm-1}^{\ul{22}}(t))^{1/2}+(\mathring{E}^s_{1-\delm}(t))^{1/2}(E^{\ul{22}}_{\delm-1}(t))^{1/2}
\end{equation*}
\begin{equation*}
\int_{\Sigma_t\cap\{r>R\}}r^{p+1}(\pd_t^s\mathcal{N}_\phi)^2 + \int_{\tilde\Sigma_t\cap\{r>R\}}r^{p+1}(\pd_t^s\mathcal{N}_\psi)^2 \lesssim \mathring{E}^s(t)B_p^{\ul{22}}(t)+\mathring{B}_p^s(t)E^{\ul{22}}(t)
\end{equation*}
\begin{multline*}
\int_{\Sigma_t\cap\{r<R\}}(\pd_t^s\mathcal{N}_\phi)^2 + \int_{\tilde\Sigma_t\cap\{r<R\}}(\pd_t^s\mathcal{N}_\psi)^2 \\
+ \int_{\Sigma_t\cap\{r\approx r_{trap}\}}|\pd_t^{s+1}\phi\pd_t^s\mathcal{N}_\phi| + \int_{\tilde\Sigma_t\cap\{r\approx r_{trap}\}}|\pd_t^{s+1}\psi\pd_t^s\mathcal{N}_\psi| \\
\lesssim \mathring{E}^s_{p'}(t)(E_{p''}^{\ul{22}}(t))^{1/2} 
\end{multline*}
We begin by observing that the second estimate is a consequence of Proposition \ref{strategy_revisited_prop} (adapted to the homogeneous case) with two unimportant changes. The first is that the norms $\mathring{E}_{p-1}^s(t)$ and $E_{p-1}^{s/2+8}(t)$ used in Proposition \ref{strategy_revisited_prop} have been replaced with the norms $\mathring{B}_p^s(t)$ and $B_p^{\ul{22}}(t)$, since they are equivalent in a region excluding the trapping radius. The second is that the factor of $\sum_{j\le 29}(\mathring{E}^{29}(t))^j$ that appears on the right hand side of the estimate from Proposition \ref{strategy_revisited_prop} has been eliminated. Given the bootstrap assumption (\ref{energy_ba}) that $\mathring{E}^{29}(t)\lesssim C_bI_0$ and the fact that $I_0$ can be chosen sufficiently small, this factor is unimportant.

The third estimate is much simpler. Since all of the integrated quantites are defined on a bounded radius, the $r$ factors can be replaced with constants and therefore we have the freedom to choose $p'$ and $p''$. Using again the fact that $C_bI_0\lesssim 1$, we ignore additional factors of any energy norm in this estimate.

The first estimate is a hybrid of the other two. The integrals on a bounded radius can be estimated by $(\mathring{E}_{1-\delm}^s(t))^{1/2}(E_{\delm-1}^{\ul{22}}(t))^{1/2}$ for the same reason as in the third estimate. The integrals over the remaining region $r>R$ can be estimated using the same strategy that is given in Proposition \ref{strategy_revisited_prop}. One simply has to check that the appropriate powers of $r$ can be assigned to each factor. \qed
\end{proof}

\begin{corollary}\label{mr_NL_absorb_bulk}
In the context of the bootstrap assumptions provided in \S\ref{bootstrap_assumptions_sec}, Theorem \ref{p_o_thm} and Lemma \ref{mr_refined_nl_lem} imply that if $s\le 29$ and $C_bI_0$ is sufficiently small, then
$$\mathring{E}_p^{s}(t_2)+\int_{t_1}^{t_2}\mathring{B}_p^{s}(t)dt \lesssim \mathring{E}_p^{s}(t_1)+\int_{t_1}^{t_2}\mathring{E}^s(t)B_p^{\ul{22}}(t)dt+(C_bI_0)^{1/2}\int_{t_1}^{t_2}\mathring{E}_p^{s}(t)T^{(\delm-3+\delp)/2}dt.$$
\end{corollary}
\begin{proof}
According to Theorem \ref{p_o_thm},
$$\mathring{E}_p^s(t_2)+\int_{t_1}^{t_2}\mathring{B}_p^s(t)dt\lesssim \mathring{E}_p^s(t_1)+\int_{t_1}^{t_2}\mathring{N}_p^s(t)dt.$$
By Lemma \ref{mr_refined_nl_lem},
$$\mathring{N}_p^s(t) \lesssim \mathring{E}^s(t)B_p^{\ul{22}}(t)+\mathring{B}_p^s(t)E^{\ul{22}}(t)+\mathring{E}^s_p(t)(E_{\delm-1}^{\ul{22}}(t))^{1/2}.$$
Using the bootstrap assumptions,
$$\int_{t_1}^{t_2}\mathring{B}_p^s(t)E^{\ul{22}}(t)dt \lesssim \int_{t_1}^{t_2}\mathring{B}_p^{s}(t)C_bI_0dt \lesssim C_bI_0\int_{t_1}^{t_2}\mathring{B}_p^{s}(t)dt.$$
Thus, if $C_bI_0$ is sufficiently small, then this error term can be absorbed into the bulk term on the left hand side.

Again using the bootstrap assumptions and the weak decay principle,
$$\int_{t_1}^{t_2}\mathring{E}_{p}^s(t)E_{\delm-1}^{\ul{22}}(t)dt \lesssim \int_{t_1}^{t_2}\mathring{E}_p^s(t)(C_bI_0)^{1/2}T^{(\delm-3+\delp)/2}dt \lesssim (C_bI_0)^{1/2}\int_{t_1}^{t_2}\mathring{E}_p^s(t)T^{(\delm-3+\delp)/2}dt.$$
These estimates are sufficient. \qed
\end{proof}

\subsubsection{Recovering boundedness of $\mathring{E}^{29}(t)$}\label{part_1_E_boundedness_sec}

By Theorem \ref{p_o_thm} and Lemma \ref{mr_refined_nl_lem},
\begin{align*}
\mathring{E}^{29}(t) \lesssim & \mathring{E}^{29}(0)+\int_0^tN^{29}(\tau)d\tau \\
\lesssim & \mathring{E}^{29}(0)+\int_0^t(\mathring{E}^{29}(\tau))^{1/2}\left((\mathring{E}^{29}(\tau))^{1/2}E_{\delm-1}^{\ul{22}}(\tau))^{1/2}+(\mathring{E}^{29}_{1-\delm}(\tau))^{1/2}(E^{\ul{22}}_{\delm-1}(\tau))^{1/2}\right)d\tau \\
\lesssim & I_0+\int_0^t(C_bI_0)^{1/2}(C_bI_0)^{1/2}(C_bT^{\delm-3+\delp}I_0)^{1/2}d\tau \\
\lesssim & (1+C_b^{3/2}I_0^{1/2})I_0.
\end{align*}
In particular, we used the weak decay principle in the third step. It follows that if $I_0$ is chosen sufficiently small so that $C_b^3I_0\lesssim 1$,
$$\mathring{E}^{29}(t)\lesssim I_0.$$
This recovers the bootstrap assumption (\ref{energy_ba}) at the highest level of derivatives.

\subsubsection{Recovering boundedness of $\mathring{E}_{2-\delp}^{29}(t)$}\label{part_1_Ep_boundedness_sec}

By Corollary \ref{mr_NL_absorb_bulk},
\begin{multline*}
\mathring{E}_{2-\delp}^{29}(t)+\int_0^t\mathring{B}_{2-\delp}^{29}(\tau)d\tau \\
\lesssim \mathring{E}_{2-\delp}^{29}(0)+\int_0^t\mathring{E}^{29}(\tau)B_{2-\delp}^{\ul{22}}(\tau)d\tau + (C_bI_0)^{1/2}\int_0^t \mathring{E}_{2-\delp}^{29}(\tau)T^{(\delm-3+\delp)/2}d\tau \\
\lesssim I_0+I_0\int_0^tB_{2-\delp}^{\ul{22}}(\tau)d\tau+(C_bI_0)^{1/2}\int_0^tC_bI_0T^{(\delm-3+\delp)/2}d\tau \\
\lesssim I_0+C_bI_0^2+(C_bI_0)^{3/2}.
\end{multline*}
 It follows that if $C_b^3I_0\lesssim 1$, then
$$\mathring{E}_{2-\delp}^{29}(t)+\int_0^t\mathring{B}_{2-\delp}^{29}(\tau)d\tau \lesssim I_0.$$
This recovers the bootstrap assumption (\ref{_29_ba}) at the highest level of derivatives.

\subsubsection{Proving decay for $\mathring{E}_{p\in[1-\delp,2-\delp]}^{28}(t)$ and $\mathring{E}_{p\in[\delm,2-\delp]}^{27}(t)$}\label{part_1_decay_sec}

We now use the heirarchy of the $p$-weighted energy estimates to prove decay in time. For convenience, we prove decay in the following lemma, which will then be repeatedly used for a different values of $p$. Essentially, this lemma states that if a $(p+1)$-weighted energy norm is bounded or decays at a certain rate, then the $p$-weighted energy norm decays at a rate with an additional factor of $T^{-1}$. This is the source of the tradeoff between a factor of $r$ and a factor of $t$.
\begin{lemma}\label{mr_NL_WE_decay}
Suppose  $p+1,p\in [\delm,2-\delp]$ and
$$\mathring{E}_{p+1}^{s+1}(t)\lesssim T^{(p+1)-2+\delp}I_0,$$
$$\int_t^\infty B_{p+1}^{\ul{22}}(\tau)d\tau \le C_bT^{(p+1)-2+\delp}I_0,$$
$$\mathring{E}_p^s(t)\le C_bT^{p-2+\delp}I_0,$$
$$\int_t^{\infty}B_p^{\ul{22}}(\tau)d\tau \le C_bT^{p-2+\delp}I_0.$$
Then if $I_0$ is sufficiently small,
$$\mathring{E}_p^s(t)\lesssim T^{p-2+\delp}I_0.$$
\end{lemma}
\begin{proof}
Using the mean value theorem, for a given $t$, let $t'\in [t/2,t]$ be the value for which $\mathring{B}_{p+1}^{s+1}(t')=\frac2t\int_{t/2}^t\mathring{B}_{p+1}^{s+1}(\tau)d\tau$. Then using Corollary \ref{mr_NL_absorb_bulk},
\begin{align*}
\mathring{E}_p^{s}(t) &\lesssim \mathring{E}_p^{s}(t')+\int_{t'}^{t}\mathring{E}^s(\tau)B_p^{\ul{22}}(\tau)d\tau+(C_bI_0)^{1/2}\int_{t'}^{t}\mathring{E}_p^{s}(\tau)T^{(\delm-3+\delp)/2}d\tau \\
&\lesssim \mathring{E}_p^s(t')+I_0\int_{t'}^tB_p^{\ul{22}}(\tau)d\tau +(C_bI_0)^{3/2}\int_{t'}^t T^{p-2+\delp}T^{(\delm-3+\delp)/2}d\tau \\
&\lesssim \mathring{E}_p^s(t')+C_bI_0^2T^{p-2+\delp}+(C_bI_0)^{3/2}T^{p-2+\delp}T^{\delm+\delp-1/2} \\
&\lesssim \mathring{E}_p^s(t')+(C_bI_0+C_b^{3/2}I_0^{1/2})T^{p-2+\delp}I_0.
\end{align*}
Now by the choice of $t'$, and another application of Corollary \ref{mr_NL_absorb_bulk},
\begin{multline*}
\mathring{E}_p^s(t')\lesssim \mathring{B}_{p+1}^{s+1}(t')=\frac2t\int_{t/2}^t\mathring{B}_{p+1}^{s+1}(\tau)d\tau \\
\lesssim t^{-1}\mathring{E}_{p+1}^{s+1}(t/2)+t^{-1}\int_{t/2}^t\mathring{E}^s(\tau)B_{p+1}^{\ul{22}}(\tau)d\tau +t^{-1}(C_bI_0)^{1/2}\int_{t/2}^t \mathring{E}_{p+1}^{s+1}(\tau)T^{(\delm-3+\delp)/2}d\tau. \\
\lesssim t^{-1}(T^{(p+1)-2+\delp}I_0+C_bI_0^2T^{(p+1)-2+\delp}+(C_bI_0)^{3/2}T^{(p+1)-2+\delp}).
\end{multline*}
Thus,
$$E_p^{s,k}(t)\lesssim (1+C_bI_0+C_b^{3/2}I_0^{1/2})T^{p-2+\delp}I_0.$$
Taking $I_0$ sufficiently small so that $C_b^3I_0\lesssim 1$ completes the proof. \qed
\end{proof}

By applying Lemma \ref{mr_NL_WE_decay} for $p=1-\delp$ and $s=28$, we obtain
$$\mathring{E}_{1-\delp}^{28}(t) \lesssim T^{-1}I_0.$$
By interpolation we obtain decay for all $p\in [1-\delp,2-\delp]$.
$$\mathring{E}_{p\in[1-\delp,2-\delp]}^{28}(t)\lesssim T^{p-2+\delp}I_0.$$
Then applying Lemma \ref{mr_NL_WE_decay} again for each $p\in [\delm,1-\delp]$ and $s=27$, we obtain
$$\mathring{E}_{p\in[\delm,2-\delp]}^{27}(t)\lesssim T^{p-2+\delp}I_0.$$
In particular, by taking $p=\delm$, this recovers the bootstrap assumption (\ref{_25_ba}) at the highest level of derivatives.

\subsubsection{Recovering weak decay for $\mathring{E}_{p\in[\delm-1,\delm]}^{26}(t)$}\label{part_1_weak_decay_sec}

To prove estimates for low $p$ (ie. $p$ in the range $[\delm-1,\delm]$), we first observe that
$$\int_t^\infty \mathring{E}_p^{26}(\tau)d\tau \lesssim \int_t^\infty \mathring{B}_{p+1}^{27}(\tau)d\tau.$$
Then by Corollary \ref{mr_NL_absorb_bulk},
\begin{align*}
\int_t^\infty \mathring{B}_{p+1}^{27}(\tau)d\tau &\lesssim \mathring{E}_{p+1}^{27}(t)+\int_t^\infty \mathring{E}^{27}(\tau)B_{p+1}^{\ul{22}}(\tau)d\tau+(C_bI_0)^{1/2}\int_t^\infty\mathring{E}_{p+1}^{27}(\tau)T^{(\delm-3+\delp)/2}d\tau \\
&\lesssim T^{(p+1)-2+\delp}I_0+I_0\int_t^\infty B_{p+1}^{\ul{22}}(\tau)d\tau+(C_bI_0)^{1/2}\int_t^\infty T^{(p+1)-2+\delp}I_0T^{(\delm-3+\delp)/2}d\tau \\
&\lesssim (1+C_bI_0+(C_bI_0)^{1/2})T^{(p+1)-2+\delp}I_0.
\end{align*}
It follows that if $I_0$ is sufficiently small so that $C_bI_0\lesssim 1$, then
$$\int_t^\infty \mathring{E}_p^{26}(\tau)d\tau \lesssim T^{(p+1)-2+\delp}I_0 = T^{p-2+\delp +1}I_0.$$
This recovers the bootstrap assumption (\ref{_23_ba}) at the highest level of derivatives.

\subsection{The general case $s+2k\le 28$}\label{general_case_s_k_sec}

As explained in the introduction of this section, this is the point where we turn to the second of two parts of the remainder of the proof. In particular, we will recover the bootstrap assumptions (\ref{energy_ba}-\ref{_23_ba}) for the norms $E^{s,k}(t)$, $E_p^{s,k}(t)$, etc. We will use results from the first part (\S\ref{homogeneous_case_sec}) to handle the homogeneous norm $\mathring{B}_p^s(t)$ appearing on the right hand side of many estimates in this second part.

The outline for \S\ref{general_case_s_k_sec} is similar to, but slightly more complicated than, the outline for \S\ref{homogeneous_case_sec}. To begin, in \S\ref{part_2_refined_estimates_sec}, we prove refined estimates for the nonlinear norms $N^{s,k}(t)$ and $N_p^{s,k}(t)$. Just like the estimates in \S\ref{part_1_refined_estimates_sec}, these estimates constitute the crucial step of the proof. The remainder of \S\ref{general_case_s_k_sec} applies these estimates to recover the bootstrap assumptions for the energy norms in a finite induction argument. The inductive assumptions are listed in \S\ref{part_2_inductive_assumptions_sec}. In \S\ref{part_2_Ep_boundedness_sec} and \S\ref{part_2_E_boundedness_sec}, the bootstrap assumptions (\ref{_29_ba}) and (\ref{energy_ba}) at the highest level $s+2k=28$ are recovered in that order. Then in \S\ref{part_2_decay_sec}, a decay lemma is proved and used to recover the bootstrap assumption (\ref{_25_ba}) at the level $s+2k=24$. Next, in \S\ref{establish_inductive_assumptions_sec} the inductive assumptions for the next step $k+1$ are established. Finally, in \S\ref{part_2_weak_decay_sec}, the bootstrap assumption (\ref{_23_ba}) at the level $s+2k=22$, which assumes a weaker form of decay, is recovered.

\subsubsection{Refined estimates for $N^{s,k}(t)$ and $N_p^{s,k}(t)$}\label{part_2_refined_estimates_sec}

The pointwise estimates given in Lemma \ref{simplified_pointwise_lemma} allow us to provide refined estimates for the nonlinear error terms. \textbf{This is the crucial step of the proof.}

\begin{lemma}\label{sk_refined_nl_lem}
In the context of the bootstrap assumptions provided in \S\ref{bootstrap_assumptions_sec}, if $s+2k\le 28$ and $C_bI_0\le 1$, then
$$N^{s,k}(t)\lesssim (E^{s,k}(t))^{1/2}\left((E^{s,k}(t))^{1/2}(E_{\delm-1}^{\ul{23}}(t))^{1/2}+(E^{s,k}_{1-\delm}(t))^{1/2}(E^{\ul{23}}_{\delm-1}(t))^{1/2}\right),$$
$$N_p^{s,k}(t)\lesssim E^{s,0}(t)B_p^{s/2+8,0}(t)+B_p^{s,0}(t)E^{s/2+8,0}(t)+E^{s,k}_{p'}(t)(E_{p''}^{\ul{23}}(t))^{1/2}.$$
\end{lemma}
\begin{remark}
The careful reader may notice that the low order energy norms in this lemma are at the level $s+2k=23$, whereas the low order energy norms in the analogous lemma for the homogeneous case (Lemma \ref{mr_refined_nl_lem}) are at the level $s+2k=22$. This difference is unimportant, but it is due to the fact that the $\pd_t$ commutators commute with the derivatives in the structures of $\mathcal{N}_\phi$ and $\mathcal{N}_\psi$, while the commutators $Q$ and $\tilde{Q}$ have angular parts that mix with the $\pd_\theta$ derivatives in some terms belonging to $\mathcal{N}_\phi$ and $\mathcal{N}_\psi$. This is manifest in the condition $s_1+...+s_{2+j}\le s+2$ that is given in Lemmas \ref{Nphi_s_k_structure_lem} and \ref{Npsi_s_k_structure_lem}. This means that the lower order factors can have up to $(s+2)/2=s/2+1$ derivatives, whereas in the homogeneous case, they can only have up to $s/2$ derivatives.
\end{remark}
\begin{proof}
First, we recall the definitions of $N^{s,k}(t)$ and $N_p^{s,k}(t)$ from Theorem \ref{p_s_k_thm}.
$$N^{s,k}(t)=(E^{s,k}(t))^{1/2}\left(\sum_{\substack{s'\le s \\ k'\le k}}||q^{-2}\tg^{k'}\Gamma^{s'}(q^2\mathcal{N}_\phi)||_{L^2(\Sigma_t)}+\sum_{\substack{s'\le s \\ k'\le k}}||q^{-2}\tg^{k'}\tilde\Gamma^{s'}(q^2\mathcal{N}_\psi)||_{L^2(\tilde{\Sigma}_t)}\right),$$
\begin{align*}
N_p^{s,k}(t) =& \sum_{\substack{s'\le s \\ k'\le k}}\int_{\Sigma_t}r^{p+1}|q^{-2}\tg^{k'}\Gamma^{s'}(q^2\mathcal{N}_\phi)|^2 + \sum_{\substack{s'\le s \\ k'\le k}}\int_{\tilde\Sigma_t}r^{p+1}|q^{-2}\tg^{k'}\tilde\Gamma^{s'}(q^2\mathcal{N}_\psi)|^2.
\end{align*}
Therefore, it suffices to prove the following three estimates.
\begin{multline*}
\sum_{\substack{s'\le s \\ k'\le k}}||q^{-2}\tg^{k'}\Gamma^{s'}(q^2\mathcal{N}_\phi)||_{L^2(\Sigma_t)}+\sum_{\substack{s'\le s \\ k'\le k}}||q^{-2}\tg^{k'}\tilde\Gamma^{s'}(q^2\mathcal{N}_\psi)||_{L^2(\tilde{\Sigma}_t)} \\
\lesssim (E^{s,k}(t))^{1/2}(E_{\delm-1}^{\ul{23}}(t))^{1/2}+(E^{s,k}_{1-\delm}(t))^{1/2}(E^{\ul{23}}_{\delm-1}(t))^{1/2}
\end{multline*}
\begin{multline*}
\sum_{\substack{s'\le s \\ k'\le k}}\int_{\Sigma_t\cap\{r>R\}}r^{p+1}|q^{-2}\tg^{k'}\Gamma^{s'}(q^2\mathcal{N}_\phi)|^2 + \sum_{\substack{s'\le s \\ k'\le k}}\int_{\tilde\Sigma_t\cap\{r>R\}}r^{p+1}|q^{-2}\tg^{k'}\tilde\Gamma^{s'}(q^2\mathcal{N}_\psi)|^2 \\
\lesssim E^{s,0}(t)B_p^{s/2+8,0}(t)+B_p^{s,0}(t)E^{s/2+8,0}(t)
\end{multline*}
\begin{multline*}
\sum_{\substack{s'\le s \\ k'\le k}}\int_{\Sigma_t\cap\{r<R\}}|q^{-2}\tg^{k'}\Gamma^{s'}(q^2\mathcal{N}_\phi)|^2 + \sum_{\substack{s'\le s \\ k'\le k}}\int_{\tilde\Sigma_t\cap\{r<R\}}|q^{-2}\tg^{k'}\tilde\Gamma^{s'}(q^2\mathcal{N}_\psi)|^2 \\
\lesssim E^{s,k}_{p'}(t)(E_{p''}^{\ul{23}}(t))^{1/2}
\end{multline*}
These estimates are analogous to the estimates in the proof of Lemma \ref{mr_refined_nl_lem}, and they can be proved the same way. \qed
\end{proof}

\begin{corollary}\label{sk_NL_absorb_bulk}
In the context of the bootstrap assumptions provided in \S\ref{bootstrap_assumptions_sec}, Theorem \ref{p_s_k_thm} and Lemma \ref{sk_refined_nl_lem} imply that if $s+2k\le 28$ and $C_bI_0$ is sufficiently small, then
\begin{multline*}
E_p^{s,k}(t_2)+\int_{t_1}^{t_2}B_p^{s,k}(t)dt \\
\lesssim E_p^{s,k}(t_1)+\int_{t_1}^{t_2}\mathring{B}_p^{s+1}(t)dt+(C_bI_0)^{1/2}\int_{t_1}^{t_2}E_p^{s,k}(t)T^{(\delm-3+\delp)/2}dt
+\sum_{\substack{s'+2k'\le s+2k \\ k'<k}}\int_{t_1}^{t_2}B_p^{s',k'}(t)dt.
\end{multline*}
whenever $s/2+8\le s+2k$ (which means $16\le s+4k$).
\end{corollary}
\begin{proof}
According to Theorem \ref{p_s_k_thm},
$$E_p^{s,k}(t_2)+\int_{t_1}^{t_2}B_p^{s,k}(t)dt\lesssim E_p^{s,k}(t_1)+\int_{t_1}^{t_2}B_p^{s+2,k-1}(t)+\mathring{B}_p^{s+1}(t)+N_p^{s,k}(t)dt.$$
By Lemma \ref{sk_refined_nl_lem},
$$N_p^{s,k}(t)\lesssim E^{s,0}(t)B_p^{s/2+8,0}(t)+B_p^{s,0}(t)E^{s/2+8,0}(t)+E^{s,k}_p(t)(E_{\delm-1}^{\ul{23}}(t))^{1/2}.$$
We now estimate each of the three terms on the right hand side. The strategy for the first term depends on whether $k=0$ or $k>0$. If $k=0$, then since $s/2+8\le s$ by assumption,
$$E^{s,0}(t)B_p^{s/2+8,0}(t) \lesssim C_bI_0 B_p^{s,0}(t).$$
It follows that if $C_bI_0$ is sufficiently small, this term can be absorbed into the bulk term on the left hand side. If instead $k> 0$, then since $s/2+8\le s+2k$ by assumption,
$$E^{s,0}(t)B_p^{s/2+8,0}(t) \lesssim C_bI_0 \sum_{\substack{s'+2k'\le s+2k \\ k'<k}}B_p^{s',k'}(t).$$
It follows that if $I_0$ is sufficiently small so that $C_bI_0\le 1$, then this term can be estimated by the lower order (in $k$) term on the right hand side.

The strategy for the second term is similar. Since $s+2k\le 28$, it follows that $s/2+8\le 28$, so
$$B_p^{s,0}(t)E^{s/2+8,0}(t)\lesssim C_bI_0B_p^{s,0}(t).$$
Again, if $k=0$, then $C_bI_0$ must be taken sufficiently small so that this term can be absorbed into the bulk term on the left hand side. If instead $k>0$, then since $s\le s+2k$, then as long as $C_bI_0\le 1$, this term can be estimated by the lower order (in $k$) term on the right hand side.

Finally, by the weak decay lemma,
$$\int_{t_1}^{t_2}E_p^{s,k}(t)(E_{\delm-1}^{\ul{23}}(t))^{1/2}dt \lesssim (C_bI_0)^{1/2}\int_{t_1}^{t_2}E_p^{s,k}(t)T^{(\delm-3+\delp)/2}dt.$$
This completes the proof. \qed
\end{proof}

\subsubsection{Inductive assumptions}\label{inductive_assumptions_sec}\label{part_2_inductive_assumptions_sec}

The remainder of the proof is a finite induction argument. First, estimates are proved for $k=0$, and then for $k=1$, etc. until $k=14$ (which saturates $s+2k\le 28$). For each $k$, it will be necessary to use estimates established for $k-1$. These inductive assumptions are listed here.

Either
$$k=0,$$
or $s+2k=28$ and
$$\sum_{\substack{s'+2k'\le 28 \\ k'<k}}\int_t^\infty B_{2-\delp}^{s',k'}(\tau)d\tau \lesssim I_0,$$
or $s+2k=26$ and
$$\sum_{\substack{s'+2k'\le 26 \\ k'<k}}\int_t^\infty B_{1-\delp}^{s',k'}(\tau)d\tau \lesssim T^{(1-\delp)-2+\delp}I_0,$$
or $s+2k=24$ and
$$\sum_{\substack{s'+2k'\le 24 \\ k'<k}}\int_t^\infty B_{\delm}^{s',k'}(\tau)d\tau \lesssim T^{\delm-2+\delp}I_0.$$

For the remainder of the proof, $k$ should be considered fixed. These estimates will be used for the fixed $k$, and eventually (in \S\ref{establish_inductive_assumptions_sec}) the corresponding estimates obtained by replacing $k$ with $k+1$ will be proved, thus closing the induction argument.

\subsubsection{Recovering boundedness of $E^{s,k}_{2-\delp}(t)$ ($s+2k=28$)}\label{part_2_Ep_boundedness_sec}

Our first application of Corollary \ref{sk_NL_absorb_bulk} is to prove boundedness of $E^{s,k}_{2-\delp}(t)$ and $\int_0^t B_{2-\delp}^{s,k}(\tau)d\tau$. With $s+2k=28$,
\begin{multline*}
E_{2-\delp}^{s,k}(t)+\int_0^t B_{2-\delp}^{s,k}(\tau)d\tau \\
\lesssim E_{2-\delp}^{s,k}(0)+\int_0^t\mathring{B}^{29}(\tau)d\tau+(C_bI_0)^{1/2}\int_0^tE_{2-\delp}^{s,k}(\tau)T^{(\delm-3+\delp)/2}d\tau+\sum_{\substack{s'+2k'\le s+2k \\ k'<k}}\int_0^tB_{2-\delp}^{s',k'}(\tau)d\tau \\
\lesssim I_0+I_0+(C_bI_0)^{1/2}\int_0^tC_bT^{(\delm-3+\delp)/2}I_0d\tau +I_0 \\
\lesssim (1+C_b^{3/2}I_0^{1/2})I_0.
\end{multline*}
It follows that if $I_0$ is sufficiently small so that $C_b^3I_0\lesssim 1$,
$$E_{2-\delp}^{s,k}(t)+\int_0^tB_{2-\delp}^{s,k}(\tau)d\tau\lesssim I_0.$$
This recovers the bootstrap assumption (\ref{_29_ba}) at the level $k$.

\subsubsection{Recovering boundedness of $E^{s,k}(t)$ ($s+2k=28$)}\label{part_2_E_boundedness_sec}

Since
$$\int_0^tB_1^{s,k}(\tau)d\tau \lesssim \int_0^tB_{2-\delp}^{s,k}(\tau)d\tau \lesssim I_0,$$
we are now able to prove that $E^{s,k}(t)$ is bounded.

Let $s+2k=28$. By Theorem \ref{p_s_k_thm} and Lemma \ref{sk_refined_nl_lem},
\begin{align*}
E^{s,k}(t) \lesssim & E^{s,k}(0)+\int_0^tB_{2-\delp}^{s+2,k-1}(\tau)+\mathring{B}_1^{s+1}(\tau)+B_1^{s,k}(\tau)+N^{s,k}(\tau)d\tau \\
\lesssim & E^{s,k}(0)+\int_0^tB_{2-\delp}^{s+2,k-1}(\tau) + \mathring{B}_1^{s+1}(\tau) + B_1^{s,k}(\tau)d\tau\\
&+\int_0^t(E^{s,k}(\tau))^{1/2}\left((E^{s,k}(\tau))^{1/2}(E_{\delm-1}^{\ul{22}}(\tau))^{1/2}+(E^{s,k}_{1-\delm}(\tau))^{1/2}(E^{\ul{22}}_{\delm-1}(\tau))^{1/2}\right)d\tau \\
\lesssim & I_0+\int_0^t(C_bI_0)^{1/2}(C_bI_0)^{1/2}(C_bT^{\delm-3+\delp}I_0)^{1/2}d\tau \\
\lesssim & (1+C_b^{3/2}I_0^{1/2})I_0.
\end{align*}
In particular, we used the weak decay principle in the third step. It follows that if $I_0$ is chosen sufficiently small so that $C_b^3I_0\lesssim 1$,
$$E^{s,k}(t)\lesssim I_0.$$
This recovers the bootstrap assumption (\ref{energy_ba}) at the level $k$.

\subsubsection{Proving decay for $E^{s,k}_{p\in[1-\delp,2-\delp]}(t)$ ($s+2k=26$) and $E^{s,k}_{p\in[\delm,2-\delp]}(t)$ ($s+2k=24$)}\label{part_2_decay_sec}

Once again, we prove a decay lemma for repeated use.
\begin{lemma}\label{sk_NL_WE_decay}
Suppose  $p+1,p\in [\delm,2-\delp]$ and
$$\int_t^\infty \mathring{B}_{p+1}^{s+3}(\tau)d\tau \lesssim T^{(p+1)-2+\delp}I_0,$$
$$E_{p+1}^{s+2,k}(t)\lesssim T^{(p+1)-2+\delp}I_0,$$
$$\sum_{\substack{s'+2k'\le s+2k+2 \\ k'<k}}\int_t^\infty B_{p+1}^{s',k'}(\tau)d\tau\lesssim T^{(p+1)-2+\delp}I_0.$$
$$\int_t^\infty \mathring{B}_p^{s+1}(\tau)d\tau \lesssim T^{p-2+\delp}I_0,$$
$$E_p^{s,k}(t)\le C_bT^{p-2+\delp}I_0,$$
$$\sum_{\substack{s'+2k'\le s+2k \\ k'<k}}\int_t^\infty B_p^{s',k'}(\tau)d\tau\lesssim T^{p-2+\delp}I_0,$$
Then if $I_0$ is sufficiently small,
$$E_p^{s,k}(t)\lesssim T^{p-2+\delp}I_0.$$
\end{lemma}

\begin{proof}
Using the mean value theorem, for a given $t$, let $t'\in [t/2,t]$ be the value for which $B_{p+1}^{s+2,k}(t')=\frac2t\int_{t/2}^tB_{p+1}^{s+2,k}(\tau)d\tau$. Then using Corollary \ref{sk_NL_absorb_bulk},
\begin{align*}
E_p^{s,k}(t) &\lesssim E_p^{s,k}(t')+\int_{t'}^t\mathring{B}_p^{s+1}(\tau)d\tau+(C_bI_0)^{1/2}\int_{t'}^tE_p^{s,k}(\tau)T^{(\delm-3+\delp)/2}d\tau+\sum_{\substack{s'+2k'\le s+2k \\ k'<k}}\int_{t'}^tB_p^{s',k'}(\tau)d\tau  \\
&\lesssim E_p^{s,k}(t')+C_bI_0\int_{t'}^t(C_bT^{p-2+\delp}I_0)T^{\delm-2+\delp}d\tau+T^{p-2+\delp}I_0 \\
&\lesssim E_p^{s,k}(t')+C_b^2I_0T^{p-2+\delp}I_0+T^{p-2+\delp}I_0.
\end{align*}
Now by the choice of $t'$,
\begin{multline*}
E_p^{s,k}(t')\lesssim B_{p+1}^{s+2,k}(t')=\frac2t\int_{t/2}^tB_{p+1}^{s+2,k}(\tau)d\tau \\
\lesssim t^{-1}E_{p+1}^{s+2,k}(t/2)+t^{-1}\int_{t/2}^t\mathring{B}_{p+1}^{s+3}(\tau)d\tau+t^{-1}C_bI_0\int_{t/2}^tE_{p+1}^{s+2,k}(\tau)T^{\delm-2+\delp}d\tau \\
+t^{-1}\sum_{\substack{s'+2k'\le s+2k+2 \\ k'<k}}\int_{t/2}^t B_{p+1}^{s',k'}(\tau)d\tau.
\end{multline*}
Thus,
$$E_p^{s,k}(t)\lesssim (1+C_b^2I_0)T^{p-2+\delp}I_0.$$
Taking $I_0$ sufficiently small so that $C_b^2I_0\lesssim 1$ completes the proof. \qed
\end{proof}

By applying Lemma \ref{sk_NL_WE_decay} for $p=1-\delp$ and $s+2k=26$, we obtain
$$E_{1-\delp}^{s,k}(t)\lesssim T^{-1}I_0.$$
By interpolation we obtain decay for all $p\in [1-\delp,2-\delp]$.
$$E_{p\in [1-\delp,2-\delp]}^{s,k}(t)\lesssim T^{p-2+\delp}I_0.$$
Then applying Lemma \ref{sk_NL_WE_decay} again for each $p\in[\delm,1-\delp]$ and $s+2k=24$, we obtain
\begin{equation*}
E_{p\in [\delm,2-\delp]}^{s,k}(t)\lesssim T^{p-2+\delp}I_0.
\end{equation*}
In particular, by taking $p=\delm$, this recovers the bootstrap assumption (\ref{_25_ba}) at the level $k$.

\subsubsection{Establishing inductive assumptions for $k+1$}\label{establish_inductive_assumptions_sec}

By Corollary \ref{sk_NL_absorb_bulk},
\begin{multline*}
\int_t^\infty B_p^{s,k}(\tau)d\tau \lesssim E_p^{s,k}(t)+\int_t^\infty\mathring{B}_p^{s+1}(\tau)d\tau+C_bI_0\int_t^\infty E_p^{s,k}(\tau)T^{\delm-2+\delp}d\tau \\
+\sum_{\substack{s'+2k'\le s+2k \\ k'<k}}\int_t^\infty B_p^{s',k'}(\tau)d\tau.
\end{multline*}
The quantities $E_p^{s,k}(t)$, $\int_t^\infty \mathring{B}_p^{s+1}(\tau)d\tau$, and $\int_t^\infty B_p^{s',k'}(\tau)d\tau$ ($s'+2k'\le s+2k$ and $k'<k$) all have the same proven decay rates. For $s+2k=28$ and $p=2-\delp$, they are bounded in time by $I_0$, for $s+2k=26$ and $p=1-\delp$, they decay at least as fast as $T^{-1}I_0$, and for $s+2k=24$ and $p=\delm$, they decay at least as fast as $T^{\delm-2+\delp}I_0$. Thus,
$$\sum_{\substack{s'+2k'\le 28 \\ k'<k+1}}\int_t^\infty B_{2-\delp}^{s',k'}(\tau)d\tau \lesssim I_0,$$
$$\sum_{\substack{s'+2k'\le 26 \\ k'<k+1}}\int_t^\infty B_{1-\delp}^{s',k'}(\tau)d\tau \lesssim T^{(1-\delp)-2+\delp}I_0,$$
$$\sum_{\substack{s'+2k'\le 24 \\ k'<k+1}}\int_t^\infty B_{\delm}^{s',k'}(\tau)d\tau \lesssim T^{\delm-2+\delp}I_0.$$
These are the inductive assumptions at the next level $k+1$.

\subsubsection{Recovering weak decay for $E_{p\in[\delm-1,\delm]}^{s,k}(t)$ ($s+2k=22$)}\label{part_2_weak_decay_sec}

Finally, set $s+2k=22$ and observe that for $p\in [\delm-1,\delm]$,
$$\int_t^\infty E_p^{s,k}(\tau)d\tau \lesssim \int_t^\infty B_{p+1}^{s+2,k}(\tau)d\tau\lesssim T^{(p+1)-2+\delp}I_0=T^{p-2+\delp+1}I_0.$$
In particular, by taking $p=\delm-1$, this recovers the bootstrap assumption (\ref{_23_ba}) at the level $k$, and thus completes the proof.

\appendix
%\section{The Covariant Formulation for Wave Maps}
%\input{xi_derivation.tex}

%\section{Derivation of the $(\phi,\psi)$ System}
%\input{phi_psi_derivation.tex}

\section{Regularity for Axisymmetric Functions}\label{regularity_sec}

In this stand-alone section, we develop a formalism for understanding terms that represent regular axisymmetric functions using certain differential operators on the unit sphere. The main motivation for this section in the context of nonlinear wave equations is summarized in \S\ref{intro_regularity_sec}. The formalism developed here can be used for a variety of wave-type problems with a coordinate degeneracy such as that in axisymmetry.

To begin, in \S\ref{gothic_operators_intro_sec}, we define the gothic operators $\fa$, $\fb$, $\fc^l$ and $\fd^l$ and prove a few basic properties. Of particular importance is Lemma \ref{cl_raise_degree_lem}, which explains how these operators act on a particular class of products of functions. In \S\ref{cl_embedding_theorem_sec}, we prove an important embedding theorem that allows us to estimate $L^2$ norms with these operators by $L^2$ norms with the standard spherical laplacian. This is important, because the spherical laplacian is a useful commutator in wave-type problems. Finally, in \S\ref{additional_regularity_lemmas_sec}, we prove a few additional related estimates that will be needed in main part of the paper.

For the remainder of this section, we will focus on functions defined on a sphere $S^2$ (or $S^6$) which depend only on the angle $\theta$ from the north pole. However, the theory discussed here easily extends to Kerr spacetimes and other axisymmetric spacetimes.

\subsection{The operators $\fa$, $\fb$, $\fc^l$, and $\fd^l$}\label{gothic_operators_intro_sec}

We begin by defining the gothic operators. The most fundamental of these are the operators $\fa$ and $\fb$, which are defined below.
\begin{definition}The operators $\fa$ and $\fb$ are
$$\fa:=\pd_\theta^2,$$
$$\fb:=\cot\theta\pd_\theta.$$
\end{definition}

In many cases, we will use these operators interchangeably, so we define the following shorthand notation.
\begin{definition}The operator $\fc$ is
$$\fc:=\fa\text{ or }\fb.$$
More generally, the family of operators $\fc^l$ is
$$\fc^l:=\fc_1...\fc_l\text{ where each }\fc_i\text{ is either }\fa\text{ or }\fb.$$
\end{definition}
So, for example, $\fc^2$ represents any of the operators $\fa^2$, $\fa\fb$, $\fb\fa$, or $\fb^2$.

Finally, we generalize the operator family $\fc^l$ slightly.
\begin{definition}The operator $\fd$ is
$$\fd:=\pd_\theta.$$
More generally, the family of operators $\fd^l$ is
$$
\fd^l := \left\{\begin{array}{cc}
\fc^{l/2} & l\in 2\mathbb{Z} \\
\pd_\theta \fc^{(l-1)/2} & l\not\in 2\mathbb{Z}.
\end{array}\right.
$$
\end{definition}
So, for example, $\fd^4$ repreesents any of the operators $\fa^2$, $\fa\fb$, $\fb\fa$, or $\fb^2$, while $\fd^5$ represents any of the operators $\pd_\theta\fa^2$, $\pd_\theta\fa\fb$, $\pd_\theta\fb\fa$, or $\pd_\theta\fb^2$.

\begin{remark}
One can think of the operator $\fd$ as a single spherical derivative and the operator family $\fd^l$ as $l$ spherical derivatives. This motivates the use of the letter d for this family. In contrast, the operator family $\fc^l$ is more like $2l$ spherical derivatives.
\end{remark}

To gain familiarity with the operator family $\fd^l$, we note the following example, which also serves as a caveat.
\begin{example}
Although
$$\fd\fd^l\subset\fd^{l+1},$$
in general
$$\fd^{l_1}\fd^{l_2}\not\subset\fd^{l_1+l_1}.$$
\end{example}
\begin{proof}
If $l$ is even, then 
$$\fd\fd^l=\pd_\theta\fc^{l/2}=\fd^{l+1}.$$
If $l$ is odd, then 
$$\fd\fd^l=\pd_\theta\pd_\theta\fc^{(l-1)/2}=\fa\fc^{(l-1)/2}\subset\fc^{(l+1)/2}=\fd^{l+1}.$$
A counterexample for the more general case of $\fd^{l_1}\fd^{l_2}$ exists when $l_1=2$ and $\l_2=1$. In this case,
$$\fb\pd_\theta\in \fd^2\fd,\text{ but }\fb\pd_\theta=\cot\theta\fa \not\in \fd^3.$$
\qed
\end{proof}

\subsubsection{The commutators $[\fb,\fa]$ and $[\fb,\fc^l]$}

The operators $\fa$ and $\fb$ satisfy a simple commutation relation.
\begin{lemma}
$$[\fb,\fa]=2\fb^2+2\fa.$$
\end{lemma}
\begin{proof}
Let $\alpha = \cot\theta$. Then $\fb=\alpha\pd_\theta$. By direct calculation,
\begin{align*}
\alpha' &= \pd_\theta\left(\frac{\cos\theta}{\sin\theta}\right) \\
&= -1-\frac{\cos^2\theta}{\sin^2\theta} \\
&= -\alpha^2-1
\end{align*}
and
\begin{align*}
\alpha'' &= (-\alpha^2-1)' \\
&= -2\alpha\alpha'.
\end{align*}
Now we compute the commutator.
\begin{align*}
[\fb,\fa](f) &= \fb\fa f -\fa\fb f \\
&= \alpha\pd_\theta^3 f - \pd_\theta^2(\alpha\pd_\theta f) \\
&= -2\alpha'\pd_\theta^2 f-\alpha''\pd_\theta f \\
&= 2\alpha^2\pd_\theta^2f + 2\pd_\theta^2 f+2\alpha\alpha'\pd_\theta f \\
&= 2\alpha\pd_\theta(\alpha\pd_\theta f)+2\pd_\theta^2 f \\
&= 2\fb^2f+2\fa f.
\end{align*}
This verifies the lemma. \qed
\end{proof}

We generalize the previous commutation relation to obtain something that will be more useful later.
\begin{lemma}\label{fb_fcl_commutation_lem}
The operators $\fb$ and $\fc^l$ satisfy the commutation relation
$$\fc^l\fb \approx \fb\fc^l+\fc^l,$$
where the $\approx$ sign indicates that the identity holds modulo constant factors for the terms on the right hand side.
\end{lemma}
\begin{proof}
The proof is a simple induction exercise. The base case is handled by the previous lemma. Now, we assume the statement of the lemma at the level $l$ and prove it for the level $l+1$.
\begin{align*}
\fc^{l+1}\fb &\approx \fc (\fb \fc^l + \fc^l) \\
&\approx \fc\fb\fc^l+\fc^{l+1} \\
&\approx (\fb\fc +\fc)\fc^l+\fc^{l+1} \\
&\approx \fb\fc^{l+1}+\fc^{l+1}.
\end{align*}
\qed
\end{proof}

\subsubsection{The graded algebra $\bigoplus_{n\in\mathbb{N}}\tau_{(n)}$}

Now, we investigate products of axisymmetric functions, and in particular we want to understand what types of products are regular on the axis. Certainly, a product of regular axisymmetric functions will also be regular, but the converse is not true. There are regular products of functions that are not necessarily regular themselves. The simplest example is the product $(\sin\theta)^2$. This product is regular on the axis, but the function $\sin\theta$ is not regular, because it behaves like $|\theta|$ in a neighborhood of the pole $\theta=0$.

We will understand regularity for products of functions by understanding how the gothic operators act on these products.
\begin{definition}
Let $\mathcal{F}$ be a fixed set of functions. Define $\tau_{(n)}=\tau_{(n)}(\mathcal{F})$ to be a family of terms of the form
$$\fd^{i_1}f_1...\fd^{i_k}f_k$$
where $f_1,...,f_k\in\mathcal{F}$ and $i_1+...+i_k=2n$. We say a term $\tau\in \tau_{(n)}$ \textit{has degree} $n$.
\end{definition}

The following example should make the meaning of the above definition slightly more clear.
\begin{example}
If $f,g,h\in \mathcal{F}$, then $\fc^i f\fc^j g\in\tau_{(i+j)}$ and $\fc^i f \pd_\theta \fc^j g \pd_\theta \fc^k h\in\tau_{(i+j+k+1)}$, however $\fc^if \fc^j g\pd_\theta \fc^k h$ is not in $\tau_{(n)}$ for any $n$.
\end{example}

%Some examples include $(\cos\theta)^2\in \tau_{(0)}$ and $\fc\cos\theta=-\cos\theta\in\tau_{(1)}$, and more interestingly

The following example should illustrate why these families of terms are useful.
\begin{example}
Let $\mathcal{F}=\{\cos\theta\}$. Then any term $\tau\in\tau_{(n)}(\mathcal{F})$ is regular on the axis. In particular, $(\sin\theta)^2=(\pd_\theta\cos\theta)^2\in\tau_{(1)}$. In contrast, the term $\cos\theta \pd_\theta\cos\theta = -\cos\theta\sin\theta$ is not in $\tau_{(n)}$ for any $n$ and is also not regular at the poles $\theta=0$ and $\theta=\pi$.
\end{example}

The main purpose for this classification of terms is the following fact.
\begin{lemma}\label{cl_raise_degree_lem}
If $\tau$ is a term of degree $n$, then $\fc\tau$ can be expressed as a sum of terms of degree $n+1$.
\end{lemma}
\begin{proof}
By definition, a term of degree $n$ is of the form
$$\fd^{i_1}f_1...\fd^{i_k}f_k,$$
where $i_1+...+i_k=2n$. We now compute
\begin{align*}
\fa(\fd^{i_1}f_1...\fd^{i_k}f_k) &= \pd_\theta^2(\fd^{i_1}f_1...\fd^{i_k}f_k) \\
&= \sum_{\substack{1\le j_1\le k \\ 1\le j_2\le k}} \fd^{i_1}f_1...\fd\fd^{i_{j_1}}f_{j_1}...\fd\fd^{i_{j_2}}f_{j_2}...\fd^{i_k}f_k \\
&= \sum_{\substack{1\le j_1\le k \\ 1\le j_2\le k}} \fd^{i_1}f_1...\fd^{i_{j_1}+1}f_{j_1}...\fd\fd^{i_{j_2}}f_{j_2}...\fd^{i_k}f_k \\
&= \sum_{\substack{1\le j_1\le k \\ 1\le j_2\le k}} \fd^{i_1}f_1...\fd^{i_{j_1}+1}f_{j_1}...\fd^{i_{j_2}+1}f_{j_2}...\fd^{i_k}f_k .
\end{align*}
The last two steps were kept separate to emphasize that the procedure works even for the terms where $j_1=j_2$. It should be clear that the final sum is a sum of terms of degree $n+1$.

We also compute
\begin{align*}
\fb(\fd^{i_1}f_1...\fd^{i_k}f_k) &= \cot\theta \pd_\theta(\fd^{i_1}f_1...\fd^{i_k}f_k) \\
&= \cot\theta \sum_{1\le j_1\le k} \fd^{i_1}f_1...\pd_\theta\fd^{i_{j_1}}f_{j_1}...\fd^{i_k}f_k \\
&= \cot\theta \sum_{1\le j_1\le k} \fd^{i_1}f_1...\fd^{i_{j_1}+1}f_{j_1}...\fd^{i_k}f_k.
\end{align*}
Now, observe that since $i_1+...+i_k=2n$, then $i_1+...+(i_{j_1}+1)+...+i_k=2n+1$. It follows that either $i_{j_1}+1$ is odd or there is some $j_2$ for which $i_{j_2}$ is odd. If $i_{j_1}+1$ is odd, then 
$$\cot\theta\fd^{i_{j_1}+1}f_{j_1}=\cot\theta\pd_\theta\fc^{i_{j_1}/2}f_{j_1}=\fb\fc^{i_{j_1}/2}f_{j_1}=\fc^{i_{j_1}/2+1}=\fd^{i_{j_1}+2}f_{j_1}.$$
Otherwise, for some $j_2$, $i_{j_2}$ is odd, so
$$\cot\theta\fd^{i_{j_2}}f_{j_2}=\cot\theta\pd_\theta\fc^{(i_{j_2}-1)/2}f_{j_2}=\fb\fc^{(i_{j_2}-1)/2}f_{j_2}=\fc^{(i_{j_2}-1)/2+1}f_{j_2}=\fd^{i_{j_2}+1}f_{j_2}.$$
So all terms can either be expressed as
$$\fd^{i_1}f_1...\fd^{i_{j_1}+2}f_{j_1}...\fd^{i_k}j_k$$
or
$$\fd^{i_1}f_1...\fd^{i_{j_1}+1}f_{j_1}...\fd^{i_{j_2}+1}f_{j_2}...\fd^{i_k}j_k.$$
Both of these are terms of degree $n+1$. \qed
\end{proof}

\subsection{An important embedding theorem}\label{cl_embedding_theorem_sec}
We now prove the following important embedding theorem.
\begin{theorem}\label{cl_embedding_thm}
$$||\fc^l f||_{L^2(S^2)}\lesssim \sum_{i\le l}||\sla\triangle^i f||_{L^2(S^2)},$$
$$||\pd_\theta \fc^l f||_{L^2(S^2)}\lesssim \sum_{i\le l}||\sla\nabla\sla\triangle^i f||_{L^2(S^2)}.$$
\end{theorem}

\begin{proof}
We prove both estimates by induction on $l$. The base case ($l=0$) is trivial. We assume the estimates hold true at the level $l$, and now we must prove them at the level $l+1$.

First, observe that
$$||\fa^{l+1}f||_{L^2(S^2)}=||\pd_\theta^{2(l+1)}f||_{L^2(S^2)}\lesssim ||\sla\nabla^{2(l+1)}f||_{L^2(S^2)}\lesssim \sum_{i\le l+1} ||\sla\triangle^i \phi||_{L^2(S^2)}$$
and
$$||\pd_\theta \fa^{l+1}f||_{L^2(S^2)}=||\pd_\theta\pd_\theta^{2(l+1)}f||_{L^2(S^2)}\lesssim ||\sla\nabla^{2(l+1)+1}f||_{L^2(S^2)}\lesssim \sum_{i\le l+1} ||\sla\nabla \sla\triangle^i \phi||_{L^2(S^2)}.$$
(In both lines, the second step follows from the identity $\sla\nabla_{e_1}^kf=\pd_\theta^kf$ in the orthonormal frame $e_1=\pd_\theta$ and $e_2=\csc\theta\pd_\phi$, and the third step is a standard ellipticity result.)

We now assert the following lemma.
\begin{lemma}\label{ccc_lem}
$$\fc^k \approx \fa^k +\fc^{\le k-1}\sla\triangle + \fc^{\le k-1},$$
where the $\approx$ sign indicates that the identity holds modulo constant factors for the terms on the right hand side and the expression $\fc^{\le k-1}$ represents operators obtained by adding together constant multiples of operators of the form $\fc^i$ where $i\le k-1$.
\end{lemma}
From Lemma \ref{ccc_lem}, it is straightforward to prove the estimate at the level $l+1$. Set $k=l+1$. Then any term of the form $\fc^{l+1}f$ can be expressed as a sum of terms of the form $\fa^{l+1}f$ (which we just observed to be bounded by the appropriate norm) or $\fc^{\le l}f$ (which is bounded according to the inductive hypothesis) or $\fc^{\le l}\sla\triangle f$ (which is also bounded according to the inductive hypothesis applied to the function $g=\sla\triangle f$).

It remains to prove Lemma \ref{ccc_lem}.
\begin{proof} (of Lemma \ref{ccc_lem}) To prove Lemma \ref{ccc_lem}, we first assert the following lemma.
\begin{lemma}\label{baa_lem}
For all $k\ge 1$,
$$(2k-1)\fb \fa^{k-1} = -\fa^k +\fc^{\le k-1}\sla\triangle + \fc^{\le k-1}.$$
\end{lemma}

To see how Lemma \ref{ccc_lem} follows from Lemma \ref{baa_lem}, we must represent an arbitrary operator of the form $\fc^k=\fc_1...\fc_k$ in the form given by the right hand side of the identity in Lemma \ref{ccc_lem}. First, if $\fc_1=...=\fc_k=\fa$, then trivially $\fc_1...\fc_k=\fa^k$ is in the form given in Lemma \ref{ccc_lem}. Alternatively, there is some $j$ for which $\fc_1...\fc_k=\fc_1...\fc_j\fb\fa^{k-j-1}$. According to Lemma \ref{baa_lem}, we can rewrite
\begin{align*}
\fc_1...\fc_j\fb\fa^{k-j-1} &\approx \fc_1...\fc_j(\fa^{k-j}+\fc^{\le k-j-1}\sla\triangle+\fc^{\le k-j-1}) \\
&\approx \fc_1...\fc_j\fa^{k-j} +\fc^{\le k-1}\sla\triangle +\fc^{\le k-1}.
\end{align*}
The latter two terms are in the form given by Lemma \ref{ccc_lem} and the first term $\fc_1...\fc_j\fa^{k-j}$ has \textit{strictly more} $\fa$s at the right end. Thus, by repeating this procedure, we will end up with the term $\fa^k$, which is of the form given in Lemma \ref{ccc_lem}.

It remains to prove Lemma \ref{baa_lem}.

\begin{proof} (of Lemma \ref{baa_lem})
We prove Lemma \ref{baa_lem} by induction on $k$. Suppose $k=1$. By a direct calculation,
$$\fb=-\fa+\sla\triangle.$$
We now assume that the identity of Lemma \ref{baa_lem} is true at the level $k$ and prove the identity for the level $k+1$.
\begin{align*}
(2k+1)\fb\fa^k &= (2k-1)\fb\fa^k +2\fb\fa^k \\
&= (2k-1)\fa\fb\fa^{k-1} +(2k-1)[\fb,\fa]\fa^{k-1} +2\fb\fa^k\\
&= (2k-1)\fa\fb\fa^{k-1} +(2k-1)(2\fb^2+2\fa)\fa^{k-1} +2\fb\fa^k \\
&= (\fa+2\fb)((2k-1)\fb\fa^{k-1}) + 2\fb\fa^k + \fc^{\le k} \\
&= (\fa+2\fb)(-\fa^k+\fc^{\le k-1}\sla\triangle+\fc^{\le k-1}) +2\fb\fa^k + \fc^{\le k} \\
&= -\fa^{k+1} +\fc^{\le k}\sla\triangle + \fc^{\le k}.
\end{align*}
This completes the inductive step and thus proves Lemma \ref{baa_lem}. \qed
\end{proof} % baa_lem

Since Lemma \ref{ccc_lem} was reduced to Lemma \ref{baa_lem}, this concludes the proof of Lemma \ref{ccc_lem}. \qed
\end{proof} % ccc_lem

Since Theorem \ref{cl_embedding_thm} was reduced to Lemma \ref{ccc_lem}, this concludes the proof of Theorem \ref{cl_embedding_thm}. \qed
\end{proof} % thm

\subsection{Additional regularity lemmas}\label{additional_regularity_lemmas_sec}

The $L^\infty$-type estimates used in the paper derive from the following lemma.
\begin{lemma}\label{spherical_infty_lem}
If $l$ is even, then
$$||\fd^l u||_{L^\infty(S^2)}\lesssim \sum_{i\le 2}||\fd^{l+i}u||_{L^2(S^2)}.$$
If $l$ is odd, then
$$||\fd^l u||_{L^\infty(S^2)}\lesssim \sum_{i\le 3}||\fd^{l+i}u||_{L^2(S^2)}.$$
\end{lemma}
\begin{proof}
We prove this theorem by applying the estimate
$$||u||_{L^\infty(S^2)}\lesssim ||u||_{L^2(S^2)}+||\sla\triangle u||_{L^2(S^2)}.$$
If $l$ is even, then
\begin{align*}
||\fd^lu||_{L^\infty(S^2)} &=||\fc^{l/2}u||_{L^\infty(S^2)} \\
&\lesssim ||\fc^{l/2}u||_{L^2(S^2)}+||\sla\triangle\fc^{l/2}u||_{L^2(S^2)} \\
&\lesssim ||\fc^{l/2}u||_{L^2(S^2)}+||\fc\fc^{l/2}u||_{L^2(S^2)} \\
&\lesssim ||\fd^lu||_{L^2(S^2)}+||\fd^{l+2}u||_{L^2(S^2)} \\
&\lesssim \sum_{i\le 2}||\fd^{l+i}u||_{L^2(S^2)}.
\end{align*}
If $l$ is odd, then
\begin{align*}
||\fd^lu||_{L^\infty(S^2)} &=||\pd_\theta \fc^{(l-1)/2}u||_{L^\infty(S^2)} \\
&\lesssim ||\cot\theta\pd_\theta\fc^{(l-1)/2}u||_{L^\infty(S^2)} + ||\sin\theta \pd_\theta\fc^{(l-1)/2}u||_{L^\infty(S^2)} \\
&\lesssim ||\fb\fc^{(l-1)/2}u||_{L^\infty(S^2)}+ ||\sin\theta \pd_\theta\fc^{(l-1)/2}u||_{L^2(S^2)} +||\sla\triangle(\sin\theta \pd_\theta\fc^{(l-1)/2}u)||_{L^2(S^2)} \\
&\lesssim ||\fd^{l+1}u||_{L^\infty(S^2)}+||\fd^lu||_{L^2(S^2)}+||\fc(\fd\cos\theta \fd^lu)||_{L^2(S^2)} \\
&\lesssim ||\fd^{l+1}u||_{L^2(S^2)}+||\fd^{l+3}u||_{L^2(S^2)}+||\fd^l u||_{L^2(S^2)}+\sum_{i\le 2}||\fd^{1+2-i}\cos\theta \fd^{l+i}u||_{L^2(S^2)} \\
&\lesssim \sum_{i\le 3}||\fd^{l+i}u||_{L^2(S^2)}.
\end{align*}
\qed
\end{proof}

Now we prove that $Q$ behaves similarly to the spherical Laplacian.
\begin{lemma}\label{laplacian_le_commutators_lem}
$$\int_{S^2(r)}((q^2\sla\triangle)^ku)^2 \lesssim \int_{S^2(r)}(\Gamma^{\le k}u)^2$$
\end{lemma}
\begin{proof}
Our strategy is to prove by induction (on $k$) that for all $s\ge 0$,
$$\int(\sla\triangle^{\le k}\Gamma^s u)^2 \lesssim \int (\Gamma^{\le 2k+s}u)^2.$$

Let us assume the above and prove the corresponding estimate for $k+1$. This is actually rather straightforward. Suppose for the moment we have the following additional estimate.
\begin{equation}\label{laplacian_le_commutators_lem_trivial_eqn}
\int \left(\sla\triangle^k(a^2\sin^2\theta\pd_t^2\Gamma^s u)\right)^2\lesssim \int (\sla\triangle^{\le k}\Gamma^{s+2} u)^2.
\end{equation}
Then we can perform the inductive step.
\begin{align*}
\int(\sla\triangle^{k+1}\Gamma^su)^2&=\int\left(\sla\triangle^k(Q-a^2\sin^2\theta\pd_t^2)\Gamma^su\right)^2 \\
&\lesssim \int (\sla\triangle^k Q\Gamma^su)^2+\int\left(\sla\triangle^k(a^2\sin^2\theta\pd_t^2\Gamma^s u)\right)^2 \\
&\lesssim \int (\sla\triangle^{\le k} \Gamma^{s+2}u)^2 \\
&\lesssim \int (\Gamma^{\le 2k+s+2}u)^2 = \int (\Gamma^{\le 2(k+1)+s}u)^2.
\end{align*}
So we only need to we verify the estimate (\ref{laplacian_le_commutators_lem_trivial_eqn}).
% By repeatedly applying Lemma \ref{cl_raise_degree_lem}, we obtain
%\begin{align*}
%\int \left(\sla\triangle^k(a^2\sin^2\theta\pd_t^2\Gamma^s u)\right)^2 
%&\lesssim \int \left(c^k(a^2\sin^2\theta\pd_t^2\Gamma^s u)\right)^2 \\
%&\lesssim \sum_{l\le 2k} \int (\fd^l(a^2\sin^2\theta)(\fd^{2k-l} \pd_t^2\Gamma^s u))^2
%\end{align*}
By a Sobolev estimate and an elliptic estimate,
\begin{align*}
\int \left(\sla\triangle^k(a^2\sin^2\theta\pd_t^2\Gamma^s u)\right)^2 &\lesssim ||\sin^2\theta||^2_{H^{2k}(S^2)}||\Gamma^{s+2}u||^2_{H^{2k}(S^2)} \\
&\lesssim \int (\sla\triangle^{\le k}\sin^2\theta)^2 \int (\sla\triangle^{\le k}\Gamma^{s+2}u)^2.
\end{align*}
The regularity of $\sin^2\theta$ can be verified by computing
$$\sla\triangle \sin^2\theta=4\cos^2\theta-2\sin^2\theta,$$
$$\sla\triangle \cos^2\theta=-4\cos^2\theta+2\sin^2\theta.$$
This concludes the proof of lemma \ref{laplacian_le_commutators_lem}. \qed
\end{proof}

Finally, we prove a lemma that handles the cases where the commutators act on $\sin\theta$ terms. (For example, see Lemma \ref{example_Nphi_term_lem}.)
\begin{lemma}\label{gain_sin_lem}
$$||\fd^i(\sin^2\theta)\fd^{l-i}f||_{L^2(S^2)}\lesssim \sum_{j\le i}||\sin^2\theta \fd^{l-j}f||_{L^2(S^2)}$$
\end{lemma}
\begin{proof}
First, we prove that for a general function $f(\theta)$,
\begin{equation}\label{gain_sin_lem_initial_eqn}
\int_0^\pi f^2\sin\theta d\theta \lesssim \int_0^\pi (\pd_\theta f)^2\sin^3\theta d\theta +\int_0^\pi f^2\sin^3\theta d\theta.
\end{equation}
Indeed,
\begin{align*}
\int_0^\pi f^2\sin\theta d\theta &= \frac12\int_0^\pi f^2\pd_\theta(\sin^2\theta)d\theta +\int_0^\pi f^2(1-\cos\theta)\sin\theta d\theta \\
&= - \int_0^\pi f\pd_\theta f \sin^2\theta d\theta +\int_0^\pi f^2(1-\cos\theta)\sin\theta d\theta \\
&\lesssim \epsilon^{-1}\int_0^\pi (\pd_\theta f)^2\sin^3\theta d\theta +\epsilon\int_0^\pi f^2\sin\theta d\theta +\int_0^\pi f^2\sin^3\theta d\theta.
\end{align*}
Taking $\epsilon$ sufficiently small and subtracting the second term on the right hand side proves the estimate (\ref{gain_sin_lem_initial_eqn}).

Now, we repeatedly apply the estimate (\ref{gain_sin_lem_initial_eqn}) to prove the lemma for the case $i=2$.
\begin{align*}
\int_0^\pi(\fd^{l-2}f)^2\sin\theta d\theta &\lesssim \int_0^\pi(\pd_\theta \fd^{l-2}f)^2\sin^3\theta d\theta+\int_0^\pi (\fd^{l-2}f)^2\sin^3\theta d\theta \\
&\lesssim \int_0^\pi(\pd_\theta^2 \fd^{l-2}f)^2\sin^5\theta d\theta + \int_0^\pi(\pd_\theta \fd^{l-2}f)^2\sin^5\theta d\theta+\int_0^\pi (\fd^{l-2}f)^2\sin^5\theta d\theta \\
&\lesssim \int_0^\pi (\sin^2\theta\fd^l f)^2\sin\theta d\theta +\int_0^\pi (\sin^2\theta\fd^{l-1}f)^2\sin\theta d\theta + \int_0^\pi (\sin^2\theta\fd^{l-2}f)^2\sin\theta d\theta.
\end{align*}
In the first step, estimate (\ref{gain_sin_lem_initial_eqn}) was applied once, and in the second step, estimate (\ref{gain_sin_lem_initial_eqn}) was applied twice. The final step is simply a rearrangement of terms. It should be clear how this procedure generalizes to any larger $i>2$.
\qed
\end{proof}

\bibliographystyle{alpha}
\bibliography{bib}

\end{document}